\documentclass[a4paper,12pt]{amsart}

\usepackage[mathcal]{euscript}   
\usepackage{mathrsfs}

\usepackage[shortlabels]{enumitem}
\usepackage[matrix,arrow]{xy}
\usepackage{comment}
\usepackage{amssymb,enumitem,verbatim,stmaryrd,xcolor,microtype,graphicx,needspace}
\usepackage[T1]{fontenc}
\usepackage[utf8]{inputenc}
\usepackage[english]{babel}\def\hyph{-\penalty0\hskip0pt\relax} 
\usepackage[top=3.3cm,bottom=3.3cm,left=3.25cm,right=3.25cm]{geometry}
\usepackage[bookmarksdepth=2,linktoc=page,colorlinks,linkcolor={red!80!black},citecolor={red!80!black},urlcolor={blue!80!black}]{hyperref}

\usepackage{tikz}\usetikzlibrary{matrix,arrows,decorations.markings}
\usepackage{tikz-cd}
\usetikzlibrary{automata,positioning,decorations.pathreplacing}
\newcommand\cprime\textquotesingle 

\usepackage[figuresleft]{rotating}
\usepackage{changepage}

\usepackage{resizegather}
\usepackage{multicol}
\tikzset{->-/.style={decoration={markings,mark=at position #1 with {\color{black}\arrow{>}}},postaction={decorate,very thick}}}
\tikzstyle{vertex}=[circle, draw, inner sep=0pt, minimum size=6pt]
\newcommand{\vertex}{\node[vertex]}

\pgfrealjobname{elliptic1}
\usepackage[mathcal]{euscript}                               
\usepackage{mathptmx}                                        
\usepackage[colorinlistoftodos,bordercolor=orange,backgroundcolor=orange!20,linecolor=orange,textsize=footnotesize]{todonotes} \makeatletter \providecommand \@dotsep{5} \def\listtodoname{List of Todos} \def\listoftodos{\@starttoc{tdo}\listtodoname} \makeatother 

\allowdisplaybreaks 

\usepackage{etoolbox}
\makeatletter
\patchcmd{\@startsection}{\@afterindenttrue}{\@afterindentfalse}{}{}             
\patchcmd{\part}{\bfseries}{\bfseries\LARGE}{}{}
\patchcmd{\section}{\scshape}{\bfseries}{}{}\renewcommand{\@secnumfont}{\bfseries} 
\patchcmd{\@settitle}{\uppercasenonmath\@title}{\large}{}{}
\patchcmd{\@setauthors}{\MakeUppercase}{}{}{}
\addto{\captionsenglish}{} 
\addto{\captionsenglish}{} 
\addto{\captionsenglish}{} 
\makeatother

\usepackage{fancyhdr}

\theoremstyle{plain}
\newtheorem{thm}{Theorem}[section]
\newtheorem{cor}[thm]{Corollary}
\newtheorem{lemma}[thm]{Lemma}
\newtheorem{prop}[thm]{Proposition}
\newtheorem{thmA}{Theorem}  
 
\newtheorem{resultA}[thmA]{Result} 

\newtheorem*{thm*}{Theorem}
\newtheorem*{lem*}{Lemma}

\theoremstyle{definition}
\newtheorem{df}[thm]{Definition}
\newtheorem{rem}[thm]{Remark}

\newtheorem{ex}[thm]{Example}
\newtheorem*{df*}{Definition}
\newtheorem*{ex*}{Example}
\newtheorem*{rem*}{Remark}

\DeclareRobustCommand{\gobblefour}[5]{}    

\DeclareFontFamily{OT1}{pzc}{}                                
\DeclareFontShape{OT1}{pzc}{m}{it}{<-> s * [1.10] pzcmi7t}{}
\DeclareMathAlphabet{\mathpzc}{OT1}{pzc}{m}{it}
\DeclareSymbolFont{sfoperators}{OT1}{bch}{m}{n} \DeclareSymbolFontAlphabet{\mathsf}{sfoperators} \makeatletter\def\operator@font{\mathgroup\symsfoperators}\makeatother 
\DeclareSymbolFont{cmletters}{OML}{cmm}{m}{it}
\DeclareSymbolFont{cmsymbols}{OMS}{cmsy}{m}{n}
\DeclareSymbolFont{cmlargesymbols}{OMX}{cmex}{m}{n}
\DeclareMathSymbol{\myjmath}{\mathord}{cmletters}{"7C}     \let\jmath\myjmath 
\DeclareMathSymbol{\myamalg}{\mathbin}{cmsymbols}{"71}     
\DeclareMathSymbol{\mycoprod}{\mathop}{cmlargesymbols}{"60}\let\coprod\mycoprod
\DeclareMathSymbol{\myalpha}{\mathord}{cmletters}{"0B}     \let\alpha\myalpha 
\DeclareMathSymbol{\mybeta}{\mathord}{cmletters}{"0C}      \let\beta\mybeta
\DeclareMathSymbol{\mygamma}{\mathord}{cmletters}{"0D}     \let\gamma\mygamma
\DeclareMathSymbol{\mydelta}{\mathord}{cmletters}{"0E}     \let\delta\mydelta
\DeclareMathSymbol{\myepsilon}{\mathord}{cmletters}{"0F}   \let\epsilon\myepsilon
\DeclareMathSymbol{\myzeta}{\mathord}{cmletters}{"10}      \let\zeta\myzeta
\DeclareMathSymbol{\myeta}{\mathord}{cmletters}{"11}       \let\eta\myeta
\DeclareMathSymbol{\mytheta}{\mathord}{cmletters}{"12}     \let\theta\mytheta
\DeclareMathSymbol{\myiota}{\mathord}{cmletters}{"13}      \let\iota\myiota
\DeclareMathSymbol{\mykappa}{\mathord}{cmletters}{"14}     \let\kappa\mykappa
\DeclareMathSymbol{\mylambda}{\mathord}{cmletters}{"15}    \let\lambda\mylambda
\DeclareMathSymbol{\mymu}{\mathord}{cmletters}{"16}        \let\mu\mymu
\DeclareMathSymbol{\mynu}{\mathord}{cmletters}{"17}        \let\nu\mynu
\DeclareMathSymbol{\myxi}{\mathord}{cmletters}{"18}        \let\xi\myxi
\DeclareMathSymbol{\mypi}{\mathord}{cmletters}{"19}        \let\pi\mypi
\DeclareMathSymbol{\myrho}{\mathord}{cmletters}{"1A}       \let\rho\myrho
\DeclareMathSymbol{\mysigma}{\mathord}{cmletters}{"1B}     \let\sigma\mysigma
\DeclareMathSymbol{\mytau}{\mathord}{cmletters}{"1C}       \let\tau\mytau
\DeclareMathSymbol{\myupsilon}{\mathord}{cmletters}{"1D}   \let\upsilon\myupsilon
\DeclareMathSymbol{\myphi}{\mathord}{cmletters}{"1E}       \let\phi\myphi
\DeclareMathSymbol{\mychi}{\mathord}{cmletters}{"1F}       \let\chi\mychi
\DeclareMathSymbol{\mypsi}{\mathord}{cmletters}{"20}       \let\psi\mypsi
\DeclareMathSymbol{\myomega}{\mathord}{cmletters}{"21}     \let\omega\myomega
\DeclareMathSymbol{\myvarepsilon}{\mathord}{cmletters}{"22}\let\varepsilon\myvarepsilon
\DeclareMathSymbol{\myvartheta}{\mathord}{cmletters}{"23}  \let\vartheta\myvartheta
\DeclareMathSymbol{\myvarpi}{\mathord}{cmletters}{"24}     \let\varpi\myvarpi
\DeclareMathSymbol{\myvarrho}{\mathord}{cmletters}{"25}    \let\varrho\myvarrho
\DeclareMathSymbol{\myvarsigma}{\mathord}{cmletters}{"26}  \let\varsigma\myvarsigma
\DeclareMathSymbol{\myvarphi}{\mathord}{cmletters}{"27}    \let\varphi\myvarphi

\DeclareMathOperator{\Hom}{Hom}
\DeclareMathOperator{\Aut}{Aut}
\DeclareMathOperator{\GL}{GL}
\DeclareMathOperator{\PGL}{PGL}
\DeclareMathOperator{\Bunvar}{Bun}

\DeclareMathOperator{\rk}{rk}
\DeclareMathOperator{\Coh}{{Coh}}
\DeclareMathOperator{\Bun}{{Bun}}
\DeclareMathOperator{\PBun}{{\mathbb{P}Bun}}
\DeclareMathOperator{\Pic}{{Pic}}
\DeclareMathOperator{\Ext}{{Ext}}
\DeclareMathOperator{\Fr}{{Fr}}
\DeclareMathOperator{\Norm}{{Norm}}
\newcommand{\period}{\Omega}

\newcommand\A{{\mathbb A}}

\newcommand\C{{\mathbb C}}

\newcommand\FF{{\mathbb F}}
\newcommand\F{{\mathcal F}}
\newcommand\G{{\mathcal G}}

\newcommand\N{{\mathbb N}}

\renewcommand\P{{\mathbb P}}
\newcommand\Q{{\mathbb Q}}

\newcommand\Z{{\mathbb Z}}

\newcommand\cA{{\mathcal A}}

\newcommand\cE{{\mathcal E}}
\newcommand\cF{{\mathcal F}}
\newcommand\cG{{\mathcal G}}

\newcommand\cH{{\mathcal H}}

\newcommand\cK{{\mathcal K}}
\newcommand\cL{{\mathcal L}}
\newcommand\cM{{\mathcal M}}
\newcommand\cN{{\mathcal N}}
\newcommand\cO{{\mathcal O}}

\newcommand{\scrG}{\mathscr{G}}
\newcommand\sC{{\mathsf C}}
\newcommand\sH{{\mathsf H}}
\newcommand\sU{{\mathsf U}}

\newcommand\bv{{\mathbf v}}
\newcommand\bw{{\mathbf w}}
\newcommand{\E}{\mathcal E}
\newcommand{\Line}{\mathcal L}
\newcommand{\Fq}{\mathbb{F}_q}

\renewcommand\geq{\geqslant}
\renewcommand\leq{\leqslant}
\newcommand{\gen}[1]{\langle #1 \rangle}

\renewcommand\emptyset\varnothing

\newcommand\ol{\textbf{Oliver:} }

\title{Automorphic forms for $\PGL(3)$ over elliptic function fields\\[10pt] \normalsize Part 1: Graphs of Hecke operators}

\author{Roberto Alvarenga}
\address{\rm Roberto Alvarenga, Instituto de Ci\^encias Matem\'aticas e de Computa\c{c}\~ao - USP, S\~ao Carlos, Brazil}
\email{alvarenga@icmc.usp.br}

\author{Oliver Lorscheid}
\address{\rm Oliver Lorscheid, University of Groningen, the Netherlands, and IMPA, Rio de Janeiro, Brazil}
\email{oliver@impa.br}

\author{Valdir Pereira J\'unior}
\address{\rm Valdir Pereira J\'unior, Instituto Nacional de Matem\'atica Pura e Aplicada, Rio de Janeiro, Brazil}
\email{valdirjosepereirajunior@gmail.com}

\allowdisplaybreaks

\begin{document}

\begin{abstract} 
 This is a first part of a series of papers in which we develop explicit computational methods for automorphic forms for $\GL_3$ and $\PGL_3$ over elliptic function fields. In this first part, we determine explicit formulas for the action of the Hecke operators on automorphic forms on $\GL_2$ and $\GL_3$ in terms of their graphs. Our primary result consists in a complete description of the graphs of degree $1$ Hecke operators for $\GL_3$. As complementary results, we describe the `even component' of the graphs of degree $2$ Hecke operators for $\GL_2$ and the `neighborhood of the identity' of the graphs of degree $2$ Hecke operators for $\GL_3$. In addition, we establish two dualities for Hecke operators for $\GL_n$ and $\PGL_n$, which hold for all $n$, all degrees and all function fields.

\end{abstract}

\maketitle

\begin{center}
 \textit{Dedicated to Gunther Cornelissen on the occasion of his 50th birthday.}
\end{center}



\begin{small} \tableofcontents \end{small}

\section*{Introduction}
\label{introduction}

This is the first part of a series of papers, in which we aim for an explicit understanding of the vanishing of periods of automorphic forms for $\PGL_3$ over an elliptic function field. Before we explain the results of the present paper, we outline the overall scope of this series of papers and put it into context with the literature.


\subsection*{Periodical automorphic forms for \texorpdfstring{$\PGL_2$}{PGL(2)}}

Period integrals for $\PGL_2$ are well-under\hyph stood and, in essence, vanishing periods yield (conjecturally) tempered representations. Let us survey this in more detail.

We fix a global field $F$ as our base field and let $\A$ be its adele ring. Let $H$ be an algebraic subgroup of $\PGL_2$ and $f:\PGL_2(F)\backslash\PGL_2(\A)\to\C$ an automorphic form. We call $f$ \emph{$H$-periodical} if its \emph{$H$-period}
\[
 \period_H(f)(g) \ = \ \int\limits_{H(F)\backslash H(\A)} f(hg) \,dh
\]
vanishes for all $g\in\PGL_2(\A)$. We say that an automorphic representation $\pi$ (which we always assume to be irreducible) is $H$-periodical if all of its elements are $H$-periodical. Note that this condition is invariant under conjugation of $H$.

Since $\PGL_2$ has dimension $3$, the only interesting periods appear for $1$ or $2$-dimen\hyph sional $H$, which are of the following types (up to conjugation): the Borel subgroup $B$ and its unipotent radical $N$, the diagonal torus $T$, for every separable quadratic field extension $E$ of $F$ a non-split torus $T_E$ and, in characteristic $2$, the Weil-restrictions of $\mathbb{G}_{m,E}$ for non-separable quadratic extensions. Note that $B$ is of secondary interest as the semidirect product of $N$ with $T$. The following interpretations of period integrals are known.

\subsubsection*{\texorpdfstring{$N$}{N}-periods} $N$-periodical automorphic forms are, by definition, cusp forms. As a by-product of Drinfeld's proof of the Langlands conjectures for $\GL_2$ over function fields (cf.\ \cite{Drinfeld83}, \cite{drinfeld-peterson}), $N$-periodical, or cuspidal, representations are tempered in the function field case. The Ramanujan-Petersson conjecture claims the analogous property for number fields.

\subsubsection*{\texorpdfstring{$T_E$}{T(E)}-periods} $T_E$-periods are linked to special values of $L$-functions, and $T_E$-periodical automorphic forms are called \emph{$E$-toroidal}. In the case of a cuspidal representation $\pi$ over a number field and $f$ in $\pi$, Waldspurger shows in \cite{Waldspurger85} that
\[
 \bigg| \ \int\limits_{T_E(F)\backslash T_E(\A)} f(tg)\, dt \ \bigg|^2 \quad = \quad c_{f,T_E}(g)\cdot L(\pi,\tfrac12) \cdot L(\pi\otimes\chi_E,\tfrac12)
\]
for some factor $c_{f,T_E}(g)$ that is nonzero as a function in $g$, where $L(\pi,s)$ is the $L$-series of $\pi$ and $\chi_E$ is the quadratic Hecke character associated with $E/F$. The analogous formula for function fields is established by Chuang and Wei in \cite{Chuang-Wei19}; also cf.\ \cite{Lysenko08,Lysenko20}.
 
In the case of an Eisenstein series $E(g,\chi)$, where $g\in\PGL_2(\A)$ and $\chi$ is a Hecke character, Zagier shows in \cite{zagier} that 
\[
 \int\limits_{T_E(F)\backslash T_E(\A)} E(tg,\chi)\, dt \ = \ c_{\chi,T_E}(g)\cdot L(\chi,\tfrac12)\cdot L(\chi\chi_E,\tfrac12)
\]
for some factor $c_{\chi,T_E}(g)$ that is nonzero as a function in $g$, where $L(\chi,s)$ and $L(\chi\chi_E,s)$ are Hecke $L$-series. In the function field case, the Hasse-Weil theorem (cf.\ \cite{Weil49}) asserts that the real part of $\chi$ is contained in $[-2\gamma\sqrt{q},\,2\gamma\sqrt{q}]$ if $E(g,\chi)$ is $T_E$-periodical, where $\gamma$ is the genus of $F$. This condition is equivalent with the temperedness of the automorphic representation generated by $E(g,\chi)$. The analogous claim for number fields is equivalent to the generalized Riemann hypothesis.
 
\subsubsection*{\texorpdfstring{$T$}{T}-periods} 
 Waldspurger's formula holds also for $T$-periods: the squared absolute value of the $T$-period of a cusp form $f$ in $\pi$ is a non-zero multiple of $L(\pi,\frac12)^2$. The $T$-period of an Eisenstein series diverges and therefore does not make sense if taken literally. Cornelissen and the second author show in \cite{Cornelissen-Lorscheid12} and \cite{Lorscheid13} that a suitable renormalization leads to a link with $L$-series, namely
 \[
  \int\limits_{T(F)\backslash T(\A)} E(tg,\chi)-\tfrac12\big[\period_N(E(\,\cdot\,,\chi))(tg)+\period_{N^t}(E(\,\cdot\,,\chi))(tg)\big]  \, dt \ = \ c_{\chi,T}(g)\cdot L(\chi,\tfrac12)^2
 \]
 for some factor $c_{\chi,T}(g)$ that is nonzero as a function in $g$, where $N^t$ is the group of lower triangular unipotent matrices. Thus we get the same relation as for $T_E$ between the temperedness for $T$-periodical Eisenstein series and the generalized Riemann hypothesis.


\subsection*{Generalizations to higher rank groups}

Periodical automorphic forms for higher rank groups are far less understood than for $\PGL_2$ and an active area of research. We do not attempt to give an overview of this rich field, but we would like to mention a few notable results:
\begin{itemize}
 \item Lafforgue generalizes Drinfeld's proof of the Langlands conjecture and the temperedness of cuspidal representations to $\PGL_n$ in the function field case (cf.\ \cite{Lafforgue97}, \cite{Lafforgue02}).
 \item $\GL_n\times\GL_{n'}$-periods of cusp forms for $\GL_{n+n'}$ (for $n'=n$ or $n+1$) are intensely studied (e.g., see \cite{Feigon-Martin-Whitehouse18}).
 \item Zagier's computation of $T_E$-periods of Eisenstein series is generalized by Wielonsky in \cite{Wielonsky85} to a formula for $\PGL_n$, which links a certain $1$-parameter family of residual Eisenstein series with an $L$-function of a degree $n$ extension $E$ of $F$. 
 \item Beineke and Bump (\cite{Beineke-Bump03}; also cf.\ \cite{Beineke-Bump02,Bump-Goldfeld84,Woodson94}) investigate certain renormalizations of Eisenstein series for $\PGL_3$ and the relation between $T$-periods and $L$-series. The situation is, however, much more complicated for higher rank than for $\PGL_2$.
 \item Periods of automorphic forms appear in the context of the relative trace formula. Namely, a comparision of relative trace formulas for two pairs of algebraic groups yields typically a relation between associated period integrals, and the non-vanishing of such periods is typically related to the lifting of automorphic representations. As first references, we guide the reader to \cite{Jacquet05} and \cite{Lapid06}.
\end{itemize}
In particular, the understanding of periodical automorphic forms for $\PGL_3$ is far from complete. It is largely open which periods are linked to tempered representations. In this paper and its sequels, we investigate this question in the case of elliptic function fields in terms of explicit parametrizations of automorphic forms.


\subsection*{Automorphic representations for \texorpdfstring{$\PGL_3$}{PGL(3)}}

The theory of automorphic forms for $\PGL_3$ is richer than that for $\PGL_2$, which features only cusp forms and a $1$-parameter family of Eisenstein series. The reason is that $\PGL_3$ has more parabolic subgroups that lead to induced representation: the whole group $\PGL_3$, two maximal parabolic subgroups of types $(2,1)$ and $(1,2)$ (which are connected by duality) and the Borel subgroup. The normalized induction of cuspidal representations from a Levi subgroup of a parabolic subgroup to $\PGL_3$ yields the following types of automorphic forms: cusp form (stemming from $\PGL_3$), a $1$-parameter family of `parabolic' Eisenstein series (stemming from either of the maximal parabolic subgroup) and a $2$-parameter family of `Borel' Eisenstein series; for more details, we refer to the sequel paper, as well as to \cite[App.\ G]{Laumon97} and \cite[sections 1.6, 1.7]{Pereira20} for general background.

In fact, Langlands shows in \cite[pp.\ 203--208]{Covallis1} that every automorphic representation is a subquotient of a representation that is induced from a cuspidal representation on a Levi subgroup; also cf.\ Franke (\cite{Franke98}), and Moeglin-Waldspurger (\cite{Moeglin-Waldspurger95}) for the function field case.


\subsection*{Automorphic forms in the function field case}

In the function field case, the domain of an automorphic form with fixed ramification is discrete modulo the right-action of a compact open subgroup $K$, which allows for explicit descriptions of automorphic forms. Even though a large part of our methods extends to ramified automorphic forms, we restrict ourselves to unramified automorphic forms for the sake of limiting the complexity. 

Unramified automorphic forms for function fields behave particularly nice: every unramified automorphic representation contains a unique unramified automorphic form, or spherical vector, up to scalar multiples. Therefore the unramified automorphic representations correspond to certain $1$-dimensional representations of the (commutative) spherical Hecke algebra $\cH_K$. In other words, unramified  automorphic representations correspond to eigenforms of $\cH_K$ and are therefore determined by their eigenvalues under the action of Hecke operators.


\subsection*{Geometric Langlands: Weil's theorem and Hecke operators}

The geometric Langlands program (e.g., cf.\ \cite{Frenkel20, Frenkel-Gaitsgory-Vilonen02,Gaitsgory15,Laumon87}) is rooted in the following observation that is usually attributed to Weil (e.g., cf.\ \cite[Lemma 3.1]{frenkel2}): the domain of unramified automorphic forms, modulo the right-action of the standard maximal compact subgroup $K$ of $\PGL_n(\A)$, is naturally identified with the set $\PBun_n X$ of isomorphism classes of $\P^{n-1}$-bundles\footnote{Note that an isomorphism class of a $\P^{n-1}$-bundle can be identified with an $\Pic X$-orbit of isomorphism classes of rank $n$ bundles on $X$; cf.\ \cite[Ex.\ 7.10]{biblia}.} on the smooth, projective and geometrically irreducible curve $X$ with function field $F$.

Let $|X|$ be the set of closed points of $X$. Satake shows in \cite{Satake63} that the spherical Hecke algebra $\cH_K$ of $\PGL_n(\A)$ is freely generated over $\C$ by the pairwise commuting Hecke operators $\Phi_{x,r}=\Phi^{(n)}_{x,r}$ where $x\in|X|$ and $r=1,\dotsc,n-1$. The Hecke operator $\Phi_{x,r}$ acts on functions $f:\PBun_nX\to \C$ in terms of the formula
\[
 (\Phi_{x,r}.f)(\overline\cE) \ = \ \sum_{\cE/\cF=\cK_x^{\oplus r}} f(\overline\cF)
\]
where $\cE$ is a rank $n$ bundle over $X$ representing the class $\overline\cE\in\PBun_nX$, where $\cK_x$ is the simple torsion sheaf with support at $x$ and where $\cF$ ranges over all subsheaves of $\cE$ with quotient $\cK_x^{\oplus r}$.


\subsection*{Graphs of Hecke operators}

We can subsume the information of how $\Phi_{x,r}$ acts on automorphic forms in an edge weighted graph $\overline\scrG_{x,r}^{(n)}$ with vertex set $\PBun_n X$ as follows. We define the \emph{weights $m_{x,r}(\overline\cE,\overline\cF)$ of $\overline\scrG_{x,r}^{(n)}$} as the number of subsheaves $\cF'$ of $\cE$ with quotient $\cK_x^{\oplus r}$ such that $\overline{\cF'}=\overline\cF$. Then $\overline\scrG_{x,r}^{(n)}$ has an edge from $\overline\cE$ to $\overline\cF$ of weight $m_{x,r}(\overline\cE,\overline\cF)$ if the weight is nonzero. These graphs were introduced by the first and second author in \cite{roberto-graphs} and \cite{oliver-graphs}, respectively.

We draw edges together with their weights $m=m_{x,r}(\overline\cE,\overline\cF)$ as
\[
  \beginpgfgraphicnamed{tikz/fig7}
  \begin{tikzpicture}[>=latex]
    \vertex[circle,fill,label={below:$\cE$}](00) at (0,0) {};
    \vertex[circle,fill,label={below:$\cF$}](10) at (3,0) {};
    \path[-,font=\scriptsize]
     (00) edge[->-=0.8] node[pos=0.2,auto,black] {\footnotesize $m$} (10);
    \node at (5,0) {, $\qquad$ or as $\qquad$};
    \vertex[circle,fill,label={below:$\cE$}](20) at (7,0) {};
    \vertex[circle,fill,label={below:$\cF$}](30) at (10,0) {};
    \path[-,font=\scriptsize]
     (20) edge node[pos=0.2,auto,black] {\footnotesize $m$} node[pos=0.8,auto,black] {\footnotesize $m'$} (30);
   \end{tikzpicture}
 \endpgfgraphicnamed
\]
if $m'=m_{x,r}(\overline\cF,\overline\cE)$ is also nonzero, where we label the nodes with representative rank $n$ bundles (without bars) for better readability. Note that the latter drawing convention applies in particular in the case $2r=n$ due to the first duality theorem (Theorem \ref{thm-dualities} and Corollary \ref{cor: dualities for PGL_n}):

\begin{thmA}[first duality for Hecke operators]\label{thmA}
 Let $\cE$ and $\cF$ be rank $n$ bundles on $X$. Then $m_{x,r}(\overline\cE,\overline\cF)\neq 0$ if and only if $m_{x,n-r}(\overline\cF,\overline\cE)\neq0$.
\end{thmA}

The virtue of the graph $\overline\scrG_{x,r}$ is that we can read off the \emph{$\Phi_{x,r}$-eigenvalue equations} of an automorphic eigenform $f:\PBun_nX\to\C$ as
\[
 \lambda_{x,r}\cdot f(\overline\cE) \ = \ \sum_{
  \tikz[>=latex,font=\tiny,x=20pt]
   {
    \filldraw[inner sep=3] (0,0) circle (2pt) node[label={below:$\overline\cE$}] (00) {}; 
    \filldraw[inner sep=3] (1,0) circle (2pt) node[label={below:$\overline\cF$}] (10) {}; 
    \draw[->-=0.7] (0,0) -- (1,0);
   }
  } 
  m_{x,r}(\overline\cE,\overline\cF) \cdot f(\overline\cF)
\]
where $\lambda_{x,r}$ is the eigenvalue of $f$ for $\Phi_{x,r}$ and $\overline\cF$ ranges over all neighbors of $\overline\cE$ in $\overline\scrG_{x,r}$.


\subsection*{An example}

The theory of unramified automorphic forms for rational function fields is not very exciting from a certain point of view: there are no unramified cusp forms and all vector bundles decompose into a sum of line bundles. Elliptic function fields are of a more intriguing nature. For example, let $x_0$ be chosen base point of the elliptic curve $X$, $x$ a degree one place and $D=x+(d-1)x_0$ for $d\in\{0,1\}$. Then there is a nontrivial extension of the twisted line bundle $\cL_D=\cO(D)$ by the structure sheaf $\cO=\cO_X$ and this defines an indecomposable rank $2$ bundle $\cM_{x,d}$. Another example of an indecomposable rank $2$ bundle is the trace $\cN_{y,2}$ of the line bundle $\cO_{X_2}(\tilde y)$ over constant field extension $X_2=X\otimes_{\FF_q}\FF_{q^2}$ where $\FF_q$ is the constant field of $X$ and $\tilde y$ is a place of $X_2$ that lies over the degree $2$ place $y$ of $X$.

If $X=E_q$ is an elliptic curve over $\FF_q$ with a unique degree $1$ place $x=x_0$ (there is a unique such curve for $q=2,3,4$), then $X$ has $q$ degree $2$ places $y_1,\dotsc,y_q$ and the graph $\overline\scrG_{x,1}^{(2)}$ of $\Phi_{x,1}$ for $\PGL_2$ is as illustrated in Figure \ref{fig: rank 2 graph of Phi(x,1) for class number 1}; cf.\ \cite{oliver-gunther}, \cite{oliver-elliptic}, \cite{serre} or \cite{Takahashi93}.

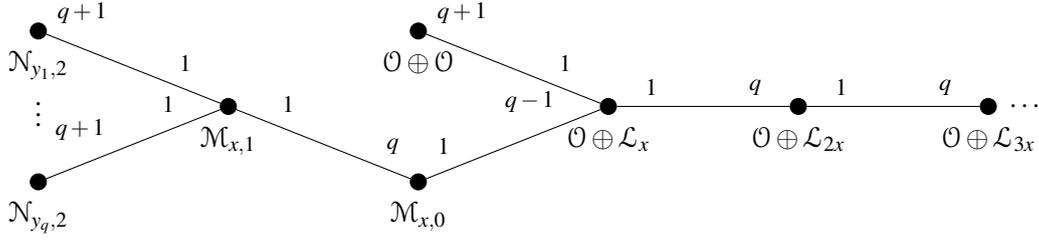
\begin{figure}[t]
 \beginpgfgraphicnamed{tikz/fig8}
  \begin{tikzpicture}[>=latex,x=2.5cm,font=\small]
    \vertex[circle,fill,label={below:${\cN_{y_1,2}}$}] (02) at (0,2) {};
    \vertex[circle,fill,label={below:${\cN_{y_q,2}}$}] (00) at (0,0) {};
    \vertex[circle,fill,label={below:${\cM_{x,1}}$}] (11) at (1,1) {};
    \vertex[circle,fill,label={below:${\cM_{x,0}}$}] (20) at (2,0) {};
    \vertex[circle,fill,label={below:${\cO\oplus\cO}$}] (22) at (2,2) {};
    \vertex[circle,fill,label={below:${\cO\oplus\cL_x}$}] (31) at (3,1) {};
    \vertex[circle,fill,label={below:${\cO\oplus\cL_{2x}}$}] (41) at (4,1) {};
    \vertex[circle,fill,label={below:${\cO\oplus\cL_{3x}}$}] (51) at (5,1) {};
    \node at (0,1) {$\vdots$};
    \node at (5.2,1) {$\dotsb$};
    \path[-,font=\scriptsize] (00) edge node[pos=0.4,auto,black] {$q+1$} node[pos=0.8,auto,black] {$1$} (11);
    \path[-,font=\scriptsize] (02) edge node[pos=0.0,auto,black] {$q+1$} node[pos=0.7,auto,black] {$1$} (11);
    \path[-,font=\scriptsize] (11) edge node[pos=0.2,auto,black] {$1$} node[pos=0.8,auto,black] {$q$} (20);
    \path[-,font=\scriptsize] (20) edge node[pos=0.2,auto,black] {$1$} node[pos=0.8,auto,black] {$q-1$} (31);
    \path[-,font=\scriptsize] (22) edge node[pos=0.0,auto,black] {$q+1$} node[pos=0.7,auto,black] {$1$} (31);
    \path[-,font=\scriptsize] (31) edge node[pos=0.2,auto,black] {$1$} node[pos=0.8,auto,black] {$q$} (41);
    \path[-,font=\scriptsize] (41) edge node[pos=0.2,auto,black] {$1$} node[pos=0.8,auto,black] {$q$} (51);
  \end{tikzpicture}
 \endpgfgraphicnamed
 \caption{The graph $\overline\scrG_{x,1}^{(2)}$ for $X=E_q$ with unique degree $1$ place $x$}
 \label{fig: rank 2 graph of Phi(x,1) for class number 1}
\end{figure}

Given an eigenform $f:\PBun_2 X\to\C$ with eigenvalue $\lambda=\lambda_{x,1}$ for $\Phi_{x,1}$, the graph determines the eigenvalue equations for $f$, such as
\begin{align*}
 \lambda f(\cO\oplus\cO) \ &= \  (q+1)f(\cO\oplus\cL_x), & \lambda f(\cM_{x,0}) \ &= \ qf(\cM_{x,1})+f(\cO\oplus\cL_x), \\ 
 \lambda f(\cN_{y_i})    \ &= \  (q+1)f(\cM_{x,1}),       & \lambda f(\cM_{x,1}) \ &= \ f(\cM_{x,0})+\sum_{j=1}^q f(\cN_{y_j}),
\end{align*}
for $i=1,\dotsc,q$. Note that this system is underdetermined since its $q+3$ equations involve $q+4$ values of $f$. Indeed, we know that every eigenvalue $\lambda$ occurs for some Eisenstein series.

The vanishing of a period imposes an additional condition, which typically limits $\lambda$ to a finite number of values if we aim for a (nonzero) eigenform $f$ in the solution space. Since $f$ lies in a tempered representation if and only if $\lambda\in[-2\sqrt{q},\ 2\sqrt{q}]$ (cf.\ \cite[Lemma 9.3]{Lorscheid13}), this yields a method to verify the temperedness for $H$-periodical eigenforms, which we inspect in more detail for the following two types of periods.

\subsubsection*{Cusp forms}

Let $N$ be the unipotent radical of the Borel subgroup of $\PGL_2$. Then an unramified cusp form $f$ satisfies 
\[
 0 \ = \ \period_N(f)(\overline{\cO\oplus\cO}) \ = \ f(\overline{\cO\oplus\cO}) + (q-1) f(\overline{\cM_{x,0}}).
\]
Solving a system of linear equations shows that we find a nonzero solution only if $\lambda=0$. 
Since $\lambda\in[2\sqrt{q},\ 2\sqrt{q}]$, this shows that every unramified cuspidal representation for $\PGL_2(\A)$ is tempered.

\subsubsection*{Toroidal automorphic forms}
Let $E={\FF_{q^2}F}$. An unramified $T_E$-periodical eigenform $f$ satisfies
\[
 0 \ = \ \period_{T_E}(f)(\overline{\cO\oplus\cO}) \ = \ f(\overline{\cO\oplus\cO}) + 2f(\overline{\cN_{y_1,2}}) + \dotsb + 2f(\overline{\cN_{y_q,2}}).
\]
Solving a system of system of linear equations shows that it contains nonzero solutions only if $\lambda=\pm q$ (cf.\ \cite{oliver-gunther}). Thus $\lambda\in[-2\sqrt{q},\ 2\sqrt{q}]$ for $q=2,3,4$, which shows that every unramified $T_E$-periodical automorphic representation for $\PGL_2(\A)$ is tempered.


\subsection*{The aim of our study}

In the light of the previous examples, we can phrase the overall objective of this series of papers as follows:\\[5pt]
\begin{quote}
 \it We intend to extend the methods from the previous example to $\PGL_3$. In particular, we aim to probe into the question of which (combinations of) period integrals lead to tempered automorphic representations.\\[5pt]
\end{quote}
As a pioneering example, the third author has executed in \cite{Pereira20} this program for the elliptic curve $E_2$ over $\FF_2$ with one rational point. To name an example, all $T_{\FF_{q^2}F}$-periodical Eisenstein series as considered by Wielonsky in \cite{Wielonsky85} are tempered.


\subsection*{The main results of this paper}

The scope of this paper is the computation of the weights $m_{x,r}(\overline\cE,\overline\cF)$ as far as we will use them in our sequel papers. In fact, we will compute refined weights\footnote{In fact, $m_{x,r}(\cE,\cF)$ appears as a multiplicity for the action of $\Phi_{x,r}$ on automorphic forms for $\GL_n$.} $m_{x,r}(\cE,\cF)$, which are defined as the number of subsheaves $\cF'$ of $\cE$ with quotient isomorphic to $\cK_x^{\oplus r}$ such that $\cF'\simeq\cF$ (as abstract sheaves). This determines the weights of $\overline\scrG_{x,r}^{(n)}$ as
\[
 m_{x,r}(\overline\cE,\overline\cF) \ = \ \sum_{\cF'\in\overline\cF} m_{x,r}(\cE,\cF').
\]
The primary result that forms the basis for most of our applications is Theorem \ref{maintheorem} whose content is as follows.\footnote{Since the descriptions of the following results extend over several pages, we omit the full statements in the introduction.} We fix an arbitrary elliptic curve $X$ over $\FF_q$.

\begin{resultA}[main theorem]\label{resultB}
 A description of the weights $m_{x,1}(\cE,\cF)$ for a degree $1$ place $x$ and all rank $3$ bundles $\cE$ and $\cF$ on $X$. 
\end{resultA}

Note that the weights $m_{x,2}(\cE,\cF)$ are fully determined by the weights $m_{x,1}(\cE',\cF')$ thanks to our second duality result (see Theorem \ref{thm-dualities}), which holds, in fact, for every smooth projective curve.

\begin{thmA}[second duality for Hecke operators]\label{thmC}
 Let $\cE_1$ and $\cE_2$ be rank $n$ bundles on $X$ and let $\cE_1^\vee$ and $\cE_2^\vee$ be their respective duals and $x$ a place of $X$. Then $m_{x,r}(\cE_1,\cE_2)=m_{x,n-r}(\cE_1^\vee,\cE_2^\vee\otimes\cO_X(x))$.
\end{thmA}

Our first side result, which is used to determine the cusp eigenforms of $\PGL(2)$ and the induced parabolic Eisenstein series for $\PGL(3)$, is Theorem \ref{evencomponent} whose content is as follows.

\begin{resultA}[first complementary theorem]\label{resultD}
 A description of the weights $m_{y,1}(\cE,\cF)$ for degree $2$ places $y$ of $X$ and all rank $2$ bundles $\cE$ and $\cF$ of even degree.
\end{resultA}

Our second side result is Theorem \ref{thm-rank3-degree2-graph} whose content is as follows.

\begin{resultA}[second complementary theorem]\label{resultE}
 A description of the weights $m_{y,1}(\cO^{\oplus3},\cF)$ for all degree $2$ places $y$ and all rank $3$ bundles $\cF$.
\end{resultA}


\subsection*{Method of proof: comparison with Hall algebras}

The Hall algebra of a curve defined over a finite field was introduced by Kapranov in \cite{kapranov} and posteriorly studied by Kapranov, Schiffmann and Vasserot in \cite{Kapranov-Schiffmann-Vasserot17}. Based on structure theorems for the Hall algebra of an elliptic curve by Burban-Schiffmann (\cite{olivier-elliptic1}) and Fratila (\cite{dragos}), the first author has developed in \cite{roberto-elliptic} a method to compute the weights $m_{x,r}(\cE,\cF)$. In the following, we describe how this method works.


\subsection*{The Hall algebra}

The Hall algebra of $X$ is defined as the complex vector space
\[
 \sH_X \ = \ \bigoplus_{\cF\in\Coh(X)} \C\cdot \cF
\]
that is freely generated by the isomorphism classes $\cF$ of coherent sheaves on $X$. The product of two classes is defined as
\[
 \cG\ast\cF \ = \ q^{\frac12\gen{\cG,\cF}} \cdot \ \sum_{\cE} h_{\cG,\cF}^\cE \cdot \cE
\]
where $\gen{\cG,\cF}=\dim\Hom(\cG,\cF)-\dim\Ext^1(\cG,\cF)$ is the Euler form and 
\[
 h_{\cG,\cF}^\cE \ = \ \#\big\{ \text{subsheaves $\cF'\simeq\cF$ of $\cE$ such that $\cE/\cF'\simeq\cG$} \}.
\]
In particular, the weight $m_{x,r}(\cE,\cF)$ equals the structure constant $h_{\cK_x^{\oplus r},\cF}^\cE$ in the Hall algebra. Thus the product $\cK_x^{\oplus r}\ast\cF$ determines the weights $m_{x,r}(\cE,\cF)$ simultaneously for all rank $n$ bundles $\cE$.


\subsection*{Structure results}

The elliptic Hall algebra $\sH_X$ decomposes as a restricted tensor product
of the subalgebras $\sH_X^{(\mu)}$ linearly spanned by isomorphism classes of semi-stable sheaves of  slope $\mu \in \Q \cup \{\infty\}$, cf. \cite[Lemma 2.6]{olivier-elliptic1}. By Atiyah's classification \cite[Thm.\ 7]{atiyah-elliptic} of vector bundles on an elliptic curve, the category $\sC_\mu$ of semi-stable coherent sheaves of slope $\mu$ is equivalent with $\sC_\infty$, the category of torsion sheaves. As a consequence, $\sH_X^{(\mu)}$ is isomorphic to $\sH_X^{(\infty)}$ for every slope $\mu$. 

The subalgebra $\sH_X^{(\infty)}$ is generated by certain elements $T_{(0,d),x}$, which average over torsion sheaves of degree $d$ with support in $x$. Transporting these generators via the isomorphisms $\sH_X^{(\mu)}\simeq\sH_X^{(\infty)}$ to $\sH_X^{(\mu)}$ yields elements $T_{\bv,x}$ in $\sH_X^{(\mu)}$ where $\bv=(n,d)$ ranges over all pairs with $n>0$, $d\in\Z$ and slope $\mu(\bv)=\frac dn$ equal to  $\mu$.

Let $\rho$  be a primitive character of $\Pic^0(X_n)$, i.e.\ it is not induced from a character on $\Pic^0(X_m)$ with $m<n$, and let $\tilde\rho$ be its Galois orbit. Taking a $\tilde\rho$-weighted average of those $T_{\bv,x}$ for which the degree of $x$ divides $\gcd(n,d)$ yields elements $T_{\bv}^{\tilde\rho}$, which generate the twisted spherical Hall algebra $\sU_X^{\tilde\rho}$. The Hall algebra decomposes into a restricted tensor product $\sH_X=\bigotimes'\sU_X^{\tilde\rho}$ where $\tilde\rho$ varies over all Galois orbits of primitive characters.

This means, in particular, that the elements of $\sU_X^{\tilde\rho}$ and $\sU_X^{\tilde\sigma}$ commute for distinct $\tilde\rho\neq\tilde\sigma$. The commutators of elements in $\sU_X^{\tilde\rho}$ have an explicit description in terms of coefficients of a generating series in the $T_{\bv}^{\tilde\rho}$; cf.\ \cite{roberto-elliptic}, \cite{olivier-elliptic1}, \cite{dragos} for more details.


\subsection*{The algorithm}

In broad strokes, the algorithmic method to compute $h_{\cK_x^{\oplus r},\cF}^\cE$ from the first author's paper \cite{roberto-elliptic} is as follows:
\begin{enumerate}
 \item \textbf{Base change:} Express $\cK_x^{\oplus r}$ and $\cF$ as linear combinations of products of the $T_\bv^{\tilde\rho}$. \\[2pt]  
    This uses a reduction to $\sH_X^{(\infty)}$ via $\sH_X^{(\mu)}\simeq\sH_X^{(\infty)}$ and calculus in the Macdonald algebra to express Hall-Littlewood symmetric functions, which correspond to torsion sheaves, as sums of products of power-sum functions, which correspond to the $T_{(0,d),x}$, followed by a base change to the $T_{\bv}^{\tilde\rho}$.\\[2pt]
    \textbf{Output:} $\cK_x^{\oplus r}\ast\cF=\sum a_i T_{\bv_{i,1}}^{\tilde\rho_{i,1}}\dotsb T_{\bv_{i,k_i}}^{\tilde\rho_{i,{k_i}}}$. \\[-2pt]
 \item \textbf{Order by slopes:} Exchange factors in the products, so that slopes increase.\\[2pt]
    Exchanging the order of $T_{\bv}^{\tilde\rho}\cdot T_{\bv'}^{\tilde\rho}$ brings the commutator $[T_{\bv}^{\tilde\rho},T_{\bv'}^{\tilde\rho}]$ into play, which leads to the subdivision of certain triangles in the lattice $\Z^2$ and the computation of coefficients of a certain generating series, which describe the commutator if the triangle with vertices $(0,0)$, $\bv$ and $\bv'$ has no inner lattice points.\\[2pt]
    \textbf{Output:} $\cK_x^{\oplus r}\ast\cF=\sum b_i T_{\bw_{i,1}}^{\tilde\rho_{i,1}}\dotsb T_{\bw_{i,l_i}}^{\tilde\rho_{i,{l_i}}}$ such that $\mu(\bw_{i,1})\leq \dotsc\leq\mu(\bw_{i,l_i})$ for all $i$. \\[-2pt]
 \item \textbf{Reverse base changes and multiplication:} Replace the $T_{\bw}^{\tilde\rho}$ by the $T_{\bv,x}$, multiply and replace by the $\cF$.\\[2pt]
    Both reverse base changes as well as the multiplication of two factors are immediate, with the exception of the products of the form $T_{\bv,x}T_{\bv',x'}$ with $\mu(\bv)=\mu(\bv')$ and $x=x'$. These latter products can be computed by expressing products of power-sums in the Macdonald algebra in terms of Hall-Littlewood polynomials.\\[2pt]
    \textbf{Output:} $\cK_x^{\oplus r}\ast\cF=\sum h_{\cK_x^{\oplus r},\cF}^{\cE_i} \cE_i$.\\[-5pt]
\end{enumerate}
Note that this description is slightly inaccurate since, in fact, we compute the commutator $[\cK_x^{\oplus r},\cF]$ of $\cK_x^{\oplus r}$ and $\cF$ in $\sH_X$, which agrees with $\cK_x^{\oplus r}\ast\cF$ up to coherent sheaves that are not vector bundles. Instead of digging into the details, we explain the algorithm in a concrete case. We refer the interested reader to the introduction and section 4 in \cite{roberto-elliptic} for a comprehensive account.


\subsection*{An example}

Let $q\in\{2,3,4\}$ and let $X=E_q$ be the elliptic curve over $\FF_q$ with $1$ rational point $x$. Let $\cN_{y,2}$ be the trace of a line bundle $\cL(\tilde y)$ on $E_{q,2}=E_q\otimes_{\FF_q}\FF_{q^2}$ where $\tilde y$ is a place of $E_{q,2}$ over $y$. We compute $\cK_x\ast\cF$ for $\cF=\cN_{y,2}\otimes\cO_X(-x)$ in the following.

We begin with the observation that we are only interested in the extensions of $\cK_x$ by $\cF$ that are vector bundles, i.e.\ we want to compute
\[
 \pi^{\mathrm{vec}}(\cK_x\ast\cF) \ = \ q^{\frac12\gen{\cK_x,\cF}} \cdot \ \sum_{\substack{\text{\tiny vector}\\\text{\tiny bundles }\cE}} h_{\cK_x,\cF}^\cE \cE.
\]
Note that $\gen{\cK_x,\cF}=-2$ and thus $q^{\frac12\gen{\cK_x,\cF}}=q^{-1}$. Since the only extension of $\cK_x$ by $\cF$ that is not a vector bundle is $\cF\oplus\cK_x$, which is equal to $q^{-1}\cF\ast\cK_x$, we gain an equality
\[
 \pi^{\mathrm{vec}}(\cK_x\ast\cF) \ = \ \cK_x\ast\cF - \cF\ast\cK_x \ = \ [\cK_x,\cF].
\]
In order to compute the commutator $[\cK_x,\cF]$, we express the involved bundles in terms of the elements $T_{(n,d)}^{\tilde\rho}$:\footnote{In order to avoid a digression into definition of the $T_{(n,d)}^{\tilde\rho}$, we omit the details of this part.}
\[
 \cK_x \ = \ T_{(0,1)}^{\tilde{{\mathbf{1}}}} \qquad \text{and} \qquad \cF \ = \ c\cdot \sum_{\rho\in\rm{P}_2} \tfrac12\big(\rho(\tilde y)+\rho(\tilde y')\big) T_{(2,0)}^{\tilde\rho}
\]
where $\mathbf{1}$ denotes the trivial character of $\Pic(E_{q})$, $\rm{P}_2$ is the group of characters of $\Pic(E_{q,2})$ that are trivial on $\cO_X(x)$, the place $\tilde y'$ is the Galois conjugate of $\tilde y$ and $c= \frac{2 \sqrt{q}}{  (q+1) \#E_q(\FF_{q^2}) }$. Using that
\[
 \big[T_{(0,1)}^{\tilde{{\mathbf{1}}}}, T_{(2,0)}^{\tilde\rho}\big] \ = \ \begin{cases}
                                                                c^{-1}q^{-1}\cdot T_{(2,1)}^{\tilde{{\mathbf{1}}}} & \text{if $\tilde\rho=\tilde{{\mathbf{1}}}$,} \\
                                                                0                                 & \text{if not,}
                                                               \end{cases}
\]
and $T_{(2,1)}^{\tilde{{\mathbf{1}}}}=\cM_{x,1}$, we compute
\[
 [\cK_x,\cF] \ = \ c\cdot\sum_{\rho\in P_2} \tfrac12\big(\rho(\tilde y)+\rho(\tilde y')\big) \big[T_{(0,1)}^{\tilde{{\mathbf{1}}}}, T_{(2,0)}^{\tilde\rho}\big] \ = \ q^{-1}\cdot T_{(2,1)}^{\tilde{{\mathbf{1}}}} \ = \ q^{\frac12\gen{\cK_x,\cF}}\cdot\cM_{x,1},
\]
which determines $m_x(\cE,\cF)$ as $1$ for $\cE\simeq\cM_{x,1}$ and as $0$ for $\cE\not\simeq\cM_{x,1}$.


\subsection*{Content overview}
The text is structured as follows. In section \ref{sec-graphs}, we introduce graphs of Hecke operators for Hecke operators of $\GL_n$ and $\PGL_n$ over a global function field, and we explain the relation to Hall algebras. In section \ref{sec-dualities}, we establish the duality theorems (Theorems \ref{thmA} and \ref{thmC}) for Hecke operators. In section \ref{sec3}, we recall and establish several facts around Hecke operators for elliptic function fields. In section \ref{sec-mainthm}, we state the main theorem (Result \ref{resultB}) of the text, which we prove in section \ref{sec-proofmainthm}. In section \ref{sec-degreetwo}, we state and proof the two complimentary theorems (Results \ref{resultD} and \ref{resultE}).


\subsection*{Acknowledgements}
The first author was supported by FAPESP [grant number 2017/ 21259-3]. The second author was supported by a Marie Sk\l odowska-Curie Individual Fellowship.

\section{Graphs of Hecke operators}
\label{sec-graphs}

In this section, we set up the notations that are used throughout the paper. We review the notion of a graph of a Hecke operator and its connection with a Hall algebra.

Let $X$ be a smooth projective and geometrically irreducible curve over a finite field $\mathbb{F}_q$, where $q$ is a prime power. In later sections, we specify $X$ to elliptic curves, but the content of this section holds for all curves $X$.

Let $F$ be the function field of $X$. We denote by $|X|$ the set of closed points of $X$ or, equivalently, the set of places of $F$. We denote the residue field of $X$ at $x \in |X|$ by $\Fq(x)$ and by $|x|=[\Fq(x):\Fq]$ the degree of $x$. Let $\mathbb{A}$ be the adele ring of $F$ and $\mathcal{O}_{\mathbb{A}}$ the ring of the adelic integers.

We denote by $G$ the general linear algebraic group $\GL_n$. In particular, we use $G(F) = \GL_n (F)$ and $G(\mathbb{A}) = \GL_n(\mathbb{A})$. We denote the center of $G$ by $Z$; thus $Z(\mathbb{A})$ is the center of $G(\mathbb{A})$. With respect to its adelic topology, $G(\mathbb{A})$ is a locally compact group. Hence $G(\mathbb{A})$ carries a (left or right) Haar measure that is unique up to constant. Note that $G(\A)$ is unimodular i.e. the left and right Haar measure coincide. Let $K = \GL_n(\mathcal{O}_{\mathbb{A}})$ be the standard maximal compact open subgroup of $G(\mathbb{A})$. We fix the Haar measure on $G(\mathbb{A})$ for which $\mathrm{vol}(K)=1$.

The complex vector space $\mathcal{H}$ of all smooth compactly supported functions $\Phi : G(\mathbb{A}) \rightarrow \C$ together with the convolution product
$$\Phi_1 \ast \Phi_2: g \longmapsto \int\limits_{G(\mathbb{A})} \Phi_1(gh^{-1})\Phi_2(h)dh$$
for $\Phi_1, \Phi_2 \in \mathcal{H}$ is called the Hecke algebra for $G(\mathbb{A})$. Its elements are called Hecke operators.
The zero element of $\mathcal{H}$ is the zero function, but there is no multiplicative unit. We define $\mathcal{H}_{K}$ to be the subalgebra of all bi-$K$-invariant elements. This subalgebra has a multiplicative unit, namely, the characteristic function $\epsilon_{K} =  \mathrm{char}_{K}$ of $K$ acts as the identity on $\mathcal{H}_{K}$ by convolution. We call $\mathcal{H}_K$ the unramified part of $\mathcal{H}$ and its elements are called unramified Hecke operators. The algebra $\mathcal{H}_{K}$ is also known as the spherical Hecke algebra; also cf.\ the introduction. 

A Hecke operator $\Phi \in \mathcal{H}$ acts on the space $\mathcal{V} := C^0( G(G)\setminus G(\A))$ of continuous functions $f: G(G)\setminus G(\A) \rightarrow \C$ by the formula
\[ \Phi(f)(g) := \int\limits_{G(\A)} \Phi(h) f(hg) dh. \]
The above action restricts to an action of $\mathcal{H}_K$ on $\mathcal{V}^K$, the space of right $K$-invariant functions.

\begin{prop}[{\cite[Proposition 1.3]{oliver-graphs}}] \label{propI} For any  unramified Hecke operator $\Phi$, there are unique complex numbers $m_1, \ldots,m_r \in \C^{*}$ and pairwise distinct classes $g_1, \ldots, g_r \in G(F) \setminus G(\mathbb{A}) / K $ such that for all $f \in \mathcal{V}^K $
\[
\Phi(f)(g) = \sum_{i=1}^{r} m_i f(g_i).
\]
\end{prop}

\begin{rem} The previous proposition (and hence the upcoming definitions) remain true if we consider $G(F) Z(\A)\setminus G(\mathbb{A}) / K$  instead  $G(F) \setminus G(\mathbb{A}) / K$. In this case, we denote the classes $g_1, \ldots, g_r $ by $\overline{g_1}, \ldots, \overline{g_r}. $
This reflects the situation when $G$ equals $\PGL_n$ instead of $\GL_n$ in our previous discussion. In this text, we consider both cases $\GL_n$ and $\PGL_n$. In order to avoid confusion, we stress whenever $G$ stands for $\PGL_n$.
\end{rem}

\begin{rem} The space $\mathcal{A}$ of automorphic forms  on $G(\A)$ is a subspace of $\mathcal{V}$ that is invariant by the action of $\cH$ (see eg. \cite[Sec. 5]{JacBor}).
In the parts 2 and 3 of this project, 
we will use Proposition \ref{propI}  mostly in the case that $f$ is an automorphic form. 
\end{rem}

\begin{df} For $g,g_1, \ldots,g_r \in G(F) \setminus G(\mathbb{A}) / K$ and $\Phi \in \mathcal{H}_K$, we define
$$\mathcal{V}_{\Phi,K}(g) := \big\{(g,g_i,m_i) \big\}_{i=1, \ldots, r}$$
where the $g_i$'s and $m_i$'s are as in Proposition \ref{propI}.
The \emph{graph $\mathscr{G}_{\Phi,K}^{(n)}$ of $\Phi$ (relative to $K$)} is the edge weighted graph with vertex set
$$\mathrm{Vert} \;\mathscr{G}_{\Phi,K}^{(n)} = G(F) \setminus G(\mathbb{A}) / K$$
and with edge set
$$\mathrm{Edge}\; \mathscr{G}_{\Phi,K}^{(n)} = \bigcup_{g \in \mathrm{Vert} \mathscr{G}_{\Phi,K}} \mathcal{V}_{\Phi,K}(g)$$
where a triple $(g,g_i,m_i)$ denotes an oriented edge from vertex $g$ to vertex $g_i$ with weight $m_i$. The classes $g_i$ are called the \emph{$\Phi$-neighbors of $g$ (relative to $K$)}.
\end{df}


\subsection{Drawing convention} If $\mathcal{V}_{\Phi,K}(g) := \big\{(g,g_1,m_1), \ldots,  (g,g_r,m_r)\big\}$, we make the following drawing conventions to  illustrate the graph $\mathscr{G}_{\Phi,K}^{(n)} $: vertices are represented by 
labeled dots, and an edge $(g,g_i,m_i)$ together with its origin $g$ and its terminus $g_i$ is drawn as
\[
  \beginpgfgraphicnamed{tikz/fig6}
  \begin{tikzpicture}[>=latex, scale=2]
        \vertex[circle,fill,label={below:$g$}](00) at (0,0) {};
        \vertex[circle,fill,label={below:$g_i$}](10) at (2,0) {};
    \path[-,font=\scriptsize]
    (00) edge[->-=0.8] node[pos=0.2,auto,black] {\normalsize $m$} (10)
    ;
   \end{tikzpicture}
 \endpgfgraphicnamed
\]
Hence the $\Phi$-neighborhood of $g$ is thus illustrated as
\[
  \beginpgfgraphicnamed{tikz/fig1}
  \begin{tikzpicture}[>=latex, scale=2]
        \vertex[circle,fill,label={below:$g$}](00) at (0,0) {};
        \vertex[circle,fill,label={below:$g_1$}](11) at (2,0.8) {};
        \vertex[circle,fill,label={below:$g_r$}](10) at (2,-0.8) {};
\draw (1.5,0.2) circle (0.015cm);
\fill  (1.5,0.2) circle (0.015cm);      
\draw (1.5,0) circle (0.015cm);
\fill  (1.5,0) circle (0.015cm);  
\draw (1.5,-0.2) circle (0.015cm);
\fill  (1.5,-0.2) circle (0.015cm);   
            \path[-,font=\scriptsize]
    (00) edge[->-=0.8] node[pos=0.5,auto,black,swap] {\normalsize $m_r$} (10)
    (00) edge[->-=0.8] node[pos=0.5,auto,black] {\normalsize  $m_1$} (11)
    ;
   \end{tikzpicture}
 \endpgfgraphicnamed
\]
If there is an edge from $g$ to $g'$ as well as an edge from $g'$ to $g$, then we draw
\[
  \beginpgfgraphicnamed{tikz/fig2}
  \begin{tikzpicture}[>=latex, scale=1,x=0.98cm]
        \vertex[circle,fill,label={below:$g$}](g) at (0,0) {};
        \vertex[circle,fill,label={below:$g'$}](g') at (2,0) {};
    \path[-,font=\scriptsize] (g) edge[] node[pos=0.2,above,black] {\scalebox{.8}{ $m$} }(g') ;
    \path[-,font=\scriptsize] (g') edge[] node[pos=0.2,above,black] {\scalebox{.8}{$m'$} }(g) ;
    \node[align=right] at (3.5,0) {in place of};

        \vertex[circle,fill,label={below:$g$}](g) at (5,0) {};
        \vertex[circle,fill,label={below:$g'$}](g') at (7,0) {};
    \path[-,font=\scriptsize, bend left=50] (g) edge[->-=0.8] node[pos=0.2,above,black] { \scalebox{.8}{ $m$} }(g');
    \path[-,font=\scriptsize, bend left=50] (g') edge[->-=0.8] node[pos=0.2,above,black] { \scalebox{.8}{ $m'$} }(g) ;
        \node[align=left] at (8,0) {and};

  \vertex[circle,fill,label={below:$g$}](00) at (9.5,-0.25) {};
  \draw[-] (9.5,.25) circle  (0.5cm) node at (9.5,.25){ \scalebox{.8}{ $m$} };
          \node[align=left] at (11.5,0) {in place of};

        \vertex[circle,fill,label={below:$g$}](00) at (13.5,-.25) {};
  \draw[->-=0.5] (13.5,0.25) circle  (0.5cm) node at (14.1,.6){ \scalebox{.8}{ $m$} };
   \end{tikzpicture}
 \endpgfgraphicnamed 
\]


\subsection{Graphs of unramified Hecke operators}
Fix an integer $n \geq 1$. Let $I_k$ be the $k \times k$-identity matrix and $\Phi_{x,r}$ the characteristic function of
$$K\left( \begin{array}{cc}
\pi_x I_r &  \\
 & I_{n-r}
\end{array}    \right)K$$
for a place $x\in |X|$. The unramified Hecke algebra $\cH_K$ is isomorphic to
\[\C[\Phi_{x,1}, \ldots, \Phi_{x,n}, \Phi_{x,n}^{-1}]_{x \in |X|};\]
cf.\ see \cite[ Chapter 12, 1.6]{dennis} for details. In the following, we use the shorter notations $\mathcal{V}_{x,r}(g) :=\mathcal{V}_{\Phi_{x,r},K}(g)$ and $\mathscr{G}_{x,r}^{(n)} :=\mathscr{G}_{\Phi_{x,r},K}^{(n)}$. From the preceding discussion, we conclude that descriptions of the graphs $\mathscr{G}_{x,r}^{(n)}$, for $x\in|X|$ and $0 \leq r \leq n-1$, yields a complete knowledge of the action of $\cH_K$ on $\cA^K$; also see \cite[Proposition 1.8]{roberto-graphs}. This allows us to restrict our attention to the graphs $\mathscr{G}_{x,r}^{(n)}$.

\begin{thm*}[{\cite[Thm. 2.6]{roberto-graphs}}] \label{thm-multi} Let $\mathscr{G}_{x,r}^{(n)}$ denote the graph of the Hecke operator $\Phi_{x,r}$ over $\GL_n$. If  $g \in \mathrm{Vert} \;\mathscr{G}_{x,r}^{(n)}$, then the multiplicities of the edges originating in $g$ sum up to $\# \mathrm{Gr}(n-r,n)(\Fq(x))$, the number of points in the Grassmannian over the residual field of $x$.
\end{thm*}

The previous theorem is also true if we replace $\GL_n$ by $\PGL_n$. In more detail, we have a natural action of $\Z(\A)$ on $\mathscr{G}_{x,r}^{(n)}$, which preserves the weight of the edges. Locally around a vertex, the quotient graph looks like $\mathscr{G}_{x,r}^{(n)}$ -- with the exception of loops, which appear if two adjacent vertices are identified -- and the weights of the outgoing edges at a vertex still sum up to the same value. We shall denote $\overline{\mathscr{G}}_{x,r}^{(n)}$ when the graph under consideration is taken over $\PGL_n$.


\subsection{Algebraic geometry perspective of graphs of  Hecke operators} Our goal is to describe the graphs of Hecke operators for $\GL_3$ (and for $\PGL_3$) for any elliptic curve. Our methods use tools from algebraic geometry. In this subsection, we derive an alternative description of the graphs $\mathscr{G}_{x,r}^{(n)}$ in terms of algebraic geometry.

By a theorem of Andr\'{e} Weil, there is a bijection between $ G(F) \setminus G(\mathbb{A}) / K$ with the set $\Bunvar_n X$ of isomorphism classes of  rank $n$ vector bundles on  $X$. Similarly, a bijection between $G(F) Z(\A)\setminus G(\mathbb{A}) / K$ with the set $\PBun_n X$ of isomorphism classes of projective  rank $n$ vector bundles on  $X$; cf.\ section 5 in \cite{oliver-graphs} for details.  This theorem allows us to determine the action of an unramified Hecke operator $\Phi_{x,r}$ in terms of the equivalence classes of short exact sequences of coherent sheaves on $X$. Namely, we consider exact sequences of the form
$$0 \longrightarrow \E' \longrightarrow \E \longrightarrow \mathcal{K}_{x}^{\oplus r} \longrightarrow 0$$
where $\E'$ and $\E$ are rank $n$ vector bundles, $x$ is a closed point of $X,$ and $\mathcal{K}_{x}^{\oplus r}$ is the skyscraper sheaf on $x$ whose stalk is $\Fq(x)^{\oplus r}$. Let $m_{x,r}(\E,\E')$ be the number of isomorphism classes of exact sequences
$$0 \longrightarrow \E'' \longrightarrow \E \longrightarrow \mathcal{K}_{x}^{\oplus r} \longrightarrow 0$$
with fixed $\E$ such that $\E'' \cong \E'.$ Equivalently, $m_{x,r}(\E,\E')$ is the number of subsheaves $\E''$ of $\E$ that are isomorphic to $\E'$ and  which the quotient $\E/\E''$ is isomorphic to $\mathcal{K}_{x}^{\oplus r}$. We denote by  $\mathcal{V}_{x,r}(\E)$ the set of triples $\big(\E,\E',m_{x,r}(\E,\E')\big)$ such that there exists an exact sequence of the above type, i.e.\  $m_{x,r}(\E,\E') \neq 0$.  Then we have
$$\mathrm{Vert}\;\mathscr{G}_{x,r}^{(n)} = \Bunvar_n X \hspace{0.3cm}\text{ and } \hspace{0.3cm} \mathrm{Edge}\; \mathscr{G}_{x,r}^{(n)} = \coprod_{\E \in \Bunvar_n X} \mathcal{V}_{x,r}(\E);$$
see \cite[Theorem 3.4 ]{roberto-graphs}.

For $G = \PGL_n$, i.e.\ if we consider the graphs over  $G(F) Z(\A)\setminus G(\mathbb{A}) / K$, we obtain in consequence
$$\mathrm{Vert}\;\overline{\mathscr{G}}_{x,r}^{(n)} = \PBun_n X \hspace{0.3cm}\text{ and } \hspace{0.3cm} \mathrm{Edge}\; \overline{\mathscr{G}}_{x,r}^{(n)} = \coprod_{\overline{\E}\in \PBun_n X} \mathcal{V}_{x,r}(\overline{\E}),$$
where $\overline{\E}$ denotes the class of the rank $n$ vector bundle $\E$ on $\PBun_n X$.


\subsection{Hall algebras}
We denote the category  of coherent sheaves on $X$ by $\mathrm{Coh}(X)$.
The Hall algebra $\mathsf{H}_X$ of coherent sheaves on a smooth projective curve $X,$ as introduced by Kapranov in \cite{kapranov}, encodes the extensions of coherent sheaves on $X$. Let $v$ be a square root of $q^{-1}$. The Hall algebra of $X$ is defined as the  $\C$-vector space
$$\mathsf{H}_X = \bigoplus_{\F} \C \cdot \F$$
where $\F$ runs through the set of isomorphism classes of objects in $\mathrm{Coh}(X)$, together with the product
$$\F \ast \G \ = \ v^{-\langle \F,\G \rangle} \ \sum_{\mathcal{E}} \quad h_{\F,\G}^{\mathcal{E}} \;\mathcal{E}$$
where $\langle \F,\G \rangle = \dim_{\Fq} \Ext^{0}(\F,\G) -   \dim_{\Fq} \Ext^{1}(\F,\G)$ and
\[ h_{\F,\G}^{\mathcal{E}} = \frac{\big|\big\{ 0 \longrightarrow \G \longrightarrow \mathcal{E} \longrightarrow \F \longrightarrow 0\big\}\big|}{|\Aut(\F)| \; |\Aut(\G)|} \]

The link between Hall algebras and graphs of Hecke operators is as follows. By \cite[Lemma 2.1]{roberto-elliptic}, we have
\[
 m_{x,r}(\E,\E') \ = \ h_{\mathcal{K}_{x}^{ \oplus r},\E'}^{\E}
\]
for $\E, \E' \in \Bun_n X$. Thus we can recover the multiplicities $m_{x,r}(\E,\E')$ from the products $\mathcal{K}_{x}^{\oplus r}\ast\E'$ in the Hall algebra of $X$; more precisely, for fixed $x$ and $r$ the graph $\mathscr{G}_{x,r}^{(n)}$ is determined by the products $\mathcal{K}_{x}^{\oplus r}\ast \E'$ where $\E'$ runs through the set of rank $n$ vector bundles on $X$. In order to prove the main theorem of this article,  we apply a method developed in \cite{roberto-elliptic}, which is based on results of Fratila \cite{dragos} and Burban-Schiffmann \cite{olivier-elliptic1} about the Hall algebra of an elliptic curve, to calculate the former products for all $\E' \in \Bunvar_3 X$, $r=1$  and $x$ a place of degree $1$. Note that we can derive a description of $\overline{\mathscr{G}}_{x,r}^{(n)}$ from $\mathscr{G}_{x,r}^{(n)}$.

\section{Duality}
\label{sec-dualities}

Hecke operators and their graphs satisfy two dualities, which we will explore in this section.

\begin{lemma}\label{lem-dual1} Given $\E,\E' \in \Bun_n X$, a closed point $x$ of $X$ and $1\leq r<n$. Then for every short exact sequence of the form
\[
 0 \longrightarrow \E' \stackrel{\varphi}{\longrightarrow} \E \longrightarrow \mathcal{K}_{x}^{\oplus r} \longrightarrow 0,
\]
there exists a canonical short exact sequence
\[ 
 0 \longrightarrow \E \stackrel{\varphi^{*}}{\longrightarrow} \E'(x) \longrightarrow \mathcal{K}_{x}^{\oplus n-r} \longrightarrow 0 
\]
where $ \E'(x) := \E ' \otimes_{\mathcal{O}_X} \mathcal{O}_X(x).$

\begin{proof} Let $V = X - \{x\}$ be an open set in $X$, then localizing the given short exact sequence in $V$ yields
$$\xymatrix@R5pt@C7pt{ 0 \ar[rr] & & \E'|_{V}  \ar[rr]^{\varphi |_{V}}_{\sim}  \ar[dd]^{\sim}_{i |_V} & & \E |_{V} \ar[rr] \ar[ddll]^{\varphi^{*}|_{V}} & & 0  \\
& & & & & & \\
0  \ar[rr] & & \E'(x)|_{V}   & &    & &
} $$
which defines $\varphi^{*}|_{V}$ as $i|_{V} \circ (\varphi|_{V})^{-1}$.

Let $U$ be a sufficiently small affine neighborhood of $x$, such that $\mathcal{O}_X(U)$ is a principal ideal domain. Then localizing the given short exact sequence yields the exact sequence
\[0 \longrightarrow \E'(U) \stackrel{\varphi{\text{\tiny$(U)$}}}{\longrightarrow} \E(U) \longrightarrow \ \Fq(x)^{\oplus r}\longrightarrow 0 \]
of $\mathcal{O}_X(U)$-modules.

By the elementary division theorem (see \cite[Thm. 7.8]{lang-algebra}), there are bases for $\E'(U)$ and $\E(U)$ such that $\varphi(U)$ is represented by a matrix
$$\begin{pmatrix}
a_1       & 0    & \cdots &  0 \\
 0         & a_2 & \cdots & 0 \\
 \vdots &  \vdots     &   \ddots         & \vdots \\
 0 & 0 & \cdots & a_n
\end{pmatrix}
 \in \mathrm{Mat}_n(\mathcal{O}_X(U))$$
with $\langle a_n \rangle \subseteq \langle a_{n-1} \rangle \subseteq \cdots \subseteq \langle a_1 \rangle.$ Moreover, since $\mathrm{coker} \varphi(U) = \Fq(x)^{\oplus r}$, we can assume

\[ \langle a_i \rangle = \begin{cases} \langle 1 \rangle  &\text{ for } i =1, \ldots, n-r\\
\langle \pi_x \rangle &\text{ for } i=n-r+1, \ldots, n,
\end{cases} \]
where $\pi_x$ is a uniformizer for $\mathcal{O}_{X,x}$ in $\mathcal{O}_X(U)$.

Therefore the diagonal maps
\[
 \psi_x := 
 \begin{pmatrix}
  a_1^{-1} \pi_x       & 0    & \cdots &  0 \\
  0         & a_2^{-1} \pi_x   & \cdots & 0 \\
  \vdots &  \vdots     &   \ddots         & \vdots \\
  0 & 0 & \cdots & a_n^{-1} \pi_x
 \end{pmatrix},
\]
\[
 \phi_x := 
 \begin{pmatrix}
  \pi_x       & 0    & \cdots &  0 \\
  0         & \pi_x   & \cdots & 0 \\
  \vdots &  \vdots     &   \ddots         & \vdots \\
  0 & 0 & \cdots & \pi_x
 \end{pmatrix}
 \qquad\text{and}\qquad
 \gamma_x := 
 \begin{pmatrix}
  a_1      & 0    & \cdots &  0 \\
  0         & a_2 & \cdots & 0 \\
  \vdots &  \vdots     &   \ddots         & \vdots \\
  0 & 0 & \cdots & a_n
 \end{pmatrix}
\]
between a dimension-$n$ vector space $F^n$ over the function field $F$ of $X$ yields the following diagram
$$\xymatrix@R5pt@C7pt{
                                                                                    &    &                                                                             &     & \E'(x)(U) \ar@{_{(}->}[dd]  \\
 \E'(U) \ar[urrrr] \ar[drr]^{\text{\tiny$\varphi(U)$}} \ar@{^{(}->}[dd]  &    &                                                                             &     &  \\
                                                                                    &    & \E(U) \ar@{^{(}->}[dd] \ar@{-->}[uurr]_{\exists !} &     & F^n   \\
 F^n \ar[urrrr]^<<<<<<<<{\phi_x} \ar[drr]_{\gamma_x}                                                  &    &                                                                             &     &   \\
                                                                                    &    & F^n \ar[uurr]_{\psi_x}                                           &     &
} $$
And that diagram induces a map from $\E(U)$ to $\E'(x)(U).$ This shows that there is a unique $\psi: \E(U) \rightarrow \E'(x)(U)$ such that
$$\xymatrix@R5pt@C7pt{
\E'(U) \ar[rr]^{} \ar[rdd]_{\text{\tiny$\varphi(U)$}} &          & \E'(x)(U) \\
                                  &          &  \\
                                  & \E(U) \ar[uur]_{\psi}&
}$$
commutes. To conclude,  there is a unique  morphism $\varphi^{*} : \E \rightarrow \E'(x)$ such that
$$\xymatrix@R5pt@C7pt{
\E' \ar[rr] \ar[rdd]_{\varphi} &          & \E'(x) \\
                                  &          &  \\
                                  & \E \ar[uur]_{\varphi^{*}}&
}$$
commutes. Moreover, it follows from the local considerations
\[ \varphi^{*}(U) \sim
 \begin{pmatrix}
a_1^{-1} \pi_x       & 0    & \cdots &  0 \\
 0         & a_2^{-1} \pi_x   & \cdots & 0 \\
 \vdots &  \vdots     &   \ddots         & \vdots \\
 0 & 0 & \cdots & a_n^{-1} \pi_x
\end{pmatrix}, \]
that $\varphi^{*}$ is a monomorphism with cokernel $\Fq(x)^{n-r}.$
\end{proof}
\end{lemma}

\begin{rem} In fact, there is a choice involved when considering $\E' \rightarrow \E'(x)$, which is a nonzero constant in the base field $\Fq$.
Up to this choice, the matrix
$$  \begin{pmatrix}
 \pi_x       & 0    & \cdots &  0 \\
 0         & \pi_x   & \cdots & 0 \\
 \vdots &  \vdots     &   \ddots         & \vdots \\
 0 & 0 & \cdots & \pi_x
\end{pmatrix}$$
is well defined as a representative of the (restriction of the) morphism $\E' \rightarrow \E'(x)$ (to the open neighborhood $U$ of $x$).
This choice does not have consequences for the graphs since it does not change the subbundles of $\E'(x).$
\end{rem}

We gain the following dualities for $\mathscr{G}_{x,r}^{(n)}$ and $\mathscr{G}_{x,n-r}^{(n)}$ over $\GL_n.$

\begin{thm} \label{thm-dualities} 
 Let $x \in |X|$, $1\leq r<n$ and $\E,\E' \in \Bunvar_n X.$ Then we have:
 \begin{align*}
  &\textup{\textbf{First duality:}}   &&  m_{x,r}(\E, \E') \neq 0\text{ if and only if }m_{x,n-r}(\E'(x), \E) \neq 0;  \\
  &\textup{\textbf{Second duality:}}  &&  m_{x,r}(\E, \E') = m_{x,n-r}(\E^{\vee}(x), \E'^{\vee}) = m_{x,n-r}(\E^{\vee}, \E'^{\vee}(-x)); \\
  &\textup{\textbf{Derived duality:}} &&  m_{x,r}(\E, \E') \neq 0\text{ if and only if }m_{x,r}(\E'^{\vee}, \E^{\vee}) \neq 0.
 \end{align*}
\end{thm}

\begin{proof} 
 The first duality is a consequence of Lemma \ref{lem-dual1}. The second duality can be proven as follows. Consider a short exact sequence
\[ 0 \longrightarrow \E' \stackrel{}{\longrightarrow} \E \longrightarrow \mathcal{K}_{x}^{\oplus r} \longrightarrow 0. \]
Applying $(-)^\vee=\mathcal{H}om(-,\mathcal{O}_X)$ yields the long exact sequence
\[
 0 \longrightarrow \mathcal{H}om(\mathcal{K}_{x}^{\oplus r},\mathcal{O}_X) \longrightarrow \E^{\vee} \stackrel{\varphi^{\vee}}{\longrightarrow}  {\E'}^{\vee} \stackrel{\delta}{\longrightarrow}  \mathcal{E}xt(\mathcal{K}_{x}^{\oplus r}, \mathcal{O}_X) \longrightarrow \cdots 
\]
But $\mathcal{H}om(\mathcal{K}_{x}^{\oplus r},\mathcal{O}_X)$ vanish and $\mathcal{E}xt(\mathcal{K}_{x}^{\oplus r}, \mathcal{O}_X) \cong \mathcal{K}_{x}^{\oplus r}$. Moreover $(\varphi^{\vee})^{\vee} = \varphi$ and $\mathcal{E}xt(\E, \mathcal{O}_X)=0$, thus $\delta$ must be surjective. This yields the short exact sequence
\[ 
 0 \longrightarrow \E^{\vee} \stackrel{}{\longrightarrow} \E'^{\vee} \longrightarrow \mathcal{K}_{x}^{\oplus r} \longrightarrow 0.
\]
Applying the functor $(-)^{*}$ from Lemma \ref{lem-dual1} to this latter sequence yields
\[ 
 0 \longrightarrow \E'^{\vee} \stackrel{}{\longrightarrow} \E^{\vee}(x) \longrightarrow \mathcal{K}_{x}^{\oplus n-r} \longrightarrow 0.
\]
This establishes a bijection
\[
\left\{
\begin{array}{c}
\text{ rank $n$ locally free subsheaves } \\
  \text{ of $\E$ with cokernel } \mathcal{K}_{x}^{\oplus r} \\
\end{array}
\right\}
 \stackrel{\big((-)^{\vee}\big)^{*}}{\longleftrightarrow}
\left\{
\begin{array}{c}
\text{ rank $n$ locally free subsheaves } \\
\text{ of $\E^{\vee}(x)$ with cokernel } \mathcal{K}_{x}^{\oplus n-r}  \\
\end{array} \right\} \]
which proves the second duality. The dervived duality results from the composition of the first with the second duality.
\end{proof}

\begin{rem} In general, $m_{x,r}(\E, \E') \neq m_{x,n-r}(\E'(x), \E)$, which are the quantities that appear in the first duality. 
See for example \cite[Thm. 6.2]{roberto-elliptic} for $\E = \Line \oplus \Line$ and $\E' = \Line(-x) \oplus \Line$ in $\mathscr{G}_{x,1}^{(2)}$, where $X$ is an elliptic curve, $x$ is a closed point of degree one and $\Line$ is any line bundle. 

This is because $m_{x,r}(\E, \E')$ is the number of isomorphism classes of short exact sequences
\[ 0 \longrightarrow \E' \longrightarrow \E \longrightarrow \mathcal{K}_{x}^{\oplus r} \longrightarrow 0, \]
with fixed $\E$. Which is the same as the number of rank $n$ locally free subsheaves of $\E$ with cokernel isomorphic to $\mathcal{K}_{x}^{\oplus r}.$ While $m_{x,r}(\E'^{\vee}, \E^{\vee})$ is the number of isomorphism classes of short exact sequences
\[ 0 \longrightarrow \E^{\vee} \longrightarrow \E'^{\vee} \longrightarrow \mathcal{K}_{x}^{\oplus r} \longrightarrow 0, \]
with fixed $\E'^{\vee} $ i.e.\ the number of rank $n$ locally free subsheaves of $\E'^{\vee} $ with cokernel isomorphic to $\mathcal{K}_{x}^{\oplus r}.$
\end{rem}

We can deduce from Theorem \ref{thm-dualities} analogous dualities for the graphs $\overline\scrG_{x,r}^{(n)}$.

\begin{cor}\label{cor: dualities for PGL_n}
 Let $x\in|X|$, $1\leq r<n$ and $\cE,\cE'\in\Bun_n X$. Then we have:
 \begin{align*}
  &\textup{\textbf{First duality:}}   &&  m_{x,r}(\overline\E, \overline{\E'}) \neq 0\text{ if and only if }m_{x,n-r}(\overline{\E'}, \overline\E) \neq 0;  \\
  &\textup{\textbf{Second duality:}}  &&  m_{x,r}(\overline\E, \overline{\E'}) = m_{x,n-r}(\E^{\vee}(x), \E'^{\vee}) = m_{x,n-r}(\overline\E^{\vee}, \overline\E'^{\vee}); \\
  &\textup{\textbf{Derived duality:}} &&  m_{x,r}(\overline\E, \overline{\E'}) \neq 0\text{ if and only if }m_{x,r}(\overline{\E'}^{\vee}, \overline\E^{\vee}) \neq 0.
 \end{align*}
\end{cor}

\begin{proof}
 Since $m_{x,r}(\overline{\E},\overline{\E'} )=\sum_{\cE''\in\overline{\cE'}} m_{x,r}({\E},{\E''} )$ and $m_{x,r}({\E},{\E'})\geq0$, the first and derived dualities follow at once from Theorem \ref{thm-dualities}. To establish the second duality, note that 
\begin{align*}
\left\{
\begin{array}{c}
\text{ rank $n$ locally free subsheaves } \\
  \text{ of $\E$ with cokernel } \mathcal{K}_{x}^{\oplus r} \\
\end{array}
\right\}
 &\stackrel{\big((-)^{\vee}\big)^{*}}{\longleftrightarrow}
\left\{
\begin{array}{c}
\text{ rank $n$ locally free subsheaves } \\
\text{ of $\E^{\vee}(x)$ with cokernel } \mathcal{K}_{x}^{\oplus n-r}  \\
\end{array} \right\} \\
 &\stackrel{\otimes \mathcal{O}_X(-x)}{\longleftrightarrow}
 \left\{
\begin{array}{c}
\text{ rank $n$ locally free subsheaves } \\
\text{ of $\E^{\vee}$ with cokernel } \mathcal{K}_{x}^{\oplus n-r}
\end{array}
\right\}
\end{align*}
\end{proof}

\begin{rem} Note that $(-)^{*}$ in the Lemma \ref{lem-dual1} is indeed a duality in the sense that
\[ \big( 0 \longrightarrow \E' \stackrel{}{\longrightarrow} \E \longrightarrow \mathcal{K}_{x}^{\oplus r} \longrightarrow 0  \big)^{**} \]
is equals to
\[  0 \longrightarrow \E'(x) \stackrel{}{\longrightarrow} \E(x) \longrightarrow \mathcal{K}_{x}^{\oplus r} \longrightarrow 0. \]
Hence $(-)^{**}$ is isomorphic to $- \otimes_{\mathcal{O}_X} \mathcal{O}_X(x)$ as a functor on short exact sequences of the form $0\to\E'\to\E\to\cK_x^{\oplus r}\to 0$.
\end{rem}

\section{Hecke operators for elliptic curves}\label{sec3}

From this point on, we restrict our attention to elliptic curves. In particular, $X$ will denote an elliptic curve over the finite field $\mathbb{F}_q$ for the rest of the paper. In the following we summarize some well known facts about coherent sheaves on $X$.


\subsection{Coherent sheaves on elliptic curves} \label{subsec-cohsheaves}

 We denote the rank and the degree of a coherent sheaf $\cF$ by $\rk (\F)$ and $\deg \cF$, respectively. We denote the \textit{slope} of a nonzero sheaf $\F$ in $\mathrm{Coh}(X)$ by $\mu(\F) = \deg(\F)\big/  \rk(\F)$. A sheaf $\F$ is \textit{semistable} (resp. \textit{stable}) if for any proper nontrivial subsheaf $\G \subset \F$ we have $\mu(\G) \leq \mu(\F)$ (resp. $ \mu(\G) < \mu(\F)$). The full subcategory $\mathsf{C}_{\mu}$ of $\mathrm{Coh}(X)$ consisting of all semistable sheaves of a fixed slope $\mu \in \Q \cup \{\infty\}$ is abelian, artinian and closed under extension. Moreover, if $\F,\G$ are semistable with $\mu(\F) < \mu(\G)$, then $\mathrm{Hom}(\G,\F)= \Ext(\F,\G)=0.$ Any sheaf $\F$ possesses a unique filtration, called the \emph{Harder-Narasimhan filtration} or \emph{HN-filtration},
$$0 = \F^{r+1} \subset \F^{r} \subset \cdots \subset \F^{1}=\F$$
for which $\F^{i}\big/\F^{i+1}$ is semistable of slope $\mu_i$ and $\mu_1 < \cdots <\mu_r$, (cf.\ \cite{hardernarasimhan}). Observe that $\mathsf{C}_{\infty}$ is the category of torsion sheaves and hence equivalent to the product category $\prod_{x} \mathrm{Tor}_x$ where $x$ runs through the set of closed points of $X$ and $\mathrm{Tor}_x$ denotes the category of torsion sheaves supported at $x.$ Since $\mathrm{Tor}_x$ is equivalent to the category of finite length modules over the local ring $\mathcal{O}_{X,x}$ of the point $x$, the torsion sheaf $\mathcal{K}_x$ is the unique simple sheaf in $\mathrm{Tor}_x$. Moreover, for each $l \in \Z_{>0}$ there exists a unique (up to isomorphism) indecomposable torsion sheaf $\mathcal{K}_{x}^{(l)}$ of length $l$ and  for each partition $\lambda = (l_1, \ldots, l_r)$ we denote by $\mathcal{K}_{x}^{(\lambda)}$  the unique (up to permutation) torsion sheaf $\mathcal{K}_{x}^{(l_1)} \oplus \cdots \oplus \mathcal{K}_{x}^{(l_r)}$  supported at $x$ associated to $\lambda$. 

\begin{thm}[{\cite{atiyah-elliptic}}] The following holds.
\begin{enumerate}[label=\textbf{{\upshape(\roman*)}}]
\item The HN-filtration of any coherent sheaf splits (noncanonically). In particular, any indecomposable coherent sheaf is semistable;

\item The set of stable sheaves of slope $\mu$ is the class of simple objects of $\mathsf{C}_{\mu}.$
\end{enumerate}
\end{thm}

Let $\E \in \Bunvar_nX$. We call a subsheaf $\cE'$ of $\cE$ a subbundle if the quotient bundle $\cE/\cE'$ is locally free. Note that every subsheaf $\cE'$ is contained in a unique minimal subbundle $\widehat{\cE'}$ of $\cE$, and this subbundle has the same rank as $\cE'$, though the degree might increase, i.e.\ $\deg(\widehat{\E'}) \geq \deg(\E')$.
For brevity, we call a rank $k$ (sub)bundle also an \emph{$k$-(sub)bundle}. Following \cite{roberto-graphs} we define
$$\delta(\mathcal{E}', \mathcal{E}) := \mathrm{rk}(\E) \deg(\mathcal{E}') - \mathrm{rk}(\E') \deg(\mathcal{E})$$
for a subbundle $\mathcal{E'}$ of $\E$,
$$\delta_k(\mathcal{E}) := \sup_{\substack{\mathcal{E'} \hookrightarrow \mathcal{E} \\ k\text{\tiny-subbundle}} } \delta(\mathcal{E}', \mathcal{E}).$$
for $k=1, \ldots, n-1,$ and
$$\delta(\E) := \max\big\{\delta_1(\E), \ldots, \delta_{n-1}(\E)\big\}.$$
The definition of these $\delta$-invariants is motivated by \cite{serre}.

\begin{rem} A vector bundle $\E$ is semistable if and only if $\delta(\E)\leq 0$, and it is stable if and only if $\delta(\E) < 0$.
\end{rem}


\subsection{Edges in rank \texorpdfstring{$3$}{3} coming from rank \texorpdfstring{$2$}{2}}
Next we present our first results on the graphs $\mathscr{G}_{x,1}^{(3)}$ which we derive from results on $\mathscr{G}_{x,1}^{(2)}$.
Namely, we shall describe the edges $\mathcal{V}_{x,1}(\E)$ of $\mathscr{G}_{x,1}^{(3)}$ for all vertices of the form $\E =\mathcal{M} \oplus \mathcal{L} \in \Bunvar_3 X$ where $\mathcal{M} \in \Bunvar_2 X$ and $\Line \in \Pic X$ with $|\deg \mathcal{M} - 2\deg \Line| \gg 0$, where the precise bound depends on $\cM$ and $\cL$, as explained in the following statement.

\begin{lemma} \label{lem1} Let $x\in X$ be a closed point of degree $|x|$ and $\mathcal{L}\in \Pic X$.  Let $\mathcal{M}, \mathcal{M}^{\prime}\in \Bun_2X$ be such that $m_{x,1}(\mathcal{M},\mathcal{M}^{\prime})\neq 0$.

\begin{enumerate}[label=\textbf{{\upshape(\roman*)}}, ]
 \item \label{lem1-i} If either
\begin{enumerate}
\item $\mathcal{M}$ is indecomposable with $2\deg\mathcal{L}-\deg\mathcal{M}>0$ or

\item $\mathcal{M}=\mathcal{L}_1\oplus\mathcal{L}_2$ is decomposable with $\mathcal{L}_i\in \Pic X$ and $\deg\mathcal{L}>\deg\mathcal{L}_i$ for $i=1,2$,
\end{enumerate}
 then $\Ext(\mathcal{M},\mathcal{L})=0$ and $\Ext(\mathcal{M}^{\prime},\mathcal{L})=0$.
\smallskip

  \item \label{lem1-ii} If either
\begin{enumerate}
\item $\mathcal{M}$ is indecomposable with $2\deg\mathcal{L}-\deg\mathcal{M}>2|x|$ or

\item $\mathcal{M} =\mathcal{L}_1\oplus\mathcal{L}_2$ is decomposable with $\mathcal{L}_i\in \Pic X$ and $\deg\mathcal{L}>\deg\mathcal{L}_i+|x|$ for $i=1,2$,
\end{enumerate}
  then $\Ext(\mathcal{M},\mathcal{L})=0$ and $\Ext(\mathcal{M},\mathcal{L}(-x))=0$.
\smallskip

 \item \label{lem1-iii} If either
\begin{enumerate}
\item $\mathcal{M}$ is indecomposable with $2\deg\mathcal{L}-\deg\mathcal{M}<0$ or 

\item $\mathcal{M}=\mathcal{L}_1\oplus\mathcal{L}_2$ is decomposable with $\mathcal{L}_i\in \Pic X$ and $\deg\mathcal{L}<\deg\mathcal{L}_i$ for $i=1,2$,
\end{enumerate}
then $\Ext(\mathcal{L},\mathcal{M})=0$ and $\Ext(\mathcal{L}(-x),\mathcal{M})=0$.
\smallskip

\item \label{lem1-iv} If either
\begin{enumerate}
\item $\mathcal{M}, \mathcal{M}'$ is indecomposable with $2\deg\mathcal{L}-\deg\mathcal{M}<-|x|$ or  

\item $\mathcal{M}$ is indecomposable with $\mathcal{M}'$  decomposable and $2\deg\mathcal{L}-\deg\mathcal{M}<-2|x|$ or 

\item $\mathcal{M}=\mathcal{L}_1\oplus\mathcal{L}_2$ is decomposable with $\mathcal{L}_i\in \Pic X$  and $\deg\mathcal{L}<\deg\mathcal{L}_i - |x|$ for $i=1,2$,
\end{enumerate}
then $\Ext(\mathcal{L},\mathcal{M})=0$ and $\Ext(\mathcal{L},\mathcal{M}^{\prime})=0$.

\end{enumerate}
\end{lemma}

\begin{proof}
\begin{enumerate}[label={\textbf{\upshape(\roman*)}}, wide=0pt]
\item We first suppose $\mathcal{M}$ is indecomposable. Then $\mathcal{M}$ is semistable and by hypothesis
$\mu(\mathcal{M})=\deg\mathcal{M}/ 2<\deg\mathcal{L}=\mu(\mathcal{L}),$
which implies $\Ext(\mathcal{M},\mathcal{L})=\Hom(\mathcal{L},\mathcal{M})=0$.

If $\mathcal{M}^{\prime}$ is indecomposable, then $\mathcal{M}'$ is semistable and by hypothesis
$$\mu(\mathcal{M}^{\prime})=(\deg\mathcal{M}-|x|)/ 2<\deg\mathcal{L}=\mu(\mathcal{L}),$$ which implies $\Ext(\mathcal{M}^{\prime},\mathcal{L})=0$.

If $\mathcal{M}^{\prime}$ is decomposable, we write $\mathcal{M}^{\prime}=\mathcal{L}_1^{\prime}\oplus\mathcal{L}_2^{\prime}$ with $\deg\mathcal{L}_2^{\prime}-\deg\mathcal{L}_1^{\prime}=\delta_1(\mathcal{M}^{\prime})\geq 0$ (see \cite[Prop. 7.7]{oliver-graphs}). Since $\deg\mathcal{L}_2^{\prime}+\deg\mathcal{L}_1^{\prime}=\deg\mathcal{M}-|x|$ it follows that
\[ \deg\mathcal{L}_2^{\prime}=\frac{\deg\mathcal{M}-|x|+\delta_1(\mathcal{M}^{\prime})}{2} \]
From \cite[Lemma 8.2]{oliver-graphs} and $\mathcal{M}$ semi-stability (i.e.\ $\delta_1(\mathcal{M})\leq 0$), it follows that $\delta_1(\mathcal{M}^{\prime})\leq |x|$.
Therefore,
      $$\deg\mathcal{L}_1^{\prime}\leq\deg\mathcal{L}_2^{\prime}=\frac{\deg\mathcal{M}-|x|+\delta_1(\mathcal{M}^{\prime})}{2}<\deg\mathcal{L},$$
      which implies that $\Ext(\mathcal{M}^{\prime},\mathcal{L})=H^0(X,\mathcal{M}^{\prime}\otimes\mathcal{L}^{\vee})=0$.

If $\mathcal{M}$ is decomposable, then $\Ext(\mathcal{M},\mathcal{L})=\Hom(\mathcal{L},\mathcal{M})=H^0(X,\mathcal{M}\otimes\mathcal{L}^{\vee})=0$.
 If $\mathcal{M}^{\prime}$  is  indecomposable we proceed as above. If $\mathcal{M}^{\prime}$  is decomposable and $\mathcal{M}$ is decomposable, we can write $\mathcal{M}^{\prime}=\mathcal{L}_1^{\prime}\oplus\mathcal{L}_2^{\prime}$ with $\deg\mathcal{L}_i^{\prime}\leq\deg\mathcal{L}_i$, $i=1,2$. Therefore, $\Ext(\mathcal{M}^{\prime},\mathcal{L})=H^0(X,\mathcal{M}^{\prime}\otimes\mathcal{L}^{\vee})=0$.\\

\item If $\mathcal{M}$ is indecomposable we have
$$\mu(\mathcal{M})=\deg\mathcal{M}/ 2<\deg\mathcal{L}-|x|=\mu(\mathcal{L}(-x))<\mu(\mathcal{L}),$$
which implies $\Ext(\mathcal{M},\mathcal{L})=0$ and $\Ext(\mathcal{M},\mathcal{L}(-x))=0$.

If $\mathcal{M}$ is decomposable, by Serre duality $\Ext(\mathcal{M},\mathcal{L})=H^0(X,\mathcal{M}\otimes\mathcal{L}^{\vee})=0$. Similarly $\Ext(\mathcal{M},\mathcal{L}(-x))=0$.\\

\item If $\mathcal{M}$ is indecomposable,
$$\deg\mathcal{L}-|x|=\mu(\mathcal{L}(-x))<\mu(\mathcal{L})<\mu(\mathcal{M})=\deg\mathcal{M}/2,$$
which implies $\Ext(\mathcal{L},\mathcal{M})=\Ext(\mathcal{L}(-x),\mathcal{M})=0$.

   If $\mathcal{M}$ is decomposable, by Serre duality we obtain $\Ext(\mathcal{L},\mathcal{M})=H^0(X,\mathcal{M}^{\vee}\otimes\mathcal{L})=0$. Similarly $\Ext(\mathcal{L}(-x),\mathcal{M})=0.$ \\

\item If $\mathcal{M}$ and $\mathcal{M}'$ are indecomposable, then they are semistable. By hypothesis
\[ \mu(\mathcal{L}) < (\deg \mathcal{M} - |x|)/2  < \mu(\mathcal{M}) \quad \text{ and } \quad
\mu(\mathcal{L}) < (\deg \mathcal{M} - |x|)/2 = \mu(\mathcal{M}'), \]
 which implies $\Ext(\mathcal{L},\mathcal{M})=\Ext(\mathcal{L},\mathcal{M}')=0$.

If $\mathcal{M}$ is indecomposable and $\mathcal{M}^{\prime}$ is decomposable, we write $\mathcal{M}^{\prime}=\mathcal{L}_1^{\prime}\oplus\mathcal{L}_2^{\prime}$ with $\deg\mathcal{L}_2^{\prime}-\deg\mathcal{L}_1^{\prime}=\delta_1(\mathcal{M}^{\prime})\geq 0$.  Moreover $\deg\mathcal{L}_2^{\prime}+\deg\mathcal{L}_1^{\prime}=\deg\mathcal{M}-|x|$ and $\delta_1(\mathcal{M}^{\prime})\leq |x|$ since $\delta_1(\mathcal{M})\leq 0$.Thus
$$\deg\mathcal{L}_2^{\prime}\geq\deg\mathcal{L}_1^{\prime}=\frac{\deg\mathcal{M}-|x|-\delta_1(\mathcal{M}^{\prime})}{2}>\deg\mathcal{L},$$
which implies $\Ext(\mathcal{L},\mathcal{M}^{\prime})=0$. That $\Ext(\mathcal{L},\mathcal{M})=0$ follows from $\mu(\mathcal{L}) <  \mu(\mathcal{M})$ and the semi-stability of $\mathcal{M}$.

If $\mathcal{M}$ is decomposable, then $\Ext(\mathcal{L},\mathcal{M})=\Hom(\mathcal{M},\mathcal{L})=H^0(X,\mathcal{L}\otimes\mathcal{M}^{\vee})=0$.
Moreover, if $\mathcal{M}^{\prime}$ is indecomposable,
$$\mu(\mathcal{L}) = \deg \Line <  (\deg\mathcal{L}_1 + \deg\mathcal{L}_2 -2 |x|)/2 < \mu(\mathcal{M}^{\prime}),$$
which implies $\Ext(\mathcal{L},\mathcal{M}^{\prime})=0$.
Otherwise, if $\mathcal{M}^{\prime}$ is decomposable, we can write $\mathcal{M}^{\prime}=\mathcal{L}_1^{\prime}\oplus\mathcal{L}_2^{\prime}$ with $\deg\mathcal{L}_i^{\prime}\leq\deg\mathcal{L}_i$, $i=1,2$. Since $\deg \mathcal{M} = \deg \Line_1 + \deg \Line_2 - |x|$, therefore $\deg\mathcal{L}_i^{\prime}\geq\deg\mathcal{L}_i-  |x|>\deg\mathcal{L}$, which implies $\Ext(\mathcal{L},\mathcal{M}^{\prime})=H^0(X,\mathcal{M}^{\prime\vee}\otimes\mathcal{L})=0$. \qedhere
\end{enumerate}
\end{proof}

\begin{thm} \label{thm1} Let $x\in X$ be a closed point of degree $|x|$ and  $q_x := q^{|x|}.$ Let $\mathcal{L}\in \Pic X$  and $\mathcal{M}, \mathcal{M}^{\prime}\in \Bun_2X$ be such that $m_{x,1}(\mathcal{M},\mathcal{M}^{\prime}) =m \neq 0$.
In  $\rm{\emph{\textbf{(i)}}} -  \rm{\emph{\textbf{(iv)}}}$ below we consider the corresponding hypotheses of Lemma \ref{lem1}. Then

\begin{enumerate}[label={\textbf{\upshape(\roman*)}}]
  \item \label{thm1-i} $m_{x,1}(\mathcal{M}\oplus\mathcal{L},\mathcal{M}^{\prime}\oplus\mathcal{L})\geq m$; \smallskip
  \item \label{thm1-ii} $m_{x,1}(\mathcal{M}\oplus\mathcal{L},\mathcal{M}\oplus\mathcal{L}(-x)) = q_x^2$ and equality holds in part \ref{thm1-i};
  \smallskip
   \item \label{thm1-iii}$m_{x,1}(\mathcal{M}\oplus\mathcal{L},\mathcal{M}\oplus\mathcal{L}(-x))\geq 1$;
   \smallskip
  \item  \label{thm1-iv} $m_{x,1}(\mathcal{M}\oplus\mathcal{L},\mathcal{M}^{\prime}\oplus\mathcal{L}) \geq m\cdot q_x$.
  \end{enumerate}
Moreover, equality holds in \ref{thm1-iv} if $\mathcal{M}$ and $\Line$ satisfy Lemma \ref{lem1} \ref{lem1-iv} for every $\mathcal{M}'$ such that
$m_{x,1}(\mathcal{M}, \cM') \neq 0.$  In this case, we also have an equality in part \ref{thm1-iii}. 

\begin{proof}
Let $\mathcal{K}_x$ be the skyscraper sheaf at $x$. If $\mathcal{M}_1=\mathcal{M},\ldots,\mathcal{M}_r\in \mathrm{Coh}(X)$ are the extensions of $\mathcal{M}^{\prime}$ by $\mathcal{K}_x$ and $m_i :=h_{\mathcal{K}_x,\mathcal{M}^{\prime}}^{\mathcal{M}_i}$ for $i=1,\ldots r$. Then $m=m_1$,
$$\mathcal{K}_x\ast \mathcal{M}^{\prime}=v^{2|x|} \big(m_1\mathcal{M}_1+\cdots+m_r\mathcal{M}_r\big), \quad
\mathcal{K}_x\ast\mathcal{L}=v^{|x|} \big(\mathcal{L}(x)+\mathcal{L}\oplus\mathcal{K}_x\big)$$
and
$$ \mathcal{K}_x\ast \mathcal{L}(-x)=v^{|x|}\big(\mathcal{L}+\mathcal{L}(-x)\oplus\mathcal{K}_x\big).$$
We denote by $\pi^{\mathrm{vec}}(-)$ the vector bundle part in the Hall algebra products. For $\mathcal{E}\in \Bun_nX$, we have
$\pi^{\mathrm{vec}}(\mathcal{K}_x *\mathcal{E})=[\mathcal{K}_x,\mathcal{E}]$
, where the commutator is taken in the Hall algebra $\mathsf{H}_X$ (cf. \cite[Proposition 2.9]{dragos}).

\begin{enumerate}[label={\textbf{\upshape(\roman*)}}, wide=0pt]
  \item Since $\Ext(\mathcal{M}^{\prime},\mathcal{L})=0$, we have in $\mathsf{H}_X$ that
  $$\mathcal{M}^{\prime}\oplus\mathcal{L}=v^{2\deg\mathcal{L}-\deg\mathcal{M}+|x|}\mathcal{M}^{\prime}\ast\mathcal{L}.$$
  Hence 
 \begin{align*}
  \mathcal{K}_x\ast (\mathcal{M}^{\prime}\oplus\mathcal{L})      &=  v^{2\deg\mathcal{L}-\deg\mathcal{M}+|x|} \big(\mathcal{K}_x\ast\mathcal{M}^{\prime}\big) \ast\mathcal{L} \\
      &= v^{2\deg\mathcal{L}-\deg\mathcal{M}+|x|}v^{2|x|} \big( m_1\mathcal{M}_1+\cdots+m_r\mathcal{M}_r \big)\ast\mathcal{L}\\
      &= m v^{2\deg\mathcal{L}-\deg\mathcal{M}+3|x|}\mathcal{M}\ast\mathcal{L}+ v^{2\deg\mathcal{L}-\deg\mathcal{M}+3|x|}\Big( \sum_{i=2}^{r} m_i\mathcal{M}_i \Big)\ast \mathcal{L} \\
      &= m v^{2\deg\mathcal{L}-\deg\mathcal{M}+3|x|}v^{-(2\deg\mathcal{L}-\deg\mathcal{M})}\mathcal{M}\oplus\mathcal{L} \\
      &\quad + v^{2\deg\mathcal{L}-\deg\mathcal{M}+3|x|} \Big( \sum_{i=2}^{r} m_i\mathcal{M}_i \Big)\ast \mathcal{L} \\
      &= mv^{3|x|}\mathcal{M}\oplus\mathcal{L} + v^{2\deg\mathcal{L}-\deg\mathcal{M}+3|x|}\big(m_2\mathcal{M}_2+\cdots+m_r\mathcal{M}_r\big)\ast \mathcal{L}
       \end{align*}
Therefore $m_{x,1}\big(\mathcal{M}\oplus\mathcal{L},\mathcal{M}^{\prime}\oplus\mathcal{L} \big)\geq v^{-3|x|}mv^{3|x|}=m.$\\

  \item  Since $\Ext(\mathcal{M},\mathcal{L}(-x))=0$, thus in $\mathsf{H}_X$
  $$\mathcal{M}\oplus\mathcal{L}(-x)=v^{2(\deg\mathcal{L}-|x|)-\deg\mathcal{M}}\mathcal{M}\ast \mathcal{L}(-x).$$
Hence
\begin{align*}
 \mathcal{K}_x\ast (\mathcal{M}\oplus\mathcal{L}(-x))     &=v^{2(\deg\mathcal{L}-|x|)-\deg\mathcal{M}} \; \mathcal{K}_x\ast\mathcal{M}\ast \mathcal{L}(-x)\\
      &=v^{2(\deg\mathcal{L}-|x|)-\deg\mathcal{M}} \big(\mathcal{M}\ast \mathcal{K}_x+[\mathcal{K}_x,\mathcal{M}]\big)\ast\mathcal{L}(-x)\\
      &=v^{2(\deg\mathcal{L}-|x|)-\deg\mathcal{M}} \big(\mathcal{M}\ast \mathcal{K}_x\ast \mathcal{L}(-x)+[\mathcal{K}_x,\mathcal{M}]\ast \mathcal{L}(-x) \big)\\
      &=v^{2(\deg\mathcal{L}-|x|)-\deg\mathcal{M}} \Big(\mathcal{M}\ast  v^{|x|}\mathcal{L}+\mathcal{M}\ast v^{|x|}\big(\mathcal{K}_x\oplus\mathcal{L}(-x)\big) \\
      &\quad +\pi^{\mathrm{vec}}(\mathcal{K}_x\ast \mathcal{M})\ast \mathcal{L}(-x) \Big)\\
      &=v^{-|x|}\mathcal{M}\oplus\mathcal{L}+v^{2(\deg\mathcal{L}-|x|)-\deg\mathcal{M}} \Big(\mathcal{M}\ast v^{|x|} \big(\mathcal{K}_x\oplus\mathcal{L}(-x) \big) \\
      &\quad +\pi^{\mathrm{vec}}(\mathcal{K}_x\ast\mathcal{M})\ast \mathcal{L}(-x)\Big)
      \end{align*}
Therefore, $ m_{x,1} \big(\mathcal{M}\oplus\mathcal{L},\mathcal{M}\oplus\mathcal{L}(-x)\big)\geq v^{-3|x|}v^{-|x|}=q^{2|x|}=q_x^2.$ Moreover, if $\mathcal{M}$ and $\Line$ are such that satisfies Lemma \ref{lem1}\ref{lem1-ii}, then it satisfies Lemma \ref{lem1}\ref{lem1-i} and hence \ref{thm1-i} holds. Summing up over all multiplicities
$$ \sum_{\E' \in \Bunvar_3 X} m_{x,1}\big(\mathcal{M}\oplus\mathcal{L}, \E' \big) \geq q_x^2 + \sum_{\mathcal{M}' \in \Bunvar_2 X} m_{x,1}\big(\mathcal{M}, \mathcal{M}' \big) = q_x^2 + q_x +1$$
 where $\E'$ runs over the neighbors of $\mathcal{M}\oplus\mathcal{L}$ in $\mathscr{G}_{x,1}^{(3)}$ obtained on \ref{thm1-i} and \ref{thm1-ii}. The identities follow from Theorem \ref{thm-multi}. \\

  \item  Since $\Ext(\mathcal{L}(-x),\mathcal{M})=0$, we have in $\mathsf{H}_X$ that
  $$\mathcal{L}(-x)\oplus\mathcal{M}=v^{\deg\mathcal{M}-2(\deg\mathcal{L}-|x|)}\mathcal{L}(-x)\ast \mathcal{M}.$$
  Hence,
  \begin{align*}
      \mathcal{K}_x\ast \big(\mathcal{L}(-x)\oplus\mathcal{M}\big) &=v^{\deg\mathcal{M}-2(\deg\mathcal{L}-|x|)}\mathcal{K}_x\ast \mathcal{L}(-x)\ast \mathcal{M}\\
      &=v^{\deg\mathcal{M}-2(\deg\mathcal{L}-|x|)}v^{|x|}v^{-(\deg\mathcal{M}-2\deg\mathcal{L})}\mathcal{L}\oplus\mathcal{M} \\
      &\quad + v^{\deg\mathcal{M}-2(\deg\mathcal{L}-|x|)}v^{|x|} \big(\mathcal{K}_x\oplus\mathcal{L}(-x) \big)\ast \mathcal{M}.
  \end{align*}
Therefore we conclude that $m_{x,1} \big(\mathcal{L}\oplus\mathcal{M},\mathcal{L}(-x)\oplus\mathcal{M}\big)\geq v^{-3|x|}v^{3|x|}=1.$      \\

  \item  Since $\Ext(\mathcal{L},\mathcal{M}^{\prime})=0$, we have $\mathcal{M}^{\prime}\oplus\mathcal{L}=v^{\deg\mathcal{M}^{\prime}-2\deg\mathcal{L}}\mathcal{L}\ast \mathcal{M}^{\prime}$. Hence,
\begin{align*}
    \mathcal{K}_x\ast \big(\mathcal{M}^{\prime}\oplus\mathcal{L} \big) & =v^{\deg\mathcal{M}^{\prime}-2\deg\mathcal{L}} \mathcal{K}_x\ast \mathcal{L}\ast\mathcal{M}^{\prime}\\
      &=v^{\deg\mathcal{M}-|x|-2\deg\mathcal{L}} \big(\mathcal{L}\ast\mathcal{K}_x\ast \mathcal{M}^{\prime}+[\mathcal{K}_x, \mathcal{L}]\ast\mathcal{M}^{\prime}\big)\\
      &=v^{\deg\mathcal{M}-|x|-2\deg\mathcal{L}} \Big(\mathcal{L}\ast v^{2|x|}\Big(\sum_{i=1}^{r} m_i\mathcal{M}_i \Big)+\pi^{\mathrm{vec}}(\mathcal{K}_x\ast\mathcal{L})\ast\mathcal{M}^{\prime}\Big)\\
      &=m v^{|x|}\mathcal{L}\oplus\mathcal{M}+v^{\deg\mathcal{M}+|x|-2\deg\mathcal{L}} \mathcal{L}\ast \Big(\sum_{i=2}^{r} m_i\mathcal{M}_i \Big) \\
      &\quad  +v^{\deg\mathcal{M}-|x|-2\deg\mathcal{L}} \pi^{\mathrm{vec}}(\mathcal{K}_x\ast\mathcal{L})\ast\mathcal{M}^{\prime}.
\end{align*}
Therefore, $m_{x,1}\big(\mathcal{M}\oplus\mathcal{L},\mathcal{M}^{\prime}\oplus\mathcal{L} \big)\geq v^{-3|x|}v^{|x|}m=m q_x.$ Moreover, if $\mathcal{M}$ and $\Line$ are such that satisfies Lemma \ref{lem1}\ref{lem1-iv}, then it satisfies Lemma \ref{lem1}\ref{lem1-iii} and hence \ref{thm1-iii} above holds. Summing up over all multiplicities
$$ \sum_{\E' \in \Bunvar_3 X} m_{x,1} \big(\mathcal{M}\oplus\mathcal{L}, \E' \big) \geq \sum_{\mathcal{M}' \in \Bunvar_2 X} m_{x,1} \big(\mathcal{M}, \mathcal{M}' \big) q_x +  1 = q_x^2 + q_x +1$$
 where $\E'$ runs over the neighbors of $\mathcal{M}\oplus\mathcal{L}$ in $\mathscr{G}_{x,1}^{(3)}$ obtained on \ref{thm1-iii} and \ref{thm1-iv}. The identities follow from Theorem \ref{thm-multi}.\qedhere
 \end{enumerate}
\end{proof}
 \end{thm}

\begin{cor} \label{cor-thm1} Let $\mathcal{M} \in \Bunvar_2 X$ and $\Line \in \Pic X$.  We denote by $\E$  the rank $3$ vector bundle $\mathcal{M}  \oplus \Line.$
\begin{enumerate}[label={\textbf{\upshape(\roman*)}}]
\item \label{cor-thm1-i}  Suppose that $\mathcal{M}$ and $\Line$ verify the hypothesis in Lemma \ref{lem1} \ref{lem1-ii}. Then
\[ \mathcal{V}_{x,1}(\E) = \Big\{ \big(\E, \mathcal{M}  \oplus \Line(-x), q_x^2 \big),  \big(\E, \mathcal{M}'  \oplus \Line, m\big) \;\big| \;
\big( \mathcal{M},\mathcal{M}', m \big) \in \mathcal{V}_{x,1}(\mathcal{M})  \Big\}. \]

\item \label{cor-thm1-ii}  Suppose that $\mathcal{M}$ and $\Line$ verify the hypothesis in Lemma \ref{lem1} \ref{lem1-iv} for every $\mathcal{M}' \in \Bunvar_2 X$ such that
$\big( \mathcal{M}, \mathcal{M}', m\big) \in \mathcal{V}_{x,1}(\mathcal{M})$. Then
\[ \mathcal{V}_{x,1}(\E) = \Big\{ \big(\E, \mathcal{M}  \oplus \Line(-x), 1 \big),  \big(\E, \mathcal{M}'  \oplus \Line, m q_x \big) \;\big| \;
\big( \mathcal{M},\mathcal{M}', m \big) \in \mathcal{V}_{x,1}(\mathcal{M})  \Big\}. \]
\end{enumerate}

\begin{proof} It follows immediately from the proofs of Theorem \ref{thm1} \ref{thm1-ii} and \ref{thm1-iv} and from the fact
 \[ \sum_{\E' \in \Bunvar_3 X} m_{x,1}(\E, \E' ) = q_x^2 +q_x +1, \]
 that is proved in Theorem \ref{thm-multi}.
\end{proof}
\end{cor}

\begin{cor}   Let $\mathcal{M} \in \Bunvar_2 X$, $\Line \in \Pic X$ and set $\E$ the rank $3$ vector bundle  $\mathcal{M}  \oplus \Line.$

\begin{enumerate}[label={\textbf{\upshape(\roman*)}}]
\item Suppose that either $\mathcal{M}$ is indecomposable and $2 \deg \Line - \deg \mathcal{M} > 2|x|$ or  $\mathcal{M} = \Line_1 \oplus \Line_2$ with
$\deg \Line_1 \leq \deg \Line_2 \leq \deg \Line$ and $\deg \Line - \deg \Line_2 > |x|$. Then
\[ \mathcal{V}_{x,2}(\E) = \Big\{ \big(\E, \mathcal{M}(-x)  \oplus \Line, 1 \big),  \big(\E, (\mathcal{M}'^{\vee}  \oplus \Line)(-x), m q_x\big) \;\big| \;
\big( \mathcal{M}^{\vee},\mathcal{M}', m \big) \in \mathcal{V}_{x,1}(\mathcal{M}^{\vee})  \Big\}. \]

\item  Suppose that either $\mathcal{M}$ is indecomposable and $2 \deg \Line - \deg \mathcal{M} < -2|x|$ or  $\mathcal{M} = \Line_1 \oplus \Line_2$ with
$\deg \Line \leq \deg \Line_1 \leq \deg \Line_2 $ and $\deg \Line_1 - \deg \Line > |x|$ . Then
\[ \mathcal{V}_{x,2}(\E) = \Big\{ \big(\E, \mathcal{M}(-x)  \oplus \Line, q_x^2 \big),  \big(\E, (\mathcal{M}'^{\vee}  \oplus \Line)(-x), m \big) \;\big| \;
\big( \mathcal{M}^{\vee},\mathcal{M}', m \big) \in \mathcal{V}_{x,1}(\mathcal{M}^{\vee})    \Big\}. \]
\end{enumerate}

\begin{proof} If $\mathcal{M}, \Line$ satisfy the hypothesis of $\rm{\bf{(i)}}$ (resp. $\rm{\bf{(ii)}}$), thus  $\E^{\vee}$ satisfies the conditions of Lemma \ref{lem1} \ref{lem1-iv} (resp. \ref{lem1-ii}). Hence item
$\rm{\bf{(i)}}$ (resp. $\rm{\bf{(ii)}}$) follows from Corollary \ref{cor-thm1} and Theorem \ref{thm-dualities}.
\end{proof}

\end{cor}

We conclude this section with an improvement of \cite[Theorem 4.15]{roberto-graphs} in rank $3$ case.

\begin{prop} Let $\E,\E', \E''\in \Bunvar_3 X$ such that $\E'$ (resp.\ $\E''$) is a neighbor of $\E$ in $\mathscr{G}_{x,1}^{(3)}$ (resp.\ $\mathscr{G}_{x,2}^{(3)}$). Then:
\smallskip

\begin{enumerate}[label=\textbf{{\upshape(\roman*)}}]

\item $\delta_1(\E') \in \big\{ \delta_1(\E) - 2|x| , \ldots, \delta_1(\E) +|x|\big\}$\; and \;$\delta_1(\E') - \delta_1(\E) \equiv |x| \; (\mathrm{mod}3);$ \medskip
\item $\delta_2(\E') \in \big\{ \delta_2(\E) - |x| , \ldots, \delta_2(\E) +2|x|\big\}$\; and \;$\delta_2(\E') - \delta_2(\E) \equiv 2|x| \; (\mathrm{mod}3);$ \medskip
\item $\delta_1(\E'') \in \big\{ \delta_1(\E) - |x| , \ldots, \delta_1(\E) +2|x|\big\}$\; and \;$\delta_1(\E'') - \delta_1(\E) \equiv 2|x| \; (\mathrm{mod}3);$ \medskip
\item $\delta_2(\E'') \in \big\{ \delta_2(\E) - 2|x| , \ldots, \delta_2(\E) +|x|\big\}$\; and \;$\delta_2(\E'') - \delta_2(\E) \equiv |x| \; (\mathrm{mod}3);$ \medskip
\end{enumerate}
where $x \in |X|$ is a closed point of degree $|x|$.

\begin{proof} Let $\mathcal{M}$ be a $2$-subbundle of $\E$ and $\mathcal{L} := \E/\mathcal{M}.$ Hence, $\mathcal{L}$ is a line bundle and, by duality, $\mathcal{L}^{\vee}$ is a line subbundle of $\E^{\vee}$. Moreover, $\E^{\vee} / \mathcal{L}^{\vee}$ is isomorphic to $\mathcal{M}^{\vee}$. We see that to maximize $\deg \mathcal{M}$ is equivalent to maximize $\deg \mathcal{L}^{\vee}$. Let $\mathcal{M}$ as above such that $\deg \mathcal{L}^{\vee}$ attains the maximum, then
\[ \deg \mathcal{L}^{\vee} = \frac{\deg \mathcal{E}^{\vee} + \delta_1(\E^{\vee})}{3}  \hspace{0.5cm}  \text{ and } \hspace{0.5cm}  \deg \mathcal{M} = \deg \E + \deg \mathcal{L}^{\vee}. \]
Therefore,
$$\delta_2(\E) = 3 \deg \mathcal{M} - 2\deg \E = \delta_1(\E^{\vee}).$$

If $\E''$ is a neighbor of $\E$ in $\mathscr{G}_{x,2}^{(3)}$, by \cite[Theorem 4.15]{roberto-graphs} and Theorem \ref{thm-dualities} we obtain
\[ \delta_2(\E'') = \delta_1(\E''^{\vee}) = \delta_1(\E''^{\vee}(-x)) \in \big\{ \delta_2(\E) - |x|, \ldots, \delta_2(\E)+2|x| \big\}. \]
Analogously we prove that $\delta_2(\E') \in \big\{ \delta_2(\E) - |x|, \ldots, \delta_2(\E) + 2|x| \big\}$ if $\E'$ is a neighbor of $\E$ in $\mathscr{G}_{x,1}^{(3)}.$ This together with \cite[Theorem 4.15]{roberto-graphs} finishes the proof.
\end{proof}

\end{prop}

\section{Graphs of rank \texorpdfstring{$3$}{3} and degree \texorpdfstring{$1$}{1}}
\label{sec-mainthm}

In this section we state the main theorem of this work, which consists of a description of the graphs $\mathscr{G}_{x,1}^{(3)}$ of the Hecke operators $\Phi_{x,1}$ (over $\GL_3$) for a closed point $x$ of degree one over an elliptic curve. We begin recalling Atiyah's famous theorem about classification on vector bundles on elliptic curves and then we provide a description of all rank $3$ vector bundles on an elliptic curve.

\begin{thm}[Atiyah] \label{atiyahtheorem} Let $X$ be an elliptic curve defined over $\Fq$. The choice of any rational point $x_0 \in X(\mathbb{F}_q)$ induces an exact equivalence of abelian categories $\epsilon_{\nu,\mu}: \mathsf{C}_{\mu} \rightarrow \mathsf{C}_{\nu}$ for any $\mu,\nu \in \Q \cup \{\infty\}$.
\end{thm}

Atiyah's proof of Theorem \ref{atiyahtheorem} also provides an algorithm to compute the equivalences, cf.\ \cite[Thm. 7]{atiyah-elliptic}. His proof holds for an algebraically closed field of any characteristic. Burban and Schiffmann in \cite[Thm. 1.1]{olivier-elliptic1}  revamp the statement and extend Atiyah's theorem to finite fields.


\subsection{Classification of rank \texorpdfstring{$3$}{3} bundles}

We introduce the following notation that will be used throughout the paper.  From Atiyah's classification, an indecomposable vector bundle $\E$ of rank $n$ and degree $d$ corresponds to a torsion sheaf $\mathcal{K}_{y}^{(l)}$, as introduced in subsection \ref{subsec-cohsheaves}, for some  $y \in |X|$ and $l \in \Z_{>0}$ such that $l |y| = \gcd(n,d)$. Therefore $\E$ is completely determined by its rank $n$, degree $d$, a positive integer $l$ and a closed point $y.$ We shall denote $\E$ by $\E_{(y,l)}^{(n,d)}$.

\begin{rem} \label{rem-traces}
An extension of scalars $\mathbb{F}_{q^m}F/F$, or geometrically $\pi: X_m  \rightarrow X$ where $X_m := X \otimes \mathbb{F}_{q^m},$ defines the constant extension of vector bundles 
\[ \pi^{*} : \Bun_n X \longrightarrow \Bun_n X_m\]
as pull backs along $\pi$. As a consequence of Lang's theorem \cite[Cor. to Thm. 1]{lang}, $\pi^{*}$ is injective. On the other hand, $\pi: X_m  \rightarrow X$ defines the \textit{trace} of vector bundles
\[ \pi_{*} : \Bun_n X_m\longrightarrow \Bun_{mn} X\]
as the direct image. As shown in  \cite[Thm. 1.8]{arason}, every indecomposable vector bundle over $X$ is the trace of a geometrically indecomposable bundle over some constant extension $X_m$ of $X$.

\end{rem}

\begin{ex}
The trace of a rank $n$ bundle over $X_m$ is the direct image sheaf with respect to $\pi:X_m\to X$, which is a rank $mn$ bundle. If $y \in |X|$ of degree $m$ and $\tilde y$ is a degree $1$ place of $X_m$ over $y$, then $\cE_{(y,l)}^{(n',d')}$ is the trace of the geometrically indecomposable bundle $\cE_{(\tilde y,l)}^{(n,d)}$,  for $n' = mn$ and $d'=md$.
\end{ex}

\begin{ex}
 The (indecomposable) line bundles over $X$ are the bundles $\cE_{(x,1)}^{(1,d)}$ with $x\in|X|$ of degree $1$ and $d\in\Z$. For $n=2,3$, the indecomposable rank $n$ bundles over $X$ are the bundles $\cE_{(x,l)}^{(n,d)}$ with $x\in|X|$ of degree $1$, $d\in\Z$ and $l=\gcd(n,d)$ (geometrically indecomposable bundles) and $\cE_{(y,1)}^{(n,nd)}$ with $y\in|X|$ of degree $n$ and $d\in\Z$ (traces of line bundles over $X_n$). 
\end{ex}

In the following, we give a complete list of isomorphism classes of rank $3$ bundles on $X$.
To begin with, the Krull-Schmidt theorem  \cite[Thm. 2]{atiyah-krull} holds for the category of vector bundles over $X$, i.e.\ every vector bundle on $X$  has a unique decomposition into a direct sum of indecomposable subbundles, up to permutation of factors. 

By our previous discussion, every indecomposable bundle on $X$ is the trace $\pi_\ast\cE$ of a bundle $\cE$ on a finite extension $X_m$ of $X$, which satisfies $\rk(\pi_\ast\cE)=m\cdot\rk(\cE)$. Thus the indecomposable bundles of rank at most $3$ on $X$ are geometrically indecomposable bundles of rank at most $3$ and traces of line bundles over $X_2$ and $X_3$.

Taking direct sums of all possible combinations of indecomposable bundles yields the following list of rank $3$ bundles on $X$:
\begin{align*}
&  \E_{(y,l)}^{(3,d)} && \text{ for $|y|=1$ or $3$ and $d \in \Z$,} \\
& \E_{(x_1,1)}^{(1,d_1)} \oplus \E_{(x_2,l_2)}^{(2,d_2)} && \text{ for $|x_1|=1$, $|x_2|=1$ or $2$ and $d_1,d_2 \in \Z$,} \\
& \E_{(x_1,1)}^{(1,d_1)} \oplus \E_{(x_2,1)}^{(1,d_2)}  \oplus \E_{(x_3,1)}^{(1,d_3)}  && \text{ for $|x_1|=|x_2|=|x_3|=1$ and $d_1,d_2,d_3 \in \Z$. }
\end{align*}

Note that this justifies that Theorem \ref{maintheorem} provides a full description of the graphs  $\mathscr{G}_{x,1}^{(3)}$ where $x$ is a degree one place of $X$.

\begin{rem}
 Note that we can derive readily a list of representatives for $\PBun_3 X$ from the above description, which is as follows:
 \begin{align*}
&  \E_{(y,l)}^{(3,d)} && \text{ for $|y|=1$ or $3$ and $d \in \{0,1,2\}$} \\
& \E_{(x_1,1)}^{(1,d_1)} \oplus \E_{(x_2,l_2)}^{(2,d_2)} && \text{ for $|x_1|=1$, $|x_2|=1$ or $2$, $d_1,\in \Z$ and $d_2 \in \{0,1\}$} \\
& \E_{(x_1,1)}^{(1,0)} \oplus \E_{(x_2,1)}^{(1,d_2)}  \oplus \E_{(x_3,1)}^{(1,d_3)}  && \text{ for $|x_1|=|x_2|=|x_3|=1$ and $d_2,d_3 \in \Z$. }
\end{align*}
This allows us to deduce a description of $\overline{\mathscr{G}}_{x,1}^{(3)}$ from Theorem \ref{maintheorem}.
\end{rem}


\subsection{The main theorem}
We are prepared to state the central result of this paper. Identities of divisors in the theorem are taken modulo the subgroup $\gen{x_0}$ of $\Pic X$ where $x_0$ is the base point of our elliptic curve $X$; e.g.\ if we write $z=x+y$, then this has to be read as the equality $z-|z|\cdot x_0=(x-|x|\cdot x_0)+(y-|y|\cdot x_0)$ in $\Pic^0 X$.

\begin{thm} \label{maintheorem} Let $X$ be an elliptic curve defined over a finite field $\Fq$ and $x$ be a degree one closed point at $X$. Then the  graph $\mathscr{G}_{x,1}^{(3)}$  is given as follows.

\begin{enumerate}[label=\textbf{{\upshape(\roman*)}}, wide=0pt]
\item \label{i1} Let  $\E := \E_{(y,1)}^{(3,0)}$, then
\[ \mathcal{V}_{x,1}(\E) = \Big\{ \big(\E, \E_{(x',1)}^{(3,-1)}, q^2 + q +
 1\big)  \Big\}, \]
where  $x'= y -x$.\\

\item \label{i2} Let $\E := \E_{(y,3)}^{(3,0)}$ , then
\[\mathcal{V}_{x,1}(\E) = \Big\{ \big( \E, \E_{(x',1)}^{(3,-1)}, q^2\big), \big(\E, \E_{(x'',1)}^{(2,-1)} \oplus \E_{(y,1)}^{(1,0)}, q\big), \big( \E, \E_{(y-x,1)}^{(1,-1)} \oplus \E_{(y,2)}^{(2,0)}, 1 \big)  \Big\},\]
where $x' = 3y-x$ and $x'' = 2y-x.$  \\

\item \label{i3}  Let $\E := \E_{(y,1)}^{(3,1)}$, then
\begin{align*}
\mathcal{V}_{x,1}(\E) = \Big\{ &\big(\E, \E_{(z,3)}^{(3,0)}, 1\big), \big(\E, \E_{(w,1)}^{(3,0)}, 1\big) ,
 \big(\E, \E_{(x_1,1)}^{(1,0)} \oplus \E_{(x_2,1)}^{(1,0)} \oplus \E_{(x_3,1)}^{(1,0)}, 1\big), \\
   & \big(\E, \E_{(z_1,1)}^{(1,0)} \oplus \E_{(z_2,1)}^{(2,0)} , 1\big), \big(\E, \E_{(y',1)}^{(1,0)} \oplus \E_{(y'',2)}^{(2,0)}, 1\big) \Big\}
\end{align*}
where $3z = y-x$, $w = y-x,$ $ x_1, x_2, x_3$ are pairwise distinct and $ x_1 + x_2 + x_3 = y-x,$
$z_2 + z_1 = y-x$, and $ y'\neq y''$ with $y' + 2y'' = y-x$. \\

\item \label{i4}  Let $\E : = \E_{(y,1)}^{(3,2)}$, then
\[ \mathcal{V}_{x,1}(\E) =  \Big\{ \big(\E, \E_{(y-x,1)}^{(3,1)}, q^2 + q + 1 - N_1 \big), \big(\E, \E_{(x',1)}^{(1,0)} \oplus \E_{(x'',1)}^{(2,1)} , 1\big) \Big\},\]
where $x',x''\in |X|$ are such that $ x' + x'' = y-x.$
\\

\item \label{i5} Let $\E :=  \E_{(y,2)}^{(2,0)} \oplus \E_{(y',1)}^{(1,0)}  $ with $y \neq y' $, then,
\begin{align*}
\mathcal{V}_{x,1}(\E) &=  \Big\{ \big(\E, \E_{(z,1)}^{(3,-1)}, q^2 -q\big), \big(\E, \E_{(y'-x,1)}^{(1,-1)} \oplus \E_{(y,2)}^{(2,0)},  1 \big), \big(\E, \E_{(y-x,1)}^{(2,-1)} \oplus \E_{(y,1)}^{(1,0)},  q \big), \\
& \hspace*{0.8cm}  \big(\E, \E_{(x',1)}^{(2,-1)} \oplus \E_{(x'',1)}^{(1,0)},  q-1 \big), \big(\E, \E_{(y-x,1)}^{(1,-1)} \oplus \E_{(y,1)}^{(1,0)} \oplus \E_{(y',1)}^{(1,0)}  ,  1 \big) \Big\},
\end{align*}
where $z = y' + 2y-x$, $x' = y + y'-x$  and  $x'' = x+y-y'$. \\

\item \label{i6} Let $\E := \E_{(y,2)}^{(2,0)} \oplus \E_{(y,1)}^{(1,0)},  $ then,

\begin{align*}
\mathcal{V}_{x,1}(\E) &=  \Big\{ \big(\E, \E_{(y-x,1)}^{(1,-1)} \oplus \E_{(y,1)}^{(1,0)} \oplus \E_{(y,1)}^{(1,0)}, 1 \big),
 \big(\E, \E_{(y-x,1)}^{(2,-1)} \oplus \E_{(y,1)}^{(1,0)}, q^2 \big), \\
 & \hspace*{0.8cm} \big(\E, \E_{(y-x,1)}^{(1,-1)} \oplus \E_{(y,2)}^{(2,0)} , q \big) \Big\}.
\end{align*}

\item \label{i7} Let $\E :=  \E_{(y,2)}^{(2,0)} \oplus \E_{(y',1)}^{(1,1)}  $ with $y \neq y' - x$, then
\begin{align*}
\mathcal{V}_{x,1}(\E) &=  \Big\{ \big(\E, \E_{(2y-x,1)}^{(2,-1)} \oplus \E_{(y',1)}^{(1,1)} , q\big),
\big(\E, \E_{(y-x,1)}^{(1,-1)} \oplus \E_{(y,1)}^{(1,0)}  \oplus \E_{(y',1)}^{(1,1)} ,  1 \big), \\
& \hspace*{0.8cm} \big(\E, \E_{(y,2)}^{(2,0)}  \oplus \E_{(y'-x,1)}^{(1,0)},  q^2 \big) \Big\}.
\end{align*}

\item \label{i8} Let $\E := \E_{(y,2)}^{(2,0)} \oplus \E_{(y',1)}^{(1,1)},  $ with $y = y' - x$, then

\begin{align*}
\mathcal{V}_{x,1}(\E) &=  \Big\{ \big(\E, \E_{(2y-x,1)}^{(2,-1)} \oplus \E_{(y',1)}^{(1,1)} , q\big),
\big(\E, \E_{(y-x,1)}^{(1,-1)} \oplus \E_{(y,1)}^{(1,0)}  \oplus \E_{(y',1)}^{(1,1)} ,  1 \big), \\
& \hspace*{0.8cm} \big(\E,  \E_{(y,3)}^{(3,0)} ,  q^2-q \big),  \big(\E,  \E_{(y,2)}^{(2,0)} \oplus \E_{(y,1)}^{(1,0)}  ,  q \big) \Big\}.
\end{align*}

\item \label{i9} Let $\E := \E_{(y,2)}^{(2,0)} \oplus \E_{(y',1)}^{(1,-1)}$,  with $y \neq y' + x$, then
\begin{align*}
 \mathcal{V}_{x,1}(\E) &=  \Big\{ \big(\E, \E_{(y',1)}^{(1,-1)} \oplus \E_{(2y-x,1)}^{(2,-1)} , q^2\big),
\big(\E, \E_{(y'-x,1)}^{(1,-2)} \oplus \E_{(y,2)}^{(2,0)}  ,  1 \big), \\
 & \hspace*{0.8cm}\big(\E, \E_{(y',1)}^{(1,-1)}  \oplus \E_{(y-x,1)}^{(1,-1)} \oplus \E_{(y,1)}^{(1,0)},  q \big) \Big\}.
\end{align*}

\item \label{i10} Let $\E := \E_{(y,2)}^{(2,0)} \oplus \E_{(y',1)}^{(1,-1)}$,  with $y = y' + x$, then
\begin{align*}
\mathcal{V}_{x,1}(\E) &=  \Big\{ \big(\E, \E_{(y',1)}^{(1,-1)} \oplus \E_{(2y-x,1)}^{(2,-1)} , q^2\big),
\big(\E, \E_{(y',2)}^{(2,-2)} \oplus \E_{(y,1)}^{(1,0)},  q-1 \big),\\
&\hspace*{0.8cm}   \big(\E, \E_{(y'-x,1)}^{(1,-2)} \oplus \E_{(y,2)}^{(2,0)}, 1 \big),
\big(\E, \E_{(y',1)}^{(1,-1)} \oplus \E_{(y',1)}^{(1,-1)} \oplus \E_{(y,1)}^{(1,0)},  1 \big)\Big\}.
\end{align*}

\item \label{i11} Let $\E := \E_{(y,2)}^{(2,0)} \oplus \E_{(y',1)}^{(1,d)}$, with $|d| \geq 2,$ then \smallskip
\begin{align*}
\mathcal{V}_{x,1}(\E) &=  \Big\{ \big(\E, \E_{(y,2)}^{(2,0)}  \oplus \E_{(y'-x,1)}^{(1,d-1)},  q^2 \big) , \\
& \hspace*{0.8cm} \big(\E, \E_{(2y-x,1)}^{(2,-1)} \oplus \E_{(y',1)}^{(1,d)} , q\big),
\big(\E, \E_{(y-x,1)}^{(1,-1)} \oplus \E_{(y,1)}^{(1,0)}  \oplus \E_{(y',1)}^{(1,d)} ,  1 \big)\Big\} 
\end{align*}
if $d \geq 2,$ and
\begin{align*}
 \mathcal{V}_{x,1}(\E) &=  \Big\{ \big(\E, \E_{(y'-x,1)}^{(1,d-1)} \oplus \E_{(y,2)}^{(2,0)}  ,  1 \big)
\big(\E, \E_{(y',1)}^{(1,d)} \oplus \E_{(2y-x,1)}^{(2,-1)} , q^2\big), \\
& \hspace*{0.8cm} \big(\E, \E_{(y',1)}^{(1,d)}  \oplus \E_{(y-x,1)}^{(1,-1)} \oplus \E_{(y,1)}^{(1,0)},  q \big) \Big\} 
\end{align*}
if $d \leq - 2.$\\

\item \label{i12} Let $\E :=  \E_{(y,1)}^{(2,0)} \oplus \E_{(y',1)}^{(1,0)},$ then
\begin{align*}
\mathcal{V}_{x,1}(\E) &=  \Big\{ \big(\E, \E_{(x',1)}^{(3,-1)}, q^2- 1 \big) , \big(\E, \E_{(y,1)}^{(2,0)} \oplus \E_{(y'-x,1)}^{(1,-1)}, 1 \big),
 \big(\E, \E_{(z,1)}^{(2,-1)} \oplus \E_{(y',1)}^{(1,0)}, q+1 \big)  \Big\}
\end{align*}
where  $x' = y'+ y - x$ and  $z=  y -x $. \\

\item \label{i13} Let $\E := \E_{(y,1)}^{(2,0)} \oplus \E_{(y',1)}^{(1,d)}$, with $|d| \geq 1,$ then \smallskip
\[
\mathcal{V}_{x,1}(\E) =   \Big\{ \big(\E, \E_{(y,1)}^{(2,0)} \oplus \E_{(y'-x,1)}^{(1,d-1)}, q^2 \big),
 \big(\E, \E_{(z,1)}^{(2,-1)} \oplus \E_{(y',1)}^{(1,d)}, q+1\big)  \Big\}  \]
if $d \geq 1,$ and
\[ \mathcal{V}_{x,1}(\E) =   \Big\{ \big(\E, \E_{(y,1)}^{(2,0)} \oplus \E_{(y'-x,1)}^{(1,d-1)}, 1 \big),
 \big(\E, \E_{(z,1)}^{(2,-1)} \oplus \E_{(y',1)}^{(1,d)}, q^2 + q\big)  \Big\}  \]
if $d \leq -1$, where  $z=  y-x $. \\


\item \label{i14} Let $\E := \E_{(y,1)}^{(2,1)} \oplus \E_{(y',1)}^{(1,0)}$, then \smallskip
\begin{align*}
\mathcal{V}_{x,1}(\E)  & =  \Big\{  \big(\E, \E_{(y',1)}^{(1,0)} \oplus \E_{(x_1,1)}^{(1,0)} \oplus \E_{(x_2,1)}^{(1,0)}, q \big),
\big(\E, \E_{(y',1)}^{(1,0)} \oplus \E_{(z,1)}^{(2,0)}, q\big), \\
 &  \hspace*{0.8cm}  \big(\E, \E_{(y',1)}^{(1,0)} \oplus \E_{(w,2)}^{(2,0)}, q\big),   \big(\E, \E_{(x',1)}^{(1,0)} \oplus \E_{(y',2)}^{(2,0)}, q-1\big),
 \big(\E, \E_{(y,1)}^{(2,1)} \oplus \E_{(y'-x,1)}^{(1,-1)}, 1\big) \Big\}    \\
 & \bigcup  \Big\{    \big(\E, \E_{(y',1)}^{(1,0)} \oplus \E_{(y',1)}^{(1,0)} \oplus \E_{(y-y'-x,1)}^{(1,0)}, 1\big)     \;\big|\; \text{ if }  y \neq y'  \Big\}    \\
& \bigcup  \Big\{   \big(\E, \E_{(y',1)}^{(1,0)} \oplus \E_{(y',2)}^{(2,0)}, 1\big), \big(\E, \E_{(y',3)}^{(3,0)}, q-1\big) \;\big|\; \text{ if }  y=2y'+x \Big\}
\end{align*}
where $y', x_1, x_2$ are pairwise distinct and $x_1 + x_2 = y-x,$ $z$ is such that $z=y-x$,
$2w=y-x$ with $w\neq y'$ and  $x'+ x= y-y'$ with $x'\neq y'$. \\

\item \label{i15} Let $\E := \E_{(y,1)}^{(2,1)} \oplus \E_{(y',1)}^{(1,1)}$, then \smallskip
\begin{align*}
\mathcal{V}_{x,1}(\E)  & =  \Big\{  \big(\E, \E_{(z,1)}^{(2,0)} \oplus \E_{(y',1)}^{(1,1)}, 1 \big),
\big(\E, \E_{(x_1,1)}^{(1,0)} \oplus \E_{(x_2,1)}^{(1,0)} \oplus \E_{(y',1)}^{(1,1)}, 1 \big), \\
& \hspace*{0.8cm} \big(\E, \E_{(w,2)}^{(2,0)} \oplus \E_{(y',1)}^{(1,1)}, 1 \big),\big(\E, \E_{(y,1)}^{(2,1)} \oplus \E_{(y'-x,1)}^{(1,0)}, q \big),
\big(\E, \E_{(z',1)}^{(3,1)}, q^2-q \big) \Big\}
\end{align*}
where $z \in |X|$ of degree two is such that $z= y-x$, $x_1 \neq x_2$ are such that  $x_1 +x_2 =y-x$, $2w=y-x$ and  $z' = y+y'-x$.\\

\item \label{i16} Let $\E := \E_{(y,1)}^{(2,1)} \oplus \E_{(y',1)}^{(1,d)}$, then \smallskip
\begin{align*}
\mathcal{V}_{x,1}(\E)  & =  \Big\{  \big(\E, \E_{(z,1)}^{(2,0)} \oplus \E_{(y',1)}^{(1,d)}, 1 \big),
\big(\E, \E_{(x_1,1)}^{(1,0)} \oplus \E_{(x_2,1)}^{(1,0)} \oplus \E_{(y',1)}^{(1,d)}, 1 \big), \\
& \hspace*{0.8cm} \big(\E, \E_{(w,2)}^{(2,0)} \oplus \E_{(y',1)}^{(1,d)}, 1 \big), \big(\E, \E_{(y,1)}^{(2,1)} \oplus \E_{(y'-x,1)}^{(1,d-1)}, q^2 \big)\Big\}
\end{align*}
if $d \geq 2$, and
\begin{align*}
\mathcal{V}_{x,1}(\E)  & =  \Big\{  \big(\E, \E_{(z,1)}^{(2,0)} \oplus \E_{(y',1)}^{(1,d)}, q \big),
\big(\E, \E_{(x_1,1)}^{(1,0)} \oplus \E_{(x_2,1)}^{(1,0)} \oplus \E_{(y',1)}^{(1,d)}, q \big), \\
& \hspace*{0.8cm} \big(\E, \E_{(w,2)}^{(2,0)} \oplus \E_{(y',1)}^{(1,d)}, q \big), \big(\E, \E_{(y,1)}^{(2,1)} \oplus \E_{(y'-x,1)}^{(1,d-1)}, 1 \big)\Big\}
\end{align*}
if $d \leq -1$, where in each case $z \in |X|$ of degree two is such that $z = y-x$, $x_1 \neq x_2$ are such that  $x_1 +x_2 =y-x$ and $2w=y-x$. \\

\item \label{i17}  Let $\E := \E_{(y,1)}^{(1,0)} \oplus \E_{(y,1)}^{(1,0)} \oplus \E_{(y,1)}^{(1,0)}$, then
$$ \mathcal{V}_{x,1}(\E) = \Big\{ \big( \E, \E_{(y-x,1)}^{(1,-1)} \oplus \E_{(y,1)}^{(1,0)} \oplus \E_{(y,1)}^{(1,0)}, q^2 + q + 1 \big) \Big\}.  $$ \smallskip

\item Let \label{i18}  $\E = \E_{(y,1)}^{(1,0)} \oplus \E_{(y,1)}^{(1,0)} \oplus \E_{(y',1)}^{(1,0)}$ with $y \neq y'$, then
\begin{align*}
\mathcal{V}_{x,1}(\E) &=  \Big\{ \big(\E, \E_{(y,1)}^{(1,0)} \oplus \E_{(y,1)}^{(1,0)} \oplus \E_{(y'-x,1)}^{(1,-1)},  1 \big),
 \big(\E,\E_{(y-x,1)}^{(1,-1)} \oplus \E_{(y,1)}^{(1,0)} \oplus \E_{(y',1)}^{(1,0)},  q+1 \big) , \\
 & \hspace*{0.8cm} \big(\E, \E_{(2y-x,1)}^{(2,-1)} \oplus \E_{(y,1)}^{(1,0)},  q^2-1 \big) \Big\}.
\end{align*} \smallskip

\item \label{i19} Let $\E := \E_{(y_1,1)}^{(1,0)} \oplus \E_{(y_2,1)}^{(1,0)} \oplus \E_{(y_3,1)}^{(1,0)}$, with $y_1,  y_2, y_3$ pairwise distinct, then
\begin{align*}
\mathcal{V}_{x,1}(\E) &=  \Big\{ \big(\E, \E_{(z,1)}^{(3,-1)}, q^2 -2 q + 1 \big),
\big(\E, \E_{(y_i - x,1)}^{(1,-1)} \oplus \E_{(y_j,1)}^{(1,0)} \oplus \E_{(y_k,1)}^{(1,0)},  1 \big), \\
& \hspace*{0.8cm} \big(\E, \E_{(z_i,1)}^{(2,-1)}  \oplus \E_{(y_i,1)}^{(1,0)},  q-1 \big) \Big\}
\end{align*}
where $z = y_1 +  y_2 + y_3 - x $  and $z_i = y_j+y_k - x $ for $i,j,k = 1,2,3$ are such that $\{ i,j,k \} = \{1,2,3\}$.\\

\item \label{i20} Let $\E := \E_{(y,1)}^{(1,0)} \oplus \E_{(y_1,1)}^{(1,0)} \oplus \E_{(y_2,1)}^{(1,1)}$, then
\begin{align*}
\mathcal{V}_{x,1}(\E) &=  \Big\{ \big(\E, \E_{(y,1)}^{(1,0)} \oplus \E_{(y_1,1)}^{(1,0)} \oplus \E_{(y_2-x,1)}^{(1,0)}, q^2  \big),
\big(\E, \E_{(y-x,1)}^{(1,-1)} \oplus \E_{(y_1,1)}^{(1,0)} \oplus \E_{(y_2,1)}^{(1,1)},  1 \big),  \\
& \hspace*{0.8cm} \big(\E, \E_{(y_1-x,1)}^{(1,-1)} \oplus \E_{(y,1)}^{(1,0)} \oplus \E_{(y_2,1)}^{(1,1)},  1 \big),
\big(\E, \E_{(y+y_1-x,1)}^{(2,-1)} \oplus  \E_{(y_2,1)}^{(1,1)},  q-1 \big) \Big\}
\end{align*}
if $y,y_1, y_2 -x$ are pairwise distinct,
\begin{align*}
\mathcal{V}_{x,1}(\E) &=  \Big\{ \big(\E, \E_{(y,1)}^{(1,0)} \oplus \E_{(y_2,2)}^{(2,0)}, q^2-q  \big),
\big(\E, \E_{(y,1)}^{(1,0)} \oplus \E_{(y_1,1)}^{(1,0)} \oplus \E_{(y_1,1)}^{(1,0)},  q  \big) \\
& \hspace*{0.8cm} \big(\E, \E_{(y-x,1)}^{(1,-1)} \oplus \E_{(y_1,1)}^{(1,0)} \oplus \E_{(y_2,1)}^{(1,1)},  1 \big),
\big(\E, \E_{(y_1-x,1)}^{(1,-1)} \oplus \E_{(y,1)}^{(1,0)} \oplus \E_{(y_2,1)}^{(1,1)},  1 \big), \\
& \hspace*{0.8cm} \big(\E, \E_{(y+y_1-x,1)}^{(2,-1)} \oplus  \E_{(y_2,1)}^{(1,1)},  q-1 \big) \Big\}
\end{align*}
if $y \neq y_1$ and $y_2 = y_1 +x$,
\begin{align*}
\mathcal{V}_{x,1}(\E) &=  \Big\{ \big(\E, \E_{(y,1)}^{(1,0)} \oplus \E_{(y,2)}^{(2,0)}, q^2-1  \big),
\big(\E, \E_{(y,1)}^{(1,0)} \oplus \E_{(y,1)}^{(1,0)} \oplus \E_{(y,1)}^{(1,0)},  1  \big) \\
& \hspace*{0.8cm}
\big(\E, \E_{(y-x,1)}^{(1,-1)} \oplus \E_{(y,1)}^{(1,0)} \oplus \E_{(y_2,1)}^{(1,1)},  q+1 \big) \Big\}
\end{align*}
if $y=y_1$ and $y_2 = y +x$, and
\begin{align*}
\mathcal{V}_{x,1}(\E) &=  \Big\{ \big(\E,  \E_{(y,1)}^{(1,0)} \oplus \E_{(y,1)}^{(1,0)} \oplus \E_{(y_2-x,1)}^{(1,0)}, q^2  \big),
\big(\E, \E_{(y-x,1)}^{(1,-1)} \oplus \E_{(y,1)}^{(1,0)} \oplus \E_{(y_2,1)}^{(1,1)},  q+1 \big) \Big\}
\end{align*}
if $y=y_1$ and $y_2 \neq y +x$.\\

\item \label{i21} Let $\E := \E_{(y,1)}^{(1,0)} \oplus \E_{(y_1,1)}^{(1,0)} \oplus \E_{(y_2,1)}^{(1,-1)}$, then
\begin{align*}
\mathcal{V}_{x,1}(\E) &=  \Big\{ \big(\E, \E_{(y,1)}^{(1,0)} \oplus \E_{(y-x,1)}^{(1,-1)} \oplus \E_{(y_2,1)}^{(1,-1)}, q^2+q  \big),
\big(\E, \E_{(y_2-x,1)}^{(1,-2)} \oplus \E_{(y,1)}^{(1,0)} \oplus \E_{(y,1)}^{(1,0)},  1 \big)\Big\}
\end{align*}
if $y=y_1$ and $y \neq y_2 +x$;
\begin{align*}
\mathcal{V}_{x,1}(\E) &=  \Big\{ \big(\E, \E_{(y_2,2)}^{(2,-2)} \oplus \E_{(y,1)}^{(1,0)} , q^2 -1  \big),
\big(\E, \E_{(y_2,1)}^{(1,-1)} \oplus \E_{(y_2,1)}^{(1,-1)} \oplus \E_{(y,1)}^{(1,0)},  q+1 \big) \\
& \hspace*{0.8cm} \big(\E, \E_{(y_2-x,1)}^{(1,-2)} \oplus \E_{(y,1)}^{(1,0)} \oplus \E_{(y,1)}^{(1,0)},  1 \big)\Big\}
\end{align*}
if $y=y_1$ and $y = y_2 +x$;
\begin{align*}
\mathcal{V}_{x,1}(\E) &=  \Big\{ \big(\E, \E_{(y_2,1)}^{(1,-1)} \oplus \E_{(y+y_1-x,1)}^{(2,-1)} , q^2 -q  \big),
\big(\E, \E_{(y_2,1)}^{(1,-1)} \oplus \E_{(y_1 -x,1)}^{(1,-1)} \oplus \E_{(y,1)}^{(1,0)},  q \big) \\
& \hspace*{0.8cm}
\big(\E, \E_{(y_2,1)}^{(1,-1)} \oplus \E_{(y -x,1)}^{(1,-1)} \oplus \E_{(y_1,1)}^{(1,0)},  q \big)
\big(\E, \E_{(y_2-x,1)}^{(1,-2)} \oplus \E_{(y,1)}^{(1,0)} \oplus \E_{(y_1,1)}^{(1,0)},  1 \big)\Big\}
\end{align*}
if $y_2, y_1-x, y-x$ are pairwise distinct;
\begin{align*}
\mathcal{V}_{x,1}(\E) &=  \Big\{  \big(\E, \E_{(y_2,1)}^{(1,-1)} \oplus \E_{(y_1 -x,1)}^{(1,-1)} \oplus \E_{(y,1)}^{(1,0)},  q \big),
\big(\E, \E_{(y_2,2)}^{(2,-2)} \oplus \E_{(y_1,1)}^{(1,0)} , q -1  \big), \\
& \hspace*{0.8cm} \big(\E, \E_{(y_2,1)}^{(1,-1)} \oplus \E_{(y+y_1-x,1)}^{(2,-1)} , q^2 -q  \big),
\big(\E, \E_{(y_2,1)}^{(1,-1)} \oplus \E_{(y_2,1)}^{(1,-1)} \oplus \E_{(y_1,1)}^{(1,0)},  1 \big) , \\
& \hspace*{0.8cm}  \big(\E, \E_{(y_2-x,1)}^{(1,-2)} \oplus \E_{(y,1)}^{(1,0)} \oplus \E_{(y_1,1)}^{(1,0)},  1 \big) \Big\}
\end{align*}
if $y \neq y_1$ and $y_1 \neq y_2 + x$;
\begin{align*}
\mathcal{V}_{x,1}(\E) &=  \Big\{  \big(\E, \E_{(y_2,1)}^{(1,-1)} \oplus \E_{(y -x,1)}^{(1,-1)} \oplus \E_{(y,1)}^{(1,0)},  q \big),
\big(\E, \E_{(y_2,2)}^{(2,-2)} \oplus \E_{(y,1)}^{(1,0)} , q -1  \big), \\
& \hspace*{0.8cm} \big(\E, \E_{(y_2,1)}^{(1,-1)} \oplus \E_{(y+y_1-x,1)}^{(2,-1)} , q^2 -q  \big),
\big(\E, \E_{(y_2,1)}^{(1,-1)} \oplus \E_{(y_2,1)}^{(1,-1)} \oplus \E_{(y,1)}^{(1,0)},  1 \big) , \\
& \hspace*{0.8cm}  \big(\E, \E_{(y_2-x,1)}^{(1,-2)} \oplus \E_{(y,1)}^{(1,0)} \oplus \E_{(y_1,1)}^{(1,0)},  1 \big) \Big\}
\end{align*}
if $y \neq y_1$ and $y \neq y_2 + x$.  \\

\item \label{i22} Let $\E := \E_{(y,1)}^{(1,0)} \oplus \E_{(y_1,1)}^{(1,0)} \oplus \E_{(y_2,1)}^{(1,d)}$ for $|d| \geq 2$, then
\begin{align*}
\mathcal{V}_{x,1}(\E) &=  \Big\{ \big(\E, \E_{(y,1)}^{(1,0)} \oplus \E_{(y_1,1)}^{(1,0)} \oplus \E_{(y_2-x,1)}^{(1,d-1)}, q^2 \big),
\big(\E, \E_{(y-x,1)}^{(1,-1)} \oplus \E_{(y_1,1)}^{(1,0)} \oplus \E_{(y_2,1)}^{(1,d)},  1 \big),\\
& \hspace*{0.8cm} \big(\E, \E_{(y,1)}^{(1,0)} \oplus \E_{(y_1-x,1)}^{(1,-1)} \oplus \E_{(y_2,1)}^{(1,d)},  1 \big),
\big(\E, \E_{(y+y_1-x,1)}^{(2,-1)} \oplus \E_{(y_2,1)}^{(1,d)},  q-1 \big)
\Big\}
\end{align*}
if $d \geq 2$ and $y \neq y_1$;
\begin{align*}
\mathcal{V}_{x,1}(\E) &=  \Big\{ \big(\E, \E_{(y,1)}^{(1,0)} \oplus \E_{(y_1,1)}^{(1,0)} \oplus \E_{(y_2-x,1)}^{(1,d-1)}, q^2 \big),
\big(\E, \E_{(y-x,1)}^{(1,-1)} \oplus \E_{(y,1)}^{(1,0)} \oplus \E_{(y_2,1)}^{(1,d)},  q+1 \big)
\Big\}
\end{align*}
if $d \geq 2$ and $y = y_1$;
\begin{align*}
\mathcal{V}_{x,1}(\E) &=  \Big\{ \big(\E, \E_{(y,1)}^{(1,0)} \oplus \E_{(y_1,1)}^{(1,0)} \oplus \E_{(y_2-x,1)}^{(1,d-1)}, 1 \big),
\big(\E, \E_{(y-x,1)}^{(1,-1)} \oplus \E_{(y_1,1)}^{(1,0)} \oplus \E_{(y_2,1)}^{(1,d)},  q \big),\\
& \hspace*{0.8cm} \big(\E, \E_{(y,1)}^{(1,0)} \oplus \E_{(y_1-x,1)}^{(1,-1)} \oplus \E_{(y_2,1)}^{(1,d)},  q \big),
\big(\E, \E_{(y+y_1-x,1)}^{(2,-1)} \oplus \E_{(y_2,1)}^{(1,d)},  q^2-q \big)
\Big\}
\end{align*}
if $d \leq - 2$ and $y \neq y_1$;
\begin{align*}
\mathcal{V}_{x,1}(\E) &=  \Big\{ \big(\E, \E_{(y,1)}^{(1,0)} \oplus \E_{(y_1,1)}^{(1,0)} \oplus \E_{(y_2-x,1)}^{(1,d-1)}, 1 \big),
\big(\E, \E_{(y-x,1)}^{(1,-1)} \oplus \E_{(y,1)}^{(1,0)} \oplus \E_{(y_2,1)}^{(1,d)},  q^2+q \big)
\Big\}
\end{align*}
if $d \leq -2$ and $y = y_1$. \\

\item \label{i23} Let $\E := \E_{(y,1)}^{(1,0)} \oplus \E_{(y_1,1)}^{(1,1)} \oplus \E_{(y_2,1)}^{(1,d)}$ for $d \leq -1$ or $d \geq 2$, then
\begin{align*}
\mathcal{V}_{x,1}(\E) &=  \Big\{ \big(\E, \E_{(y,1)}^{(1,0)} \oplus \E_{(y_1,1)}^{(1,1)} \oplus \E_{(y_2-x,1)}^{(1,d-1)}, q^2 \big),
\big(\E, \E_{(y-x,1)}^{(1,-1)} \oplus \E_{(y_1,1)}^{(1,1)} \oplus \E_{(y_2,1)}^{(1,d)},  1 \big),\\
& \hspace*{0.8cm} \big(\E, \E_{(y,1)}^{(1,0)} \oplus \E_{(y_1-x,1)}^{(1,0)} \oplus \E_{(y_2,1)}^{(1,d)},  q \big)
\Big\}
\end{align*}
if $d > 2$ and $y \neq y_1-x$;
\begin{align*}
\mathcal{V}_{x,1}(\E) &=  \Big\{ \big(\E, \E_{(y,1)}^{(1,0)} \oplus \E_{(y_1,1)}^{(1,1)} \oplus \E_{(y_2-x,1)}^{(1,d-1)}, q^2 \big),
\big(\E, \E_{(y-x,1)}^{(1,-1)} \oplus \E_{(y_1,1)}^{(1,1)} \oplus \E_{(y_2,1)}^{(1,d)},  1 \big),\\
& \hspace*{0.8cm} \big(\E, \E_{(y,1)}^{(1,0)} \oplus \E_{(y_1-x,1)}^{(1,0)} \oplus \E_{(y_2,1)}^{(1,d)},  1 \big),
\big(\E, \E_{(y,2)}^{(2,0)} \oplus  \E_{(y_2,1)}^{(1,d)},  q-1 \big)
\Big\}
\end{align*}
if $d > 2$ and $y = y_1-x$;
\begin{align*}
\mathcal{V}_{x,1}(\E) &=  \Big\{ \big(\E, \E_{(y,1)}^{(1,0)} \oplus \E_{(y_1,1)}^{(1,1)} \oplus \E_{(y_2-x,1)}^{(1,d-1)}, 1 \big),
\big(\E, \E_{(y-x,1)}^{(1,-1)} \oplus \E_{(y_1,1)}^{(1,1)} \oplus \E_{(y_2,1)}^{(1,d)},  q \big),\\
& \hspace*{0.8cm} \big(\E, \E_{(y,1)}^{(1,0)} \oplus \E_{(y_1-x,1)}^{(1,0)} \oplus \E_{(y_2,1)}^{(1,d)},  q^2 \big)
\Big\}
\end{align*}
if $d < -1$ and $y \neq  y_1-x$;
\begin{align*}
\mathcal{V}_{x,1}(\E) &=  \Big\{ \big(\E, \E_{(y,1)}^{(1,0)} \oplus \E_{(y_1,1)}^{(1,1)} \oplus \E_{(y_2-x,1)}^{(1,d-1)}, 1 \big),
\big(\E, \E_{(y-x,1)}^{(1,-1)} \oplus \E_{(y_1,1)}^{(1,1)} \oplus \E_{(y_2,1)}^{(1,d)},  q \big),\\
& \hspace*{0.8cm} \big(\E, \E_{(y,1)}^{(1,0)} \oplus \E_{(y_1-x,1)}^{(1,0)} \oplus \E_{(y_2,1)}^{(1,d)},  q \big),
\big(\E, \E_{(y,2)}^{(2,0)} \oplus  \E_{(y_2,1)}^{(1,d)},  q^2-q \big)
\Big\}
\end{align*}
if $d < -1$ and $y = y_1-x$;
\begin{align*}
\mathcal{V}_{x,1}(\E) &=  \Big\{ \big(\E, \E_{(y,1)}^{(1,0)} \oplus \E_{(y_1,1)}^{(1,1)} \oplus \E_{(y_2-x,1)}^{(1,d-1)}, q^2 \big),
\big(\E, \E_{(y-x,1)}^{(1,-1)} \oplus \E_{(y_1,1)}^{(1,1)} \oplus \E_{(y_2,1)}^{(1,d)},  1 \big),\\
& \hspace*{0.8cm} \big(\E, \E_{(y,1)}^{(1,0)} \oplus \E_{(y_1-x,1)}^{(1,0)} \oplus \E_{(y_2,1)}^{(1,d)},  q \big)
\Big\}
\end{align*}
if $d =2$, $y_1 \neq  y + x$ and $y_1 \neq y_2 -x$;
\begin{align*}
\mathcal{V}_{x,1}(\E) &=  \Big\{ \big(\E, \E_{(y,1)}^{(1,0)} \oplus \E_{(y_1,1)}^{(1,1)} \oplus \E_{(y_2-x,1)}^{(1,d-1)}, q \big),
\big(\E, \E_{(y-x,1)}^{(1,-1)} \oplus \E_{(y_1,1)}^{(1,1)} \oplus \E_{(y_2,1)}^{(1,d)},  1 \big),\\
& \hspace*{0.8cm} \big(\E, \E_{(y,1)}^{(1,0)} \oplus \E_{(y_1-x,1)}^{(1,0)} \oplus \E_{(y_2,1)}^{(1,d)},  q \big),
\big(\E, \E_{(y,1)}^{(1,0)} \oplus  \E_{(y_1,2)}^{(2,2)} ,  q^2-q \big)
\Big\}
\end{align*}
if $d =2$, $y_1 \neq  y + x$ and $y_1 = y_2 -x$;
\begin{align*}
\mathcal{V}_{x,1}(\E) &=  \Big\{ \big(\E, \E_{(y,1)}^{(1,0)} \oplus \E_{(y_1,1)}^{(1,1)} \oplus \E_{(y_2-x,1)}^{(1,d-1)}, q^2 \big),
\big(\E, \E_{(y-x,1)}^{(1,-1)} \oplus \E_{(y_1,1)}^{(1,1)} \oplus \E_{(y_2,1)}^{(1,d)},  1 \big),\\
& \hspace*{0.8cm} \big(\E, \E_{(y,1)}^{(1,0)} \oplus \E_{(y_1-x,1)}^{(1,0)} \oplus \E_{(y_2,1)}^{(1,d)},  1 \big),
\big(\E, \E_{(y,2)}^{(2,0)} \oplus \E_{(y_2,1)}^{(1,d)}  ,  q-1 \big)
\Big\}
\end{align*}
if $d =2$, $y_1 = y + x$ and $y_1 \neq y_2 -x$;
\begin{align*}
\mathcal{V}_{x,1}(\E) &=  \Big\{ \big(\E, \E_{(y,1)}^{(1,0)} \oplus \E_{(y_1,1)}^{(1,1)} \oplus \E_{(y_2-x,1)}^{(1,d-1)}, q \big),
\big(\E, \E_{(y-x,1)}^{(1,-1)} \oplus \E_{(y_1,1)}^{(1,1)} \oplus \E_{(y_2,1)}^{(1,d)},  1 \big),\\
& \hspace*{0.8cm} \big(\E, \E_{(y,1)}^{(1,0)} \oplus \E_{(y_1-x,1)}^{(1,0)} \oplus \E_{(y_2,1)}^{(1,d)},  1 \big),
\big(\E, \E_{(y,2)}^{(2,0)} \oplus \E_{(y_2,1)}^{(1,d)}  ,  q-1 \big), \\
& \hspace*{0.8cm} \big(\E, \E_{(y,1)}^{(1,0)} \oplus  \E_{(y_1,2)}^{(2,2)} ,  q^2-q \big)
\Big\}
\end{align*}
if $d =2$, $y_1 = y + x$ and $y_1 = y_2 -x$;

\begin{align*}
\mathcal{V}_{x,1}(\E) &=  \Big\{ \big(\E, \E_{(y,1)}^{(1,0)} \oplus \E_{(y_1,1)}^{(1,1)} \oplus \E_{(y_2-x,1)}^{(1,d-1)}, 1 \big),
\big(\E, \E_{(y-x,1)}^{(1,-1)} \oplus \E_{(y_1,1)}^{(1,1)} \oplus \E_{(y_2,1)}^{(1,d)},  q \big),\\
& \hspace*{0.8cm} \big(\E, \E_{(y,1)}^{(1,0)} \oplus \E_{(y_1-x,1)}^{(1,0)} \oplus \E_{(y_2,1)}^{(1,d)},  q^2  \big)
\Big\}
\end{align*}
if $d =-1$, $y \neq  y_2 + x$ and $y \neq y_1 -x$;
\begin{align*}
\mathcal{V}_{x,1}(\E) &=  \Big\{ \big(\E, \E_{(y,1)}^{(1,0)} \oplus \E_{(y_1,1)}^{(1,1)} \oplus \E_{(y_2-x,1)}^{(1,d-1)}, 1 \big),
\big(\E, \E_{(y-x,1)}^{(1,-1)} \oplus \E_{(y_1,1)}^{(1,1)} \oplus \E_{(y_2,1)}^{(1,d)},  q \big),\\
& \hspace*{0.8cm} \big(\E, \E_{(y,1)}^{(1,0)} \oplus \E_{(y_1-x,1)}^{(1,0)} \oplus \E_{(y_2,1)}^{(1,d)},  q \big),
\big(\E, \E_{(y_2,1)}^{(1,d)} \oplus  \E_{(y,2)}^{(2,0)} ,  q^2-q \big)
\Big\}
\end{align*}
if $d =-1$, $y \neq  y_2 + x$ and $y = y_1 -x$;
\begin{align*}
\mathcal{V}_{x,1}(\E) &=  \Big\{ \big(\E, \E_{(y,1)}^{(1,0)} \oplus \E_{(y_1,1)}^{(1,1)} \oplus \E_{(y_2-x,1)}^{(1,d-1)}, 1 \big),
\big(\E, \E_{(y-x,1)}^{(1,-1)} \oplus \E_{(y_1,1)}^{(1,1)} \oplus \E_{(y_2,1)}^{(1,d)},  1 \big),\\
& \hspace*{0.8cm} \big(\E, \E_{(y,1)}^{(1,0)} \oplus \E_{(y_1-x,1)}^{(1,0)} \oplus \E_{(y_2,1)}^{(1,d)},  q^2 \big),
\big(\E, \E_{(y_2,2)}^{(2,-2)} \oplus \E_{(y_1,1)}^{(1,1)}  ,  q-1 \big)
\Big\}
\end{align*}
if $d =-1$, $y =  y_2 + x$ and $y \neq y_1 -x$;
\begin{align*}
\mathcal{V}_{x,1}(\E) &=  \Big\{ \big(\E, \E_{(y,1)}^{(1,0)} \oplus \E_{(y_1,1)}^{(1,1)} \oplus \E_{(y_2-x,1)}^{(1,d-1)}, 1 \big),
\big(\E, \E_{(y-x,1)}^{(1,-1)} \oplus \E_{(y_1,1)}^{(1,1)} \oplus \E_{(y_2,1)}^{(1,d)},  1 \big),\\
& \hspace*{0.8cm} \big(\E, \E_{(y,1)}^{(1,0)} \oplus \E_{(y_1-x,1)}^{(1,0)} \oplus \E_{(y_2,1)}^{(1,d)},  q \big),
\big(\E, \E_{(y_2,2)}^{(2,-2)} \oplus \E_{(y_1,1)}^{(1,1)}, q-1 \big) ,\\
& \hspace*{0.8cm}  \big(\E, \E_{(y_2,1)}^{(1,d)} \oplus  \E_{(y,2)}^{(2,0)} ,  q^2-q \big)
\Big\}
\end{align*}
if $d =-1$, $y =  y_2 + x$ and $y = y_1 -x$.\\

\item \label{i24} Let $\E := \E_{(y,1)}^{(1,d)} \oplus \E_{(y_1,1)}^{(1,d_1)} \oplus \E_{(y_2,1)}^{(1,d_2)}$ with  $d_2 -d_1 > 1$ and $d_1 -d > 1$, then
\begin{align*}
\mathcal{V}_{x,1}(\E) &=  \Big\{ \big(\E, \E_{(y-x,1)}^{(1,d-1)} \oplus \E_{(y_1,1)}^{(1,d_1)} \oplus \E_{(y_2,1)}^{(1,d_2)}, 1 \big),
\big(\E, \E_{(y,1)}^{(1,d)} \oplus \E_{(y_1-x,1)}^{(1,d_1-1)} \oplus \E_{(y_2,1)}^{(1,d_2)},  q \big),\\
& \hspace*{0.8cm} \big(\E, \E_{(y,1)}^{(1,d)} \oplus \E_{(y_1,1)}^{(1,d_1)} \oplus \E_{(y_2-x,1)}^{(1,d_2 -1)},  q^2 \big)
\Big\}.
\end{align*}


\end{enumerate}

\end{thm}

\begin{proof} The proof is based on the algorithm developed in \cite{roberto-elliptic}. For better readability  we separated the proof in several lemmas and present it in the next section.

Item \ref{i1} follows from \cite[Thm. 5.3]{roberto-elliptic}.
Item \ref{i2} follows from \cite[Thm. 5.3]{roberto-elliptic} and Lemmas \ref{idX} and \ref{idL}.
Item \ref{i3} follows from Lemmas \ref{idC}, and \ref{idB}, \cite[Thm. 5.1]{roberto-elliptic}, and Lemmas \ref{idE} and \ref{idF}.
Item \ref{i4} follows from Lemmas \ref{idA.3} and \ref{idS}.
Item \ref{i5} follows from  \cite[Thm. 5.3]{roberto-elliptic} and Lemmas \ref{idX}, \ref{idL} and \ref{idK}.
Item \ref{i6} follows from Lemmas \ref{idI}, \ref{idX} and \ref{idL}.
Item \ref{i7} follows from \cite[Lemma 6.7]{roberto-elliptic} joint with Theorem \ref{thm1}, \cite[Lemma 6.8]{roberto-elliptic} joint with Theorem \ref{thm1} and Lemma \ref{idF}.
Item \ref{i8} follows from \cite[Lemma 6.7]{roberto-elliptic} joint with Theorem \ref{thm1}, \cite[Lemma 6.8]{roberto-elliptic} joint with Theorem \ref{thm1}, Lemmas \ref{idC} and \ref{idF2}.
Item \ref{i9} follows from Theorem \ref{thm1} \ref{thm1-iv} joint with  \cite[Lemma 6.7]{roberto-elliptic}, Theorem \ref{thm1} \ref{thm1-iii}  and  Lemma \ref{idC.0}.
Item \ref{i10} follows from Theorem \ref{thm1} \ref{thm1-iv} joint with \cite[Lemma 6.7]{roberto-elliptic}, Lemmas \ref{idB.0} and \ref{Prop26/01}, and  Theorem \ref{thm1} \ref{thm1-iii}.

For \ref{i11} observes that if $d \geq 2$ (resp. $d \leq 2$), then $\E$ satisfies the hypothesis of Corollary \ref{cor-thm1} \ref{cor-thm1-i} (resp. \ref{cor-thm1-ii}) and thus \cite[Lemmas 6.7, 6.8]{roberto-elliptic} completes the proof.

For \ref{i12} observe that since $y \neq y'$, Atiyah's classification \ref{atiyahtheorem} implies
$$\Ext\big(\E_{(y',1)}^{(1,0)},\E_{(y,1)}^{(2,0)} \big) = \Ext\big(\E_{(y'-x,1)}^{(1,-1)},\E_{(y,1)}^{(2,0)}\big)  = 0.$$
Hence $\E$ satisfy the hypothesis of Lemma \ref{lem1} \ref{lem1-iii} and thus Theorem \ref{thm1} \ref{thm1-iii} with  \cite[Thm. 5.3]{roberto-elliptic} and Lemma \ref{idX} complete the proof.

Item \ref{i13} follows from Lemma \ref{idE} and Theorem \ref{thm1} \ref{thm1-i} joint with \cite[Lemma 6.7]{roberto-elliptic} for $d=1$ and Corollary \ref{cor-thm1} for $d \geq 2$. Since $\E_{(y,1)}^{(2,0)} $ has only one neighbor in the rank $2$ graph $\mathscr{G}_{x,1}$ and it is indecomposable, Corollary \ref{cor-thm1} implies the case $d \leq -1.$

Item \ref{i14} follows from Lemmas \ref{idC}, \ref{idB}, \ref{idE}, \ref{idF}. \ref{idF2}, \ref{idG}, \ref{idPO}, \ref{idD}, \ref{idSA} and \cite[Thm. 5.1]{roberto-elliptic}.
Item \ref{i15} follows from Lemmas \ref{idN.1}, \ref{idPF}, \ref{idR.O}, \ref{idS}, \ref{idA.3} and \ref{idH}.
Item \ref{i16} follows from Corollary \ref{cor-thm1} \ref{cor-thm1-i} if $d \geq 2$ and Corollary \ref{cor-thm1} \ref{cor-thm1-ii} if $d \leq -1$.
Item \ref{i17} follows from Lemma \ref{idI}$\mathrm{(ii)}$.
Item \ref{i18} follows from Lemmas \ref{idI}$\mathrm{(i)}$, \ref{idK}$\mathrm{(iv)}$ and \ref{idX}$\mathrm{(vi)}$.
 Item \ref{i19} follows from \cite[Thm. 5.3]{roberto-elliptic} and Lemmas \ref{idK}$\mathrm{(v)}$ and \ref{idX}.
Item \ref{i20} follows from Lemmas \ref{idB}, \ref{idF}, \ref{idF2}, \ref{idG}, \ref{idD} and Theorem \ref{thm1} joint with \cite[Thm. 6.2]{roberto-elliptic}.
Item \ref{i21} follows from Lemmas \ref{idR20}, \ref{idN20}, \ref{idB.0}, \ref{Prop26/01}, \ref{idS}, \ref{idC20}, \ref{idE21} and by observing that
$m_{x,r}(\E,\E') = m_{x,r}(\E \otimes \Line,\E' \otimes \Line)$ for every $\Line \in \Pic X$.
Item \ref{i22} follows from Corollary \ref{cor-thm1} joint with \cite[Thm. 6.2]{roberto-elliptic}.
Item \ref{i23} for $d> 2$ or $d< -1$ follows from Corollary \ref{cor-thm1} coupled with \cite[Thm. 6.2]{roberto-elliptic}. For $d=-1, 2$ item \ref{i23} follows from Theorem \ref{thm1} joint with \cite[Thm. 6.2]{roberto-elliptic} and Lemmas \ref{idL}, \ref{idK}, \ref{idI}.
Item \ref{i24} follows from \cite[Thm. 5.2]{roberto-elliptic}.
\end{proof}


\begin{figure}[p]

\newgeometry{left=0mm, right=0mm}\begin{center}\hspace{-6.5cm}\includegraphics[angle=270,scale=0.95]{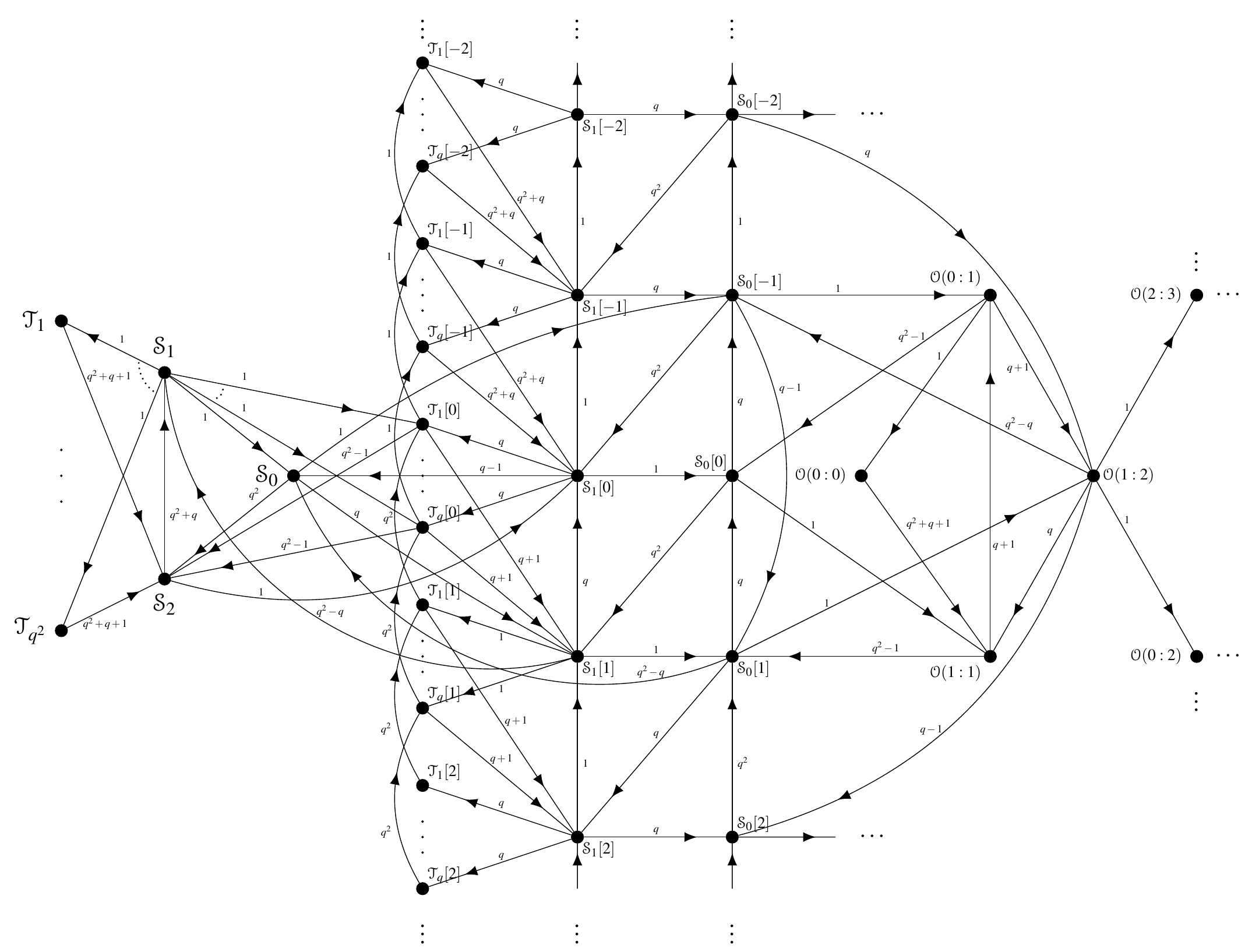}\end{center}\restoregeometry

\caption{The graph $\overline{\mathscr{G}}^{(3)}_{3,1}$ for the elliptic curve $E_q$ over $\Fq$ for $q=2,3,4$}
\label{fig: G31}

\end{figure}


Before we state and prove the lemmas that compose to the proof of our main theorem, we give an example that illustrates the graph of a Hecke operator for an elliptic curve with only one rational point over $\PGL(3)$.

\begin{ex}  \label{ex1}  In this example we describe the projective rank $3$ and degree $1$ graph $\overline{\mathscr{G}}_{x,1}^{(3)}$ for an elliptic curve with only one rational point $x$.
There are up to an isomorphism three such elliptic curves (see for example \cite[2.4.4 and Ex. 3 of 2.4]{serre} and \cite{oliver-gunther}): $E_2$ over $\mathbb{F}_2$ defined by the Weierstrass equation $y^2 + y = x^3 + x + 1,$ $E_{3}$ over $\mathbb{F}_3$ defined by the Weierstrass equation $y^2 = x^3 + 2x +2 $ and $E_{4}$ over $\mathbb{F}_4$ defined by the Weierstrass equation $y^2 + y = x^3 + \alpha$, where $\mathbb{F}_4 = \mathbb{F}_2(\alpha)$.

The number of traces of line bundles of constant field extensions of $E_q$ depends on the number $\FF_{q^n}$-rational points of $E_q$. A comparison of the coefficients of the zeta functions of $E_q$ and $E_{q,n}=E_q\times_{\FF_q}\FF_{q^n}$ shows that the cardinality of $E_q(\FF_{q^n})$ is determined by the cardinality of $E_q(\FF_q)$. Since the latter number is $1$, we obtain $\#E_q(\mathbb{F}_{q^2}) = 2q+1 $ and  $\#E_q(\mathbb{F}_{q^3}) = 3q^2+1$ in our situation; cf. \cite[Exerc. 5.13]{silverman} for details. We denote by $y_1, \ldots, y_q$ the $q$ closed points of $E_q$ of degree $2$ and by $z_1, \ldots, z_{q^2}$ the $q^2$ closed points of $E_q$ of degree $3$.

Since $\Pic X$ has trivial torsion, it is particularly easy to describe a system of representatives for $\PBun_3 X$. Namely, $\overline\cE=\overline{\cE'}$ and $\deg\cE=\deg\cE'$ imply that $\cE\simeq\cE'$ as vector bundles, which allows us to choose the following set of representatives. If $\cE$ is indecomposable, then $\overline\cE$ has a unique representative $\cE'$ with $0\leq\deg\cE'\leq 2$. If $\cE=\cM\oplus\cL$ for an indecomposable rank $2$ bundle $\cM$, then $\overline\cE$ has a unique representative of the form $\cM'\oplus\cL'$ with $0\leq\deg\cM\leq1$. If $\cE$ is the sum of three line bundles, then $\overline\cE$ has a unique representative of the form $\cO\oplus\cO(kx)\oplus\cO(lx)$ with $0\leq k\leq l$.

For better readability, we adopt the following notation for the vertices of  $\overline{\mathscr{G}}_{x,1}^{(3)}$:
 \[ \mathcal{T}_i := \big[\E_{(z_i,1)}^{(3,0)}\big] , \; i=1, \ldots, q^2, \quad \mathcal{S}_0 := \big[\E_{(x,3)}^{(3,0)}\big],  \quad \mathcal{S}_1 := \big[\E_{(x,1)}^{(3,1)}\big],
 \quad \mathcal{S}_2 := \big[\E_{(x,1)}^{(3,2)}\big], \]
 \[ \mathcal{T}_i[k] := \big[ \E_{(y_i,1)}^{(2,0)} \oplus \mathcal{O}_{E_q}(kx) \big], \; i=1, \ldots, q, \quad \mathcal{S}_0[k] := \big[\E_{(x,2)}^{(2,0)} \oplus \mathcal{O}_{E_q}(kx) \big],  \]
 \[  \mathcal{S}_1[k] := \big[\E_{(x,1)}^{(2,1)} \oplus  \mathcal{O}_{E_q}(kx) \big] \quad \text{ and } \quad \mathcal{O}(k:\ell) := \big[ \mathcal{O}_{E_q} \oplus \mathcal{O}_{E_q}(kx) \oplus \mathcal{O}_{E_q}( \ell x) \big], \]
where $k, \ell \in \Z.$

Finally we remark that $m_{x,1}(\overline\cE,\overline{\cE'})=m_{x,1}(\cE,\cE')$ since two subbundles $\cE'$ and $\cE''$ of $\cE$ with quotient $\cE/\cE'\simeq\cK_x\simeq\cE/\cE''$ have the same degree and therefore $\overline\cE=\overline{\cE'}$ implies that $\cE\simeq\cE'$ as vector bundles. Thanks to these deductions, the graph $\overline{\mathscr{G}}_{x,1}^{(3)}$ for $E_q$ is determined by Theorem \ref{maintheorem} and as illustrated in Figure \ref{fig: G31}.

\end{ex}

\section{Proof of the main theorem}
\label{sec-proofmainthm}

This section is devoted to the proof of Theorem \ref{maintheorem}. Throughout the section, $X$ is an elliptic curve with a fixed base point $x_0$ and $x$ stands for a second (possibly equal) degree $1$ place of $X$.


\subsection{Norm maps on the Picard group of \texorpdfstring{$X$}{X}} \label{subsec1-sec-proofmainthm} 

Let $X_n$ be the base field extension of $X$ to $\FF_{q^n}$. By the Riemann-Roch theorem, the map
\begin{eqnarray*}
X(\mathbb{F}_{q^n}) = X_n(\mathbb{F}_{q^n}) & \stackrel{1:1}{\longrightarrow}  & \Pic^{0}(X_n) \\
x  & \longmapsto & \mathcal{O}_{X_n}(x-x_0)
\end{eqnarray*}
is bijective. This allows us to transport the group  structure  from $\Pic^{0}(X_n) $ to $ X_n(\mathbb{F}_{q^n})$  such that $x_0$ is the neutral element. Moreover, the inclusions $ X(\mathbb{F}_{q^m}) \subseteq  X(\mathbb{F}_{q^n})$ for $m|n$ are compatible with the respective group structures. To shorten notation,  we let $N_n$ stand for $|X(\mathbb{F}_{q^n})|$, the number of rational points of $X$ over $\mathbb{F}_{q^n}$.

We consider the (geometric) Frobenius endomorphism $\Fr_X$ on the scheme $(X, \mathcal{O}_X)$, which is the identity at the level of (topological) points and the raising at the $q$-th power at the level of functions. 
The Frobenius endomorphism $\Fr_X$ acts on the set $X(\mathbb{F}_{q^n})$ of $\mathbb{F}_{q^n}$-rational points of $X$ by composition $\Fr_X \circ x : \mathrm{Spec}(\mathbb{F}_{q^n}) \rightarrow X,$ where $x \in X(\mathbb{F}_{q^n})$. This action is compatible with the group structure of $X(\mathbb{F}_{q^n})$.
If $m|n$ are positive integers, then we have the identification $X(\mathbb{F}_{q^m}) = X(\mathbb{F}_{q^n})^{\Fr_{X}^m}.$ In particular $X(\mathbb{F}_{q^n}) = X(\overline{\Fq})^{\Fr_{X}^{n}}.$ We  can also identify the set  $|X|$  with $X(\overline{\Fq})\big/ \Fr_X$, where the quotient means that two $\overline{\mathbb{F}_q}$-rational points of $X$ are identified if they have the same orbit under the Frobenius action. Similarly, we can identify $|X_n|$ with $X(\overline{\Fq})\big/ \Fr_{X}^n.$ These two equalities allow us to define an action of $\Fr_X$ on the closed points of $X_n.$

 The above isomorphism $X(\mathbb{F}_{q^n})  \cong \Pic^{0}(X_n)$  yields the commutative diagram
\[ 
 \xymatrix@R5pt@C7pt{ \Pic^{0}(X_n) \ar[dd]^{\wr} \ar[rr]^{\Fr_{X,n}^{*}} & & \Pic^{0}(X_n) \ar[dd]^{\wr}\\
                      & & \\
                      X(\mathbb{F}_{q^n}) \ar[rr]^{\Fr_{X}} & & X(\mathbb{F}_{q^n}).
                      } 
\]
Let $m,n$ be positive integers such that $m|n.$ We define the relative norm map
\[ 
 \Norm_{m}^{n}:  \Pic^{0}(X_n) \longrightarrow \Pic^{0}(X_m) 
\]
by
\[ 
 \Norm_{m}^{n}(\Line) := \bigotimes_{i=1}^{n/m} (\Fr_{X,n}^{*})^{mi}(\Line) .
\]
The fact that this map is well defined follows from Galois descent \cite[Proposition 3.4]{arason}.


\subsection{Characters on the Picard group of \texorpdfstring{$X$}{X}} \label{subsec2-sec-proofmainthm} 
For a finite abelian group $G$, we denote by $\widehat{G}$ its characters group.
The Frobenius $\Fr_{X,n}^{*}$  acts by duality on each group {\small $\widehat{\Pic^{0}(X_n)}$} and we will denote this action simply by $\Fr_{X,n}.$ 

We  denote {\small $\widehat{\Pic^{0}(X_n)}$}  \big(resp.  {\small $\widehat{\Pic^{0}(X_n)} \big/ \Fr_{X,n}$} \big)  shortly by $\mathrm{P}_n$    (resp. $\widetilde{\mathrm{P}}_n$). For a character $\rho \in \mathrm{P}_n$, we denote by $\widetilde{\rho}$  its class on $\widetilde{\mathrm{P}}_n$. Moreover, $\widetilde{X}$ stands for  $\bigsqcup_{n } \widetilde{\mathrm{P}}_n$ and we shall abuse the language and call the elements on $\widetilde{X}$  by characters even though they are actually orbits of characters.  

By dualizing the norm maps  we also obtain the relative norm maps between the character groups for which we will use the same notation
\[ \Norm_{m}^{n}:  \widehat{\Pic^{0}(X_m)} \longrightarrow \widehat{\Pic^{0}(X_n)}. \]

Since we have fixed a rational point $x_0$ on $X$, for any integers $n \geq 1$ and $d \in \Z$ we can identify $\Pic^{d}(X_n)$ with $\Pic^{0}(X_n)$ by subtract $d x_0.$ This allows us to extend any character $\rho$ of $\Pic^{0}(X_n)$ to a character of $\Pic(X_n)$ by putting $\rho (\mathcal{O}_{X_n}(x_0))=1.$ Unless otherwise specified we will view the characters of $\Pic^{0}(X_n)$ as characters of $\Pic(X_n).$


\subsection{A grading on the elliptic Hall algebra} The Grothendieck group $K_0(X)$ of $\mathrm{Coh}(X)$ is isomorphic to $\Z \oplus \Pic(X)$ via the map $\F \mapsto \big(\mathrm{rk}(\F), \mathrm{det}(\F)\big)$. Moreover, if we compose this morphism with the one sending a line bundle to its degree, then we get a map
\begin{eqnarray*}
\varphi : \Coh(X) &\longrightarrow& K_{0}'(X):=\Z^2 \\
\F &\longmapsto& \big(\mathrm{rk}(\F), \deg(\F)\big).
\end{eqnarray*}
into the numerical Grothendieck group $K_{0}'(X)$. Its image is
$$\mathbf{Z} := \big\{(n,d) \in K_{0}'(X) \big| n> 0 \text{ or } n=0 \text{ and } d \geq 0\big\}.$$
The Hall algebra $\mathsf{H}_X$ is naturally graded over $\mathbf Z$ with homogeneous parts $\mathsf{H}_X[\mathbf{v}] := \bigoplus_{[\F]=\mathbf{v}} \C\cdot \F$ for $\mathbf{v} \in \mathbf{Z}.$ 
If $\mathbf{v}=(n,d) \in \mathbf{Z}$, we denote by $\gamma(\mathbf{v})$ the integer $\gcd(n,d).$

We shall denote by $\mathsf{H}_{X}^{\mathrm{tor}} := \bigoplus_{d \geq 0} \mathsf{H}_X[0,d]$ the subalgebra of torsion sheaves and by $\mathsf{H}_{X}^{\mathrm{vec}} := \bigoplus_{(n,d) \in \mathbf{Z}, n>0} \mathsf{H}_X[n,d]$ the subalgebra of vector bundles. We denote by $\pi^{\mathrm{vec}}$ the projection map $\mathsf{H}_X \rightarrow \mathsf{H}_{X}^{\mathrm{vec}}.$


\subsection{A special set of generators for the elliptic Hall algebra}
Next we introduce a special element of the Hall algebra of $X$ that will be essential to the proof of the following lemmas. Let $y$ be a closed point of $X$ and $m$ a positive integer.  We set
$$[m] := \frac{v^m - v^{-m}}{v - v^{-1}} \qquad \text{ and }  \qquad  c_{i} := \frac{v^i [i] N_i}{i},$$
 where $v^2 = q^{-1}$ and $N_i := |X(\mathbb{F}_{q^i})|.$
Let $\lambda=(\lambda_1,\dotsc,\lambda_l)$ be a partition of length $l(\lambda)=l$ and weight $|\lambda|=\lambda_1+\dotsb+\lambda_l$.
We define 
\[T_{(0,m),y} := \begin{cases} \frac{[m]|y|}{m} \displaystyle\sum_{|\lambda|=m/|y|} n_{u_y}(l(\lambda)-1)\mathcal{K}_{y}^{(\lambda)} & \text{ if $|y|$ divides $ m$} \smallskip \\
0 & \text{ otherwise,} 
\end{cases}\]
where $n_{u_y}(l)= \prod_{i=1}^{l}(1-u_{y}^{-2i})$, $u_y$ is a square root of  $q_{y}^{-1}$ and $\mathcal{K}_{y}^{(\lambda)}$ the unique torsion sheaf supported at $y$ associated to $\lambda$.

For any $\mu \in \Q \cup \{\infty \}$, we consider the subspace $\mathsf{H}_{X}^{(\mu)} \subset \mathsf{H}_X$ that is spanned by classes $\big\{\mathcal{F} \;|\; \F \in \mathsf{C}_{\mu}\big\}.$ Since the category $\mathsf{C}_{\mu}$ is stable under extensions, $\mathsf{H}_{X}^{(\mu)}$ is a subalgebra of $\mathsf{H}_X.$ 
The Hall algebra $\mathsf{H}_{X}$ is isomorphic to the restrict tensor product of the subalgebras $\mathsf{H}_{X}^{(\mu)}$ ordered from left to right in increasing order for $\mu \in \Q \cup \{\infty\}$, see \cite[Lemma 2.6]{olivier-elliptic1}.

The exact equivalence $\epsilon_{\mu,\nu}$ from Theorem \ref{atiyahtheorem} gives rise to an algebra isomorphism $\epsilon_{\mu,\nu} : \mathsf{H}_{X}^{(\nu)} \stackrel{\sim}{\rightarrow} \mathsf{H}_{X}^{(\mu)}.$
Note that $T_{(0,m),y} \in \mathsf{H}_{X}^{(\infty)}.$ For every $\mathbf{v} \in \mathbf{Z}$ with slope $\mu,$ we define the element $T_{\mathbf{v},y}$ by $\epsilon_{\mu,\infty}(T_{(0,m),y}).$ Namely, there is $f \in \mathrm{SL}_2(\Z)$ such that $f(0,m)=\mathbf{v}$ and thus $T_{\mathbf{v},y} := f \cdot T_{(0,m),y}$,  see also \cite[pag. 16]{bburban}.

\begin{df} For a character $\widetilde{\rho} \in \widetilde{\mathrm{P}}_n$ and a closed point $y \in X$, we define
$$\widetilde{\rho}(y) := \frac{1}{n} \sum_{i=0}^{n-1}\rho((\Fr_{X,n})^{i}(\mathcal{O}_{X_n}(y'))) = \frac{1}{d} \sum_{i=0}^{d-1} \rho((\Fr_{X,n})^{i}(\mathcal{O}_{X_n}(y')))$$
where $y' \in X_n$ is a closed point that sits above $y$ and $d = \gcd(|y|,n).$
\end{df}

\begin{df} \label{defofT} For a character $\widetilde{\rho} \in \widetilde{X}$ and for a point $\mathbf{v} \in \mathbf{Z}$ we define the element
$$T_{\mathbf{v}}^{\widetilde{\rho}} := \sum_{y \in |X|} \widetilde{\rho}(y) T_{\mathbf{v},y}.$$
\end{df}
From \cite[Prop. 3.4]{dragos} the elements $T_{\mathbf{v}}^{\widetilde{\rho}}$ in the above definition generates the elliptic Hall algebra $\mathsf{H}_X$ for $\mathbf{v} \in \mathbf{Z}$ and $\widetilde{\rho} \in \widetilde{\mathrm{P}}_{d}$ with $d=\gamma(\mathbf{v})$.


\subsection{Edges terminating in geometrically indecomposable bundles}
For the rest of this section, we fix a degree one place $x$ of $X$. In order to keep the exposition compact, we refer the reader to \cite[Sec. 4]{roberto-elliptic} (or more specific \cite[Thms. 4.2, 4.6 and 4.8]{roberto-elliptic}) for further details on the algorithm that we use in the following proofs.

\begin{rem} 
 Note that the following lemmas cover all edges of $\cG_{x,1}^{(3)}$ between rank $3$ bundles $\cE$ and $\cE'$. In particular, this means that if $\cE$ does not appear in the lemma listing the arrows of $\cG_{x,1}^{(3)}$ with target $\cE'$, then $m_{x,1}(\E, \E'  ) =0$.

\end{rem}


\begin{lemma} \label{idC} Let $\E' := \E_{(z,3)}^{(3,0)}$ be a rank $3$ vector bundle on $X$. Then \smallskip

\begin{enumerate}[label={\textbf{\upshape(\roman*)}}, ]

    \item $m_{x,1}\big( \E_{(z,2)}^{(2,0)} \oplus \E_{(x+z,1)}^{(1,1)}   , \E'  \big) =q^2 -q$;  \smallskip

    \item $m_{x,1}\big( \E_{(z,1)}^{(1,0)} \oplus \E_{(x+2z,1)}^{(2,1)}   , \E'  \big) =q-1$;  \smallskip

    \item $m_{x,1}\big( \E_{(x+3z,1)}^{(3,1)} , \E'  \big) =1$;  \smallskip
\end{enumerate}

\begin{proof} Via Atiyah's classification \ref{atiyahtheorem}, the vector bundle $\E'$ corresponds to the torsion sheaf $\mathcal{K}_{x}^{(3)}$. Hence
\[ \E'= v^{-6}(v^4 -2v^2 +1)  \big( T_{(1,0),z}\big)^3 + v^{-6}(v^4-1)  T_{(2,0),z}  T_{(1,0),z} + v^{-6}(v^4+v^2+1)  T_{(3,0),z}. \]
Base changing and calculating the product $\mathcal{K}_x * \E'$ in the Hall algebra $\mathsf{H}_X$ yields
\begin{align*}
\pi^{\mathrm{vec}} (\mathcal{K}_x * \E') & = N_1^{-4}(v^{-2} -2v^{-4} +v^{-6}) \\
 & \quad \cdot \sum_{\substack{ \sigma, \rho_i \in \mathrm{P}_1 \\ i=1,2,3}} \widetilde{\sigma}(-x) \widetilde{\rho}_1(-z)\widetilde{\rho}_2(-z)\widetilde{\rho}_3(-z) \Big[T_{(0,1)}^{\widetilde{\sigma}}, T_{(1,0)}^{\widetilde{\rho}_1} T_{(1,0)}^{\widetilde{\rho}_2} T_{(1,0)}^{\widetilde{\rho}_3} \Big] \\
&+N_1^{-2}N_2^{-1}(v^{-2}-v^{-6}) \sum_{\substack{ \sigma, \rho_1 \in \mathrm{P}_1 \\ \rho_2 \in \mathrm{P}_2}} \widetilde{\sigma}(-x) \widetilde{\rho}_1(-z)\widetilde{\rho}_2(-z)
\Big[T_{(0,1)}^{\widetilde{\sigma}}, T_{(2,0)}^{\widetilde{\rho}_2} T_{(1,0)}^{\widetilde{\rho}_1} \Big]  \\
&+N_1^{-1}N_3^{-1}(v^{-6}+v^{-4}+v^{-2}) \sum_{\substack{ \sigma \in \mathrm{P}_1,  \rho_3 \in \mathrm{P}_3}} \widetilde{\sigma}(-x) \widetilde{\rho}_3(-z)
\Big[T_{(0,1)}^{\widetilde{\sigma}}, T_{(3,0)}^{\widetilde{\rho}_3}\Big].
\end{align*}
Denote the first, second and third term on the above identity by $A, B$ and $C$, respectively. Then
\begin{align*}
A &=  \tfrac{v^5+v^3}{2}\text{ \small $(v^{-2} -2v^{-4} +v^{-6})$}  \E_{(z,1)}^{(1,0)} \oplus \E_{(z,1)}^{(1,0)}  \oplus \E_{(x+z,2)}^{(1,1)}  \\
   & \quad + \text{ \small $3v^{-1}(v^{-2} -2v^{-4} +v^{-6})$} \E_{(z,2)}^{(2,0)} \oplus \E_{(x+z,1)}^{(1,1)} \\
   & \quad + \text{ \small $3 v (v^{-2} -2v^{-4} +v^{-6})$} \E_{(z,2)}^{(1,0)} \oplus \E_{(x+2z,1)}^{(2,1)} + \text{ \small $v^3 (v^{-2} -2v^{-4} +v^{-6})$}  \E_{(x+3z,1)}^{(3,1)},  \\[5pt]
B &= \text{ \small $ (v^{-2} - v^{-6})$} \tfrac{v^5(v^{-2}-1)}{2}   \E_{(z,1)}^{(1,0)} \oplus \E_{(z,1)}^{(1,0)}  \oplus \E_{(x+z,2)}^{(1,1)} +
\text{ \small $ (v^{-2} - v^{-6})$}  \tfrac{-v^5}{2}  \E_{(z,2)}^{(2,0)} \oplus \E_{(x+z,1)}^{(1,1)}  \\
& \quad  +  \text{ \small $ (v^{-2} - v^{-6})$}  \tfrac{-v^7}{2}  \E_{(x+2z,2)}^{(2,1)} \oplus \E_{(z,1)}^{(1,0)}  +  \text{ \small $ (v^{-2} - v^{-6})$} \tfrac{-v^9}{2} \E_{(x+3z,1)}^{(3,1)} \\ \medskip
\text{and } &\\
C &= \text{ \small $(v^4 + v^2 + 1)$ } \tfrac{v^3}{3} \E_{(x+3z,1)}^{(3,1)}.
\end{align*}
Putting all together we have precisely the assertion of the lemma.
\end{proof}
\end{lemma}


\begin{lemma} \label{idA.3} Let $\E' := \E_{(z,1)}^{(3,1)}$ be a rank $3$ vector bundle on $X.$ Then \smallskip

\begin{enumerate}[label={\textbf{\upshape(\roman*)}}, ]

  \item $m_{x,1}\big(\E_{(x+z,1)}^{(3,2)}, \E'  \big) = q^2+q+1 - N_1$; \smallskip

    \item $m_{x,1}\big(\E_{(x',1)}^{(2,1)} \oplus \E_{(x'',1)}^{(1,1)}, \E'  \big) = q^2-q$;
\end{enumerate}
where $x',x'' \in |X|$ are such that $x'+x'' = z+x.$

\begin{proof} We first observe that $\E' = N_{1}^{-1 } \sum_{\rho\in \mathrm{P}_1} \widetilde{\rho}(-z) T_{(3,1)}^{\widetilde{\rho}}.$ Thus
\[ \pi^{\mathrm{vec}} (\mathcal{K}_x * \E')  =  N_1^{-2}  \sum_{\substack{ \sigma, \rho \in \mathrm{P}_1 \\ }} \text{ \small $\widetilde{\sigma}(-x) \widetilde{\rho}(-z)$} \Big[T_{(0,1)}^{\widetilde{\sigma}}, T_{(3,1)}^{\widetilde{\rho}} \Big]. \]
From  \cite[Thm. 5.2]{dragos},
$\big[T_{(0,1)}^{\widetilde{\sigma}}, T_{(3,1)}^{\widetilde{\rho}}\big]$ vanishes unless $\widetilde{\sigma} = \widetilde{\rho}$. In this case we have
\[ \Big[T_{(0,1)}^{\widetilde{\sigma}}, T_{(3,1)}^{\widetilde{\sigma}} \Big] = \text{ \small $(c_2 - c_1+c_2^2 (v^{-1}-v)2^{-1})$} T_{(3,2)}^{\widetilde{\sigma}} + \text{ \small $ c_1(v^{-1}-v)$} T_{(2,1)}^{\widetilde{\sigma}} T_{(1,1)}^{\widetilde{\sigma}}. \]
Therefore 
$$\pi^{\mathrm{vec}} (\mathcal{K}_x * \E')  = v([3]-v c_1) \E_{(x+z,1)}^{(3,2)} + (v^{-1}-v) \E_{(x',1)}^{(2,1)} \oplus \E_{(x'',1)}^{(1,1)}.$$
 Multiplying the above equation by $v^{-3}$ yields the desired statement.
\end{proof}
\end{lemma}


\subsection{Edges terminating in a sum of rank 2 bundle with a line bundle }

\begin{lemma} \label{idX} Let $\E' := \E_{(y,1)}^{(1,0)} \oplus \E_{(y',1)}^{(2,-1)}   $ be a rank $3$ vector bundle on $X$. Then \smallskip

\begin{enumerate}[label={\textbf{\upshape(\roman*)}}, ]
\item $m_{x,1}\big(\E_{(y,1)}^{(1,0)}  \oplus \E_{(z,1)}^{(2,0)}, \E'  \big) = q+1$, where   $z = x+y'$; \smallskip

\item   $m_{x,1}\big(\E_{(y,1)}^{(1,0)}  \oplus \E_{(z,2)}^{(2,0)}, \E'  \big) = q$,  where   $2z=x+y'$ and $z \neq y$;
\smallskip

\item $m_{x,1}\big(\E_{(y,1)}^{(1,0)}  \oplus \E_{(y,2)}^{(2,0)}, \E'  \big) = q^2$,  if   $2y=x+y'$;
\smallskip

\item $m_{x,1}\big(\E_{(y,3)}^{(3,0)}, \E'  \big) = q$,  if   $2y=x+y'$;
\smallskip

\item $m_{x,1}\big(\E_{(y,1)}^{(1,0)}  \oplus \E_{(z,1)}^{(1,0)} \oplus \E_{(z',1)}^{(1,0)}, \E'  \big) = q-1$, where  $y,z,z'$ are pairwise distinct and  $z+z' = x+y'$;
\smallskip

\item $m_{x,1}\big(\E_{(z,1)}^{(1,0)}  \oplus \E_{(y,1)}^{(1,0)} \oplus \E_{(y,1)}^{(1,0)}, \E'  \big) = q^2-1$, where  $z \neq y$  and  $z+y = x+y'$;
\smallskip

\item   $m_{x,1}\big(\E_{(z,1)}^{(1,0)}  \oplus \E_{(y,2)}^{(2,0)}, \E'  \big) = q-1$,  where   $z \neq y$ and $z + y = x+ y'$;
\smallskip

\item   $m_{x,1}\big(\E_{(x+y,1)}^{(1,1)}  \oplus \E_{(y',1)}^{(2,-1)}, \E'  \big) = q^2$.
\smallskip
  \end{enumerate}

\begin{proof} By \cite[Thm. 4.6]{roberto-elliptic}, we may write
\[ \E' = v N_{1}^{-2} \sum_{\rho, \rho_1 \in \mathrm{P}_1} \widetilde{\rho}_1 (-y')\widetilde{\rho}(-y) T_{(2,-1)}^{\widetilde{\rho}_1} T_{(1,0)}^{\widetilde{\rho}} \]
and
\[  \mathcal{K}_x = N_{1}^{-1 } \sum_{\sigma\in \mathrm{P}_1} \widetilde{\sigma}(-x) T_{(0,1)}^{\widetilde{\sigma}}.\]
See \cite[Section 4]{roberto-elliptic} for further details on the following algorithm. Then
\[ \pi^{\mathrm{vec}} (\mathcal{K}_x * \E') = v N_{1}^{-3} \sum_{\sigma, \rho_1, \rho \in \mathrm{P}_1} \widetilde{\sigma}(-x) \widetilde{\rho}_1(-y')\widetilde{\rho}(-y) \Big[T_{(0,1)}^{\widetilde{\sigma}}, T_{(2,-1)}^{\widetilde{\rho}_1} T_{(1,0)}^{\widetilde{\rho}} \Big].\]
Since
\[ \Big[T_{(0,1)}^{\widetilde{\sigma}}, T_{(2,-1)}^{\widetilde{\rho}_1} T_{(1,0)}^{\widetilde{\rho}} \Big] = c_1 \frac{\Theta_{(2,0)}}{(v^{-1}-v)} T_{(1,0)}^{\widetilde{\rho}} + c_1 T_{(2,-1)}^{\widetilde{\rho}_1} T_{(1,1)}^{\widetilde{\sigma}} \]
and
\[  \Theta_{(2,0)} = (v^{-1}-v) T_{(2,0)}^{\Norm_{1}^{2}\tilde{\sigma}} + \tfrac{(v^{-1}-v)^{2}}{2} T_{(1,0)}^{\widetilde{\sigma}} T_{(1,0)}^{\widetilde{\sigma}}, \]
we obtain
\[  \pi^{\mathrm{vec}} (\mathcal{K}_x * \E') = \Big(  v^2 T_{(2,0), z} + \tfrac{v^2 (v^{-1}-v)}{2} T_{(1,0),z} T_{(1,0),z'} \Big) T_{(1,0),y} + v^2 T_{(2,-1),y'} T_{(1,1),x+y},\]
where $z+z' = x+y'$, and $z \in |X|$ is such that  $z = x+y$. We calculated each term separated.
By the definition,
\[T_{(2,0), z} = \begin{cases} [2] \E_{(z,1)}^{(2,0)}  \hfill &\text{ if } |z|=2 \smallskip\\
\frac{[2]}{2} \Big( (1-v^{-2}) \E_{(z,1)}^{(1,0)} \oplus \E_{(z,1)}^{(1,0)} + \E_{(z,2)}^{(2,0)} \Big)   \hfill &\text{ if }  |z|=1.
\end{cases} \]
Thus,
\[ T_{(2,0), z} T_{(1,0),y} =   [2] \; \E_{(y,1)}^{(1,0)}  \oplus \E_{(z,1)}^{(2,0)} \]
if $|z|=2$,  
 \[ T_{(2,0), z} T_{(1,0),y} =  \tfrac{[2](1-v^{-2})}{2} \E_{(z,1)}^{(1,0)}  \oplus \E_{(z,1)}^{(1,0)} \oplus \E_{(y,1)}^{(1,0)} + \tfrac{v^2[2]}{2}  \E_{(y,1)}^{(1,0)}  \oplus \E_{(z,2)}^{(2,0)} \]
if $  |z|=1 \text{ and } z\neq y$, and 
\[  T_{(2,0), z} T_{(1,0),y} =    \tfrac{[2]}{2} \Big( (1-v^{-6}) \E_{(y,1)}^{(1,0)}  \oplus \E_{(y,1)}^{(1,0)} \oplus \E_{(y,1)}^{(1,0)} + \E_{(y,1)}^{(1,0)}  \oplus \E_{(y,2)}^{(2,0)} + \E_{(y,3)}^{(3,0)} \Big)\]
if $ |z|=1 \text{ and } z =y.$
On the other hand, 
\[  T_{(1,0),z} T_{(1,0),z'} T_{(1,0),y} =  \E_{(z,1)}^{(1,0)}  \oplus \E_{(z',1)}^{(1,0)} \oplus \E_{(y,1)}^{(1,0)} \] 
if $z,z',y$ are pairwise distinct, 
\[ T_{(1,0),z} T_{(1,0),z'} T_{(1,0),y} =  \Big( (1+v^{-2}) \E_{(z,1)}^{(1,0)}  \oplus \E_{(z,1)}^{(1,0)} \oplus \E_{(y,1)}^{(1,0)} + \E_{(y,1)}^{(1,0)}  \oplus \E_{(z,2)}^{(2,0)} \Big) \]
if $ z=z' \text{ and } z \neq y$, 
\[  T_{(1,0),z} T_{(1,0),z'} T_{(1,0),y} = 2\Big( (1+ v^{-2}) \E_{(z,1)}^{(1,0)}  \oplus \E_{(y,1)}^{(1,0)} \oplus \E_{(y,1)}^{(1,0)} + \E_{(z,1)}^{(1,0)}  \oplus \E_{(y,2)}^{(2,0)} \Big) \]
if $z \neq z'$ and  $z'=y$ (or equivalently if $z \neq z'$ and  $z=y$, interchanging $z$ by $z'$ in the above equation), and finally 
\begin{align*}
T_{(1,0),z} T_{(1,0),z'} T_{(1,0),y} & =  \Big( (1+2v^{-2} + 2v^{-4} + v^{-6})  \E_{(y,1)}^{(1,0)}  \oplus \E_{(y,1)}^{(1,0)} \oplus \E_{(y,1)}^{(1,0)} \\
& \quad + (1+2v^{-2}) \E_{(y,1)}^{(1,0)}  \oplus \E_{(y,2)}^{(2,0)} + \E_{(y,3)}^{(3,0)} \Big) 
\end{align*} 
 if $z=z'=y.$
Taking care of summing up the multiplicities of a same vertex, this proves the lemma.
\end{proof}

\end{lemma}


\begin{lemma} \label{idL} Let $\E' := \E_{(y,2)}^{(2,0)} \oplus \E_{(y',1)}^{(1,-1)} $ be a rank $3$ vector bundle on $X$.\smallskip

\begin{enumerate}[label={\textbf{\upshape(\roman*)}}, ]
\item If $y \neq x+y'$, then \smallskip
  \begin{enumerate}
\item $m_{x,1}\big(\E_{(x+y',1)}^{(1,0)}  \oplus \E_{(y,2)}^{(2,0)}, \E'  \big) =1$;  \smallskip

\item $m_{x,1}\big(\E_{(y',1)}^{(1,-1)}  \oplus \E_{(y,1)}^{(1,0)} \oplus \E_{(x+y,1)}^{(1,1)}, \E'  \big) = q^2 - q$;
\smallskip

\item   $m_{x,1}\big(\E_{(y',1)}^{(1,-1)}  \oplus \E_{(x+2y,1)}^{(2,1)}, \E'  \big) = q$. \smallskip

\end{enumerate}

\item If $y =x+y'$, then
  \begin{enumerate}
\item $m_{x,1}\big( \E_{(y,1)}^{(1,0)} \oplus \E_{(y,2)}^{(2,0)}  , \E'  \big) =q$;  \smallskip

\item   $m_{x,1}\big(\E_{(y',1)}^{(1,-1)}  \oplus \E_{(y,1)}^{(1,0)} \oplus \E_{(x+y,1)}^{(1,1)} , \E'  \big) = q^2-q$; \smallskip

\item   $m_{x,1}\big(\E_{(y',1)}^{(1,-1)}  \oplus \E_{(x+2y,1)}^{(2,1)} , \E'  \big) = q$; \smallskip

\item $m_{x,1}\big(\E_{(y,3)}^{(3,0)} , \E'  \big) = 1$. \smallskip

\end{enumerate}
\end{enumerate}

\begin{proof} Writing $\E'$ is terms of elements $T_{\mathbf{v}}^{\widetilde{\rho}}$ (see Definition \ref{defofT}) we get
\begin{align*}
 \E' &= v^2  \Big( \tfrac{N_1^{-3}(1-v^2)}{2} \sum_{\rho, \rho_1, \chi \in \mathrm{P}_1} \widetilde{\chi}(-y') \widetilde{\rho}(-y)\widetilde{\rho}_1(-y) T_{(1,-1)}^{\widetilde{\chi}} T_{(1,0)}^{\widetilde{\rho} } T_{(1,0)}^{\widetilde{\rho}_1} \\
& \quad + \tfrac{N_1^{-1}N_2^{-1}(v^2+1)}{[2]} \sum_{\chi \in \mathrm{P}_1, \rho_2 \in \mathrm{P}_2} \widetilde{\chi}(-y') \widetilde{\rho}_2(-y) T_{(1,-1)}^{\widetilde{\chi} } T_{(2,0)}^{\widetilde{\rho}_2} \Big).
\end{align*}
Hence
\begin{align*} \pi^{\mathrm{vec}} (\mathcal{K}_x * \E') & =  \tfrac{ N_{1}^{-4}(v^2 -v^4)}{2} \hspace{-8pt} \sum_{\sigma, \rho_1, \rho, \chi \in \mathrm{P}_1} \hspace{-8pt} \widetilde{\sigma}(-x) \widetilde{\chi}(-y')\widetilde{\rho}_1(-y)\widetilde{\rho}(-y) \Big[T_{(0,1)}^{\widetilde{\sigma}}, T_{(1,-1)}^{\widetilde{\chi}} T_{(1,0)}^{\widetilde{\rho}} T_{(1,0)}^{\widetilde{\rho}_1} \Big] \\
& \quad + \tfrac{ N_{1}^{-2} N_{2}^{-1}(v^2 + v^4)}{[2]} \sum_{\sigma, \chi \in \mathrm{P}_1, \rho_2 \in \mathrm{P}_2} \widetilde{\sigma}(-x) \widetilde{\chi}(-y')\widetilde{\rho}_2(-y) \Big[T_{(0,1)}^{\widetilde{\sigma}}, T_{(1,-1)}^{\widetilde{\chi}} T_{(2,0)}^{\widetilde{\rho}_2} \Big].
\end{align*}
Denote the first term in the above identity by $A$ and the second term by $B$. Using the algorithm in \cite{roberto-elliptic} we obtain
{ \small \[A=
\tfrac{(v^3)-v^5}{2} T_{(1,0),x+y'} T_{(1,0),y} T_{(1,0),y} + {\text{\tiny $(v^3-v^5)$}} T_{(1,-1),y'} T_{(1,0),y} T_{(1,1),x+y} + \tfrac{v^4 - v^6}{2} T_{(1,-1),y'} T_{(2,1),x+2y}.
\] }
We have two cases to analyze. If $y \neq x+ y'$,
\begin{align*}
A & = \tfrac{(v-v^5)}{2} \E_{(x+y',1)}^{(1,0)}  \oplus \E_{(y,1)}^{(1,0)} \oplus \E_{(y,1)}^{(1,0)} + \tfrac{(v^3-v^5)}{2} \E_{(x+y',1)}^{(1,0)}  \oplus \E_{(y,2)}^{(2,0)} \\ \smallskip
&\quad + (v^{-1}-v) \E_{(y',1)}^{(1,-1)}  \oplus \E_{(y,1)}^{(1,0)} \oplus \E_{(x+y,1)}^{(1,1)} + \tfrac{(v-v^3)}{2} \E_{(y',1)}^{(1,-1)}  \oplus \E_{(x+2y,1)}^{(2,1)}.
\end{align*}
If $y = x+ y'$,
\begin{align*}
A & = \tfrac{(-v^5-v^3+ v^{-1} + v^{-3})}{2} \E_{(y,1)}^{(1,0)}  \oplus \E_{(y,1)}^{(1,0)} \oplus \E_{(y,1)}^{(1,0)} + \tfrac{(-v^3-v^5+2v)}{2} \E_{(y,1)}^{(1,0)}  \oplus \E_{(y,2)}^{(2,0)} \\ \smallskip
&\quad + \tfrac{(v^{3}-v^5)}{2} \E_{(y,3)}^{(3,0)}   + (v^{-1}-v) \E_{(y',1)}^{(1,-1)}  \oplus \E_{(y,1)}^{(1,0)} \oplus  \E_{(x+y,1)}^{(1,1)}
+ \tfrac{(v-v^3)}{2}  \E_{(y',1)}^{(1,-1)} \oplus \E_{(x+2y,1)}^{(2,1)}.
\end{align*}
The second term of $\pi^{\mathrm{vec}} (\mathcal{K}_x * \E')$ yields
\[ B= \tfrac{N_1^{-2}N_2^{-1}(v^4+v^2)}{[2]} \big( c_1 N_1 N_2 T_{(1,0),x+y} T_{(2,0),y} + c_2 N_1^{2} T_{(1,-1),y} T_{(2,1),x+y} \big). \]
Hence
\[ B = \tfrac{(v^5-v)}{2} \E_{(x+y',1)}^{(1,0)}  \oplus \E_{(y,1)}^{(1,0)} \oplus \E_{(y,1)}^{(1,0)} + \tfrac{(v^3+v^5)}{2} \E_{(x+y',1)}^{(1,0)}  \oplus \E_{(y,2)}^{(2,0)}
+ \tfrac{(v+v^3)}{2} \E_{(y',1)}^{(1,-1)}  \oplus \E_{(x+2y,1)}^{(2,1)},\]
for $y \neq x+y'$. The case where $y = x+y'$ yields
\begin{align*}
B & = \tfrac{(v^5+v^3- v^{-1} - v^{-3})}{2} \E_{(y,1)}^{(1,0)}  \oplus \E_{(y,1)}^{(1,0)} \oplus \E_{(y,1)}^{(1,0)} + \tfrac{(v^3+v^5)}{2} \E_{(y,1)}^{(1,0)}  \oplus \E_{(y,2)}^{(2,0)}
\\
& \quad + \tfrac{(v^{3}+v^5)}{2} \E_{(y,3)}^{(3,0)}   + \tfrac{(v+v^3)}{2}  \E_{(y',1)}^{(1,-1)} \oplus \E_{(x+2y,1)}^{(2,1)}.
\end{align*}
Therefore
{ \small  \[ \pi^{\mathrm{vec}} (\mathcal{K}_x * \E')  = {v^3} \E_{(x+y',1)}^{(1,0)}  \oplus \E_{(y,2)}^{(2,0)}  + (v^{-1}-v) \E_{(y',1)}^{(1,-1)}  \oplus \E_{(y,1)}^{(1,0)} \oplus \E_{(x+y,1)}^{(1,1)}  + v \E_{(y',1)}^{(1,-1)}  \oplus \E_{(x+2y,1)}^{(2,1)} \] }
for $y \neq x+y'$, and
\begin{align*}
 \pi^{\mathrm{vec}} (\mathcal{K}_x * \E') & = v \; \E_{(y,1)}^{(1,0)}  \oplus \E_{(y,2)}^{(2,0)} + v^3 \E_{(y,3)}^{(3,0)}  +v \E_{(y',1)}^{(1,-1)} \oplus \E_{(x+2y,1)}^{(2,1)} \\
 & \quad + (v^{-1}-v) \E_{(y',1)}^{(1,-1)}  \oplus \E_{(y,1)}^{(1,0)} \oplus  \E_{(x+y,1)}^{(1,1)} 
\end{align*}
for $y=x+y'.$ Multiplying above identities by $v^{-3}$ yields the desired assertion.
\end{proof}
\end{lemma}



\begin{lemma} \label{idE} Let $\E' := \E_{(y',1)}^{(1,0)} \oplus \E_{(y,1)}^{(2,0)}$ be a rank $3$ vector bundle on $X$. Then \smallskip

\begin{enumerate}[label={\textbf{\upshape(\roman*)}}, ]
\item $m_{x,1}\big(\E_{(y,1)}^{(2,0)} \oplus \E_{(x+y',1)}^{(1,1)}, \E'  \big) = q^2$; \smallskip

\item $m_{x,1}\big(\E_{(y',1)}^{(1,0)}  \oplus \E_{(z,1)}^{(2,1)}, \E'  \big) = q$, with  $z=x+ y$;
\smallskip

\item   $m_{x,1}\big(\E_{(w,1)}^{(3,1)} , \E'  \big) = 1$,  with  $w=x+y'+y$.
\end{enumerate}

\begin{proof} Our proof starts with the observation that
\[ \E' := \E_{(y',1)}^{(1,0)} \oplus \E_{(y,1)}^{(2,0)} = \E_{(y',1)}^{(1,0)} * \E_{(y,1)}^{(2,0)}.\]
Since
$$ \E_{(y',1)}^{(1,0)} = N_1^{-1} \sum_{\rho_1\in \mathrm{P}_1}\text{ \small $ \widetilde{\rho}_1(-y')$} T_{(1,0)}^{\widetilde{\rho}_1}
\  \  \text{ and  } \  \ \E_{(y,1)}^{(2,0)} =2 [2]^{-1} N_2^{-1} \sum_{\rho_2\in \mathrm{P}_2}\text{ \small $ \widetilde{\rho}_2(-y)$} T_{(2,0)}^{\widetilde{\rho}_2}, $$
we obtain
$$\E' = 2 [2]^{-1} N_2^{-1} N_1^{-1} \sum_{\rho_1\in \mathrm{P}_1, \rho_2\in \mathrm{P}_2} \text{ \small $ \widetilde{\rho}_1(-y') \widetilde{\rho}_2(-y)$} T_{(1,0)}^{\widetilde{\rho}_1} T_{(2,0)}^{\widetilde{\rho}_2}.$$
From the algorithm in \cite{roberto-elliptic}, it follows that
\begin{align*} \pi^{\mathrm{vec}} (\mathcal{K}_x * \E') &= 2 [2]^{-1} N_2^{-1} N_1^{-2}   \sum_{\substack{ \sigma, \rho_1 \in \mathrm{P}_1 \\ \rho_2\in \mathrm{P}_2}} \text{ \small $\widetilde{\sigma}(-x) \widetilde{\rho}_1(-y')\widetilde{\rho}_2(-y)$} \Big[T_{(0,1)}^{\widetilde{\sigma}}, T_{(1,0)}^{\widetilde{\rho}_1} T_{(2,0)}^{\widetilde{\rho}_2}\Big] \\
& = 2 [2]^{-1} N_2^{-1} N_1^{-2} \Big( c_1 \sum_{\substack{ \sigma \in \mathrm{P}_1 \\ \rho_2\in \mathrm{P}_2}} \text{ \small $ \widetilde{\sigma}(-x-y') \widetilde{\rho}_2(-y) $} T_{(2,0)}^{\widetilde{\rho}_2} T_{(1,1)}^{\widetilde{\sigma}} \\
& \quad + c_2 \sum_{ \sigma, \rho_1 \in \mathrm{P}_1} \text{ \small $ \widetilde{\sigma}(-x-y) \widetilde{\rho}_1(-y')$} T_{(1,0)}^{\widetilde{\rho}_1} T_{(2,1)}^{\widetilde{\sigma}}
+ c_1c_2 \sum_{\sigma \in \mathrm{P}_1}\text{ \small $ \widetilde{\sigma}(-x-y'-y)$} T_{(3,1)}^{\widetilde{\sigma}} \Big)
\end{align*}
Thus
\begin{align*}
\pi^{\mathrm{vec}} (\mathcal{K}_x * \E') & = \tfrac{v}{[2]} T_{(2,0),y} T_{(1,1), x+y'}
 +  v^2  T_{(1,0),y'} T_{(2,1), x+y}  + v^3  T_{(3,1),x+y'+y}  \\
& = v^{-1} \E_{(y,1)}^{(2,0)} \oplus \E_{(x+y',1)}^{(1,1)} + v \E_{(y',1)}^{(1,0)}  \oplus \E_{(x+y,1)}^{(2,1)} + v^3 \E_{(x+y'+y,1)}^{(3,1)}.
\end{align*}
Multiplying the above equation by $v^{-3}$ completes the proof.
\end{proof}
\end{lemma}


\begin{lemma} \label{idF} Let $\E' := \E_{(y,1)}^{(1,0)} \oplus \E_{(y',2)}^{(2,0)}$ be a rank $3$ vector bundle on $X$ with $y \neq y'$. Then \smallskip

\begin{enumerate}[label={\textbf{\upshape(\roman*)}}, ]

\item $m_{x,1}\big(\E_{(y',2)}^{(2,0)}  \oplus \E_{(x+y,1)}^{(1,1)}, \E'  \big) = q^2$;
\smallskip

\item   $m_{x,1}\big(\E_{(y,1)}^{(1,0)} \oplus \E_{(y',1)}^{(1,0)} \oplus \E_{(x+y',1)}^{(1,1)}, \E'  \big) = q^2-q$; \smallskip

\item $m_{x,1}\big(\E_{(y',1)}^{(1,0)}  \oplus \E_{(z,1)}^{(2,1)}, \E'  \big) = q-1$, where $z=x+y+y'$; \smallskip

\item $m_{x,1}\big(\E_{(y,1)}^{(1,0)}  \oplus \E_{(z',1)}^{(2,1)}, \E'  \big) = q$, where $z'=x+2y'$; \smallskip

\item $m_{x,1}\big(\E_{(w,1)}^{(3,1)}, \E'  \big) = 1$, where $w=x+y+2y'$. \smallskip
\end{enumerate}

\begin{proof} Since $y \neq y'$, we can write $\E' = \E_{(y,1)}^{(1,0)} * \E_{(y',2)}^{(2,0)} $ in $\mathsf{H}_X$. Moreover,
\[ \E_{(y,1)}^{(1,0)} = N_1^{-1} \sum_{\rho_1\in \mathrm{P}_1}\text{ \small $ \widetilde{\rho}_1(-y)$} T_{(1,0)}^{\widetilde{\rho}_1} \]
  and
\begin{multline*}
 \E_{(y',2)}^{(2,0)} = \tfrac{(1-v^2)N_1^{-2}}{2} \sum_{\rho, \chi \in \mathrm{P}_1} \text{ \small $ \widetilde{\rho}(-y') \widetilde{\chi}(-y')$} T_{(1,0)}^{\widetilde{\rho}}  T_{(1,0)}^{\widetilde{\chi}} \\
+ \text{ \small $ [2]^{-1}(v^2+1)N_2^{-1}$} \sum_{\rho_2 \in \mathrm{P}_2} \text{ \small $ \widetilde{\rho}_2(-y') $ } T_{(2,0)}^{\widetilde{\rho}_2}.
\end{multline*}
Hence
\begin{align*} \pi^{\mathrm{vec}} (\mathcal{K}_x * \E') & = \tfrac{(1-v^2)}{2  N_1^{4}}   \sum_{\substack{ \sigma, \rho_1, \rho, \chi \in \mathrm{P}_1 \\ }} \text{ \small $\widetilde{\sigma}(-x) \widetilde{\rho}_1(-y)\widetilde{\rho}(-y')\widetilde{\chi}(-y')$} \Big[T_{(0,1)}^{\widetilde{\sigma}}, T_{(1,0)}^{\widetilde{\rho}_1} T_{(1,0)}^{\widetilde{\rho}} T_{(1,0)}^{\widetilde{\chi}} \Big] \\
& \quad +  \tfrac{(1+v^2)}{[2]  N_1^{2} N_2}  \sum_{\substack{ \sigma, \rho_1  \in \mathrm{P}_1 \\ \rho_2  \in \mathrm{P}_2 }} \text{ \small $\widetilde{\sigma}(-x) \widetilde{\rho}_1(-y)\widetilde{\rho}_2(-y')$} \Big[T_{(0,1)}^{\widetilde{\sigma}}, T_{(1,0)}^{\widetilde{\rho}_1} T_{(1,0)}^{\widetilde{\rho}_1} T_{(2,0)}^{\widetilde{\rho}_2} \Big].
\end{align*}
We denote by $A$ (resp. $B$) the first (resp. second) term in the above summation. Again using the algorithm in \cite{roberto-elliptic}, we obtain
\begin{multline*}
 A = \tfrac{(v^{-3} - v)}{2}  \E_{(y',1)}^{(1,0)} \oplus \E_{(y',1)}^{(1,0)} \oplus \E_{(x+y,1)}^{(1,1)}
+  \tfrac{(v^{-1} - v)}{2}  \E_{(y',2)}^{(2,0)}  \oplus \E_{(x+y,1)}^{(1,1)}
+\tfrac{(v^3 - v^{5})}{2} \E_{(w,1)}^{(3,1)} \\
  +  (v - v^{3}) \E_{(y',1)}^{(1,0)}  \oplus \E_{(z,1)}^{(2,1)} + \tfrac{(v - v^{3})}{2} \E_{(y,1)}^{(1,0)}  \oplus \E_{(z',1)}^{(2,1)}
  + (v^{-1} - v) \E_{(y,1)}^{(1,0)} \oplus \E_{(y',1)}^{(1,0)} \oplus \E_{(x+y',1)}^{(1,1)} \\
\end{multline*}
and
\begin{align*} 
B & =   \tfrac{(v-v^{-3} )}{2}  \E_{(y',1)}^{(1,0)} \oplus \E_{(y',1)}^{(1,0)} \oplus \E_{(x+y,1)}^{(1,1)} +  \tfrac{(v^3 + v^{5})}{2} \E_{(w,1)}^{(3,1)} \\ 
 & \quad + \tfrac{(v+v^{3} )}{2} \E_{(y,1)}^{(1,0)}  \oplus \E_{(z',1)}^{(2,1)} +\tfrac{(v^{-1} + v)}{2}  \E_{(y',2)}^{(2,0)}  \oplus \E_{(x+y,1)}^{(1,1)} 
\end{align*}
where $z=x+y+y'$, $z'=x+2y'$ and $w=x+y+2y'$. Adding $A$  to $B$ yields precisely the assertion of the lemma.
 \end{proof}

\end{lemma}


\begin{lemma} \label{idS}  Let $\E' := \E_{(x',1)}^{(1,0)} \oplus \E_{(x'',1)}^{(2,1)} $ be a rank $3$ vector bundle on $X.$ Then \smallskip

\begin{enumerate}[label={\textbf{\upshape(\roman*)}},]

  \item $m_{x,1}\big(\E_{(x+x',1)}^{(1,1)} \oplus \E_{(x'',1)}^{(2,1)}, \E'  \big) = q$; \smallskip

  \item  $m_{x,1}\big(\E_{(z,1)}^{(3,2)}, \E'  \big) = 1$, for $z=x+x'+x''$; \smallskip

    \item $m_{x,1}\big(\E_{(x',1)}^{(1,0)} \oplus \E_{(w,2)}^{(2,2)}, \E'  \big) = q^2 $, \; if $|w|=1$; \smallskip

    \item $m_{x,1}\big(\E_{(x',1)}^{(1,0)} \oplus \E_{(w,1)}^{(2,2)}, \E'  \big) = q^2 +q$,\; if $|w|=2$;\smallskip

        \item $m_{x,1}\big(\E_{(x',1)}^{(1,0)} \oplus \E_{(y_1,1)}^{(1,1)} \oplus \E_{(y_2,1)}^{(1,1)}, \E'  \big) = q^2 - q$,

\end{enumerate}
where $w \in |X|$ is such that $w = x+x''$  and $y_1 \neq y_2$ are such that $y_1+y_2 =x+y.$

\begin{proof} In the Hall algebra of $X$, we write
\[ \E' = v \;\E_{(x',1)}^{(1,0)} \cdot \E_{(x'',1)}^{(2,1)} = v N_1^{-2} \sum_{\rho_1, \rho_2 \in \mathrm{P}_1} \widetilde{\rho}_1(-x') \widetilde{\rho}_2(-x'') T_{(1,0)}^{\widetilde{\rho}_1} T_{(2,1)}^{\widetilde{\rho}_2} .\]
Hence, using \cite[Section 4]{roberto-elliptic},
\begin{align*}
 \pi^{\mathrm{vec}} (\mathcal{K}_x * \E')  &= v N_1^{-3}  \sum_{\sigma, \rho_1, \rho_2 \in \mathrm{P}_1} \widetilde{\sigma}(-x) \widetilde{\rho}_1(-x') \widetilde{\rho}_2(-x'')
\Big[T_{(0,1)}^{\widetilde{\sigma}}, T_{(1,0)}^{\widetilde{\rho}_1} T_{(2,1)}^{\widetilde{\rho}_2} \Big] \\
& = v N_1^{-3} \sum_{\substack{\sigma, \rho_1, \rho_2 \in \mathrm{P}_1 \\ \sigma = \rho_1}}  \widetilde{\sigma}(-x-x') \widetilde{\rho}_2(-x'') c_1 T_{(2,1)}^{\widetilde{\rho}_2}  T_{(1,1)}^{\widetilde{\sigma}}  \\
&  \quad +  v N_1^{-3} \sum_{\substack{\sigma, \rho_1, \rho_2 \in \mathrm{P}_1 \\ \sigma = \rho_1 =\rho_2}}  \widetilde{\sigma}(-x- x' -x'') c_1^2 T_{(3,2)}^{\widetilde{\sigma}}  \\
 & \quad + v N_1^{-3} \sum_{\substack{\sigma, \rho_1, \rho_2 \in \mathrm{P}_1\\ \sigma = \rho_2}}  \widetilde{\sigma}(-x-x'') \widetilde{\rho}_1(-x') c_1 T_{(1,0)}^{\widetilde{\rho}_1}   \tfrac{\Theta_{(2,2)}}{(v^{-1} -v)}
\end{align*}
where $\Theta_{(2,2)} := (v^{-1} -v) T_{(2,2)}^{\Norm_{1}^{2}(\widetilde{\sigma})}  + \frac{(v^{-1}-v)^2}{2} T_{(1,1)}^{\widetilde{\sigma}}  T_{(1,1)}^{\widetilde{\sigma}}.$
Thus
\begin{align*}
\pi^{\mathrm{vec}} (\mathcal{K}_x * \E') & = v N_1^{-3} \Big( c_1 N_1^{2}  T_{(2,1), x''} T_{(1,1), x+x'} + c_1^2 N_1 T_{(3,2),z}  + c_1 N_1^{2}  T_{(1,0), x'} T_{(2,2), w}  \\
&  \quad  +  c_1 N_1^{2} \tfrac{(v^{-1}-v)}{2}  T_{(1,0), x'} T_{(1,1), y_1}  T_{(1,1), y_2}  \Big),
\end{align*}
where $z=x+x'+x''$,  $w = x+x''$, and
$y_1 +y_2 = x+ x''$.
Since
\[T_{(2,2), w} = \begin{cases} [2] \E_{(w,1)}^{(2,2)}  \hfill &\text{ if } |w|=2 \smallskip \\
\frac{[2]}{2} \Big( (1-v^{-2}) \E_{(w,1)}^{(1,1)} \oplus \E_{(w,1)}^{(1,1)} + \E_{(w,2)}^{(2,2)} \Big)   \hfill &\text{ if }  |w|=1,
\end{cases} \]
and
\[ T_{(1,1), y_1}  T_{(1,1), y_2} = \begin{cases} \E_{(y,1)}^{(1,1)} \oplus \E_{(y,1)}^{(1,1)} \hfill &\text{ if } y_1 \neq y_2 \smallskip \\
(1+v^{-2}) \E_{(y,1)}^{(1,1)} \oplus \E_{(y,1)}^{(1,1)} + \E_{(y,2)}^{(2,2)} \hfill &\text{ if } y = y_1 = y_2 \end{cases} \]
we obtain
\begin{align*}
\pi^{\mathrm{vec}} (\mathcal{K}_x * \E') = v \E_{(x+x',1)}^{(1,1)} \oplus \E_{(x'',1)}^{(2,1)} + v^3 \E_{(z,1)}^{(3,2)} + A + B
\end{align*}
where
\[ A=\begin{cases}  (v+v^{-1}) \E_{(x',1)}^{(1,0)}  \oplus \E_{(w,1)}^{(2,2)} \hfill & \text{ if }  |w|=2; \smallskip \\
\tfrac{(v- v^{-3})}{2}  \E_{(x',1)}^{(1,0)} \oplus \E_{(w,1)}^{(1,1)} \oplus \E_{(w,1)}^{(1,1)} + \tfrac{(v+v^{-1})}{2} \E_{(x',1)}^{(1,0)} \oplus \E_{(w,2)}^{(2,2)}  \hfill & \text{ if }  |w|=1 \end{cases}\]
and
 \[ B =\begin{cases}  (v^{-1}-v) \E_{(x',1)}^{(1,0)}  \oplus \E_{(y_1,1)}^{(1,1)}  \oplus \E_{(y_2,1)}^{(1,1)} \hfill & \text{ if }  y_1 \neq y_2; \smallskip \\
\tfrac{(v^{-3}-v)}{2}  \E_{(x',1)}^{(1,0)} \oplus \E_{(y,1)}^{(1,1)} \oplus \E_{(y,1)}^{(1,1)} + \tfrac{(v^{-1}-v)}{2} \E_{(x',1)}^{(1,0)} \oplus \E_{(y,2)}^{(2,2)}  \hfill & \text{ if }  y=y_1=y_2. \end{cases}\]
The lemma follows from putting all together and multiplying by $v^{-3}.$
\end{proof}
\end{lemma}


\begin{lemma} \label{idF2} Let $\E' := \E_{(y,1)}^{(1,0)} \oplus \E_{(y,2)}^{(2,0)} $ be a rank $3$ vector bundle on $X$. Then, \smallskip

\begin{enumerate}[label={\textbf{\upshape(\roman*)}}, ]
  \item $m_{x,1}\big(\E_{(y,1)}^{(1,0)} \oplus \E_{(y,1)}^{(1,0)} \oplus \E_{(x+y,1)}^{(1,1)}, \E'  \big) =q^2-1$;  \smallskip

 \item $m_{x,1}\big(\E_{(y,2)}^{(2,0)} \oplus \E_{(x+y,1)}^{(1,1)}, \E'  \big) =q$;  \smallskip

  \item $m_{x,1}\big(\E_{(y,1)}^{(1,0)} \oplus \E_{(x+2y,1)}^{(2,1)}, \E'  \big) =1$.  \smallskip
  \end{enumerate}

\begin{proof} Following the algorithm in \cite[Section 4]{roberto-elliptic}, we write $\E'$ in terms of elements in Definition \ref{defofT}:
\begin{align*}
\E' & = \tfrac{N_{1}^{-3}(2v^2-v^6-v^4) }{6} \sum_{\substack{ \sigma_i \in \mathrm{P}_1 \\ i=1,2,3}} \text{\small $\widetilde{\sigma}_1(-y)\widetilde{\sigma}_2(-y) \widetilde{\sigma}_3(-y) $}
T_{(1,0)}^{\widetilde{\sigma}_1}  T_{(1,0)}^{\widetilde{\sigma}_2} T_{(1,0)}^{\widetilde{\sigma}_3}    \\
& \quad + \tfrac{N_{1}^{-1} N_{2}^{-1}(v^6+ v^4 )}{[2]}  \sum_{\substack{ \rho_1 \in \mathrm{P}_1 \\ \rho_2 \in \mathrm{P}_2}} \text{\small $\widetilde{\rho}_1(-y)\widetilde{\rho}_2(-y) $}  T_{(2,0)}^{\widetilde{\rho}_2}  T_{(1,0)}^{\widetilde{\rho}_1} +  \tfrac{N_{3}^{-1} (-v^6 - v^4 - v^2)}{[3]}  \sum_{ \rho_3 \in \mathrm{P}_3} \text{\small $\widetilde{\rho}_3(-y)$} T_{(3,0)}^{\widetilde{\rho}_3}.
\end{align*}
Thus
\begin{align*}
\pi^{\mathrm{vec}} (\mathcal{K}_x * \E')  & = \tfrac{N_{1}^{-4}(2v^2-v^6-v^4) }{6} \hspace{-5pt} \sum_{\substack{ \sigma, \sigma_i \in \mathrm{P}_1 \\ i=1,2,3}} \hspace{-5pt} \text{\small $\widetilde{\sigma}(-x) \widetilde{\sigma}_1(-y)\widetilde{\sigma}_2(-y) \widetilde{\sigma}_3(-y) $}
\Big[ T_{(0,1)}^{\widetilde{\sigma}}, T_{(1,0)}^{\widetilde{\sigma}_1}  T_{(1,0)}^{\widetilde{\sigma}_2} T_{(1,0)}^{\widetilde{\sigma}_3} \Big]   \\
& \quad + \tfrac{N_{1}^{-2} N_{2}^{-1}(v^6+ v^4 )}{[2]}  \sum_{\substack{ \sigma, \rho_1 \in \mathrm{P}_1 \\ \rho_2 \in \mathrm{P}_2}} \text{\small $\widetilde{\sigma}(-x) \widetilde{\rho}_1(-y)\widetilde{\rho}_2(-y) $}  \Big[ T_{(0,1)}^{\widetilde{\sigma}}, T_{(2,0)}^{\widetilde{\rho}_2}  T_{(1,0)}^{\widetilde{\rho}_1} \Big] \\
& \quad +  \tfrac{N_{3}^{-1}N_{1}^{-1}   (-v^6 - v^4 - v^2)}{[3]}  \sum_{ \sigma \in \mathrm{P}_1, \rho_3 \in \mathrm{P}_3} \text{\small $\widetilde{\sigma}(-x) \widetilde{\rho}_3(-y)$}
\Big[ T_{(0,1)}^{\widetilde{\sigma}}, T_{(3,0)}^{\widetilde{\rho}_3} \Big].
\end{align*}
Calculating the Lie brackets in $\mathsf{H}_X$ and base changing back, the first term of the above sum is given by
\begin{align*}
& \tfrac{(v-v^5-2v^3+2v^{-1})}{2} \E_{(y,1)}^{(1,0)} \oplus \E_{(y,1)}^{(1,0)} \oplus \E_{(x+y,1)}^{(1,1)}
+  \tfrac{(2v-v^5-v^3)}{2} \E_{(y,2)}^{(2,0)} \oplus \E_{(x+y,1)}^{(1,1)} \\
&+ \tfrac{(2v^3-v^5-v^7)}{2}  \E_{(y,1)}^{(1,0)} \oplus \E_{(x+2y,1)}^{(2,1)}
+  \tfrac{(2v^5-v^9 - v^7)}{6}  \E_{(x+3y,1)}^{(3,1)}.
\end{align*}
The second term in $\pi^{\mathrm{vec}} (\mathcal{K}_x * \E') $ becomes
 \begin{multline*}
 \tfrac{(v^5-v)}{2} \E_{(y,1)}^{(1,0)} \oplus \E_{(y,1)}^{(1,0)} \oplus \E_{(x+y,1)}^{(1,1)} +  \tfrac{(v^9 + v^7)}{2}  \E_{(x+3y,1)}^{(3,1)} +  \tfrac{(v^5+v^3)}{2} \E_{(y,2)}^{(2,0)} \oplus \E_{(x+y,1)}^{(1,1)} \\
+ \tfrac{(v^5+v^7)}{2}  \E_{(y,1)}^{(1,0)} \oplus \E_{(x+2y,1)}^{(2,1)},
 \end{multline*} 
and the third term in $\pi^{\mathrm{vec}} (\mathcal{K}_x * \E') $ is given by 
 \text{ \small $\tfrac{(-v^9 - v^7 -v^5)}{3} \E_{(x+3y,1)}^{(3,1)}.$ }
Putting all together concludes the proof.
\end{proof}
\end{lemma}


\begin{lemma} \label{idB.0} Let $\E' := \E_{(y',2)}^{(2,-2)} \oplus \E_{(y,1)}^{(1,0)}$ be a rank $3$ vector bundle on $X$. Then \smallskip

\begin{enumerate}[label={\textbf{\upshape(\roman*)}}, ]

   \item $m_{x,1}\big(\E_{(y',2)}^{(2,-2)}\oplus \E_{(x+y,1)}^{(1,1)}, \E'  \big) =q^2$;  \smallskip

 \item $m_{x,1}\big(\E_{(x+2y',1)}^{(2,-1)} \oplus \E_{(y,1)}^{(1,0)}, \E'  \big) =1$;  \smallskip

  \item $m_{x,1}\big(\E_{(y',1)}^{(1,-1)} \oplus \E_{(y'+x,1)}^{(1,0)}  \oplus \E_{(y,1)}^{(1,0)}, \E'  \big) =q -1 $ if $x+y'  \neq y$;  \smallskip

    \item $m_{x,1}\big(\E_{(y',1)}^{(1,-1)} \oplus \E_{(y,1)}^{(1,0)}  \oplus \E_{(y,1)}^{(1,0)}, \E'  \big) =q^2 -1 $ if $x+y' =y$;  \smallskip

  \item $m_{x,1}\big(\E_{(y',1)}^{(1,-1)} \oplus \E_{(y,2)}^{(2,0)}, \E'  \big) =q -1 $ if $x+y' =y$.

  \end{enumerate}

\begin{proof} Since $\mu(\E_{(y',2)}^{(2,-2)} ) < \mu(\E_{(y,1)}^{(1,0)})$ and $\E_{(y',2)}^{(2,-2)} $ corresponds to $\mathcal{K}_{y'}^{(2)}$ via Atiyah's classification of vector bundles on elliptic curves, we obtain
\begin{align*}
\E' &= \tfrac{N_{1}^{-3} v^2 (1-v^2)}{2} \sum_{\substack{ \rho,\rho_i \in \mathrm{P}_1 \\ i=1,2,}} \text{\small $\widetilde{\rho}_1(-y')\widetilde{\rho}_2(-y') \widetilde{\rho}(-y) $}
T_{(1,-1)}^{\widetilde{\rho}_1}  T_{(1,-1)}^{\widetilde{\rho}_2} T_{(1,0)}^{\widetilde{\rho}} \\
& \quad + \tfrac{N_{2}^{-1}  N_{1}^{-1} v^2 (1+v^2)}{[2]}  \sum_{\substack{ \rho \in \mathrm{P}_1 \\ \chi \in \mathrm{P}_2}} \text{\small $\widetilde{\chi}(-y')\widetilde{\rho}(-y) $}  T_{(2,-2)}^{\widetilde{\chi}} T_{(1,0)}^{\widetilde{\rho}}.
\end{align*}
Hence
\begin{align*}
\pi^{\mathrm{vec}} (\mathcal{K}_x * \E')  & = \tfrac{N_{1}^{-4} v^2 (1-v^2)}{2}\hspace{-8pt} \sum_{\substack{ \sigma,\rho,\rho_i \in \mathrm{P}_1 \\ i=1,2,}} \hspace{-8pt}\text{\small $ \widetilde{\sigma}(-x)\widetilde{\rho}_1(-y')\widetilde{\rho}_2(-y') \widetilde{\rho}(-y) $}
\Big[ T_{(0,1)}^{\widetilde{\sigma}} , T_{(1,-1)}^{\widetilde{\rho}_1}  T_{(1,-1)}^{\widetilde{\rho}_2} T_{(1,0)}^{\widetilde{\rho}} \Big] \\
&\quad + \tfrac{N_{2}^{-1}  N_{1}^{-2} v^2 (1+v^2)}{[2]}  \sum_{\substack{ \sigma, \rho \in \mathrm{P}_1 \\ \chi \in \mathrm{P}_2}} \text{\small $\widetilde{\sigma}(-x)\widetilde{\chi}(-y')\widetilde{\rho}(-y) $}  \Big[ T_{(0,1)}^{\widetilde{\sigma}} ,T_{(2,-2)}^{\widetilde{\chi}} T_{(1,0)}^{\widetilde{\rho}} \Big] .
\end{align*}
We denote by $A$ (resp. $B$) the first (resp. second) term in the above sum. Proceeding as the algorithm in \cite[Section 4]{roberto-elliptic}, we obtain
\begin{align*}
A & =  (v^3 -v^5) T_{(1,-1),y'}  T_{(1,0), x+y'} T_{(1,0),y}  +\tfrac{ (v^3 -v^5)}{2}  T_{(1,-1),y'}  T_{(1,-1),y'}  T_{(1,1),x+y}  \\
& \quad  +\tfrac{ (v^4 -v^6)}{2 }T_{(2,-1),2y'+x}  T_{(1,0),y}  \\
& = \tfrac{(v^3 - v^5)}{2} \Big( v^{-6}(v^2+1) \E_{(y',1)}^{(1,-1)} \oplus \E_{(y',1)}^{(1,-1)}  \oplus \E_{(x+y,1)}^{(1,1)} + v^{-4} \E_{(y',2)}^{(2,-2)}\oplus \E_{(x+y,1)}^{(1,1)} \Big) \\
&  \quad +  \tfrac{(v^3 - v^5)}{2}  \E_{(x+2y',1)}^{(2,-1)} \oplus \E_{(y,1)}^{(1,0)} + A'
\end{align*}
where
\[ A' = \begin{cases}
(v-v^3) \E_{(y',1)}^{(1,-1)} \oplus \E_{(y'+x,1)}^{(1,0)}  \oplus \E_{(y,1)}^{(1,0)}   &\text{ if }  x+y'  \neq y \medskip \\
(v-v^3) \Big( (1+v^{-2}) \E_{(y',1)}^{(1,-1)} \oplus \E_{(y,1)}^{(1,0)}  \oplus \E_{(y,1)}^{(1,0)} + \E_{(y',1)}^{(1,-1)} \oplus \E_{(y,2)}^{(2,0)} \Big) &\text{ if }  x+y'  = y.
\end{cases} \]
We further obtain
\begin{align*}
B & = \tfrac{v^4+ v^2}{[2]} \Big( \tfrac{v^2 [2]}{2} T_{(2,-1),2y'+x}  T_{(1,0),y} + v\; T_{(2,-2),2y'}  T_{(1,1),x+y} \Big) \\
& = \tfrac{(v + v^{-1})}{2} \Big( (1-v^{-2})  \E_{(y',1)}^{(1,-1)} \oplus \E_{(y',1)}^{(1,-1)}  \oplus \E_{(x+y,1)}^{(1,1)}  + \E_{(y',2)}^{(2,-2)}\oplus \E_{(x+y,1)}^{(1,1)} \Big)  \\
& \quad  + \tfrac{(v^3 + v^5)}{2}  \E_{(x+2y',1)}^{(2,-1)} \oplus \E_{(y,1)}^{(1,0)}.
\end{align*}
Adding the multiplicities of a same vector bundle that shows up in $A$ and $B$ and multiplying by $v^{-3}$ completes the proof.
\end{proof}
\end{lemma}


\begin{lemma} \label{idPO} Let $\E' := \E_{(y',1)}^{(1,-1)} \oplus \E_{(y,1)}^{(2,1)} $ be a rank $3$ vector bundle on $X$.  Then \smallskip

\begin{enumerate}[label={\textbf{\upshape(\roman*)}}, ]

   \item $m_{x,1}\big(\E_{(y'+x,1)}^{(1,0)} \oplus \E_{(y,1)}^{(2,1)} , \E'  \big) =1$;  \smallskip

   \item $m_{x,1}\big(\E_{(y',1)}^{(1,-1)} \oplus \E_{(z,2)}^{(2,2)} , \E'  \big) =q^2$;\smallskip

       \item $m_{x,1}\big(\E_{(y',1)}^{(1,-1)} \oplus \E_{(w,2)}^{(2,2)} , \E'  \big) =q^2+q$;
   \end{enumerate}
  where $|z|=1$ with $z=x+y$ and $|w|=2$ with $w=x+y.$

\begin{proof} In $\mathsf{H}_X$ we may write
\[\E' =  v^3 N_{1}^{-2}\sum_{\substack{ \rho_i \in \mathrm{P}_1 \\ i=1,2,}} \text{\footnotesize $\widetilde{\rho}_1(-y')\widetilde{\rho}_2(-y) $}
T_{(1,-1)}^{\widetilde{\rho}_1}  T_{(2,1)}^{\widetilde{\rho}_2}. \]
Thus
\begin{align*}
& \pi^{\mathrm{vec}} (\mathcal{K}_x * \E')    =   v^3 N_{1}^{-3}\sum_{\substack{\sigma, \rho_i \in \mathrm{P}_1 \\ i=1,2}} \text{\small $\widetilde{\sigma}(-x)  \widetilde{\rho}_1(-y')\widetilde{\rho}_2(-y) $} \Big[ T_{(0,1)}^{\widetilde{\sigma}}, T_{(1,-1)}^{\widetilde{\rho}_1}  T_{(2,1)}^{\widetilde{\rho}_2} \Big]  \\
&= v^4 \Big(  T_{(1,0),x+y'} T_{(2,1),y} + \sum_{\substack{ x' \in |X| \\ x'=x+y }} T_{(1,-1),y'} \big( T_{(2,2),x'} + \tfrac{(v^{-1}-v)}{2} T_{(1,1),x'} T_{(1,1),x'} \big) \Big)\\
&= v^3 \E_{(y'+x,1)}^{(1,0)} \oplus \E_{(y,1)}^{(2,1)} + v^{-1} \E_{(y',1)}^{(1,-1)} \oplus \E_{(z,2)}^{(2,2)}  +  (v+v^{-1}) \E_{(y',1)}^{(1,-1)} \oplus \E_{(w,2)}^{(2,2)},
\end{align*}
where $|z|=1$ with $z=x+y$ and $|w|=2$ with $w=x+y.$ Note that the above sum on $x'$ vanishes if $|x'|>2$. The proof follows after multiplying by $v^{-3}.$
\end{proof}
\end{lemma}


\begin{lemma} \label{idR.O} Let $\E' := \E_{(y,2)}^{(2,0)} \oplus \E_{(y',1)}^{(1,1)}  $ be a rank $3$ vector bundle on $X$. Then \smallskip

\begin{enumerate}[label={\textbf{\upshape(\roman*)}}, ]

   \item $m_{x,1}\big( \E_{(y,2)}^{(2,0)} \oplus \E_{(y'+x,1)}^{(1,2)}  , \E'  \big) =q^2$;  \smallskip

   \item $m_{x,1}\big(\E_{(x+2y,1)}^{(2,1)} \oplus \E_{(y',1)}^{(1,1)}   , \E'  \big) =1$; \smallskip

    \item $m_{x,1}\big(\E_{(y,1)}^{(1,0)} \oplus \E_{(y+x,1)}^{(1,1)} \oplus \E_{(y',1)}^{(1,1)}   , \E'  \big) =q-1$, if $y+ x \neq y'$; \smallskip

    \item $m_{x,1}\big(\E_{(y,1)}^{(1,0)} \oplus \E_{(y',1)}^{(1,1)} \oplus \E_{(y',1)}^{(1,1)}   , \E'  \big) =q^2-1$, if $y+ x = y'$; \smallskip

   \item $m_{x,1}\big(\E_{(y,1)}^{(1,0)} \oplus \E_{(y',2)}^{(2,2)}   , \E'  \big) = q -1$, if $y+ x = y'$.
   \end{enumerate}

\begin{proof} In the Hall algebra of $X$, $\E' = v^2 \E_{(y,2)}^{(2,0)} *\E_{(y',1)}^{(1,1)} , $
\[ \E_{(y,2)}^{(2,0)}  = \tfrac{N_{1}^{-2}(1-v^2)}{2} \sum_{\rho,\rho_1\in \mathrm{P}_1} \text{\small $\widetilde{\rho}(-y)\widetilde{\rho}_1(-y) $}
T_{(1,0)}^{\widetilde{\rho}}  T_{(1,0)}^{\widetilde{\rho}_1} + \tfrac{N_{2}^{-1}(1+v^2)}{[2]} \sum_{\rho_2\in \mathrm{P}_2} \text{\small $ \widetilde{\rho}_2(-y) $}  T_{(2,0)}^{\widetilde{\rho}_2} \]
and
\[  \E_{(y',1)}^{(1,1)} = N_{1}^{-1} \sum_{\rho' \in \mathrm{P}_1} \text{\small $\widetilde{\rho}'(-y')$} T_{(1,1)}^{\widetilde{\rho}'}. \]
Therefore
\begin{align*}
\pi^{\mathrm{vec}} (\mathcal{K}_x * \E') & = \tfrac{N_{1}^{-4}(v^2-v^4)}{2} \hspace{-5pt}\sum_{\sigma, \rho,\rho_1, \rho'\in \mathrm{P}_1} \hspace{-5pt}\text{\small $\widetilde{\sigma}(-x) \widetilde{\rho}(-y)\widetilde{\rho}_1(-y) \widetilde{\rho}'(-y')$}
\Big[T_{(0,1)}^{\widetilde{\sigma}}, T_{(1,0)}^{\widetilde{\rho}}  T_{(1,0)}^{\widetilde{\rho}_1} T_{(1,1)}^{\widetilde{\rho}'} \Big] \\
& \quad + \tfrac{N_{1}^{-2} N_{2}^{-1}(v^4+v^2)}{[2]} \sum_{\substack{ \sigma. \rho' \in \mathrm{P}_1 \\ \rho_2\in \mathrm{P}_2}} \text{\small $ \widetilde{\sigma}(-x)\widetilde{\rho}_2(-y)\widetilde{\rho}'(-y')$}  \Big[T_{(0,1)}^{\widetilde{\sigma}},  T_{(2,0)}^{\widetilde{\rho}_2} T_{(1,1)}^{\widetilde{\rho}'} \Big].
\end{align*}
Let us set $A$ (resp. $B$) the first (resp. second) term in the previous sum. By \cite[Section 4]{roberto-elliptic}, we obtain
\begin{align*}
A &= \tfrac{(v^3-v^5)}{2} \left( (v^{-4}+v^{-6}) \E_{(y,1)}^{(1,0)} \oplus \E_{(y,1)}^{(1,0)} \oplus \E_{(y'+x,1)}^{(1,2)} + v^{-4}   \E_{(y,2)}^{(2,0)} \oplus \E_{(y'+x,1)}^{(1,2)}  \right) \\
& \quad +  \tfrac{(v^3-v^5)}{2} \E_{(x+2y,1)}^{(2,1)} \oplus \E_{(y',1)}^{(1,1)}  + A'
\end{align*}
where
\[ A' = \begin{cases} (v-v^3) \left( (1+v^{-2}) \E_{(y,1)}^{(1,0)} \oplus \E_{(y',1)}^{(1,1)} \oplus \E_{(y',1)}^{(1,1)}  + \E_{(y,1)}^{(1,0)} \oplus \E_{(y',2)}^{(2,2)} \right) \hfill & \text{if } x+y=y', \medskip \\
(v-v^3) \left( \E_{(y,1)}^{(1,0)} \oplus \E_{(y+x,1)}^{(1,1)} \oplus \E_{(y',1)}^{(1,1)} \right) \hfill & \text{if } x+y \neq y'.  \end{cases} \]
Analogously,
{ \small \[ B = \tfrac{(v^3+v^5)}{2} \E_{(x+2y,1)}^{(2,1)} \oplus \E_{(y',1)}^{(1,1)}  + \tfrac{(v+v^{-1})}{2}\left( (1-v^{-2})    \E_{(y,1)}^{(1,0)} \oplus \E_{(y,1)}^{(1,0)} \oplus \E_{(y'+x,1)}^{(1,2)}  + \E_{(y,2)}^{(2,0)} \oplus \E_{(y'+x,1)}^{(1,2)}  \right).
\]}
Multiplying the sum $A + B$ by $v^{-3}$ yields the desired result.
\end{proof}
\end{lemma}


\begin{lemma} \label{idN.1} Let $\E' := \E_{(y,1)}^{(2,0)} \oplus \E_{(y',1)}^{(1,1)}  $ be a rank $3$ vector bundle on $X$.  Then \smallskip

\begin{enumerate}[label={\textbf{\upshape(\roman*)}}, ]

   \item $m_{x,1}\big(\E_{(y+x,1)}^{(2,1)} \oplus \E_{(y',1)}^{(1,1)}   , \E'  \big) =1$;  \smallskip

   \item $m_{x,1}\big(\E_{(y,1)}^{(2,0)} \oplus \E_{(y'+x,1)}^{(1,2)}   , \E'  \big) =q^2$.
   \end{enumerate}

\begin{proof} In the Hall algebra  $\mathsf{H}_X$, we can write $\E'$ as follows
\[ \E' = v^2 \big( \E_{(y,1)}^{(2,0)} * \E_{(y',1)}^{(1,1)} \big) =
\tfrac{2 v^2 N_{2}^{-1} N_{1}^{-1}}{[2] }  \sum_{\substack{\rho_i \in \mathrm{P}_i \\ i=1,2,}} \text{\small $\widetilde{\rho}_1(-y')\widetilde{\rho}_2(-y) $}
T_{(2,0)}^{\widetilde{\rho}_2}  T_{(1,1)}^{\widetilde{\rho}_1}.\]
Hence
\[ \pi^{\mathrm{vec}} (\mathcal{K}_x * \E')  = \tfrac{2 v^2 N_{2}^{-1} N_{1}^{-2}}{[2] }  \sum_{\substack{\sigma, \rho_i \in \mathrm{P}_i \\ i=1,2,}} \text{\small $\widetilde{\sigma}(-x)\widetilde{\rho}_1(-y')\widetilde{\rho}_2(-y) $}  \Big[ T_{(0,1)}^{\widetilde{\sigma}}, T_{(2,0)}^{\widetilde{\rho}_2}  T_{(1,1)}^{\widetilde{\rho}_1} \Big].\]
Since
\[\Big[ T_{(0,1)}^{\widetilde{\sigma}} T_{(2,0)}^{\widetilde{\rho}_2}   \Big] = \begin{cases}  c_2 T_{(2,1)}^{\widetilde{\sigma}} \hfill & \text{if } \rho_2 = \Norm_{1}^{2}(\sigma) \\
0 & \text{otherwise} \end{cases}\]
and
\[\Big[ T_{(0,1)}^{\widetilde{\sigma}} T_{(2,0)}^{\widetilde{\rho}_1}   \Big] = \begin{cases}  c_1 T_{(1,2)}^{\widetilde{\sigma}} \hfill & \text{if } \rho_1 = \sigma \\
0 & \text{otherwise}, \end{cases}\]
we have
$$\pi^{\mathrm{vec}} (\mathcal{K}_x * \E') = v^3 \E_{(y+x,1)}^{(2,1)} \oplus \E_{(y',1)}^{(1,1)}  + v^{-1} \E_{(y,1)}^{(2,0)} \oplus \E_{(y'+x,1)}^{(1,2)}.$$
 Therefore the lemma follows from multiplying $\pi^{\mathrm{vec}} (\mathcal{K}_x * \E')$ by $v^{-3}.$
\end{proof}
\end{lemma}


\begin{lemma} \label{idSA} Let $\E' := \E_{(y,1)}^{(2,-1)} \oplus \E_{(y',1)}^{(1,1)}  $ be a rank $3$ vector bundle on $X$.  Then \smallskip

\begin{enumerate}[label={\textbf{\upshape(\roman*)}}, ]
   \item $m_{x,1}\big(\E_{(y+x,2)}^{(2,0)} \oplus \E_{(y',1)}^{(1,1)} , \E'  \big) =q$;  \smallskip

   \item $m_{x,1}\big(\E_{(y,1)}^{(2,-1)} \oplus \E_{(x+y',1)}^{(1,2)} , \E'  \big) =q^2$.
   \end{enumerate}
\end{lemma}

\begin{proof} 
We can write $\E'$ in $\mathsf{H}_X$ as follows:
\[ \E' = v^3 \big( \E_{(y,1)}^{(2,-1)} * \E_{(y',1)}^{(1,1)} \big) =
v^3 N_{1}^{-2}  \sum_{\substack{\rho_i \in \mathrm{P}_1 \\ i=1,2,}} \text{\small $\widetilde{\rho}_1(-y)\widetilde{\rho}_2(-y') $}
T_{(2,-1)}^{\widetilde{\rho}_1}  T_{(1,1)}^{\widetilde{\rho}_2}.\]
Thus
\begin{align*}
 \pi^{\mathrm{vec}} (\mathcal{K}_x * \E')  &= v^3 N_{1}^{-3}   \sum_{\substack{\sigma, \rho_i \in \mathrm{P}_1 \\ i=1,2,}} \text{\small $\widetilde{\sigma}(-x)  \widetilde{\rho}_1(-y)\widetilde{\rho}_2(-y') $}
\Big[ T_{(0,1)}^{\widetilde{\sigma}}, T_{(2,-1)}^{\widetilde{\rho}_1}  T_{(1,1)}^{\widetilde{\rho}_2}\Big] \\
& = v^4 \Big( \big(T_{(2,0),w} +  \tfrac{(v^{-1}-v)}{2} T_{(1,0),w} T_{(1,0),w} \big) T_{(1,1),y'} +  T_{(2,-1),y} T_{(1,2),w'}    \Big)  \\
&= v \; \E_{(w,2)}^{(2,0)} \oplus \E_{(y',1)}^{(1,1)}  + v^{-1} \E_{(y,1)}^{(2,-1)} \oplus \E_{(w',1)}^{(1,2)},
\end{align*}
where $w=x+y$ and $w'=x+y'$. The proof follows from multiplying above equation by $v^{-3}.$
\end{proof}


\subsection{Edges terminating in a sum of three  line bundles }

\begin{lemma} \label{idB} Let $\E' := \E_{(x_1,1)}^{(1,0)} \oplus \E_{(x_2,1)}^{(1,0)} \oplus \E_{(x_3,1)}^{(1,0)}$ be a rank $3$ vector bundle  and $x_1,x_2,x_3$ be pairwise distinct closed points on $X$. Then \smallskip

\begin{enumerate}[label={\textbf{\upshape(\roman*)}},]
\item $m_{x,1}\big(\E_{(x_i,1)}^{(1,0)} \oplus \E_{(x_j,1)}^{(1,0)} \oplus \E_{(x+x_k,1)}^{(1,1)}, \E'  \big) = q^2$, where  $\{i,j,k\}=\{1,2,3\}$; \smallskip

\item $m_{x,1}\big(\E_{(x_i,1)}^{(1,0)}  \oplus \E_{(y,1)}^{(2,1)}, \E'  \big) = q$, where  $y=x+x_j+x_k$ and $\{i,j,k\}=\{1,2,3\}$;
\smallskip

\item   $m_{x,1}\big(\E_{(z,1)}^{(3,1)} , \E'  \big) = 1$,  where   $z=x+x_1+x_2+x_3$.
\end{enumerate}

\begin{proof} Applying \cite[Thm. 4.6]{roberto-elliptic} we can write
\[ \E' = N_1^{-3} \sum_{\sigma_1,\sigma_2,\sigma_3 \in \mathrm{P}_1}\text{ \small $ \widetilde{\sigma}_1(-x_1)\widetilde{\sigma}_2(-x_2)\widetilde{\sigma}_3(-x_3)$} T_{(1,0)}^{\widetilde{\sigma}_1} T_{(1,0)}^{\widetilde{\sigma}_2} T_{(1,0)}^{\widetilde{\sigma}_3}.\]
Hence
\[ \pi^{\mathrm{vec}} (\mathcal{K}_x * \E') = N_1^{-4} \hspace{-10pt} \sum_{\sigma, \sigma_1,\sigma_2,\sigma_3 \in \mathrm{P}_1}\hspace{-10pt}\text{ \small $\widetilde{\sigma}(-x) \widetilde{\sigma}_1(-x_1)\widetilde{\sigma}_2(-x_2)\widetilde{\sigma}_3(-x_3)$} \Big[T_{(0,1)}^{\widetilde{\sigma}}, T_{(1,0)}^{\widetilde{\sigma}_1} T_{(1,0)}^{\widetilde{\sigma}_2} T_{(1,0)}^{\widetilde{\sigma}_3} \Big].\]
By the algorithm in \cite[Section 4]{roberto-elliptic} we obtain that
\begin{align*}
\pi^{\mathrm{vec}} (\mathcal{K}_x * \E') & = v \Big(  \E_{(x_1,1)}^{(1,0)} * \E_{(x_2,1)}^{(1,0)} * \E_{(x+x_3,1)}^{(1,1)} +
\E_{(x_1,1)}^{(1,0)} *\E_{(x_3,1)}^{(1,0)} * \E_{(x+x_2,1)}^{(1,1)} \\
&\quad + \E_{(x_2,1)}^{(1,0)} * \E_{(x_3,1)}^{(1,0)} * \E_{(x+x_1,1)}^{(1,1)} \Big) + v^2 \Big( \E_{(x_1,1)}^{(1,0)} * \E_{(x+x_2+x_3,1)}^{(2,1)} \\
&\quad + \E_{(x_2,1)}^{(1,0)}  * \E_{(x+x_1+x_3,1)}^{(2,1)} + \E_{(x_3,1)}^{(1,0)}  * \E_{(x+x_1+x_2,1)}^{(2,1)} \Big) + v^3  \E_{(z,1)}^{(3,1)}
\end{align*}
where $z=x+x_1+x_2+x_3$.
By definition of the product of the Hall algebra, multiplying the above identity by $v^{-3}$ yields the desired assertion.
\end{proof}
\end{lemma}


\begin{lemma} \label{idK} Let $\E' := \E_{(y_1,1)}^{(1,-1)} \oplus \E_{(y_2,1)}^{(1,0)} \oplus \E_{(y_3,1)}^{(1,0)} $ be a rank $3$ vector bundle on $X$ with $y_2 \neq y_3$.  Then, \smallskip

\begin{enumerate}[label={\textbf{\upshape(\roman*)}}, ]

  \item $m_{x,1}\big(\E_{(y_1,1)}^{(1,-1)} \oplus \E_{(y_2,1)}^{(1,0)} \oplus \E_{(x+y_3,1)}^{(1,1)}, \E'  \big) =q^2$;  \smallskip

    \item $m_{x,1}\big(\E_{(y_1,1)}^{(1,-1)} \oplus \E_{(y_3,1)}^{(1,0)} \oplus \E_{(x+y_2,1)}^{(1,1)}, \E'  \big) =q^2$;  \smallskip

      \item $m_{x,1}\big(\E_{(y_1,1)}^{(1,-1)} \oplus \E_{(x+y_2+y_3,1)}^{(2,1)}, \E'  \big) =q$.  \smallskip
\end{enumerate}

\noindent
Moreover, depending on $y_1, y_2, y_3$   we have the following additional edges:

\begin{enumerate}[label={\textbf{\upshape(\roman*)}}, resume]

\item If $x+y_1=y_i$, then \smallskip

  \begin{enumerate}
\item $m_{x,1}\big(\E_{(y_i,1)}^{(1,0)} \oplus \E_{(y_i,1)}^{(1,0)} \oplus \E_{(y_j,1)}^{(1,0)}, \E'  \big) =q +1$;  \smallskip

\item $m_{x,1}\big(\E_{(y_i,2)}^{(2,0)}  \oplus \E_{(y_j,1)}^{(1,0)}, \E'  \big) = 1$,
\end{enumerate} \smallskip

\noindent
 for $\{i,j\} = \{2,3\}.$ 
 \smallskip

\item If $x+y_1 \neq y_2$ and $x+y_1 \neq y_3$, then
   $$m_{x,1}\big(\E_{(y_1+x,1)}^{(1,0)} \oplus \E_{(y_2,1)}^{(1,0)} \oplus \E_{(y_3,1)}^{(1,0)}, \E'  \big) =1.$$

\end{enumerate}

\begin{proof} Since $y_2 \neq y_3$, we might write $\E'$ in the Hall algebra $\mathsf{H}_X$ as follows
\begin{align*}
\E' &= v^2 \; \E_{(y_1,1)}^{(1,-1)} * \E_{(y_2,1)}^{(1,0)} * \E_{(y_3,1)}^{(1,0)}  \\
&= \text{ \small $v^2 N_1^{-3}$} \sum_{ \substack{ \rho_i \in \mathrm{P}_1 \\ i=1,2,3}} \text{ \small $\widetilde{\rho}_1(-y_1) \widetilde{\rho}_2(-y_2) \widetilde{\rho}_3(-y_3)$} T_{(1,-1)}^{\widetilde{\rho}_1}  T_{(1,0)}^{\widetilde{\rho}_2} T_{(1,0)}^{\widetilde{\rho}_3}.
\end{align*}
By the algorithm in \cite{roberto-elliptic}, we conclude that
\begin{align*}
\pi^{\mathrm{vec}} (\mathcal{K}_x * \E') & = v^2 N_1^{-4} \Big( c_1 N_1^{3}  T_{(1,0), x+y_1} T_{(1,0), y_2} T_{(1,0), y_3} \\
&\quad + c_1 N_1^{3}  T_{(1,-1), y_1} T_{(1,0), y_2} T_{(1,1), x+ y_3}   \\
&\quad  +  c_1 N_1^{3}  T_{(1,-1), y_1} T_{(1,0), y_3} T_{(1,1), x+ y_2}  + c_1^2 N_1^{2}  T_{(1,-1), y_1} T_{(2,1), x+ y_2+y_3} \Big).
\end{align*}
We denote the four terms in the above identity by 
$$A:= v^3 T_{(1,0), x+y_1} T_{(1,0), y_2} T_{(1,0), y_3}, \quad B:= v^3  T_{(1,-1), y_1} T_{(1,0), y_2} T_{(1,1), x+ y_3},$$  
$$C := v^3  T_{(1,-1), y_1} T_{(1,0), y_3} T_{(1,1), x+ y_2}  \  \ \text{ and }  \ \ D:= v^4 T_{(1,-1), y_1} T_{(2,1), x+ y_2+y_3}.$$ 
Since the elements in $B, C$ and $ D$ have slopes strictly increasing, using the definition of Hall algebras product and multiplying by $v^{-3}$ yields the first three assertions.

If $x+y_1 = y_2$ in $A$, thus
\[ A = (v^3+v)    \E_{(y_2,1)}^{(1,0)} \oplus \E_{(y_2,1)}^{(1,0)} \oplus \E_{(y_3,1)}^{(1,0)} + v^3 \E_{(y_2,2)}^{(2,0)}  \oplus \E_{(y_3,1)}^{(1,0)}.\]
Analogous for $x+y_1 = y_3$. Multiplying by $v^{-3}$ we have the desired last two assertions.
\end{proof}
\end{lemma}


\par\needspace{2\baselineskip}  
\begin{lemma} \label{idI} Let $\E' := \E_{(y',1)}^{(1,-1)} \oplus \E_{(y,1)}^{(1,0)} \oplus \E_{(y,1)}^{(1,0)} $ be a rank $3$ vector bundle on $X$. \smallskip

\begin{enumerate}[label={\textbf{\upshape(\roman*)}}, ]

\item If $x+y' \neq y$, then \smallskip

\begin{enumerate}
  \item $m_{x,1}\big(\E_{(y'+x,1)}^{(1,0)} \oplus \E_{(y,1)}^{(1,0)} \oplus \E_{(y,1)}^{(1,0)}, \E'  \big) =1$;  \smallskip

    \item $m_{x,1}\big(\E_{(y',1)}^{(1,-1)} \oplus \E_{(y,1)}^{(1,0)} \oplus \E_{(x+y,1)}^{(1,1)}, \E'  \big) =q$.  \smallskip
\end{enumerate}

\item  If $x+y' =  y$, then \smallskip

\begin{enumerate}
  \item $m_{x,1}\big(\E_{(y,1)}^{(1,0)} \oplus \E_{(y,1)}^{(1,0)} \oplus \E_{(y,1)}^{(1,0)}, \E'  \big) =q^2+q+1$;  \smallskip

    \item $m_{x,1}\big(\E_{(y,2)}^{(2,0)} \oplus \E_{(y,1)}^{(1,0)} , \E'  \big) =1$;  \smallskip

  \item $m_{x,1}\big(\E_{(y-x,1)}^{(1,-1)} \oplus \E_{(y,1)}^{(1,0)} \oplus \E_{(x+y,1)}^{(1,1)}, \E'  \big) =q$.
\end{enumerate}

\end{enumerate}

\begin{proof} We first write $\E'$ in the terms of elements $T_{\mathbf{v}}^{\widetilde{\rho}}$ in the Definition \ref{defofT}.  Proceeding as in the proof of \cite[Thms. 4.2 and 4.6]{roberto-elliptic}, we obtain
{\small \begin{align*}
 \E'    &= v^2 \E_{(y',1)}^{(1,-1)} * \Big(  \E_{(y,1)}^{(1,0)} * \E_{(y,1)}^{(1,0)} \Big)  \\
 & = v^2 \Big( N_{1}^{-1 } \sum_{\chi \in \mathrm{P}_1} \text{\tiny $\widetilde{\chi}(-y')$} T_{(1,-1)}^{\widetilde{\chi}} \Big) *
 \Big( \tfrac{v^2 N_{1}^{-2 } }{2} \sum_{\rho, \rho_1 \in \mathrm{P}_1} \text{\tiny $ \widetilde{\rho}(-y) \widetilde{\rho}_1(-y)$} T_{(1,0)}^{\widetilde{\rho}} T_{(1,0)}^{\widetilde{\rho}_1}
 - \tfrac{v^2 N_{2}^{-1 } }{[2]} \sum_{\rho_2 \in \mathrm{P}_2}  \text{\tiny $ \widetilde{\rho}_2(-y)$} T_{(2,0)}^{\widetilde{\rho}_2} \Big)  \\
 & = \tfrac{v^4}{2} \Big( N_{1}^{-3 } \sum_{\chi, \rho, \rho_1 \in \mathrm{P}_1} \text{\tiny $ \widetilde{\chi}(-y') \widetilde{\rho}(-y) \widetilde{\rho}_1(-y)$} T_{(1,-1)}^{\widetilde{\chi}} T_{(1,0)}^{\widetilde{\rho}} T_{(1,0)}^{\widetilde{\rho}_1}
 -  \tfrac{2 N_{1}^{-1 }  N_{2}^{-1 } }{[2]} \sum_{\substack{ \chi \in \mathrm{P}_1 \\ \rho_2 \in \mathrm{P}_2}} \text{\tiny $\widetilde{\chi}(-y')\widetilde{\rho}_2(-y) $} T_{(1,-1)}^{\widetilde{\chi}}  T_{(2,0)}^{\widetilde{\rho}_2}  \Big).
 \end{align*}}
Hence the product $\mathcal{K}_x * \E'$ in $\mathsf{H}_X^{\mathrm{vec}} $ is given by
\begin{align*}
\pi^{\mathrm{vec}} (\mathcal{K}_x * \E') & = \tfrac{v^4 N_{1}^{-4 } }{2}  \sum_{\sigma, \chi, \rho, \rho_1 \in \mathrm{P}_1} \text{\small $\widetilde{\sigma}(-x) \widetilde{ \chi}(-y') \widetilde{\rho}(-y) \widetilde{\rho}_1(-y)$} \Big[ T_{(0,1)}^{\widetilde{\sigma}}, T_{(1,-1)}^{\widetilde{\chi}} T_{(1,0)}^{\widetilde{\rho}} T_{(1,0)}^{\widetilde{\rho}_1} \Big] \\
& \quad- \tfrac{v^4 N_{1}^{-2 }  N_{2}^{-1 } }{[2]} \sum_{\substack{ \sigma, \chi \in \mathrm{P}_1 \\ \rho_2 \in \mathrm{P}_2}} \text{\small $\widetilde{\sigma}(-x) \widetilde{\chi}(-y')\widetilde{\rho}_2(-y) $} \Big[ T_{(0,1)}^{\widetilde{\sigma}},T_{(1,-1)}^{\widetilde{\chi}}  T_{(2,0)}^{\widetilde{\rho}_2} \Big].
\end{align*}
We calculate each term of the above sum separatedly. We denote by $A$ the first term and by $B$ the second one.

Suppose $x+y'\neq y$. Applying the algorithm developed in \cite[Section 4]{roberto-elliptic} yields
\begin{align*}
A &= \tfrac{(v^3+v^5)}{2} \E_{(y'+x,1)}^{(1,0)} \oplus \E_{(y,1)}^{(1,0)} \oplus \E_{(y,1)}^{(1,0)}
+ \tfrac{v^5}{2} \E_{(y'+x,1)}^{(1,0)} \oplus \E_{(y,2)}^{(2,0)} + \tfrac{v^3}{2}  \E_{(y',1)}^{(1,-1)} \oplus \E_{(x+2y,1)}^{(2,1)}  \\
 &  \quad + v \E_{(y',1)}^{(1,-1)} \oplus \E_{(y,1)}^{(1,0)} \oplus \E_{(x+y,1)}^{(1,1)}
\intertext{and }
B &= \tfrac{(v^3-v^5)}{2} \E_{(y'+x,1)}^{(1,0)} \oplus \E_{(y,1)}^{(1,0)} \oplus \E_{(y,1)}^{(1,0)}
- \tfrac{v^5}{2} \E_{(y'+x,1)}^{(1,0)} \oplus \E_{(y,2)}^{(2,0)}
- \tfrac{v^3}{2}  \E_{(y',1)}^{(1,-1)} \oplus \E_{(x+2y,1)}^{(2,1)}.
\end{align*}
Suppose $x+y' = y$. Applying the algorithm in \cite[Section 4]{roberto-elliptic} once again yields
\begin{align*}
A &= 2^{-1}(v^5+2v^3+2v+v^{-1}) \E_{(y,1)}^{(1,0)} \oplus \E_{(y,1)}^{(1,0)} \oplus \E_{(y,1)}^{(1,0)} + \tfrac{(v^5+2v^3)}{2} \E_{(y,2)}^{(2,0)} \oplus \E_{(y,1)}^{(1,0)} \\
&\quad + \tfrac{v^5}{2} \E_{(y,3)}^{(3,0)} + v \E_{(y-x,1)}^{(1,-1)} \oplus \E_{(y,1)}^{(1,0)} \oplus \E_{(x+y,1)}^{(1,1)} + \tfrac{v^3}{2} \E_{(y',1)}^{(1,-1)} \oplus \E_{(x+2y,1)}^{(2,1)}
\end{align*}
and
\begin{align*}
B &= \tfrac{(v^{-1}-v^5)}{2} \E_{(y,1)}^{(1,0)} \oplus \E_{(y,1)}^{(1,0)} \oplus \E_{(y,1)}^{(1,0)} - \tfrac{v^5}{2} \E_{(y,2)}^{(2,0)} \oplus \E_{(y,1)}^{(1,0)}   - \tfrac{v^5}{2} \E_{(y,3)}^{(3,0)} \\
  &\quad - \tfrac{v^3}{2} \E_{(y',1)}^{(1,-1)} \oplus \E_{(x+2y,1)}^{(2,1)}.
\end{align*}
Adding the multiplicities for each vector bundle proves the lemma.
\end{proof}

\end{lemma}


\begin{lemma} \label{idC.0} \label{idN20} Let $\E' := \E_{(y',1)}^{(1,-1)} \oplus \E_{(z,1)}^{(1,-1)} \oplus \E_{(y,1)}^{(1,0)}  $ be a rank $3$ vector bundle on $X$ with $z \neq y'$. Then, \smallskip

\begin{enumerate}[label={\textbf{\upshape(\roman*)}}, ]
    \item $m_{x,1}\big(\E_{(y',1)}^{(1,-1)} \oplus \E_{(z,1)}^{(1,-1)} \oplus \E_{(y+x,1)}^{(1,1)}, \E'  \big) =q^2$;  \smallskip

      \item $m_{x,1}\big(\E_{(w,1)}^{(2,-1)} \oplus \E_{(y,1)}^{(1,0)}, \E'  \big) =1$, with $w = y'+z+x$;  \smallskip

  \item $m_{x,1}\big(\E_{(z,1)}^{(1,-1)} \oplus \E_{(x+y',1)}^{(1,0)} \oplus \E_{(y,1)}^{(1,0)}, \E'  \big) =q$, if $y'+x \neq y;$ \smallskip

    \item $m_{x,1}\big(\E_{(z,1)}^{(1,-1)} \oplus \E_{(y,1)}^{(1,0)} \oplus \E_{(y,1)}^{(1,0)}, \E'  \big) =q^2 +q$, if $y'+x = y;$ \smallskip

 \item $m_{x,1}\big(\E_{(z,1)}^{(1,-1)} \oplus \E_{(y,2)}^{(2,0)}, \E'  \big) =q$,  if $y'+x = y;$\smallskip

  \item $m_{x,1}\big(\E_{(y',1)}^{(1,-1)} \oplus \E_{(x+z,1)}^{(1,0)} \oplus \E_{(y,1)}^{(1,0)}, \E'  \big) =q$, if $z+x \neq y;$ \smallskip

    \item $m_{x,1}\big(\E_{(y',1)}^{(1,-1)} \oplus \E_{(y,1)}^{(1,0)} \oplus \E_{(y,1)}^{(1,0)}, \E'  \big) =q^2 +q$, if $z+x = y;$ \smallskip

 \item $m_{x,1}\big(\E_{(y',1)}^{(1,-1)} \oplus \E_{(y,2)}^{(2,0)}, \E'  \big) =q$,  if $z+x = y.$\smallskip

  \end{enumerate}

\begin{proof} Since $z \neq y',$ in the Hall algebra of $X$ we can write
\[ \E' =N_{1}^{-3} v^2  \sum_{\substack{ \rho_i \in \mathrm{P}_1 \\ i=1,2,3}} \text{\small $\widetilde{\rho}_1(-y')\widetilde{\rho}_2(-z) \widetilde{\rho}_3(-y) $}
T_{(1,-1)}^{\widetilde{\rho}_1}  T_{(1,-1)}^{\widetilde{\rho}_2} T_{(1,0)}^{\widetilde{\rho}_3}.    \]
Thus
\begin{align*}
\pi^{\mathrm{vec}} (\mathcal{K}_x * \E')  & = N_{1}^{-4} v^2\sum_{\substack{ \sigma, \rho_i \in \mathrm{P}_1\\ i=1,2,3}} \text{\small $\widetilde{\sigma}(-x) \widetilde{\rho}_1(-y')\widetilde{\rho}_2(-z) \widetilde{\rho}_3(-y) $}
\Big[ T_{(0,1)}^{\widetilde{\sigma}}, T_{(1,-1)}^{\widetilde{\rho}_1}  T_{(1,-1)}^{\widetilde{\rho}_2} T_{(1,0)}^{\widetilde{\rho}_3} \Big]   \\
& = v^3 \Big( T_{(1,-1),z} T_{(1,0), x+y'} T_{(1,0),y} + T_{(1,-1),y'} T_{(1,0),x+z}  T_{(1,0),y}  \\
&\quad   +v \; T_{(2,-1), x+y'+z} T_{(1,0), y}  +  T_{(1,-1),y'} T_{(1,-1), z} T_{(1,1),x+y}   \Big) \\
& =   v^{-1} \E_{(y',1)}^{(1,-1)} \oplus \E_{(z,1)}^{(1,-1)} \oplus \E_{(y+x,1)}^{(1,1)} +
v^3  \E_{(x+y'+z,1)}^{(2,-1)} \oplus \E_{(y,1)}^{(1,0)}  \\
&  \quad + v^3 T_{(1,-1),z} T_{(1,0), x+y'} T_{(1,0),y} + v^3 T_{(1,-1),y'} T_{(1,0),x+z}  T_{(1,0),y} .
\end{align*}
We define $A := v^3 T_{(1,-1),z} T_{(1,0), x+y'} T_{(1,0),y}$ and $B:= v^3 T_{(1,-1),y'} T_{(1,0),x+z}  T_{(1,0),y}.$ Analyzing the cases in which $x+y'$ and  $x+z$ are equal to and different from $y$, we obtain
\[ A = \begin{cases} v \; \E_{(z,1)}^{(1,-1)} \oplus \E_{(x+y',1)}^{(1,0)} \oplus \E_{(y,1)}^{(1,0)} \hfill &\text{ if } y'+x \neq y; \medskip \\

(v+v^{-1}) \E_{(z,1)}^{(1,-1)} \oplus \E_{(y,1)}^{(1,0)} \oplus \E_{(y,1)}^{(1,0)} + v\;\E_{(z,1)}^{(1,-1)} \oplus \E_{(y,2)}^{(2,0)} \hfill &\text{ if } y'+x = y
\end{cases} \]
and
\[ B = \begin{cases} v \; \E_{(y',1)}^{(1,-1)} \oplus \E_{(x+z,1)}^{(1,0)} \oplus \E_{(y,1)}^{(1,0)} \hfill &\text{ if } z+x \neq y; \medskip \\
(v+v^{-1}) \E_{(y',1)}^{(1,-1)} \oplus \E_{(y,1)}^{(1,0)} \oplus \E_{(y,1)}^{(1,0)} + v\;\E_{(y',1)}^{(1,-1)} \oplus \E_{(y,2)}^{(2,0)} \hfill &\text{ if } z +x = y.
\end{cases}\]
Putting all together and multiplying by $v^{-3}$ yields the desired assertion.
\end{proof}
\end{lemma}


\begin{lemma} \label{Prop26/01} Let $\E' := \E_{(y',1)}^{(1,-1)} \oplus \E_{(y',1)}^{(1,-1)}  \oplus \E_{(y,1)}^{(1,0)}$ be a rank $3$ vector bundle on $X$.  If $y=y'+x$, then \smallskip

\begin{enumerate}[label={\textbf{\upshape(\roman*)}}, ]

   \item $m_{x,1}\big(\E_{(y',1)}^{(1,-1)} \oplus \E_{(y,1)}^{(1,0)}  \oplus \E_{(y,1)}^{(1,0)}, \E'  \big) =q +1$;  \smallskip

   \item $m_{x,1}\big(\E_{(y',1)}^{(1,-1)} \oplus \E_{(y,2)}^{(2,0)}, \E'  \big) =1$;  \smallskip

      \item $m_{x,1}\big(\E_{(y',1)}^{(1,-1)} \oplus \E_{(y',1)}^{(1,-1)}  \oplus \E_{(y+x,1)}^{(1,1)}, \E'  \big) =q^2$.  \smallskip
  \end{enumerate}

\begin{proof} In $\mathsf{H}_X$ we may write
\begin{multline*}
 \E' = \tfrac{ v^4}{2} \Big( N_{1}^{-3}\sum_{\substack{ \rho,\rho_i \in \mathrm{P}_1 \\ i=1,2,}} \text{\small $\widetilde{\rho}_1(-y')\widetilde{\rho}_2(-y') \widetilde{\rho}(-y) $}
T_{(1,-1)}^{\widetilde{\rho}_1}  T_{(1,-1)}^{\widetilde{\rho}_2} T_{(1,0)}^{\widetilde{\rho}} \\
 - \tfrac{2 N_{2}^{-1}  N_{1}^{-1}}{[2]}  \sum_{\substack{ \rho \in \mathrm{P}_1 \\ \chi \in \mathrm{P}_2}} \text{\small $\widetilde{\chi}(-y')\widetilde{\rho}(-y) $}  T_{(2,-2)}^{\widetilde{\chi}} T_{(1,0)}^{\widetilde{\rho}} \Big). 
\end{multline*}
Since $ \mathcal{K}_x = N_{1}^{-1 } \sum_{\sigma\in \mathrm{P}_1} \widetilde{\sigma}(-x) T_{(0,1)}^{\widetilde{\sigma}}$,
\begin{align*}
\pi^{\mathrm{vec}} (\mathcal{K}_x * \E')  & = \tfrac{N_{1}^{-4} v^4}{2} \sum_{\substack{ \sigma,\rho,\rho_i \in \mathrm{P}_1 \\ i=1,2,}} \text{\small $ \widetilde{\sigma}(-x)\widetilde{\rho}_1(-y')\widetilde{\rho}_2(-y') \widetilde{\rho}(-y) $}
\Big[ T_{(0,1)}^{\widetilde{\sigma}} , T_{(1,-1)}^{\widetilde{\rho}_1}  T_{(1,-1)}^{\widetilde{\rho}_2} T_{(1,0)}^{\widetilde{\rho}} \Big] \\
& \quad - \tfrac{N_{2}^{-1}  N_{1}^{-2} v^4}{[2]}  \sum_{\substack{ \sigma, \rho \in \mathrm{P}_1 \\ \chi \in \mathrm{P}_2}} \text{\small $\widetilde{\sigma}(-x)\widetilde{\chi}(-y')\widetilde{\rho}(-y) $}  \Big[ T_{(0,1)}^{\widetilde{\sigma}} ,T_{(2,-2)}^{\widetilde{\chi}} T_{(1,0)}^{\widetilde{\rho}} \Big] .
\end{align*}
Let $A$ (resp. $B$) be the first (resp. second) term on the right side of this identity. Using the algorithm of \cite[Section 4]{roberto-elliptic} as the previous lemma, we obtain
\begin{align*}
A &= (v^3+v) \E_{(y',1)}^{(1,-1)} \oplus \E_{(y,1)}^{(1,0)}  \oplus \E_{(y,1)}^{(1,0)} + v^3 \E_{(y',1)}^{(1,-1)} \oplus \E_{(y,2)}^{(2,0)} \\
& \quad + \tfrac{(v+v^{-1})}{2} \E_{(y',1)}^{(1,-1)} \oplus \E_{(y',1)}^{(1,-1)}  \oplus \E_{(y+x,1)}^{(1,1)} + \tfrac{v}{2} \E_{(y',2)}^{(2,-2)} \oplus \E_{(y+x,1)}^{(1,1)}
+ \tfrac{v^5}{2} \E_{(2y'+x,1)}^{(2,-1)} \oplus \E_{(y,1)}^{(1,0)}  \\
\intertext{and}
B &= \tfrac{v^5}{2} \E_{(2y'+x,1)}^{(2,-1)} \oplus \E_{(y,1)}^{(1,0)}  +  \tfrac{v}{2} \;\E_{(y',2)}^{(2,-2)} \oplus \E_{(y+x,1)}^{(1,1)}
+ \tfrac{(v-v^{-1})}{2} \E_{(y',1)}^{(1,-1)} \oplus \E_{(y',1)}^{(1,-1)}  \oplus \E_{(y+x,1)}^{(1,1)}.
\end{align*}
Subtracting $B$ from $A$ and multiplying by $v^{-3}$ concludes the proof.
\end{proof}

\end{lemma}


\begin{lemma} \label{idG} Let $\E' := \E_{(y',1)}^{(1,0)} \oplus \E_{(y',1)}^{(1,0)}  \oplus \E_{(y,1)}^{(1,0)}$ be a rank $3$ vector bundle on $X$ with $y \neq y'$.  Then \smallskip

\begin{enumerate}[label={\textbf{\upshape(\roman*)}}, ]

   \item $m_{x,1}\big(\E_{(y',1)}^{(1,0)} \oplus \E_{(y,1)}^{(1,0)}  \oplus \E_{(y'+x,1)}^{(1,1)}, \E'  \big) =q$;  \smallskip

   \item $m_{x,1}\big(\E_{(y',1)}^{(1,0)} \oplus \E_{(y',1)}^{(1,0)}  \oplus  \E_{(y+x,1)}^{(1,1)}, \E'  \big) =q^2$;  \smallskip

      \item $m_{x,1}\big(\E_{(y',1)}^{(1,0)} \oplus \E_{(x+y+y',1)}^{(2,1)}, \E'  \big) =1$.  \smallskip
  \end{enumerate}

\begin{proof} Since $y'\neq y$, we can write in the Hall algebra of $X$
\[ \E' = \big( \E_{(y',1)}^{(1,0)} \oplus \E_{(y',1)}^{(1,0)} \big) * \E_{(y,1)}^{(1,0)}\ \ \text{ and } \  \  \E_{(y,1)}^{(1,0)} = N_{1}^{-1 } \sum_{\rho \in \mathrm{P}_1} \widetilde{\rho}(-y) T_{(1,0)}^{\widetilde{\rho}}. \]
By Atiyah's classification of vector bundles on an elliptic curve (Theorem \ref{atiyahtheorem}), the vector bundle $\E_{(y',1)}^{(1,0)} \oplus \E_{(y',1)}^{(1,0)}$ corresponds to $\mathcal{K}_{y'}^{\oplus 2}$. The isomorphism between the Hall algebra of torsion sheaves supported at a closed point and the Macdonald ring of symmetric functions (see \cite[Prop. 4.1]{olivier-elliptic1} and \cite[Chapter III]{macdonald}), coupled with the definition of $T_{\mathbf{v},y'}$ in \cite[Paragraph 4.2]{olivier-elliptic1}, yields
\[ \E_{(y',1)}^{(1,0)} \oplus \E_{(y',1)}^{(1,0)} = \tfrac{v^2}{2} \Big( T_{(1,0),y'} T_{(1,0,y')} - \tfrac{2}{[2]} T_{(2,0),y'} \Big).\]
Hence
\begin{multline*}
 \E'=  \tfrac{ N_{1}^{-3} v^2}{2} \sum_{\substack{ \rho,\rho_i \in \mathrm{P}_1 \\ i=1,2,}} \text{\footnotesize $\widetilde{\rho}_1(-y')\widetilde{\rho}_2(-y') \widetilde{\rho}(-y) $}
T_{(1,0)}^{\widetilde{\rho}_1}  T_{(1,0)}^{\widetilde{\rho}_2} T_{(1,0)}^{\widetilde{\rho}} \\
 - \tfrac{N_{2}^{-1}  N_{1}^{-1} v^2}{[2]}  \sum_{\substack{ \rho \in \mathrm{P}_1 \\ \chi \in \mathrm{P}_2}} \text{\footnotesize $\widetilde{\chi}(-y')\widetilde{\rho}(-y) $}  T_{(2,0)}^{\widetilde{\chi}} T_{(1,0)}^{\widetilde{\rho}} . 
\end{multline*}
Thus
\begin{align*}
 \pi^{\mathrm{vec}} (\mathcal{K}_x * \E')  &  = \tfrac{ N_{1}^{-4} v^2}{2} \sum_{\substack{\sigma, \rho,\rho_i \in \mathrm{P}_1 \\ i=1,2,}} \text{\small $\widetilde{\sigma}(-x)  \widetilde{\rho}_1(-y')\widetilde{\rho}_2(-y') \widetilde{\rho}(-y) $}
\Big[ T_{(0,1)}^{\widetilde{\sigma}}, T_{(1,0)}^{\widetilde{\rho}_1}  T_{(1,0)}^{\widetilde{\rho}_2} T_{(1,0)}^{\widetilde{\rho}} \Big] \\
  &\quad - \tfrac{N_{2}^{-1}  N_{1}^{-2} v^2}{[2]}  \sum_{\substack{ \sigma, \rho \in \mathrm{P}_1 \\ \chi \in \mathrm{P}_2}} \text{\small $\widetilde{\sigma}(-x) \widetilde{\chi}(-y')\widetilde{\rho}(-y) $}  \Big[ T_{(0,1)}^{\widetilde{\sigma}}, T_{(2,0)}^{\widetilde{\chi}} T_{(1,0)}^{\widetilde{\rho}} \Big] \\[1ex]
  & = v^3 T_{(1,0),y'} T_{(1,0),y} T_{(1,1),x+y'} + \tfrac{v^3}{2} T_{(1,0),y'} T_{(1,0),y'} T_{(1,1),x+y}    \\
 &\quad  + \tfrac{v^5}{2} T_{(3,1),x+2y'+y}+ v^4 T_{(1,0),y'}  T_{(2,1),x+y+y'}   + \tfrac{v^4}{2} T_{(1,0),y}  T_{(2,1),x+2y'}   \\
&\quad - \tfrac{v^3}{[2]}  T_{(2,0),y'} T_{(1,1),x+y}- \tfrac{v^4}{2}  T_{(1,0),y} T_{(2,1),x+2y'}   -  \tfrac{v^5}{2} T_{(3,1),x+2y'+y}  \\[1ex]
  &= v\; \E_{(y',1)}^{(1,0)} \oplus \E_{(y,1)}^{(1,0)}  \oplus \E_{(y'+x,1)}^{(1,1)} + \tfrac{(v+v^{-1})}{2} \E_{(y',1)}^{(1,0)} \oplus \E_{(y',1)}^{(1,0)}  \oplus  \E_{(y+x,1)}^{(1,1)}
 \\
 &\quad + v^3  \E_{(y',1)}^{(1,0)} \oplus \E_{(x+y+y',1)}^{(2,1)}  + \tfrac{v^2}{2}  \E_{(y',2)}^{(2,0)} \oplus \E_{(x+y,1)}^{(1,1)} -  \tfrac{v^2}{2}  \E_{(y',2)}^{(2,0)} \oplus \E_{(x+y,1)}^{(1,1)} \\
 &\quad  - \tfrac{(v-v^{-1})}{2} \E_{(y',1)}^{(1,0)} \oplus \E_{(y',1)}^{(1,0)}  \oplus  \E_{(y+x,1)}^{(1,1)}   \\[1ex]
  &= v \; \E_{(y',1)}^{(1,0)} \oplus \E_{(y,1)}^{(1,0)}  \oplus \E_{(y'+x,1)}^{(1,1)}  + v^{-1} \E_{(y',1)}^{(1,0)} \oplus \E_{(y',1)}^{(1,0)}  \oplus  \E_{(y+x,1)}^{(1,1)} \\
  &\quad + v^3 \E_{(y',1)}^{(1,0)} \oplus \E_{(x+y+y',1)}^{(2,1)}.
\end{align*}
 We conclude the proof by multiplying the above identity by $v^{-3}.$
\end{proof}
\end{lemma}


\begin{lemma} \label{idD} Let $\E' := \E_{(y,1)}^{(1,0)} \oplus \E_{(y,1)}^{(1,0)}  \oplus \E_{(y,1)}^{(1,0)}  $ be a rank $3$ vector bundle on $X$. Then
\[ m_{x,1}\big(\E_{(y,1)}^{(1,0)} \oplus \E_{(y,1)}^{(1,0)}  \oplus \E_{(y+x,1)}^{(1,1)} , \E'  \big) =1.\]
\end{lemma}

\begin{proof} We first observe from Lemma \ref{idI} $\mathrm{(ii)}$ and Theorem \ref{thm-dualities}  that  there exists only one $\E \in \Bunvar_3 X$ such that
$ m_{x,1}\big(\E, \E'  \big) \neq 0,$ and $\E$ must be equal to $\E_{(y,1)}^{(1,0)} \oplus \E_{(y,1)}^{(1,0)}  \oplus \E_{(y+x,1)}^{(1,1)}.$ Hence we are left with computing its multiplicity.

Atiyah's theorem (Theorem \ref{atiyahtheorem}) and the structure of the Macdonald ring of symmetric functions yields
\[ \E' = v^6 \Big( \tfrac{1}{6} \big( T_{(1,0),y} \big)^3 - \tfrac{1}{[2]} T_{(1,0),y} T_{(2,0),y} + \tfrac{1}{[3]} T_{(3,0),y} \Big).\]
Hence
\begin{align*}
 \pi^{\mathrm{vec}} (\mathcal{K}_x * \E')  &=  \tfrac{v^6 N_{1}^{-4}}{6} \sum_{\substack{\sigma, \rho_i \in \mathrm{P}_1 \\ i=1,2,3}} \text{\small $\widetilde{\sigma}(-x)  \widetilde{\rho}_1(-y)\widetilde{\rho}_2(-y) \widetilde{\rho}_3(-y) $}
\Big[ T_{(0,1)}^{\widetilde{\sigma}}, T_{(1,0)}^{\widetilde{\rho}_1}  T_{(1,0)}^{\widetilde{\rho}_2} T_{(1,0)}^{\widetilde{\rho}_3} \Big] \\
& \quad - \tfrac{v^6 N_{1}^{-2} N_{2}^{-1}}{[2]} \sum_{\substack{\sigma, \rho \in \mathrm{P}_1 \\ \chi \in \mathrm{P}_2 }} \text{\small $\widetilde{\sigma}(-x)  \widetilde{\rho} (-y)\widetilde{\chi}(-y)  $}
\Big[ T_{(0,1)}^{\widetilde{\sigma}}, T_{(1,0)}^{\widetilde{\rho}}  T_{(2,0)}^{\widetilde{\chi}} \Big] \\
& \quad + \tfrac{v^6 N_{1}^{-1} N_{3}^{-1}}{[2]} \sum_{\substack{\sigma \in \mathrm{P}_1 \\ \lambda \in \mathrm{P}_3 }} \text{\small $\widetilde{\sigma}(-x)  \widetilde{\lambda}(-y)  $}
\Big[ T_{(0,1)}^{\widetilde{\sigma}},  T_{(3,0)}^{\widetilde{\lambda}} \Big].
\end{align*}
We denote by $A,B$ and $C$ the first, second and third term in the above sum, respectively. Proceeding as before and following the algorithm in \cite{roberto-elliptic}, we obtain
\begin{multline*}
 A = \tfrac{(v^5+v^3)}{2} \E_{(y,1)}^{(1,0)} \oplus \E_{(y,1)}^{(1,0)}  \oplus \E_{(y+x,1)}^{(1,1)} \\
 + \tfrac{v^5}{2} \E_{(y,2)}^{(2,0)} \oplus \E_{(y+x,1)}^{(1,1)}
+\tfrac{v^7}{2} \E_{(y,1)}^{(1,0)} \oplus \E_{(2y+x,1)}^{(2,1)}  + \tfrac{v^9}{6} \E_{(3y+x,1)}^{(3,1)},  
\end{multline*}
\begin{multline*}
 B= \tfrac{(v^3-v^5)}{2} \E_{(y,1)}^{(1,0)} \oplus \E_{(y,1)}^{(1,0)}  \oplus \E_{(y+x,1)}^{(1,1)} \\
 - \tfrac{v^5}{2}  \E_{(y,2)}^{(2,0)} \oplus \E_{(y+x,1)}^{(1,1)}
- \tfrac{v^7}{2} \E_{(y,1)}^{(1,0)} \oplus \E_{(2y+x,1)}^{(2,1)}
-  \tfrac{v^9}{2} \E_{(3y+x,1)}^{(3,1)}  
\end{multline*}
and
$C= \tfrac{v^9}{3} \E_{(3y+x,1)}^{(3,1)}$.
Summing $A$, $B$ and $C$ and multiplying the outcome by $v^{-3}$ yields the desired assertion.
\end{proof}


\begin{lemma} \label{idPF} Let $\E' := \E_{(x_1,1)}^{(1,0)} \oplus \E_{(x_2,1)}^{(1,0)} \oplus \E_{(y',1)}^{(1,1)}  $ be a rank $3$ vector bundle on $X$ with $x_1 \neq x_2$. Then \smallskip

\begin{enumerate}[label={\textbf{\upshape(\roman*)}}, ]

   \item $m_{x,1}\big( \E_{(x_1,1)}^{(1,0)} \oplus \E_{(x_2,1)}^{(1,0)} \oplus \E_{(x+y',1)}^{(1,2)}  , \E'  \big) =q^2$;  \smallskip

   \item $m_{x,1}\big(\E_{(x+x_1+x_2,1)}^{(2,1)} \oplus \E_{(y',1)}^{(1,1)}   , \E'  \big) =1$;\smallskip

    \item $m_{x,1}\big(\E_{(x_2,1)}^{(1,0)} \oplus \E_{(x_1+x,1)}^{(1,1)} \oplus \E_{(y',1)}^{(1,1)}   , \E'  \big) =q$, if $x_1 + x \neq y'$;\smallskip

    \item $m_{x,1}\big(\E_{(x_1,1)}^{(1,0)} \oplus \E_{(x_2+x,1)}^{(1,1)} \oplus \E_{(y',1)}^{(1,1)}   , \E'  \big) =q$, if $x_2 + x \neq y'$; \smallskip

   \item $m_{x,1}\big(\E_{(x_2,1)}^{(1,0)} \oplus \E_{(y',1)}^{(1,1)} \oplus \E_{(y',1)}^{(1,1)}   , \E'  \big) = q^2 + q$, if $x_1 + x = y'$; \smallskip

   \item $m_{x,1}\big(\E_{(x_2,1)}^{(1,0)} \oplus \E_{(y',1)}^{(2,2)} , \E'  \big) = q$, if $x_1 + x = y'$; \smallskip

    \item $m_{x,1}\big(\E_{(x_1,1)}^{(1,0)} \oplus \E_{(y',1)}^{(1,1)} \oplus \E_{(y',1)}^{(1,1)}   , \E'  \big) = q^2 + q$, if $x_2 + x = y'$; \smallskip

     \item $m_{x,1}\big(\E_{(x_1,1)}^{(1,0)} \oplus \E_{(y',1)}^{(2,2)} , \E'  \big) = q$, if $x_2 + x = y'$.
   \end{enumerate}

\begin{proof} Since $x_1 \neq x_2$ and $\mu\big(\E_{(x_2,1)}^{(1,0)} \big) < \mu\big(\E_{(y',1)}^{(1,1)} \big) $, we have
\[\E' = v^2 N_{1}^{-3}   \sum_{\substack{\rho, \rho_i \in \mathrm{P}_1 \\ i=1,2,}} \text{\small $\widetilde{\rho}_1(-x_1)\widetilde{\rho}_2(-x_2) \widetilde{\rho}(-y') $}
T_{(1,0)}^{\widetilde{\rho}_1}  T_{(1,0)}^{\widetilde{\rho}_2} T_{(1,1)}^{\widetilde{\rho}}.\]
Hence
\[ \pi^{\mathrm{vec}} (\mathcal{K}_x * \E')  =  v^2 N_{1}^{-4}   \sum_{\substack{\sigma,\rho, \rho_i \in \mathrm{P}_1 \\ i=1,2,}} \text{\small $\widetilde{\sigma}(-x) \widetilde{\rho}_1(-x_1)\widetilde{\rho}_2(-x_2) \widetilde{\rho}(-y') $}
\Big[ T_{(0,1)}^{\widetilde{\sigma}}, T_{(1,0)}^{\widetilde{\rho}_1}  T_{(1,0)}^{\widetilde{\rho}_2} T_{(1,1)}^{\widetilde{\rho}} \Big], \]
where the commutator is taken in the Hall algebra of $X.$ By \cite[Thm. 5.2]{dragos}, we can compute above commutator in $\mathsf{H}_X$. Thus proceeding with the algorithm in \cite{roberto-elliptic} yields
\begin{align*}
\pi^{\mathrm{vec}} (\mathcal{K}_x * \E') & = v^3 T_{(1,0),x_2} T_{(1,1),x_1+x} T_{(1,1),y'} +  v^3 T_{(1,0),x_1} T_{(1,1),x_2+x} T_{(1,1),y'}  \\
 & \quad+ v^{-1} \E_{(x_1,1)}^{(1,0)} \oplus \E_{(x_2,1)}^{(1,0)} \oplus \E_{(x+y',1)}^{(1,2)}  + v^3 \E_{(x+x_1+x_2,1)}^{(2,1)} \oplus \E_{(y',1)}^{(1,1)}. \end{align*}
 To conclude the proof, set $A := T_{(1,0),x_2} T_{(1,1),x_1+x} T_{(1,1),y'} $, $B:= T_{(1,0),x_1} T_{(1,1),x_2+x} T_{(1,1),y'}$ and observe that
 \[ A = \begin{cases} v^{-2} \E_{(x_2,1)}^{(1,0)} \oplus \E_{(x_1+x,1)}^{(1,1)} \oplus \E_{(y',1)}^{(1,1)}  \hfill &\text{if }  x_1 + x \neq y', \medskip\\
 (v^{-2}+v^{-4})\E_{(x_2,1)}^{(1,0)} \oplus \E_{(y',1)}^{(1,1)} \oplus \E_{(y',1)}^{(1,1)} + v^{-2} \E_{(x_2,1)}^{(1,0)} \oplus \E_{(y',1)}^{(2,2)} \hfill &\text{if }  x_1 + x = y'.
 \end{cases} \]
Changing $x_1$ by $x_2$ above, we obtain an analogous expression for $B$. Putting all together and multiplying by $v^{-3}$ completes the proof.
\end{proof}
\end{lemma}


\begin{lemma} \label{idH} Let $\E' := \E_{(y',1)}^{(1,0)} \oplus \E_{(y',1)}^{(1,0)} \oplus \E_{(y,1)}^{(1,1)}  $ be a rank $3$ vector bundle on $X$. Then \smallskip

\begin{enumerate}[label={\textbf{\upshape(\roman*)}}, ]

   \item $m_{x,1}\big( \E_{(y',1)}^{(1,0)} \oplus \E_{(y',1)}^{(1,0)} \oplus \E_{(y+x,1)}^{(1,2)}   , \E'  \big) =q^2$;  \smallskip

    \item $m_{x,1}\big(\E_{(y',1)}^{(1,0)} \oplus \E_{(y'+x,1)}^{(1,1)} \oplus \E_{(y,1)}^{(1,1)}   , \E'  \big) = 1$, if $y' + x \neq y$; \smallskip

    \item $m_{x,1}\big(\E_{(y',1)}^{(1,0)} \oplus \E_{(y,1)}^{(1,1)} \oplus \E_{(y,1)}^{(1,1)}  , \E'  \big) =q+1$, if $y'+ x = y$; \smallskip

   \item $m_{x,1}\big(\E_{(y',1)}^{(1,0)} \oplus \E_{(y,2)}^{(2,2)}   , \E'  \big) = 1$, if $y'+ x = y$.
   \end{enumerate}

\begin{proof}  As the first step of the algorithm in \cite{roberto-elliptic}, we must write $\E'$ in terms of the elements $T_{\mathbf{v}}^{\widetilde{\rho}}$ in the Definition \ref{defofT}. Atiyah's theorem (Theorem \ref{atiyahtheorem}) implies that $\E_{(y,1)}^{(1,1)} $ corresponds to $\mathcal{K}_y$ and $\E_{(y',1)}^{(1,0)} \oplus \E_{(y',1)}^{(1,0)} $  corresponds to $\mathcal{K}_{y'}^{\oplus 2}.$ Hence by the definition of $T_{\mathbf{v},y'}$ (and $T_{\mathbf{v},y}$) in terms of elements in the Macdonald ring of symmetric functions (cf. \cite[Paragraph 4.2]{olivier-elliptic1}) and \cite[Prop. 3.5]{roberto-elliptic}, we can write
 \[ \E_{(y',1)}^{(1,0)} \oplus \E_{(y',1)}^{(1,0)} = \tfrac{v^2}{2} \left( N_{1}^{-2}  \sum_{\rho,\rho_1\in \mathrm{P}_1} \text{\small $\widetilde{\rho}(-y')\widetilde{\rho}_1(-y') $}
T_{(1,0)}^{\widetilde{\rho}}  T_{(1,0)}^{\widetilde{\rho}_1} - \tfrac{2 N_{2}^{-1}}{[2]} \sum_{\widetilde{\rho}_2\in \mathrm{P}_2} \text{\small $ \widetilde{\rho}_2(-y') $}  T_{(2,0)}^{\widetilde{\rho}_2} \right) \]
and
\[ \E_{(y,1)}^{(1,1)} =  N_{1}^{-1} \sum_{\chi\in \mathrm{P}_1} \text{\small $\chi(-y) $} T_{(1,1)}^{\chi} \]
in the Hall algebra $\mathsf{H}_X$. Thus $\E'  = v^2 \left( \E_{(y',1)}^{(1,0)} \oplus \E_{(y',1)}^{(1,0)} \right)  * \E_{(y,1)}^{(1,1)} $ is equal to
\begin{multline*}
v^2 \Big( \tfrac{v^2N_{1}^{-3} }{2}  \sum_{\rho,\rho_1, \chi\in \mathrm{P}_1} \text{\small $\widetilde{\rho}(-y')\widetilde{\rho}_1(-y') \widetilde{\chi}(-y)$}
T_{(1,0)}^{\widetilde{\rho}}  T_{(1,0)}^{\widetilde{\rho}_1} T_{(1,1)}^{\widetilde{\chi}} \\
- \tfrac{v^2 N_{1}^{-1} N_{2}^{-1}}{[2]} \sum_{\substack{\chi \in \mathrm{P}_1 \\ \rho_2\in \mathrm{P}_2}} \text{\small $ \widetilde{\rho}_2(-y') \widetilde{\chi}(-y)  $}  T_{(2,0)}^{\widetilde{\rho}_2} T_{(1,1)}^{\widetilde{\chi}} \Big)
\end{multline*}
and then
\begin{align*}
\pi^{\mathrm{vec}} (\mathcal{K}_x * \E') & =  \tfrac{v^4 N_{1}^{-4} }{2}  \sum_{\sigma, \rho,\rho_1, \chi\in \mathrm{P}_1} \text{\small $\widetilde{\sigma}(-x) \widetilde{\rho}(-y')\widetilde{\rho}_1(-y') \widetilde{\chi}(-y)$}  \Big[ T_{(0,1)}^{\widetilde{\sigma}}, T_{(1,0)}^{\widetilde{\rho}}  T_{(1,0)}^{\widetilde{\rho}_1} T_{(1,1)}^{\widetilde{\chi}}  \Big] \\
& \quad - \tfrac{v^4 N_{1}^{-2} N_{2}^{-1}}{[2]} \sum_{\substack{\sigma,\chi \in \mathrm{P}_1 \\ \rho_2\in \mathrm{P}_2}} \text{\small $\widetilde{\sigma}(-x) \widetilde{\rho}_2(-y') \widetilde{\chi}(-y)  $}
\Big[ T_{(0,1)}^{\widetilde{\sigma}}, T_{(2,0)}^{\widetilde{\rho}_2} T_{(1,1)}^{\widetilde{\chi}} \Big] .
\end{align*}
We denote by $A$ and $B$ the first and second term of the above sum, respectively. Next we apply \cite[Thm. 5.2]{dragos} to calculate the previous commutators in $\mathsf{H}_X$. Hence
\begin{align*}\Big[ T_{(0,1)}^{\widetilde{\sigma}}, T_{(1,0)}^{\widetilde{\rho}}  T_{(1,0)}^{\widetilde{\rho}_1} T_{(1,1)}^{\widetilde{\chi}}  \Big] & =
c_1 \; T_{(1,0)}^{\widetilde{\rho}}  T_{(1,1)}^{\widetilde{\sigma}} T_{(1,1)}^{\widetilde{\chi}}  + c_1 \; T_{(1,0)}^{\widetilde{\rho}_1}  T_{(1,1)}^{\widetilde{\sigma}} T_{(1,1)}^{\widetilde{\chi}} \\
& \quad + c_1 \; T_{(1,0)}^{\widetilde{\rho}}  T_{(1,0)}^{\widetilde{\rho}_1} T_{(1,2)}^{\widetilde{\sigma}}  +c_{1}^{2} \;T_{(2,1)}^{\widetilde{\sigma}}  T_{(1,1)}^{\widetilde{\chi}}   \end{align*}
and
\[ \Big[ T_{(0,1)}^{\widetilde{\sigma}}, T_{(2,0)}^{\widetilde{\rho}_2} T_{(1,1)}^{\widetilde{\chi}} \Big] = c_2 T_{(2,1)}^{\widetilde{\sigma}} T_{(1,1)}^{\widetilde{\chi}} + c_1 T_{(2,0)}^{\widetilde{\rho}_2} T_{(1,2)}^{\widetilde{\sigma}}. \]
Replacing the above calculations in the expression of $\pi^{\mathrm{vec}} (\mathcal{K}_x * \E')$ and  base changing back yields
\begin{align*}
A &= v^5 T_{(1,0),y'}  T_{(1,1),y'+x} T_{(1,1),y} + \tfrac{v^5}{2} T_{(1,0),y'}  T_{(1,0),y'}  T_{(1,2),y+x} +  \tfrac{v^6}{2} T_{(2,1),x+2y'} T_{(1,1),y}  \\
& =  \tfrac{v^5}{2} \E_{(x+2y',1)}^{(2,1)} \oplus \E_{(y,1)}^{(1,1)}  + \tfrac{v}{2} \left( (1+v^{-2}) \E_{(y',1)}^{(1,0)} \oplus \E_{(y',1)}^{(1,0)} \oplus \E_{(y+x,1)}^{(1,2)} +
\E_{(y',2)}^{(2,0)} \oplus \E_{(y+x,1)}^{(1,2)} \right) \\
&\quad + v^5 T_{(1,0),y'}  T_{(1,1),y'+x} T_{(1,1),y}.
\end{align*}
Let $A'$ denote $ T_{(1,0),y'}  T_{(1,1),y'+x} T_{(1,1),y} $, then
\[A' = \begin{cases}
v^{-2} \E_{(y',1)}^{(1,0)} \oplus \E_{(y'+x,1)}^{(1,1)} \oplus \E_{(y,1)}^{(1,1)} \hfill & \text{if } y' + x \neq y, \medskip \\
v^{-2} \left( (1+v^{-2}) \E_{(y',1)}^{(1,0)} \oplus \E_{(y,1)}^{(1,1)} \oplus \E_{(y,1)}^{(1,1)} +  \E_{(y',1)}^{(1,0)} \oplus \E_{(y,2)}^{(2,2)} \right)  \hfill & \text{if } y' + x = y.
\end{cases} \]
Analogous for $B$, we obtain
\begin{align*}
B & = \tfrac{v^6}{2} T_{(2,1),x+2y'} T_{(1,1),y}  + \tfrac{v^5}{[2]}   T_{(2,0),y'}  T_{(1,2),y+x} \\
& =  \tfrac{v^5}{2} \E_{(x+2y',1)}^{(2,1)} \oplus \E_{(y,1)}^{(1,1)}  + \tfrac{v}{2} \left( (1-v^{-2})   \E_{(y',1)}^{(1,0)} \oplus \E_{(y',1)}^{(1,0)} \oplus \E_{(y+x,1)}^{(1,2)}
+ \E_{(y',2)}^{(2,0)} \oplus \E_{(y+x,1)}^{(1,2)} \right).
\end{align*}
Therefore
\[ v^{-3} \; \pi^{\mathrm{vec}} (\mathcal{K}_x * \E')  = v^{-4}  \E_{(y',1)}^{(1,0)} \oplus \E_{(y',1)}^{(1,0)} \oplus \E_{(y+x,1)}^{(1,2)} + v^2 A', \]
which proves our statement since $v^{-2}=q.$
\end{proof}
\end{lemma}


\begin{lemma} \label{idR20} Let $\E' := \E_{(y',1)}^{(1,-2)} \oplus \E_{(y,1)}^{(1,0)} \oplus \E_{(y,1)}^{(1,0)}  $ be a rank $3$ vector bundle on $X$. Then \smallskip

\begin{enumerate}[label={\textbf{\upshape(\roman*)}}, ]

   \item $m_{x,1}\big( \E_{(y'+x,1)}^{(1,-1)} \oplus \E_{(y,1)}^{(1,0)} \oplus \E_{(y,1)}^{(1,0)}   , \E'  \big) =1$;  \smallskip

    \item $m_{x,1}\big(\E_{(y',1)}^{(1,-2)} \oplus \E_{(y,1)}^{(1,0)} \oplus \E_{(y+x,1)}^{(1,1)}   , \E'  \big) = q$.

   \end{enumerate}

\begin{proof} Similar to the proof of Lemma \ref{idH}, we can write $\E'$ in $\mathsf{H}_X$ as follows:
\begin{align*}
\E' & = v^4 \;\E_{(y',1)}^{(1,-2)} * \Big( \E_{(y,1)}^{(1,0)} \oplus \E_{(y,1)}^{(1,0)} \Big) \\
& = \tfrac{v^6}{2} \Big( N_{1}^{-3}   \sum_{\rho,\rho_1, \chi\in \mathrm{P}_1} \text{\small $\widetilde{\rho}_1(-y')\widetilde{\rho}(-y) \widetilde{\chi}(-y)$}
T_{(1,-2)}^{\widetilde{\rho}_1}  T_{(1,0)}^{\widetilde{\rho}} T_{(1,0)}^{\widetilde{\chi}}   \\
&\quad   - \tfrac{2  N_{1}^{-1} N_{2}^{-1}}{[2]} \sum_{\substack{\rho_1 \in \mathrm{P}_1 \\ \rho_2\in \mathrm{P}_2}} \text{\small $  \widetilde{\rho}_1(-y) \widetilde{\rho}_2(-y)  $}  T_{(1,-2)}^{\widetilde{\rho}_1} T_{(2,0)}^{\widetilde{\rho}_2} \Big).
\end{align*}
Hence
\begin{align*}
\pi^{\mathrm{vec}} (\mathcal{K}_x * \E') & =  \tfrac{v^6 N_{1}^{-4} }{2}  \sum_{\sigma, \rho,\rho_1, \chi\in \mathrm{P}_1} \text{\small $\widetilde{\sigma}(-x) \widetilde{\rho}_1(-y')\widetilde{\rho}(-y) \widetilde{\chi}(-y)$}  \Big[ T_{(0,1)}^{\widetilde{\sigma}}, T_{(1,-2)}^{\widetilde{\rho}_1}  T_{(1,0)}^{\widetilde{\rho}} T_{(1,0)}^{\widetilde{\chi}}   \Big] \\
& \quad - \tfrac{v^6 N_{1}^{-2} N_{2}^{-1}}{[2]} \sum_{\substack{\sigma,\rho_1 \in \mathrm{P}_1 \\ \rho_2\in \mathrm{P}_2}} \text{\small $\widetilde{\sigma}(-x) \widetilde{\rho}_1(-y) \widetilde{\rho}_2(-y) $}
\Big[ T_{(0,1)}^{\widetilde{\sigma}}, T_{(1,-2)}^{\widetilde{\rho}_1} T_{(2,0)}^{\widetilde{\rho}_2}\Big].
\end{align*}
Let us denote by $A$ (resp. $B$) the first (resp. second) term in the above subtraction. Proceeding with the algorithm in \cite[Section 4]{roberto-elliptic} yields
\begin{align*}
A &= \tfrac{v^5}{2} \left( (1+v^{-2})  \E_{(y'+x,1)}^{(1,-1)} \oplus \E_{(y,1)}^{(1,0)} \oplus \E_{(y,1)}^{(1,0)}  +  \E_{(y'+x,1)}^{(1,-1)} \oplus \E_{(y,2)}^{(2,0)} \right) \\
& \quad + v\; \E_{(y',1)}^{(1,-2)} \oplus \E_{(y,1)}^{(1,0)} \oplus \E_{(y+x,1)}^{(1,1)}  + \tfrac{v^3}{2} \E_{(y',1)}^{(1,-2)} \oplus \E_{(2y+x,1)}^{(2,1)}
\end{align*}
and
\begin{align*} 
B &= \tfrac{v^5}{2} \E_{(y'+x,1)}^{(1,-1)} \oplus \E_{(y,2)}^{(2,0)} + \tfrac{(v^5 - v^3)}{2} \E_{(y'+x,1)}^{(1,-1)} \oplus \E_{(y,1)}^{(1,0)} \oplus \E_{(y,1)}^{(1,0)}  + \tfrac{v^3}{2} \E_{(y',1)}^{(1,-2)} \oplus \E_{(2y+x,1)}^{(2,1)}.
\end{align*}
The proof follows from multiplying $A - B$ by $v^{-3}.$
\end{proof}
\end{lemma}


\begin{lemma} \label{idC20} Let $\E' := \E_{(z,1)}^{(1,-1)} \oplus \E_{(z,1)}^{(1,-1)} \oplus \E_{(y,1)}^{(1,0)}  $ be a rank $3$ vector bundle on $X$ with $y \neq z+x$. Then \smallskip

\begin{enumerate}[label={\textbf{\upshape(\roman*)}}, ]

   \item $m_{x,1}\big( \E_{(z,1)}^{(1,-1)} \oplus \E_{(z+x,1)}^{(1,0)} \oplus \E_{(y,1)}^{(1,0)}    , \E'  \big) =1$;  \smallskip

    \item $m_{x,1}\big(\E_{(z,1)}^{(1,-1)} \oplus \E_{(z,1)}^{(1,-1)} \oplus \E_{(y+x,1)}^{(1,1)}    , \E'  \big) = q^2$.
   \end{enumerate}

\begin{proof} As in the proof of Lemma \ref{Prop26/01},
\begin{align*}
\pi^{\mathrm{vec}} (\mathcal{K}_x * \E')  & = \tfrac{N_{1}^{-4} v^4}{2} \sum_{\substack{ \sigma,\rho,\rho_i \in \mathrm{P}_1 \\ i=1,2}} \text{\small $ \widetilde{\sigma}(-x)\widetilde{\rho}_1(-z)\widetilde{\rho}_2(-z) \widetilde{\rho}(-y) $}
\Big[ T_{(0,1)}^{\widetilde{\sigma}} , T_{(1,-1)}^{\widetilde{\rho}_1}  T_{(1,-1)}^{\widetilde{\rho}_2} T_{(1,0)}^{\widetilde{\rho}} \Big] \\
&\quad  - \tfrac{N_{2}^{-1}  N_{1}^{-2} v^4}{[2]}  \sum_{\substack{ \sigma, \rho \in \mathrm{P}_1 \\ \chi \in \mathrm{P}_2}} \text{\small $\widetilde{\sigma}(-x)\widetilde{\chi}(-z)\widetilde{\rho}(-y) $}  \Big[ T_{(0,1)}^{\widetilde{\sigma}} ,T_{(2,-2)}^{\widetilde{\chi}} T_{(1,0)}^{\widetilde{\rho}} \Big] .
\end{align*}
Let $A$ (resp. $B$) be the first (resp. second) term in the right side of above identity.
Since $T_{(1,d),z} T_{(1,d),z}  = (1+v^{-2}) \E_{(z,1)}^{(1,d)} \oplus \E_{(z,1)}^{(1,d)} + \E_{(z,2)}^{(2,2d)}$ in $\mathsf{H}_X$ for any integer $d$,
\begin{multline*}
A  = v^3 \E_{(z,1)}^{(1,-1)} \oplus \E_{(z+x,1)}^{(1,0)} \oplus \E_{(y,1)}^{(1,0)}  + \tfrac{v^5}{2} \E_{(2z+x,1)}^{(2,-1)} \oplus \E_{(y,1)}^{(1,0)} \\
+  \tfrac{(v^{-1}+ v)}{2}  \E_{(z,1)}^{(1,-1)} \oplus \E_{(z,1)}^{(1,-1)} \oplus \E_{(y+x,1)}^{(1,1)}   + \tfrac{v}{2} \E_{(z,2)}^{(2,-2)} \oplus \E_{(y+x,1)}^{(1,1)}
\end{multline*}
by the algorithm in \cite[Section 4]{roberto-elliptic}.  Moreover,
\begin{multline*}
B =  \tfrac{v^5}{2} \E_{(2z+x,1)}^{(2,-1)} \oplus \E_{(y,1)}^{(1,0)} \\
  +  \tfrac{v}{2} \Big( (1-v^{-2})  \E_{(z,1)}^{(1,-1)} \oplus \E_{(z,1)}^{(1,-1)} \oplus \E_{(y+x,1)}^{(1,1)}  + \tfrac{v}{2} \E_{(z,2)}^{(2,-2)} \oplus \E_{(y+x,1)}^{(1,1)} \Big)
\end{multline*}
since 
$$T_{(2,-2),z} := \tfrac{[2]}{2} \Big( (1-v^{-2})  \E_{(z,1)}^{(1,-1)} \oplus \E_{(z,1)}^{(1,-1)} + \E_{(z,2)}^{(2,-2)}\Big) \  \   \text{ for }  |z|=1.$$ 
Therefore we have
\[\pi^{\mathrm{vec}} (\mathcal{K}_x * \E')  = v^3 \E_{(z,1)}^{(1,-1)} \oplus \E_{(z+x,1)}^{(1,0)} \oplus \E_{(y,1)}^{(1,0)}  +
v^{-1} \E_{(z,1)}^{(1,-1)} \oplus \E_{(z,1)}^{(1,-1)} \oplus \E_{(y+x,1)}^{(1,1)}  \]
and the lemma follows from multiplying the above identity by $v^{-3}$. \end{proof}
\end{lemma}


\begin{lemma} \label{idE21} Let $\E' := \E_{(z,1)}^{(1,-2)} \oplus \E_{(y,1)}^{(1,0)} \oplus \E_{(y',1)}^{(1,0)}  $ be a rank $3$ vector bundle on $X$ with $y \neq y'$. Then \smallskip

\begin{enumerate}[label={\textbf{\upshape(\roman*)}}, ]

   \item $m_{x,1}\big( \E_{(z+x,1)}^{(1,-1)} \oplus \E_{(y,1)}^{(1,0)} \oplus \E_{(y',1)}^{(1,0)}    , \E'  \big) =1$;  \smallskip

    \item $m_{x,1}\big(\E_{(z,1)}^{(1,-2)} \oplus \E_{(y',1)}^{(1,0)} \oplus \E_{(y+x,1)}^{(1,1)}    , \E'  \big) = q^2$; \smallskip

      \item $m_{x,1}\big(\E_{(z,1)}^{(1,-2)} \oplus \E_{(y,1)}^{(1,0)} \oplus \E_{(y'+x,1)}^{(1,1)}    , \E'  \big) = q^2$; \smallskip

            \item $m_{x,1}\big(\E_{(z,1)}^{(1,-2)} \oplus \E_{(y+y'+ x,1)}^{(2,1)}   , \E'  \big) = q^2$.
   \end{enumerate}

\begin{proof}  Since $y \neq y'$, we can write in $\mathsf{H}_X$
\[ \E' = v^{4} N_{1}^{-3} \sum_{\substack{ \rho_i \in \mathrm{P}_1 \\ i=1,2,3}} \text{\small $ \widetilde{\rho}_1(-z)\widetilde{\rho}_2(-y) \widetilde{\rho}_3(-y') $}
 T_{(1,-2)}^{\widetilde{\rho}_1}  T_{(1,0)}^{\widetilde{\rho}_2} T_{(1,0)}^{\widetilde{\rho}_3}.\]
 Thus
 \begin{align*}
 \pi^{\mathrm{vec}} (\mathcal{K}_x * \E')   & = v^{4} N_{1}^{-4}  \sum_{\substack{ \sigma, \rho_i \in \mathrm{P}_1 \\ i=1,2,3}} \text{\small $ \widetilde{\sigma}(-x) \widetilde{\rho}_1(-z)\widetilde{\rho}_2(-y) \widetilde{\rho}_3(-y') $}
\Big[ T_{(0,1)}^{\widetilde{\sigma}} , T_{(1,-2)}^{\widetilde{\rho}_1}  T_{(1,0)}^{\widetilde{\rho}_2} T_{(1,0)}^{\widetilde{\rho}_3} \Big] \\
& =v^5 \Big( T_{(1,-1), x+z} T_{(1,0),y} T_{(1,0),y'} + T_{(1,-2), z} T_{(1,1),y+x} T_{(1,0),y'}  \\
&\quad  + T_{(1,-2), z} T_{(1,0),y} T_{(1,1),y'+x} + v T_{(1,-2), z} T_{(2,1),y+y' +x} \Big)\\
& = v^3 \E_{(z+x,1)}^{(1,-1)} \oplus \E_{(y,1)}^{(1,0)} \oplus \E_{(y',1)}^{(1,0)} + v^{-1} \E_{(z,1)}^{(1,-2)} \oplus \E_{(y',1)}^{(1,0)} \oplus \E_{(y+x,1)}^{(1,1)}  \\
& \quad + v^{-1} \E_{(z,1)}^{(1,-2)} \oplus \E_{(y,1)}^{(1,0)} \oplus \E_{(y'+x,1)}^{(1,1)} + v^{-1}  \E_{(z,1)}^{(1,-2)} \oplus \E_{(y+y'+ x,1)}^{(2,1)} .
\end{align*}
We complete the proof multiplying above identity by $v^{-3}.$
\end{proof}
\end{lemma}

\section{Degree \texorpdfstring{$2$}{2} graphs}
\label{sec-degreetwo}

In this last section, we describe the `even component' of $\mathscr{G}_{y,1}^{(2)}$ and the neighborhood of $\cO_X^{\oplus3}$ in $\mathscr{G}_{y,1}^{(3)}$ for a degree two place $y$ of $X$. Our results are based on the computation of the commutators of certain elements in twisted spherical Hall algebras, which we introduce in the following.


\subsection{Twisted spherical Hall algebras}

Recall from section \ref{sec-proofmainthm} that $\mathrm{P}_n$ stands for {\small $\widehat{\mathrm{Pic}^{0} X_n}$} and for $\rho \in \mathrm{P}_n$, $\tilde{\rho}$ denotes its class on $\widetilde{\mathrm{P}}_n := \text{ \small $\widehat{\mathrm{Pic}^{0} X_n}/ \mathrm{Fr}_{X,n}$}.$ 

\begin{df} We call a character $\rho \in \rm{P}_n$ primitive of degree $n$ if its orbit under the Frobenius $\Fr_{X,n}$ is of maximal cardinality, which is $n$. Equivalently, by \cite[Lemma 3.8]{dragos}, $\rho$ is primitive if there does not exist a character $\chi \in \mathrm{P}_m$ for $m|n$ such that $\rho = \Norm_{m}^{n}(\chi)$. 
\end{df}

 Note that $\rho \in \rm{P}_n$ is either primitive or that there is a primitive character $\chi \in \mathrm{P}_m$ for $m|n$ such that $\rho = \Norm_{m}^{n}(\chi)$, see also \cite[Cor.\ 3.9]{dragos}. 

\begin{df}\label{deftwistedsphericalalgebra} Let $n \geq 1$ and ${\rho} \in  \rm{P}_n$ be a primitive character. One defines the algebra $\mathbf{U}_{X}^{\widetilde{\rho}}$ as the subalgebra of $\mathsf{H}_X$ that is generated by  
$$\Big\{ T_{n\mathbf{v}}^{\Norm_{n}^{n \gamma(\mathbf{v})}(\widetilde{\rho})} \big|\; \mathbf{v} \in \mathbf{Z} \Big\},$$
where $\mathbf{Z} := \big\{(m,n) \in \Z^2 \;|\; m>0 \text{ or } m=0 \text{ and } n>0 \big\}$ and $\gamma(\mathbf{v}) := \gcd(m,n)$ for $\mathbf{v} = (m,n) \in \mathbf{Z}.$ 
This subalgebra is called the twisted spherical Hall algebra of $X$ and $\widetilde{\rho}.$ 
\end{df}

For non-collinear $\mathbf{v},\mathbf{w} \in \mathbf{Z}$, we define $\epsilon_{\mathbf{v},\mathbf{w}}:= \mathrm{sign}(\det(\mathbf{v},\mathbf{w})) \in \{\pm 1\}$.
We denote by $\Delta_{\mathbf{v},\mathbf{w}}$ the triangle formed by the vectors $\mathbf{o},\mathbf{v},\mathbf{v+w}$ where $\mathbf{o}=(0,0).$

\begin{df} \label{defpathalgebra} 
 Fix $\alpha, \beta \in \C^{*}$ with $\alpha, \beta \not\in \{\pm 1\}$.
 We set $\nu:= (\alpha \beta)^{-\frac{1}{2}}$ and 
\[ c_i(\alpha,\beta) := (\alpha^{i/2} - \alpha^{-i/2})(\beta^{i/2}-\beta^{-i/2})[i]_{\nu}\big/i \ \ \text{ where } \  \ 
[i]_{\nu} := \frac{\nu^i - \nu^{-i}}{\nu - \nu^{-1}}. \]
Let $\textbf{E}_{\alpha, \beta}^{n}$ be the $\C$-algebra generated by $\{t_{\mathbf{v}} \; |\; \mathbf{v} \in \mathbf{Z}\}$ modulo the following relations:
\begin{enumerate}
\item If $\mathbf{o},\mathbf{v},\mathbf{v'}$ are collinear then $[t_{\mathbf{v}},t_{\mathbf{v'}}]=0.$ \smallskip
\item If $\mathbf{v},\mathbf{w}$ are such that $\gamma(\mathbf{v})=1$ and that $\Delta_{\mathbf{v},\mathbf{w}}$ has no interior lattice point, then
\[[t_{\mathbf{w}}, t_\mathbf{v}] = \epsilon_{\mathbf{v},\mathbf{w}} c_{n \gamma(\mathbf{w})}(\alpha, \beta) \frac{\theta_{\mathbf{v}+\mathbf{w}}}{n(\nu^{-1} - \nu)}\]
where the elements $\theta_{\mathbf{z}}$, with  $\mathbf{z} \in \mathbf{Z}$ are defined by the following generating series:
\[\sum_{i \geq 0}\theta_{i\mathbf{z}_0} s^i = \exp\left( n(\nu^{-1}-\nu) \sum_{i \geq 1} t_{i \mathbf{z}_{0}} s^i \right)\]
for any $\mathbf{z}_0 \in \mathbf{Z}$ such that $\gamma(\mathbf{z_0})=1.$
 \end{enumerate}
Note that $\theta_{\mathbf{z}}=n(\nu^{-1}-\nu)t_{\mathbf{z}}$ whenever $\gamma(\mathbf{z})=1.$ 
\end{df} 

Given an elliptic curve $X$ defined over $\Fq$, by a theorem of Hasse (see \cite[Thm. 2.3.1]{silverman}) there exists conjugate algebraic numbers $\alpha, \overline{\alpha}$ satisfying $\alpha\overline{\alpha} =q$, such that 
\[N_m := \#X(\mathbb{F}_{q^m}) = q^m +1 -(\alpha^m + \overline{\alpha}^m)\]
for any $m \geq 1.$ The numbers $\alpha, \overline{\alpha}$ are the eigenvalues of the Frobenius automorphism acting on $H^{1}_{et}(\overline{X}, \overline{\Q_{\ell}})$ for some prime number $\ell$ distinct from the characteristic of $\Fq$ and $\overline{X} := X \times_{\mathrm{Spec}(\Fq)} \mathrm{Spec}(\overline{\Fq})$ (cf. \cite[App C]{biblia} for further details). We observe that in this case, $c_i(\alpha, \overline{\alpha}) = c_i$ as defined in subsection \ref{subsec1-sec-proofmainthm}. 
The connection between the above definition and the Hall algebra $\mathsf{H}_X$ is given by the following theorem. 

\begin{thm}[{\cite[Thm. 5.2]{dragos}}]\label{dragosmaintheorem} With the previous choice of  $\alpha, \overline{\alpha}$,
for any two integers $n,m \geq 1$ and any two different primitive characters ${\rho} \in \mathrm{P}_n$ and ${\sigma} \in \mathrm{P}_m$, we have:
\begin{enumerate}[label=\upshape(\arabic*)]
\item The twisted spherical Hall algebra $\mathbf{U}_{X}^{\widetilde{\rho}}$ is isomorphic to $\mathbf{E}_{\alpha,\overline{\alpha}}^{n}$. This isomorphism is given by
\[t_{\mathbf{v}} \longmapsto T_{n \mathbf{v}}^{\Norm_{n}^{n \gamma(\mathbf{v})}(\widetilde{\rho})}.\]

\item The algebras $\mathbf{U}_{X}^{\widetilde{\rho}}$ and $\mathbf{U}_{X}^{\widetilde{\sigma}}$ commute with each other. 

\item The Hall algebra $\mathsf{H}_X$ decomposes into a commutative restricted tensor product
\[\mathsf{H}_X \cong \bigotimes_{\widetilde{\rho} \in \mathcal{P}}\text{$'$ } \mathbf{U}_{X}^{\widetilde{\rho}} \]
where $\mathcal{P}$ is the set of equivalence classes of primitive characters on $\bigsqcup_{n \geq 1} \widetilde{\mathrm{P}}_n.$ 
\end{enumerate}

\end{thm}

Finally we can state and prove a result that is fundamental to the computation of the edges for a degree $2$ graph. 

\begin{lemma} \label{lemma-commutator} Let $\rho, \sigma \in \mathrm{P}_2$ and $\tilde{\rho}, \tilde{\sigma}$ its classes in $\widetilde{\rm{P}}_2$. Then
\[ \Big[ T_{(0,2)}^{\tilde{\sigma}} , T_{(2,0)}^{\tilde{\rho}}  \Big] = \begin{cases}  
c_2 T_{(2,2)}^{\widetilde{\rho}}  & \text{ if } \sigma \text{ is primitive and } \tilde{\rho} =\tilde{\sigma}; \medskip \\
\text{\small$\tfrac{c_2^2(v^{-1}-v)}{2c_2}$ } T_{(1,1)}^{\widetilde{\chi}} T_{(1,1)}^{\widetilde{\chi}}  + \text{ \small $\tfrac{c_2^2-2c_1c_2}{c_1}$ } T_{(2,2)}^{\widetilde{\sigma}}  & \text{ if } \sigma = \rho = \Norm_{1}^2(\chi); \medskip \\
0  & \text{ otherwise; } 
\end{cases}
\]
where $\chi \in \rm{P}_1$ and the commutator is taken in the elliptic Hall algebra $\mathsf{H}_X.$

\begin{proof} For the case that $\sigma$ is primitive, we refer \cite[Section 8.4 - Step $2_n$]{dragos}. If $\sigma$ is not primitive, then $\sigma = \Norm_{1}^2(\chi)$ for some $\chi \in \rm{P}_1$. If $\rho \neq \sigma$, then these characters belong to different twisted spherical Hall subalgebras of $\mathsf{H}_X$ and $\big[ T_{(0,2)}^{\tilde{\sigma}} , T_{(2,0)}^{\tilde{\rho}}  \big]$ vanishes by Theorem \ref{dragosmaintheorem}, part (2). 

If $\sigma = \rho = \Norm_{1}^2(\chi) $, then $\big[ T_{(0,2)}^{\tilde{\sigma}} , T_{(2,0)}^{\tilde{\rho}}  \big]$ can be calculated in  $\mathbf{U}_{X}^{\widetilde{\sigma}}$. By the isomorphism given in Theorem \ref{dragosmaintheorem}, part (1), we are left with computing $\big[t_{(0,2)}, t_{(2,0)} \big]$ in $\mathbf{E}_{\alpha,\overline{\alpha}}^{n}.$ Since 
\[ \big[t_{(0,1)}, t_{(2,-1)} \big] = c_1 \left( t_{(2,0)} + \tfrac{(v^{-1}-v)}{2} t_{(1,0)} t_{(1,0)} \right),\]
we have 
\begin{align*}
\big[t_{(0,2)}, t_{(2,0)} \big] & = \big[t_{(0,2)}, c_{1}^{-1} \big[ t_{(0,1)},t_{(2,-1)} \big]\big] - \tfrac{(v^{-1}-v)}{2} \big[t_{(0,2)}, t_{(1,0)} t_{(1,0)} \big] \\
& =  c_{1}^{-1} \big[t_{(0,1)}, \big[ t_{(0,2)},t_{(2,-1)} \big]\big] - c_2(v^{-1}-v) t_{(1,0)} t_{(1,2)} \\
& \quad -  c_1 c_2 \tfrac{(v^{-1}-v)}{2} \left(  t_{(2,2)} + \tfrac{(v^{-1}-v)}{2} t_{(1,1)} t_{(1,1)} \right). 
\end{align*}
Next we observe that $\big[ t_{(0,2)},t_{(2,-1)} \big] = c_2 (v^{-1}-v) t_{(1,0)} t_{(1,1)} + c_2([3]-v c_1-1) t_{(2,1)}.$ Therefore
\[ \big[t_{(0,2)}, t_{(2,0)} \big] = c_2 \left( \tfrac{c_2}{c_1} -2 \right) t_{(2,2)} + c_2^{2}\tfrac{(v^{-1}-v)}{2c_1}  t_{(1,1)}  t_{(1,1)}.\]
Going back to $\mathsf{H}_X$ via the isomorphism from Theorem \ref{dragosmaintheorem} concludes the proof.
\end{proof}

\end{lemma}


\subsection{The even component of graphs of rank \texorpdfstring{$2$}{2} and degree \texorpdfstring{$2$}{2}}

We continue with a description of the even component of the graph $\mathscr{G}_{y,1}^{(2)}$ for a closed point $y$ of degree two.   

\begin{thm} \label{evencomponent}Let $y \in |X|$ a degree two closed point. Then the even component of $\mathscr{G}_{y,1}^{(2)}$ is given as follows. \smallskip

\begin{enumerate}[label=\textbf{{\upshape(\roman*)}}, wide=0pt]

\item \label{y-i1} Let  $\E :=\Line_1 \oplus \Line_2$ with $\Line_i \in \Pic X, \; i=1,2$. Suppose either $\deg \Line_2 > \deg \Line_1 +2$ or $\deg \Line_2= \deg \Line_1 +2$ and $\Line_1 \otimes \cO_X(y) \not\cong \Line_2$, then
\[ \mathcal{V}_{y,1}(\E) = \Big\{ \big(\E, \Line_1(-y) \oplus \Line_2, 1\big),  \big(\E, \Line_1 \oplus \Line_2(-y), q^2\big)  \Big\}. \] 

\item \label{y-i2} Let  $\E :=\Line_1 \oplus \Line_2$ with $\Line_i \in \Pic X, \; i=1,2$. Suppose $\deg \Line_2= \deg \Line_1 +2$ and $\Line_1 \otimes \cO_X(y) \cong \Line_2$, then
\[ \mathcal{V}_{y,1}(\E) = \Big\{ \big(\E, \Line_1(-y) \oplus \Line_2, 1\big), \big(\E, \Line_1 \oplus \Line_2(-y), q\big),
\big(\E, \E_{(x,2)}^{(2,2d_1)}, q^2 -q\big)  \Big\}, \]
where  the Atiyah representation of $\Line_1$ is $\E_{(x,1)}^{(1,d_1)}.$ \\

\item \label{y-i3} Let  $\E :=\Line \oplus \Line$ with $\Line \in \Pic X$, then
\[ \mathcal{V}_{y,1}(\E) = \Big\{ \big(\E, \Line(-y) \oplus \Line, q+1\big), \big(\E, \E_{(y,1)}^{(2,2d-2)}, q^2 -q\big)  \Big\}, \]
where $d=\deg \Line$.  \\

\item \label{y-i4} Let  $\E :=\Line_1 \oplus \Line_2$ with $\Line_i \in \Pic X, \; i=1,2$. Suppose $\deg \Line_2= \deg \Line_1$, then 
\begin{align*}
\mathcal{V}_{y,1}(\E) &= \Big\{ \big(\E, \Line_1(-y) \oplus \Line_2, 1\big), \big(\E, \Line_1 \oplus \Line_2(-y), 1\big),  \big(\E, \E_{(x,2)}^{(2,2d-2)}, q-1 \big),  \\
& \hspace*{0.8cm} \big(\E, \E_{(y',1)}^{(2,2d-2)}, q-1 \big)  \Big\},
\end{align*} 
where $x= x_1+x_2 -y$, $y'+y = x_1+x_2$ and the Atiyah representations of $\Line_1$ and $\Line_2$ are $\E_{(x_1,1)}^{(1,d)}$ and  $\E_{(x_2,1)}^{(1,d)}$, respectively.\\

\item \label{y-i5} Let  $\E := \E_{(x,2)}^{(2,0)}$, then 
\begin{align*}
\mathcal{V}_{y,1}(\E) &= \Big\{ \big(\E, \E_{(x-y,1)}^{(1,-2)}  \oplus \E_{(x,1)}^{(1,0)}, 1\big), \big(\E, \E_{(x_1,1)}^{(1,-1)} \oplus \E_{(x_2,1)}^{(1,-1)} , q \big), \big(\E, \E_{(x,2)}^{(2,-2)}, q \big),\\
& \hspace*{0.8cm}  \big(\E, \E_{(y',1)}^{(2,-2)}, q \big)  \Big\},
\end{align*} 
where $x_1 \neq x_2$ with $x_1 +x_2 = x+y$ and $y'\neq y$ with $y'= 2x-y.$ \\

\item \label{y-i6} Let  $\E := \E_{(y,1)}^{(2,0)}$, then 
\begin{align*}
\mathcal{V}_{y,1}(\E) &= \Big\{ \big(\E, \E_{(x_0,1)}^{(1,-1)} \oplus \E_{(x_0,1)}^{(1,-1)}, 1\big), \big(\E, \E_{(y,1)}^{(2,-2)} , m_1 \big), \big(\E, \E_{(y',1)}^{(2,-2)} , m_2 \big), \\
&  \hspace*{0.8cm}     \big(\E, \E_{(x_1,1)}^{(1,-1)} \oplus \E_{(x_2,1)}^{(1,-1)}, q+1 \big)  \Big\},
\end{align*} 
where $y'\neq y$ and $x_1 \neq x_2$ with $x_1 + x_2 = 0$. Let $z,z' \in |X_2| $ above $y$ and $w,w'\in |X_2|$ above $y'$, the multiplicities $m_1, \; m_2$ are given as follows 
\[m_1 = \begin{cases} q+1 & \text{ if } \; 3z \neq 0 \\ 
1 & \text{ if }\; 3z =0 
\end{cases}
\quad \text{ and } \quad  m_2 = \begin{cases} q+1 & \text{ if }\; w \neq 2z \; \text{ and } \; w \neq 2z'  \\ 
1 & \text{ if }\; w = 2z \; \text{ or } \; w =2z'.
\end{cases}\]

\item \label{y-i7} Let  $\E := \E_{(y',1)}^{(2,0)}$ with $y'\neq y$, then 
\begin{align*}
\mathcal{V}_{y,1}(\E) &= \Big\{ \big(\E, \E_{(y'',1)}^{(2,-2)}, 1\big), \big(\E, \E_{(y',1)}^{(2,-2)} , m_1 \big),   \big(\E, \E_{(y,1)}^{(2,-2)} , m_2 \big), \big(\E, \E_{(x,2)}^{(2,-2)} , q+1 \big),  \\
& \hspace*{0.8cm} \big(\E, \E_{(x_1,1)}^{(1,-1)} \oplus \E_{(x_2,1)}^{(1,-1)}, q+1 \big)   \Big\},
\end{align*} 
where $y,y',y''$ are pairwise distinct, $x=y-y'$ and $x_1 \neq x_2$ with $x_1 + x_2 = y'-y$. Moreover if $z,z' \in |X_2| $ sit above $y$,  $w,w'\in |X_2|$ sit above $y'$
and $v,v'\in |X_2|$ sit above $y''$, then $y''$ is such that either $v=z+w$ or $v=z+w'$ and 
\[m_1 = \begin{cases} 1 & \text{ if } \; z=2w \; \text{ or } \; z=2w' \\ 
q+1 & \text{ if }\; z\neq 2w \; \text{ and } \; z \neq 2w'
\end{cases}
\; \text{ and } \;  m_2 = \begin{cases} 1 & \text{ if }\; w = 2z \; \text{ or } \; w = 2z'  \\ 
q+1 & \text{ if } \; w \neq 2z \; \text{ and } \; w  \neq 2z'.
\end{cases}\]

\end{enumerate}

\begin{proof} Item \ref{y-i1} follows from \cite[Thm. 5.2]{roberto-elliptic} and Lemma \ref{lemmaB}. Item \ref{y-i2} follows from Lemma \ref{Kyc0}, \cite[Thm.\ 8.5]{oliver-graphs} and Lemma \ref{Kys0}. Applying either Lemma \ref{lemmaA} or \cite[Thm.\ 8.5]{oliver-graphs} coupled with either Lemma \ref{Kyty} or \cite[Cor.\ 6.5]{oliver-elliptic}
give item \ref{y-i3}. Item \ref{y-i4} follows from Lemmas \ref{lemmaA}, \ref{Kys0}, \ref{Kyty} and \ref{Kyty'}. Item \ref{y-i5} follows from Lemmas \ref{lemmaA}, \ref{lemmaB}, \ref{Kys0} and \ref{Kyty'}. Item \ref{y-i6} follows from Lemmas \ref{Kyc0}, \ref{Kyty}, \ref{Kyty'} and \ref{lemmaB}.  Finally, item \ref{y-i7} follows from Lemmas \ref{Kyty'}. \ref{Kyty}, \ref{Kys0} and \ref{lemmaB}. 
\end{proof}
\end{thm}

Before we state and prove the lemmas implying this theorem, we provide two examples that clarify the difference in the statement of the neighbors of traces bundles. 

\begin{ex}  Let $E_2$ and $E_3$ be the elliptic curves as in the Example \ref{ex1} with only one rational point i.e.\ $E_q(\Fq)= \{ x \}$ for $q=2,3$.   
We denote by $y_1, \ldots, y_q$ the $q$ closed points of $E_q$ of degree $2$ and
\[ E_q(\mathbb{F}_{q^2}) = \big\{ x, z_1,z_1', z_2,z_2', \ldots , z_{q}, z_{q}' \big\} \]
the $2q+1$ rational points of $ E_q(\mathbb{F}_{q^2}) $   prescribed such that $z_i,z_i'$ sit above $y_i$ for $i=1, \ldots, q.$  
Moreover, since $\mathrm{Cl}^0 ( E_q \otimes \mathbb{F}_{q^2}) \simeq \Z/(2q+1) \Z$, we can suppose $nz_1 = z_n$  in $\mathrm{Cl}^0(E_q \otimes \mathbb{F}_{q^2})$,  for $n=1, \ldots, q$ and $q=2,3.$

We fix $y=y_1$ and adopt the following notation for the vertices of  $\overline{\mathscr{G}}_{y,1}^{(2)}$ 
\[ c_{nx} := \mathcal{O}_{E_q} \oplus \mathcal{O}_{E_q}(nx), \quad t_{y_i} := \E_{(y_i,1)}^{(2,0)}, \quad s_0 := \E_{(x,2)}^{(2,0)} \quad \text{ and } \quad s_{x} := \E_{(x,1)}^{(2,1)}\]
for $i=1,\ldots,q.$ Hence
\[ \P \Bun_2^{\rm{dec}} E_q = \big\{ c_{nx} \big\}_{n \in \N}, \quad  \P \Bun_2^{\rm{gi}} E_q = \big\{ s_0, s_{x}\big\}, \quad
\P \Bun_2^{\rm{tr}} E_q = \big\{ t_y, t_{y_2}, \ldots, t_{y_q} \big\}. \]
Therefore the even component of the graph $\overline{\mathscr{G}}_{y,1}^{(2)}$ over $E_2$ is given by
\begin{center}

$
\beginpgfgraphicnamed{tikz/fig4}
\begin{tikzpicture}[ >=latex, scale=1.4][h]


\vertex[circle,fill,label={left:$t_y$}] (ty) at (0,1) {}; 

\vertex[circle,fill,label={left:$t_{y_2}$}] (t2) at (0,-1) {}; 

\vertex[circle,fill,label={above:$c_0$}] (c0) at (3,1) {}; 

\vertex[circle,fill,label={above:$s_0$}] (s0) at (3,-1) {}; 

\vertex[circle,fill,label={above:$c_2$}] (c2) at (5,0) {}; 

\vertex[circle,fill,label={above:$c_4$}] (c4) at (7,0) {}; 

\draw (8,0) node[right]{$\cdots$};


 \draw[black] (0,1.25) circle  (0.25cm) node at (0,1.6) {\scalebox{.5}{ $3$ } } ;

\path[-,font=\scriptsize]
(ty) edge[] node[pos=0.2,above,black] {\scalebox{.6}{ $1$}} (c0)
(ty)  edge[] node[pos=0.2,left,black] {\scalebox{.6}{ $1$}}  (t2);

\path[-,font=\scriptsize]
(t2) edge[] node[pos=0.2,left,black] {\scalebox{.6}{ $1$}} (ty)
(t2)  edge[] node[pos=0.2,above,black] {\scalebox{.6}{ $3$}}  (s0);

 \draw[black] (0,-1.25) circle  (0.25cm) node at (0,-1.6) {\scalebox{.5}{ $1$ } } ;


\path[-,font=\scriptsize]
(s0)  edge[] node[pos=0.2,above,black] {\scalebox{.6}{$2$}}  (t2)
(s0)  edge[] node[pos=0.3,above,black]  {\scalebox{.6}{$1$}} (c2);

 \draw[] (3,-1.25) circle  (0.25cm) node at (3,-1.6) {\scalebox{.5}{ $2$ } } ;
 

\path[-,font=\scriptsize]
(c0) edge[] node[pos=0.2,above,black] {\scalebox{.6}{ $2$}} (ty)
(c0)  edge[] node[pos=0.2,above,black] {\scalebox{.6}{$3$}} (c2);


\path[-,font=\scriptsize]
(c2) edge[] node[pos=0.2,above,black]{\scalebox{.6}{ $2$}} (c0)
(c2)  edge[] node[pos=0.2,above,black] {\scalebox{.6}{$2$}} (s0)
(c2)  edge[] node[pos=0.2,above,black] {\scalebox{.6}{$1$}} (c4);


\path[-,font=\scriptsize]
(c4) edge[] node[pos=0.2,above,black] {\scalebox{.6}{$4$}} (c2)
(c4) edge[] node[pos=0.2,above,black] {\scalebox{.6}{ $1$}} (8,0);

\end{tikzpicture}
\endpgfgraphicnamed
$
\end{center}
and the even component of the graph $\overline{\mathscr{G}}_{y,1}^{(2)}$ over $E_3$ is 
\begin{center}

$
\beginpgfgraphicnamed{tikz/fig5}
\begin{tikzpicture}[ >=latex, scale=1.2][h]


\vertex[circle,fill,label={left:$t_{y_2}$}] (t2) at (0,1) {}; 

\vertex[circle,fill,label={left:$t_{y_3}$}] (t3) at (0,-1) {}; 

\vertex[circle,fill,label={below:$t_y$}] (ty) at (3,1) {}; 

\vertex[circle,fill,label={above:$s_0$}] (s0) at (3,-1) {}; 

\vertex[circle,fill,label={above:$c_0$}] (c0) at (5,0) {}; 

\vertex[circle,fill,label={above:$c_2$}] (c2) at (7,0) {}; 

\vertex[circle,fill,label={above:$c_4$}] (c4) at (9,0) {}; 

\draw (10,0) node[right]{$\cdots$};


\path[-,font=\scriptsize]
(ty) edge[] node[pos=0.2,above,black] {\scalebox{.6}{ $1$}} (c0)
(ty)  edge[] node[pos=0.2,above,black] {\scalebox{.6}{ $1$}}  (t2)
(ty)  edge[] node[pos=0.2,above,black] {\scalebox{.6}{ $4$}}  (t3);;
 \draw[black] (3,1.25) circle  (0.25cm) node at (3,1.6) {\scalebox{.5}{ $4$ } } ;

\path[-,font=\scriptsize]
(t2) edge[] node[pos=0.2,left,black] {\scalebox{.6}{ $1$}} (t3)
(t2) edge[] node[pos=0.2,above,black] {\scalebox{.6}{ $1$}} (ty)
(t2)  edge[] node[pos=0.2,above,black] {\scalebox{.6}{ $4$}}  (s0);
\draw[black] (0,1.25) circle  (0.25cm) node at (0,1.6) {\scalebox{.5}{ $4$ } } ;

\path[-,font=\scriptsize]
(t3) edge[] node[pos=0.2,left,black] {\scalebox{.6}{ $1$}} (t2)
(t3) edge[] node[pos=0.2,left,black] {\scalebox{.6}{ $4$}} (ty)
(t3)  edge[] node[pos=0.2,below,black] {\scalebox{.6}{ $4$}}  (s0);
\draw[black] (0,-1.25) circle  (0.25cm) node at (0,-1.6) {\scalebox{.5}{ $1$ } } ;


\path[-,font=\scriptsize]
(s0) edge[] node[pos=0.2,below,black,swap] {\scalebox{.6}{$3$}} (t3)
(s0)  edge[] node[pos=0.2,above,black] {\scalebox{.6}{$3$}}  (t2)
(s0)  edge[] node[pos=0.2,below,black]  {\scalebox{.6}{$1$}} (c2);

 \draw[] (3,-1.25) circle  (0.25cm) node at (3,-1.6) {\scalebox{.5}{ $3$ } } ;
 

\path[-,font=\scriptsize]
(c0) edge[] node[pos=0.2,above,black] {\scalebox{.6}{ $6$}} (ty)
(c0)  edge[] node[pos=0.2,above,black] {\scalebox{.6}{$4$}} (c2);


\path[-,font=\scriptsize]
(c2) edge[] node[pos=0.2,above,black]{\scalebox{.6}{ $3$}} (c0)
(c2)  edge[] node[pos=0.2,below,black] {\scalebox{.6}{$6$}} (s0)
(c2)  edge[] node[pos=0.2,above,black] {\scalebox{.6}{$1$}} (c4);


\path[-,font=\scriptsize]
(c4) edge[] node[pos=0.2,above,black] {\scalebox{.6}{$9$}} (c2)
(c4) edge[] node[pos=0.2,above,black] {\scalebox{.6}{ $1$}} (10,0);

\end{tikzpicture}
\endpgfgraphicnamed
$
\end{center}

We conclude the example with the following observation. The fact that there exists an edge from $\E'$ to $\E$ whenever there exists an edge from $ \E$ to $\E'$ in above graphs is not a coincidence. Indeed, let $X$ be any elliptic curve, for every place $x\in |X|$ if there exists an edge $(\E,\E', m) \in \mathrm{Edge} \; \overline{\mathscr{G}}_{x,1}^{(2)}$, then there is a positive $m'\in \Z^{\times}$  such that $(\E',\E, m') \in \mathrm{Edge} \; \overline{\mathscr{G}}_{x,1}^{(2)}$. This follows from Theorem \ref{thm-dualities}. See also \cite[section 3]{oliver-graphs} for a different approach and a connection with Bruhat-Tits trees.  

\end{ex}


\subsection{Proof of Theorem \ref{evencomponent}} The proof of this section's main theorem  is based on the algorithm developed in \cite[Sec. 4]{roberto-elliptic} and is given by the following lemmas. 

\begin{lemma} \label{lemmaA} Let $\E' := \E_{(x_1,1)}^{(1,-2)} \oplus \E_{(x_2,1)}^{(1,0)}$ be a rank $2$ vector bundle on $X$. Then, \smallskip

\begin{enumerate}[label={\textbf{\upshape(\roman*)}}, ]
\item $m_{y,1}\big(\E_{(x_1+y,1)}^{(1,0)}  \oplus \E_{(x_2,1)}^{(1,0)}, \E'  \big) = 1$, if $x_2 \neq x_1 +y$; \smallskip

\item $m_{y,1}\big(\E_{(x_2,1)}^{(1,0)}  \oplus \E_{(x_2,1)}^{(1,0)}, \E'  \big) = q+1$, if $x_2 = x_1 +y$; \smallskip

\item $m_{y,1}\big(\E_{(x_2,2)}^{(2,0)}, \E'  \big) = 1$, if $x_2 = x_1 +y$; \smallskip

\item $m_{y,1}\big(\E_{(x_1,1)}^{(1,-2)} \oplus \E_{(x_2+y,1)}^{(1,2)} , \E'  \big) = q^2$.
\end{enumerate}

\begin{proof} We first observe that
\[\E' = v^2 \E_{(x_1,1)}^{(1,-2)} * \E_{(x_2,1)}^{(1,0)} = v^2 N_1^{-2} \sum_{\rho_1, \rho_2 \in \rm{P}_1} \text{\small $ \widetilde{\rho}_1(-x_1)\widetilde{\rho}_2(-x_2) $} T_{(1,-2)}^{\widetilde{\rho}_1}  T_{(1,0)}^{\widetilde{\rho}_2} \]
and 
\[ \mathcal{K}_y = 2 N_{2}^{-2} [2]^{-1} \sum_{\sigma \in \mathrm{P}_2} \text{ \small $\widetilde{\sigma}(-y)$ } T_{(0,2)}^{\widetilde{\sigma}}. \]
Thus,
 \begin{align*}
 \pi^{\mathrm{vec}} (\mathcal{K}_y * \E')   & = \tfrac{2 v^{2} N_{1}^{-2}}{[2]N_2}  \sum_{\substack{  \rho_1, \rho_2 \in \mathrm{P}_1 \\ \sigma \in \mathrm{P}_2}} \text{\small $ \widetilde{\sigma}(-y)\widetilde{\rho}_1(-x_1)\widetilde{\rho}_2(-x_2)  $}
\Big[ T_{(0,2)}^{\widetilde{\sigma}} , T_{(1,-2)}^{\widetilde{\rho}_1}  T_{(1,0)}^{\widetilde{\rho}_2} \Big] \\
& = v^4 T_{(1,0),x_1+y} T_{(1,0),x_2} + v^4 T_{(1,-2),x_1} T_{(1,2),x_2+y}.
\end{align*}
Moreover,
\[ T_{(1,0),x_1+y} T_{(1,0),x_2} = \begin{cases} \E_{(x_1+y,1)}^{(1,0)}  \oplus \E_{(x_2,1)}^{(1,0)} & \text{ if } x_2 \neq x_1 +y \medskip\\
(1+v^{-2})  \E_{(x_2,1)}^{(1,0)}  \oplus \E_{(x_2,1)}^{(1,0)} + \E_{(x_2,2)}^{(2,0)} & \text{ if } x_2 = x_1 +y, 
\end{cases} \]
and 
$$ T_{(1,-2),x_1} T_{(1,2),x_2+y} = v^{-4} \E_{(x_1,1)}^{(1,-2)} \oplus \E_{(x_2+y,1)}^{(1,2)}.$$
The proof follows by multiplying $\pi^{\mathrm{vec}} (\mathcal{K}_y * \E') $ by $v^{-4}.$ 
\end{proof}

\end{lemma}


\begin{lemma} \label{lemmaB} Let $\E' := \E_{(x_1,1)}^{(1,0)} \oplus \E_{(x_2,1)}^{(1,0)}$ be a rank $2$ vector bundle on $X$ with $x_1 \neq x_2$. Then, \smallskip

\begin{enumerate}[label={\textbf{\upshape(\roman*)}}, ]
\item $m_{y,1}\big(\E_{(x_1+y,1)}^{(1,2)}  \oplus \E_{(x_2,1)}^{(1,0)}, \E'  \big) = q^2$; \smallskip

\item $m_{y,1}\big(\E_{(x_1,1)}^{(1,0)}  \oplus \E_{(x_2+y,1)}^{(1,2)}, \E'  \big) = q^2$; \smallskip

\item $m_{y,1}\big(\E_{(w,2)}^{(2,2)} , \E'  \big) = q$ if $w=x_1+x_2+y$; \smallskip

\item $m_{y,1}\big(\E_{(y',1)}^{(2,2)} , \E'  \big) = q+1$ if $y'=x_1+x_2+y$; \smallskip

\item $m_{y,1}\big(\E_{(x_1',1)}^{(1,1)} \oplus \E_{(x_2',1)}^{(1,1)} , \E'  \big) = q-1$ if  $x_1' + x_2' =x_1+x_2+y$ and $x_1'  \neq x_2'$. \smallskip
\end{enumerate}

\begin{proof} Since $x_1 \neq x_2$, we can write $\E'$ in the Hall algebra $\mathsf{H}_X$ as follows
\[\E' =  \E_{(x_1,1)}^{(1,0)} * \E_{(x_2,1)}^{(1,0)} =  N_1^{-2} \sum_{\rho_1, \rho_2 \in \rm{P}_1} \text{\small $ \widetilde{\rho}_1(-x_1)\widetilde{\rho}_2(-x_2) $} T_{(1,0)}^{\widetilde{\rho}_1} 
 T_{(1,0)}^{\widetilde{\rho}_2}. \] 
Hence, 
 \begin{align*}
 \pi^{\mathrm{vec}} (\mathcal{K}_y * \E')   & = \tfrac{2 N_{1}^{-2}}{[2]N_2}  \Bigg( c_2  \sum_{\substack{  \rho_1, \rho_2 \in \mathrm{P}_1 \\ \sigma = \rm{Norm}_{1}^{2}(\rho_1)}} \text{\small $ \widetilde{\sigma}(-y) \widetilde{\rho}_1(-x_1)\widetilde{\rho}_2(-x_2)  $} T_{(1,2)}^{\widetilde{\rho}_1}  T_{(1,0)}^{\widetilde{\rho}_2}  \\
 & \quad + c_2  \sum_{\substack{  \rho_1, \rho_2 \in \mathrm{P}_1 \\ \sigma = \rm{Norm}_{1}^{2}(\rho_2)}} \text{\small $ \widetilde{\sigma}(-y)\widetilde{\rho}_1(-x_1)\widetilde{\rho}_2(-x_2)  $} T_{(1,0)}^{\widetilde{\rho}_1}  T_{(1,2)}^{\widetilde{\rho}_2}  \Bigg) \\
& = v^2 T_{(1,0),x_2} T_{(1,2),x_1+y} + v^2 T_{(1,0),x_1} T_{(1,2),x_2+y} + v^3 \sum_{ \substack{ w \in |X| \\ w=x_1+x_2+y}} T_{(2,2),w} \\
&\quad  + \tfrac{(v^2-v^4)}{2} \sum_{ \substack{ x_1',x_2' \in |X| \\ x_1'+x_2'=x_1+x_2+y}}  T_{(1,1),x_1'} T_{(1,1),x_2'}.
\end{align*}
We conclude the proof by observing that 
\[ T_{(1,0),x_2} T_{(1,2),x_1+y} = v^{-2}  \E_{(x_1+y,1)}^{(1,2)}  \oplus \E_{(x_2,1)}^{(1,0)}\]
and
\[T_{(1,0),x_1} T_{(1,2),x_2+y} = v^{-2} \E_{(x_1,1)}^{(1,0)}  \oplus \E_{(x_2+y,1)}^{(1,2)}.\]
Moreover
\[ T_{(1,1),x_1'} T_{(1,1),x_2'}   = \begin{cases} (1+v^{-2} )\E_{(x_1',1)}^{(1,1)} \oplus \E_{(x_1',1)}^{(1,1)} + \E_{(x_1',2)}^{(2,2)} & \text{ if } x_1'=x_2' \medskip \\
\E_{(x_1',1)}^{(1,1)} \oplus \E_{(x_2',1)}^{(1,1)} & \text{ if } x_1'\neq x_2'  \end{cases}\] 
and 
\[ T_{(2,2),w} = \begin{cases} \E_{(w,2)}^{(2,2)} + (1-v^{-2}) \E_{(w,1)}^{(1,1)} \oplus \E_{(w,1)}^{(1,1)} & \text{ if } |w|=1 \medskip\\
[2] \E_{(w,1)}^{(2,2)} & \text{ if } |w|=1. \end{cases}\]
Putting all together and multiplying by $v^{-4}$ yields the desire statement.  
\end{proof}

\end{lemma}


\begin{lemma} \label{Kyc0} Let $\E' := \E_{(x,1)}^{(1,0)} \oplus \E_{(x,1)}^{(1,0)}$ be a rank $2$ vector bundle on $X$. Then, \smallskip

\begin{enumerate}[label={\textbf{\upshape(\roman*)}}, ]
\item $m_{y,1}\big(\E_{(x,1)}^{(1,0)}  \oplus \E_{(x+y,1)}^{(1,2)}, \E'  \big) = q$; \smallskip

\item $m_{y,1}\big(\E_{(y,1)}^{(2,2)} , \E'  \big) = 1$. \smallskip

\end{enumerate}

\begin{proof} As in \cite[Lemma 6.6]{roberto-elliptic} we can write $\E'$ in terms of elements in the twisted spherical Hall algebras as follows
\[ \E' =  \tfrac{v^2 N_1^{-2}}{2} \sum_{\rho, \rho_1 \in \rm{P}_1} \text{\small $ \widetilde{\rho}(-x)\widetilde{\rho}_1(-x) $} T_{(1,0)}^{\widetilde{\rho}} T_{(1,0)}^{\widetilde{\rho}_1} 
-  \tfrac{v^2 N_1^{-2}}{[2]} \sum_{\rho_2\in \rm{P}_2} \text{\small $ \widetilde{\rho}_2(-x) $} T_{(2,0)}^{\widetilde{\rho}_2}. \]
Thus, 
 \begin{align*}
 \pi^{\mathrm{vec}} (\mathcal{K}_y * \E')   & = \tfrac{v^2 N_{1}^{-2}}{[2]N_2}   \sum_{\substack{  \rho, \rho_1 \in \mathrm{P}_1 \\ \sigma \in \rm{P}_2 }} 
 \text{\small $ \widetilde{\sigma}(-y)\widetilde{\rho}(-x)\widetilde{\rho}_1(-x)  $} \Big[ T_{(0,2)}^{\widetilde{\sigma}}, T_{(1,0)}^{\widetilde{\rho}}  T_{(1,0)}^{\widetilde{\rho}_1} \Big] \\
 & \quad -  \tfrac{2 v^2}{[2]^2N_2^2}  \sum_{\substack{  \sigma, \rho_2 \in \mathrm{P}_2 }} \text{\small $ \widetilde{\sigma}(-y)\widetilde{\rho}_2(-x) $} 
 \Big[ T_{(0,2)}^{\widetilde{\sigma}},   T_{(2,0)}^{\widetilde{\rho}_2} \Big].  
 \end{align*}
As usual, we denote the first term in the above sum by $A$ and the second term by $B$. Let us calculate each term separated. Proceeding as previous lemma we conclude that
\[A = v^2 \E_{(x,1)}^{(1,0)}  \oplus \E_{(x+y,1)}^{(1,2)} + \tfrac{(v^6 +v^4)}{2} \E_{(y',1)}^{(2,2)}
+ \tfrac{v^4}{2} \E_{(x',2)}^{(2,2)} +  \tfrac{(v^4 - v^6)}{2} \E_{(x_1,1)}^{(1,1)} \oplus  \E_{(x_2,1)}^{(1,1)} \]
where $|y'|=2$ with $y'= y+2x$, $|x'|=1$ with $2x' = y+2x$ and $x_1 \neq x_2$ with $x_1+x_2 = y+2x.$

From Lemma \ref{lemma-commutator},
\begin{align*}
B &= \tfrac{2 v^2}{[2]^2N_2^2} \Bigg( \sum_{\chi \in \rm{P}_1}  \text{\small $ \rm{Norm}_{1}^{2}(\widetilde{\chi})(-y-x) $} 
\left( \tfrac{c_{2}^2 (v^{-1}-v)}{2c_1} T_{(1,1)}^{\widetilde{\chi}} T_{(1,1)}^{\widetilde{\chi}} 
+ \big( \tfrac{c_2^2}{c_1} -2c_2 \big) T_{(2,2)}^{ \rm{Norm}_{1}^{2}(\widetilde{\chi})} \right) \\
& \quad - 2c_2 \sum_{\substack{\sigma \in \rm{P_2} \\ \text{primitive}}}  \text{\small $\widetilde{\sigma}(-y-x) $} T_{(2,2)}^{\widetilde{\sigma}} \Bigg) \\
& = \sum_{w \in |X|} T_{(2,2),w} \Big( \tfrac{2 v^2}{[2]^2N_2^2}  \Big) \left( \tfrac{c_2^2}{c_1} \sum_{\chi \in \rm{P}_1}  \text{\small $ \rm{Norm}_{1}^{2}(\widetilde{\chi})(w-y-x) $} 
-2c_2 \sum_{\substack{\sigma \in \rm{P_2} }}  \text{\small $\widetilde{\sigma}(w-y-x) $} \right) \\
& \quad + \tfrac{(v^4 - v^6)}{4} \sum_{x_1',x_2' \in |X|}T_{(1,1),x_1'}  T_{(1,1),x_2'}  
\end{align*}
where $x_1' +x_2' = y+2x.$ Since 
\[ \sum_{\substack{\sigma \in \rm{P_2} }}  \widetilde{\sigma}(w-y-x) =  \begin{cases} N_2/2 & \text{ if } w=y \\
 0 & \text{ otherwise }  \end{cases} \]
we conclude that 
\begin{align*}
B &=  \tfrac{(v^4 - v^6)}{2} \E_{(x_1,1)}^{(1,1)} \oplus  \E_{(x_2,1)}^{(1,1)} + \tfrac{v^4}{2} \E_{(x',2)}^{(2,2)}
+ \tfrac{(v^6 + v^4)}{2} \E_{(y'',1)}^{(2,2)} +  \tfrac{(v^6 - v^4)}{2} \E_{(y,1)}^{(2,2)} 
\end{align*}
where $y''\neq y$. Therefore, 
\[ \pi^{\mathrm{vec}} (\mathcal{K}_y * \E')   = v^2 \E_{(x,1)}^{(1,0)}  \oplus \E_{(x+y,1)}^{(1,2)}  + v^4 \E_{(y,1)}^{(2,2)} \]
and we have the desired statement multiplying above equation by $v^{-4}.$
\end{proof}
\end{lemma}


\begin{lemma} \label{Kys0} Let $\E' := \E_{(x,2)}^{(2,0)} $ be a rank $2$ vector bundle on $X$. Then, \smallskip

\begin{enumerate}[label={\textbf{\upshape(\roman*)}}, ]
\item $m_{y,1}\big(\E_{(x',2)}^{(2,2)}  , \E'  \big) = q$ if $2x'=y+2x$; \smallskip

\item $m_{y,1}\big(\E_{(x_1,1)}^{(1,1)} \oplus \E_{(x_2,1)}^{(1,1)}, \E'  \big) = q-1$ for $x_1 \neq x_2$ and $x_1+x_2 =y+2x$; \smallskip

\item $m_{y,1}\big(\E_{(y',1)}^{(2,2)}  , \E'  \big) = q+1$ for $y'\neq y$ and $y'=y+2x$; \smallskip

\item $m_{y,1}\big(\E_{(x,1)}^{(1,0)} \oplus \E_{(x+y,1)}^{(1,2)}    , \E'  \big) = q^2 - q$.\smallskip
\end{enumerate}

\begin{proof} As in \cite[Lemma 6.5]{roberto-elliptic}
\[ \E' = \tfrac{(v^2 +1)}{N_2[2]} \sum_{\rho \in \rm{P}_2}  \text{\small $\widetilde{\rho}(-x) $} T_{(2,0)}^{\widetilde{\rho}} +  \tfrac{(1 -v^2)}{2 N_1^2} 
\sum_{\rho_1, \rho_2 \in \rm{P}_1}  \text{\small $\widetilde{\rho}_1(-x) \widetilde{\rho}_2(-x) $} T_{(1,0)}^{\widetilde{\rho}_1} T_{(1,0)}^{\widetilde{\rho}_2}. \] 
Thus, 
 \begin{align*}
 \pi^{\mathrm{vec}} (\mathcal{K}_y * \E')   & = \tfrac{(1-v^2) N_{1}^{-2}}{ [2] N_2}   \sum_{\substack{  \rho_1, \rho_2 \in \mathrm{P}_1 \\ \sigma \in \rm{P}_2 }} 
 \text{\small $ \widetilde{\sigma}(-y)\widetilde{\rho}_1(-x)\widetilde{\rho}_2(-x)  $} \Big[ T_{(0,2)}^{\widetilde{\sigma}}, T_{(1,0)}^{\widetilde{\rho}_1}  T_{(1,0)}^{\widetilde{\rho}_2} \Big] \\
 &\quad +  \tfrac{2 (v^2+1)}{[2]^2 N_2^2}  \sum_{\substack{  \sigma, \rho \in \mathrm{P}_2 }} \text{\small $ \widetilde{\sigma}(-y)\widetilde{\rho}(-x) $} 
 \Big[ T_{(0,2)}^{\widetilde{\sigma}},   T_{(2,0)}^{\widetilde{\rho}} \Big].  
 \end{align*}
Let $A$ (resp. $B$) denote the first (resp. second) term in the above sum. 
Analogous to the previous lemma, applying the algorithm in \cite{roberto-elliptic} yields
\[A = (1-v^2) \E_{(x,1)}^{(1,0)}  \oplus \E_{(x+y,1)}^{(1,2)} + \tfrac{(v^2 -v^6)}{2} \E_{(y'',1)}^{(2,2)}
+ \tfrac{(v^2-v^4)}{2} \E_{(x',2)}^{(2,2)} +  \tfrac{(v^2- 2v^4 + v^6)}{2} \E_{(x_1,1)}^{(1,1)} \oplus  \E_{(x_2,1)}^{(1,1)} \]
where $|y''|=2$ with  $y''= y+2x$, $|x'|=1$ with $2x' = y+2x$ and $x_1 \neq x_2$ with $x_1+x_2 = y+2x.$
On the other hand, Lemma \ref{lemma-commutator} implies,
\begin{align*}
B & = \sum_{w \in |X|} T_{(2,2),w} \Big( \tfrac{2 (v^2+1)}{[2]^2N_2^2}  \Big) \left( \tfrac{c_2^2}{c_1} \sum_{\chi \in \rm{P}_1}  \text{\small $ \rm{Norm}_{1}^{2}(\widetilde{\chi})(w-y-x) $} 
-2c_2 \sum_{\substack{\sigma \in \rm{P_2} }}  \text{\small $\widetilde{\sigma}(w-y-x) $} \right) \\
&\quad  + \tfrac{(v^2 - v^6)}{4} \sum_{x_1',x_2' \in |X|}T_{(1,1),x_1'}  T_{(1,1),x_2'}  
\end{align*}
where $x_1' +x_2' = y+2x.$ As in the proof of previous lemma we observe  
\[ \sum_{\substack{\sigma \in \rm{P_2} }}  \widetilde{\sigma}(w-y-x) =  \begin{cases} N_2/2 & \text{ if } w=y \\
 0 & \text{ otherwise }  \end{cases} \]
in order to conclude that 
$$B = \tfrac{(v^2 + v^4)}{2} \E_{(x',2)}^{(2,2)} + \tfrac{(v^2-  v^6)}{2} \E_{(x_1,1)}^{(1,1)} \oplus  \E_{(x_2,1)}^{(1,1)} 
+  \tfrac{(v^6 -v^2)}{2} \E_{(y,1)}^{(2,2)} + \tfrac{(v^6 + 2v^4 + v^2)}{2} \E_{(y',1)}^{(2,2)} $$
where $x',x_1,x_2$ are given as before and $y'$ is such that $y'=y+2x$ and $y'\neq y.$ The lemma follows by multiplying $v^{-4}$ in the sum $A + B$. 
\end{proof}

\end{lemma}


\begin{lemma} \label{Kyty} Let $\E' := \E_{(y,1)}^{(2,0)} $ be a rank $2$ vector bundle on $X$ and $z,z'\in |X_2|$ siting above $y$.  Then, \smallskip

\begin{enumerate}[label={\textbf{\upshape(\roman*)}}, ]
\item $m_{y,1}\big(\E_{(x,1)}^{(1,1)}  \oplus \E_{(x,1)}^{(1,1)}, \E'  \big) = q^2 - q$, where   $2x = 2y$; \smallskip

\item $m_{y,1}\big(\E_{(x_1,1)}^{(1,1)}  \oplus \E_{(x_2,1)}^{(1,1)}, \E'  \big) = q - 1$, where $x_1 \neq x_2$ and  $x_1+x_2 = 2y$; \smallskip

\item $m_{y,1}\big(\E_{(y,1)}^{(2,2)} , \E' \big) =m_1 $, where  $m_1 = q+1$ if $3z \neq 0$ and $m_1 = 1$ if $3z=0$; \smallskip

\item $m_{y,1}\big(\E_{(y',1)}^{(2,2)} , \E'  \big) =m_2$, where  $m_2 = 1$ if either $w=2z$ or $w=2z'$ and $m_2 = q+1$ otherwise . Whence $w \in |X_2|$ sits above $y'$ and $y'\neq y$. \smallskip

\end{enumerate}

\begin{proof} As practice, we start by writing $\E'$ in terms of elements in the twisted spherical Hall algebras, 
\[\E' = \tfrac{2}{N_2[2]} \sum_{\rho \in \rm{P}_2}  \text{\small $\widetilde{\rho}(-y) $} T_{(2,0)}^{\widetilde{\rho}}.\] 
Hence, 
 \begin{align*}
 \pi^{\mathrm{vec}} (\mathcal{K}_y * \E')   & =  \tfrac{4}{[2]^2 N_2^2}  \sum_{\substack{  \sigma, \rho \in \mathrm{P}_2 }} \text{\small $ \widetilde{\sigma}(-y)\widetilde{\rho}(-y) $} 
 \Big[ T_{(0,2)}^{\widetilde{\sigma}},   T_{(2,0)}^{\widetilde{\rho}} \Big] \\
 &= \tfrac{4}{[2]^2 N_2^2} \sum_{\substack{  \sigma, \rho \in \mathrm{P}_2 \\ \rho \text{ primitive }}} \text{\small $ \widetilde{\sigma}(-y)\widetilde{\rho}(-y) $} 
 \Big[ T_{(0,2)}^{\widetilde{\sigma}},   T_{(2,0)}^{\widetilde{\rho}} \Big]  \\
 & \quad + \tfrac{4}{[2]^2 N_2^2}  \sum_{\substack{  \sigma\in \mathrm{P}_2, \chi \in \rm{P}_1 \\ \rho = \rm{Norm}_{1}^{2}(\chi)}} \text{\small $ \widetilde{\sigma}(-y)\widetilde{\rho}(-y) $}  \Big[ T_{(0,2)}^{\widetilde{\sigma}},   T_{(2,0)}^{\widetilde{\rho}} \Big].
 \end{align*}
Lemma \ref{lemma-commutator} yields
 \begin{multline*}
 \pi^{\mathrm{vec}} (\mathcal{K}_y * \E')    = \tfrac{v^4(v^{-1} - v)}{2c_1} \sum_{\chi \in \rm{P}_1} \sum_{x_1',x_2' \in |X|} \text{\small $ \widetilde{\chi}(x_1'+x_2') \rm{Norm}_{1}^{2}(\widetilde{\chi})(-2y) $}  T_{(1,1),x_1'}  T_{(1,1),x_2'} \\
 + \sum_{x' \in |X|} T_{(2,2),x'} \Big( \tfrac{4}{[2]^2 N_2^2}   \Big) \left( \tfrac{c_2^2}{c_1} \sum_{\chi \in \rm{P}_1}  \text{\small $ \rm{Norm}_{1}^{2}(\widetilde{\chi})(x'-2y) $} -2c_2 \sum_{\substack{\sigma \in \rm{P_2} }}  \text{\small $\widetilde{\sigma}(x'-2y) $} \right).
\end{multline*}
Let $y' \in |X|$ be a degree two closed point distinct from $y$. Let $z,z' \in |X_2|$ siting above $y$ and $w,w \in |X_2|$ siting above $y'.$ A straightforward calculation shows that
\[ \sum_{\substack{\sigma \in \rm{P_2} }}  \widetilde{\sigma}(x'-2y) =  \begin{cases} N_2/2 & \text{ if } |x'|=1 \\
N_2/4 & \text{ if } x' = y' \text{ and either } w=2z \text{ or } w=2z' \\
N_2/4 & \text{ if } x'=y \text{ and } 3z=0\\
0 &\text{ otherwise}. \end{cases} \]
Therefore, 
 \begin{align*}
\pi^{\mathrm{vec}} (\mathcal{K}_y * \E')   & = (1-v^2) \E_{(x,1)}^{(1,1)}  \oplus \E_{(x,1)}^{(1,1)} + (v^2 - v^4) \E_{(x_1,1)}^{(1,1)}  \oplus \E_{(x_2,1)}^{(1,1)} \\
& \quad + m_1' \E_{(y,1)}^{(2,2)}  + m_2' \E_{(y',1)}^{(2,2)} 
\end{align*}
for $2x=2y$, $x_1 \neq x_2$ with $x_1+x_2 =2y$ and $y'\neq y$. Where 
$$m_1'= \begin{cases} v^2+v^4 &\text{ if } 3z \neq 0 \\ 
v^4 & \text{ if } 3z=0\end{cases}  \quad \text{and} \quad 
m_2'= \begin{cases} v^4 &\text{ if either } w = 2z \text{ or } w=2z'  \\ 
v^4 +v^2 & \text{ otherwise }\end{cases} $$
with $z,z',w,w' \in |X_2|$ are as above. 
Multiplying $\pi^{\mathrm{vec}} (\mathcal{K}_y * \E') $ by $v^{-4}$ yields the desired statement. 
\end{proof}

\end{lemma}


\begin{lemma} \label{Kyty'} Let $\E' := \E_{(y',1)}^{(2,0)} $ be a rank $2$ vector bundle on $X$ with $y'\neq y$. Let $z,z'\in |X_2|$ siting above $y$ and $w,w' \in |X_2|$ siting above $y'.$ Then, \smallskip

\begin{enumerate}[label={\textbf{\upshape(\roman*)}}, ]
\item $m_{y,1}\big(\E_{(x,2)}^{(2,2)}, \E'  \big) =  q$, where   $2x = y + y$; \smallskip

\item $m_{y,1}\big(\E_{(x_1,1)}^{(1,1)}  \oplus \E_{(x_2,1)}^{(1,1)}, \E'  \big) = q - 1$, where $x_1 \neq x_2$ and  $x_1+x_2 = y+y'$; \smallskip

\item $m_{y,1}\big(\E_{(y'',1)}^{(2,2)} , \E' \big) = 1 $, where  $y,y',y''$ are pairwise distinct; \smallskip

\item $m_{y,1}\big(\E_{(y,1)}^{(2,2)} , \E' \big) =m_1 $, where  $m_1 = 1$ if $w=2z$ or $w=2z'$ and  $m_1 = q+1$ otherwise; 
\smallskip

\item $m_{y,1}\big(\E_{(y',1)}^{(2,2)} , \E'  \big) =m_2$, where  $m_2 = 1$ if either $z=2w$ or $z=2w'$ and $m_2 = q+1$ otherwise. \smallskip

\end{enumerate}

\begin{proof} We proceed mutatis mutandis last lemma. Hence, 
 \begin{align*}
 \pi^{\mathrm{vec}} (\mathcal{K}_y * \E')   & =  \tfrac{4}{[2]^2 N_2^2}  \sum_{\substack{  \sigma, \rho \in \mathrm{P}_2 }} \text{\small $ \widetilde{\sigma}(-y)\widetilde{\rho}(-y') $} 
 \Big[ T_{(0,2)}^{\widetilde{\sigma}},   T_{(2,0)}^{\widetilde{\rho}} \Big] \\
 &= \tfrac{4}{[2]^2 N_2^2} \sum_{\substack{  \sigma, \rho \in \mathrm{P}_2 \\ \rho \text{ primitive }}} \text{\small $ \widetilde{\sigma}(-y)\widetilde{\rho}(-y') $} 
 \Big[ T_{(0,2)}^{\widetilde{\sigma}},   T_{(2,0)}^{\widetilde{\rho}} \Big]  \\
 &\quad  + \tfrac{4}{[2]^2 N_2^2}  \sum_{\substack{  \sigma\in \mathrm{P}_2, \chi \in \rm{P}_1 \\ \rho = \rm{Norm}_{1}^{2}(\chi)}} \text{\small $ \widetilde{\sigma}(-y)\widetilde{\rho}(-y') $}  \Big[ T_{(0,2)}^{\widetilde{\sigma}},   T_{(2,0)}^{\widetilde{\rho}} \Big]\\
  & = \tfrac{v^4(v^{-1} - v)}{2c_1} \sum_{\chi \in \rm{P}_1} \sum_{x_1',x_2' \in |X|} \text{\small $ \widetilde{\chi}(x_1'+x_2') \rm{Norm}_{1}^{2}(\widetilde{\chi})(-y-y') $}  T_{(1,1),x_1'}  T_{(1,1),x_2'} \\
& \quad + \sum_{x' \in |X|} T_{(2,2),x'} \text{ \scriptsize $ \Big( \tfrac{4}{[2]^2 N_2^2}   \Big) \left( \tfrac{c_2^2}{c_1} \sum_{\chi \in \rm{P}_1}  \text{\small $ \rm{Norm}_{1}^{2}(\widetilde{\chi})(x'-y-y') $} -2c_2 \sum_{\substack{\sigma \in \rm{P_2} }}  \text{\small $\widetilde{\sigma}(x'-y-y') $} \right)$}.
\end{align*}
where the last equality is given by Lemma \ref{lemma-commutator}. Splitting above sums in all possibilities for $x_1', x_2', x'$ and observing 
\[ \sum_{\substack{\sigma \in \rm{P_2} }}  \widetilde{\sigma}(x'-y-y') =  \begin{cases} N_2/4 & \text{ if } |x'|=2 \text{ and } x',y,y' \text{ are pairwise distinct } \\
N_2/4 & \text{ if } x' = y \text{ and either } w=2z \text{ or } w=2z' \\
N_2/4 & \text{ if } x'=y' \text{ and either } z=2w \text{ or } z=2w'\\
0 &\text{ otherwise}, \end{cases} \]
we conclude 
 \begin{align*}
 \pi^{\mathrm{vec}} (\mathcal{K}_y * \E')   & = v^2 \E_{(x,2)}^{(2,2)} + (v^2-v^4) \E_{(x_1,1)}^{(1,1)}  \oplus \E_{(x_2,1)}^{(1,1)}  + v^4 \E_{(y'',1)}^{(2,2)}  + m_1' \E_{(y,1)}^{(2,2)}  + m_2' \E_{(y',1)}^{(2,2)}. 
 \end{align*}
 where $2x = y + y$, $x_1 \neq x_2$ with $x_1+x_2 = y+y'$, $y'' \neq y'$, $y''\neq y$,
 $$m_1'= \begin{cases} v^4 &\text{ if either } w=2z \text{ or } w=2z'  \\ 
v^4+v^2 & \text{ otherwise }\end{cases}  $$ 
and
$$ m_2'= \begin{cases} v^4 &\text{ if either } z = 2w \text{ or } z=2w'  \\ 
v^4 +v^2 & \text{ otherwise }\end{cases}. $$
Multiplying  $ \pi^{\mathrm{vec}} (\mathcal{K}_y * \E') $ by $v^{-4}$ yields the proof.  
\end{proof}

\end{lemma}


\subsection{Some edges for graphs of rank \texorpdfstring{$3$}{3} and degree \texorpdfstring{$2$}{2}}
\label{subsection: edges for graphs of rank 3 and degree 2}
Let $y$ be a place of degree two. In the following, we finish the first part of this project by determining some edges and give a complete description of the $\Phi_{y,1}$-neighborhood of $\mathcal{O}_X \oplus \mathcal{O}_X \oplus \mathcal{O}_X $ in $\mathscr{G}_{y,1}^{(3)}$.

\begin{lemma} \label{idL2} Let $\E' := \E_{(x,1)}^{(1,-2)} \oplus \E_{(x',1)}^{(1,0)} \oplus \E_{(x',1)}^{(1,0)}  $ be a rank $3$ vector bundle on $X$. Then
\smallskip

\begin{enumerate}[label={\textbf{\upshape(\roman*)}}, ]

   \item $m_{y,1}\big( \E_{(x+y,1)}^{(1,0)} \oplus \E_{(x',1)}^{(1,0)} \oplus \E_{(x',1)}^{(1,0)}    , \E'  \big) =1$, if $x+y\neq x'$; \smallskip

    \item $m_{y,1}\big(\E_{(x',1)}^{(1,0)} \oplus \E_{(x',1)}^{(2,0)} , \E'  \big) = 1$,  if $x+y = x'$;\smallskip

      \item $m_{y,1}\big(\E_{(x',1)}^{(1,0)} \oplus \E_{(x',1)}^{(1,0)} \oplus \E_{(x',1)}^{(1,0)} , \E'  \big) = q^2 + q+1$, if $x+y= x'$;
      
      \item $m_{y,1}\big(\E_{(x,1)}^{(1,-2)} \oplus \E_{(x',1)}^{(1,0)} \oplus \E_{(x'+y,1)}^{(1,2)} , \E'  \big) = q^3$;
      
      \item $m_{y,1}\big(\E_{(x,1)}^{(1,-2)} \oplus \E_{(y,1)}^{(2,2)}  , \E'  \big) =  q^2$.
   \end{enumerate}
Moreover, there is no other edge with target to $\E'$ in $\mathscr{G}_{y,1}^{(3)}.$

\begin{proof} Analogous to the proofs of the lemmas in the former section, the proof is based on the algorithm in \cite{roberto-elliptic} to compute the product $\mathcal{K}_y * \E'$ in the Hall algebra $\mathsf{H}_X$ of $X$. Since $y$ is a closed point of degree two, its skyscraper sheaf is given by
\[ \mathcal{K}_y = 2 N_{2}^{-2} [2]^{-1} \sum_{\sigma \in \mathrm{P}_2} \text{ \small $\widetilde{\sigma}(-y)$ } T_{(0,2)}^{\widetilde{\sigma}} \]
in $\mathsf{H}_X.$ On the other hand, we can write $\E' = v^4 \; \E_{(x,1)}^{(1,-2)} * \big( \E_{(x',1)}^{(1,0)} \oplus \E_{(x',1)}^{(1,0)} \big) $ as
\[ \text{ \small $ \tfrac{v^6 N_1^{-3}}{2} \sum_{\chi, \rho, \rho_1 \in \mathrm{P}_1} \text{ \small $\widetilde{\chi}(-x) \widetilde{\rho}(-x') \widetilde{\rho}_1(-x')$} T_{(1,-2)}^{\widetilde{\chi}} T_{(1,0)}^{\widetilde{\rho}}  T_{(1,0)}^{\widetilde{\rho}_1} 
  -  \tfrac{v^6 N_1^{-1} }{[2]N_2}   \sum_{ \substack{ \chi \in \mathrm{P}_1 \\ \rho_2 \in \mathrm{P}_2}} \text{ \small $\widetilde{\chi}(-x) \widetilde{\rho}_2(-x')$} T_{(1,-2)}^{\widetilde{\chi}} T_{(2,0)}^{\widetilde{\rho}_2} $} \]
in terms of generators of the Hall algebra of $X$. Hence
\begin{align*}
\pi^{\mathrm{vec}} (\mathcal{K}_y * \E')   & = \tfrac{v^{6} N_{1}^{-3}}{  N_{2} [2]}   \displaystyle\sum_{\substack{ \chi, \rho, \rho_1 \in \mathrm{P}_1 \\ \sigma \in \mathrm{P}_2 }} \text{\small $ \widetilde{\sigma}(-y)\widetilde{\chi}(-x)\widetilde{\rho}(-x') \widetilde{\rho}_1(-x') $}
\Big[ T_{(0,2)}^{\widetilde{\sigma}} , T_{(1,-2)}^{\widetilde{\chi}}  T_{(1,0)}^{\widetilde{\rho}} T_{(1,0)}^{\widetilde{\rho}_1} \Big] \\
&\quad -  \tfrac{2 v^{6} N_{1}^{-1}}{  N_{2}^{2} [2]^{2}}   \displaystyle\sum_{\substack{ \chi \in \mathrm{P}_1 \\ \sigma, \rho_2 \in \mathrm{P}_2 }} \text{\small $ \widetilde{\sigma}(-y)\widetilde{\chi}(-x)\widetilde{\rho}_2(-x') $}
\Big[ T_{(0,2)}^{\widetilde{\sigma}} , T_{(1,-2)}^{\widetilde{\chi}}  T_{(2,0)}^{\widetilde{\rho}_2} \Big].
\end{align*} 
We denote by $A$ the first term in the above sum and by $B$ the second one.   In order to calculate $A$, observe
\begin{align*}
\Big[ T_{(0,2)}^{\widetilde{\sigma}} , T_{(1,-2)}^{\widetilde{\chi}}  T_{(1,0)}^{\widetilde{\rho}} T_{(1,0)}^{\widetilde{\rho}_1} \Big]
&= \Big[ T_{(0,2)}^{\widetilde{\sigma}} , T_{(1,-2)}^{\widetilde{\chi}} \Big] T_{(1,0)}^{\widetilde{\rho}} T_{(1,0)}^{\widetilde{\rho}_1} +  T_{(1,-2)}^{\widetilde{\chi}} \Big[ T_{(0,2)}^{\widetilde{\sigma}} ,T_{(1,0)}^{\widetilde{\rho}} \Big] T_{(1,0)}^{\widetilde{\rho}_1} \\
& \quad +  T_{(1,-2)}^{\widetilde{\chi}}  T_{(1,0)}^{\widetilde{\rho}} \Big[ T_{(0,2)}^{\widetilde{\sigma}} ,T_{(1,0)}^{\widetilde{\rho}_1} \Big].
\end{align*}
Moreover, 
\[ \Big[ T_{(0,2)}^{\widetilde{\sigma}} , T_{(1,-2)}^{\widetilde{\chi}} \Big] T_{(1,0)}^{\widetilde{\rho}} T_{(1,0)}^{\widetilde{\rho}_1} = c_2\; T_{(1,0)}^{\widetilde{\chi}}  T_{(1,0)}^{\widetilde{\rho}} T_{(1,0)}^{\widetilde{\rho}_1} \]  
if $\sigma = \Norm_{1}^{2}(\chi),$ 
\[ T_{(1,-2)}^{\widetilde{\chi}} \Big[ T_{(0,2)}^{\widetilde{\sigma}} ,T_{(1,0)}^{\widetilde{\rho}} \Big] T_{(1,0)}^{\widetilde{\rho}_1} =  c_2 \; T_{(1,-2)}^{\widetilde{\chi}}  T_{(1,0)}^{\widetilde{\rho}_1} T_{(1,2)}^{\widetilde{\rho}} \] 
if $\sigma = \Norm_{1}^{2}(\rho)$   and $\rho \neq \rho_1$,
\[ T_{(1,-2)}^{\widetilde{\chi}} \Big[ T_{(0,2)}^{\widetilde{\sigma}} ,T_{(1,0)}^{\widetilde{\rho}} \Big] T_{(1,0)}^{\widetilde{\rho}_1} = c_1 c_2 \; T_{(1,-2)}^{\widetilde{\chi}}  \Big( T_{(2,2)}^{\widetilde{\rho}} + \tfrac{(v^{-1}-v)}{2} T_{(1,1)}^{\widetilde{\rho}} T_{(1,1)}^{\widetilde{\rho}}\Big)  \] 
if  $\sigma = \Norm_{1}^{2}(\rho)$  with  $\rho = \rho'$, and
\[ T_{(1,-2)}^{\widetilde{\chi}}  T_{(1,0)}^{\widetilde{\rho}} \Big[ T_{(0,2)}^{\widetilde{\sigma}} ,T_{(1,0)}^{\widetilde{\rho}_1} \Big] = c_2 \; T_{(1,-2)}^{\widetilde{\chi}}  T_{(1,0)}^{\widetilde{\rho}} T_{(1,2)}^{\widetilde{\rho}_1} \]  
 if  $\sigma = \Norm_{1}^{2}(\rho_1).$ 
Hence
 \begin{align*}
 A & = \tfrac{v^8}{2}  \sum_{ \substack{x''\in |X| \\ x''=x+y}} T_{(1,0), x''} T_{(1,0),x'} T_{(1,0), x'} + v^8\; T_{(1,-2), x} T_{(1,0),x'} T_{(1,2), x'+y} \\
 & \quad+\tfrac{v^9}{2}  \sum_{ \substack{w\in |X| \\ 2w =2x'+y}} T_{(1,-2), x} T_{(2,2),w} +  \tfrac{(v^8 - v^{10})}{4} \sum_{ \substack{x_1,x_2\in |X| \\ x_1 +x_2 =2x'+y}}  T_{(1,-2), x} T_{(1,1),x_1} T_{(1,1),x_2} \\
 \end{align*}
Considering all the possibilities for the above sums, we have
  \begin{align*}
 A & = \tfrac{v^8}{2} \E_{(x'',1)}^{(1,0)} \oplus \E_{(x',2)}^{(2,0)} + \tfrac{v^8(1+v^{-2})}{2} \E_{(x'',1)}^{(1,0)} \oplus \E_{(x',1)}^{(1,0)} \oplus \E_{(x',1)}^{(1,0)} 
 + \tfrac{(v^4 +v^2)}{2} \E_{(x,1)}^{(1,-2)} \oplus \E_{(w,1)}^{(2,2)} \smallskip\\
 &\quad + \tfrac{(v^8 + 2v^6 + 2v^4 + v^2)}{2} \E_{(x',1)}^{(1,0)} \oplus \E_{(x',1)}^{(1,0)} \oplus \E_{(x',1)}^{(1,0)} + 
 \tfrac{(v^8+2v^{6})}{2} \E_{(x',1)}^{(1,0)} \oplus \E_{(x',2)}^{(2,0)}   + \tfrac{v^8}{2} \E_{(x',3)}^{(3,0)} \smallskip \\
 & \quad +  \E_{(x,1)}^{(1,-2)} \oplus \E_{(x',1)}^{(1,0)} \oplus \E_{(x'+y,1)}^{(1,2)} + \tfrac{v^2}{2}  \E_{(x,1)}^{(1,-2)} \oplus \E_{(w',2)}^{(2,2)} 
 + \tfrac{(v^2 -v^4)}{2}  \E_{(x,1)}^{(1,-2)} \oplus \E_{(x_1,1)}^{(1,1)}  \oplus \E_{(x_2,1)}^{(1,1)} 
  \end{align*}
where $x''=x+y \neq x'$,  $2w = 2x'+y$, $2w' = 2x'+y$ and  $x_1 \neq x_2$ with $x_1+x_2 = 2x'+ y.$

 Now we turn our attention to calculate $B$. By Theorem \ref{dragosmaintheorem}, we have
 \begin{align*}
\Big[ T_{(0,2)}^{\widetilde{\sigma}} , T_{(1,-2)}^{\widetilde{\chi}}  T_{(2,0)}^{\widetilde{\rho}_2} \Big]
&= \Big[ T_{(0,2)}^{\widetilde{\sigma}} , T_{(1,-2)}^{\widetilde{\chi}} \Big] T_{(2,0)}^{\widetilde{\rho}_2} +  T_{(1,-2)}^{\widetilde{\chi}} \Big[ T_{(0,2)}^{\widetilde{\sigma}} ,T_{(2,0)}^{\widetilde{\rho}_2} \Big] \\
& =c_2  T_{(1,0)}^{\widetilde{\chi}}  T_{(2,0)}^{\widetilde{\rho}_2}  + T_{(1,-2)}^{\widetilde{\chi}}  \Big[ T_{(0,2)}^{\widetilde{\sigma}} ,T_{(2,0)}^{\widetilde{\rho}_2} \Big]
\end{align*}
if $\sigma = \Norm_{1}^{2}(\chi).$ We use Lemma \ref{lemma-commutator} and the algorithm from \cite{roberto-elliptic} to conclude that
\begin{align*}
B &=  \tfrac{v^8}{2} \E_{(x'',1)}^{(1,0)} \oplus \E_{(x',2)}^{(2,0)} + \tfrac{v^8(1-v^{-2})}{2} \E_{(x'',1)}^{(1,0)} \oplus \E_{(x',1)}^{(1,0)} \oplus \E_{(x',1)}^{(1,0)}  
+  \tfrac{v^2}{2}  \E_{(x,1)}^{(1,-2)} \oplus \E_{(w',2)}^{(2,2)}   \smallskip \\
&\quad  +  \tfrac{v^8}{2} \E_{(x',1)}^{(1,0)} \oplus \E_{(x',2)}^{(2,0)} + \tfrac{v^8}{2} \E_{(x',3)}^{(3,0)} 
+ \tfrac{(v^2 -v^4)}{2}  \E_{(x,1)}^{(1,-2)} \oplus \E_{(x_1,1)}^{(1,1)}  \oplus \E_{(x_2,1)}^{(1,1)} \smallskip \\
&\quad  +\tfrac{(v^8 - v^2)}{2} \E_{(x',1)}^{(1,0)} \oplus \E_{(x',1)}^{(1,0)} \oplus \E_{(x',1)}^{(1,0)}  + \tfrac{(v^4 +v^2)}{2} \E_{(x,1)}^{(1,-2)} \oplus \E_{(y',1)}^{(2,2)} + \tfrac{(v^4 - v^2)}{2} \E_{(x,1)}^{(1,-2)} \oplus \E_{(y,1)}^{(2,2)}  
\end{align*}
where $y' \neq y$. Subtracting $B$ from $A$  and multiplying the outcome by $v^{-6}$ completes the proof.
\end{proof}
\end{lemma}

Let $x \in |X|$ and $x' \in |X_m|$ be such that $x'$ lies above $x$, i.e.\ $\pi(x')=x$. We remember that there are $\gcd(m, |x|)$ closed points in $|X_m|$ lying above $x$ and $|x'| = |x|/\gcd(m, |x|).$

\begin{thm}\label{thm-rank3-degree2-graph} 
Let $y \in |X|$ of degree two and $\E :=\E_{(x,1)}^{(1,0)} \oplus \E_{(x,1)}^{(1,0)} \oplus \E_{(x,1)}^{(1,0)}$\; for some closed point $x$ in $X$ of degree one. Then
\[ \mathcal{V}_{y,1}(\E) = \Big\{ \big(\E, \E_{(x-y,1)}^{(1,-2)} \oplus \E_{(x,1)}^{(1,0)}  \oplus \E_{(x,1)}^{(1,0)} , q^2 + q + 1\big),
 \big(\E, \E_{(y,1)}^{(2,-2)}\oplus \E_{(x,1)}^{(1,0)} , q^4 - q\big) \Big\}\]
is the $\Phi_{y,1}$-neighborhood of $\E$ in $\mathscr{G}_{y,1}^{(3)}$.

\begin{proof} Let $X_2 := X \otimes \mathbb{F}_{q^2}$ be the degree two constant extension of $X$. Let $x' \in |X_2|$ lying above $x$, i.e.\ $\pi(x')=x$, where $\pi: X_2 \rightarrow X$ is the constant extension map. Note that $x'$ is unique and $|x'| = 1.$
By \cite[Lemma 6.5]{oliver-graphs} $\pi^{*}$ restricts to an injective map from the set of traces and decomposable bundles on $X$ to the set of decomposable bundles on $X_2$. Moreover, 
$$\pi^{*}\big(\E_{(x,1)}^{(1,0)} \oplus \E_{(x,1)}^{(1,0)} \oplus \E_{(x,1)}^{(1,0)}\big) = \E_{(x',1)}^{(1,0)} \oplus \E_{(x',1)}^{(1,0)} \oplus \E_{(x',1)}^{(1,0)}.$$ 
Suppose $\E, \E' \in \Bun_3 X$ are such that $m_{y,1}(\E,\E') \neq 0$, i.e.\ there exists a short exact sequence
\[ 0 \longrightarrow \E' \longrightarrow \E \longrightarrow \mathcal{K}_{y} \longrightarrow 0. \]
Since extension of constants is an exact functor, we still have an exact sequence when we extend the scalars to $\mathbb{F}_{q^2}$
\begin{equation}\label{sho}
 0 \longrightarrow \pi^{*}(\E') \longrightarrow \pi^{*}(\E) \longrightarrow \mathcal{K}_{z} \oplus \mathcal{K}_{ z'}  \longrightarrow 0
\end{equation}
where  $z, z'$ are the two closed points of $X_2$ lying above $y$.
The above short exact sequence (\ref{sho}) splits into two exact sequences
\[  0 \rightarrow \E'' \rightarrow \pi^{*}(\E) \rightarrow \mathcal{K}_{z} \rightarrow 0 \quad \text{ and } \quad
0 \rightarrow \pi^{*}(\E')  \rightarrow \E'' \rightarrow \mathcal{K}_{z'} \rightarrow 0   \]
where $\E'' \in \Bun_3 X_2$ is the kernel of $\pi^{*}(\E) \rightarrow \mathcal{K}_{z}.$

Hence for every $\big( \E, \E', m \big)$ in $\mathrm{Edge}\; \mathscr{G}_{y,1}^{(3)}$, there are a vector bundle $\E'' \in \Bun_3 X_2$ and edges $\big( \E, \E'', m' \big)$ in $\mathrm{Edge}\; \mathscr{G}_{z,1}^{(3)}$ and
$\big( \E'', \E', m'' \big)$ in $\mathrm{Edge}\; \mathscr{G}_{z',1}^{(3)}$.

From Theorem \ref{maintheorem} item \ref{i17} (or directly Lemma \ref{idN20}) we have only one option for $\E''$ in $\mathscr{G}_{z,1}^{(3)}$, namely $\E'' = \E_{(x'-z,1)}^{(1,-1)} \oplus \E_{(x',1)}^{(1,0)} \oplus \E_{(x',1)}^{(1,0)}.$ Since $z \neq  z'$,  item \ref{i21} of our main theorem (or directly Lemma \ref{idR20}) implies that $\E''$ has two neighbors in $\mathscr{G}_{z',1}^{(3)}$. Therefore $\E$ must have at most two neighbors in $\mathscr{G}_{y,1}^{(3)}$ over $X$.

Lemma \ref{idL2} shows that
\[ m_{y,1}\big(\E, \E_{(x-y,1)}^{(1,-2)} \oplus \E_{(x,1)}^{(1,0)}  \oplus \E_{(x,1)}^{(1,0)} \big) = q^2 +q+1,\]
i.e.\ $\E_{(x-y,1)}^{(1,-2)} \oplus \E_{(x,1)}^{(1,0)}  \oplus \E_{(x,1)}^{(1,0)} $ is one of the two neighbors of $\E$ in $\mathscr{G}_{y,1}^{(3)}.$ It follows from \cite[Prop. 6.4]{oliver-elliptic} that 
$m_{y,1}\big(\E_{(x,1)}^{(1,0)}  \oplus \E_{(x,1)}^{(1,0)},  \E_{(y,1)}^{(2,-2)}\big)$ is non-zero in $\mathscr{G}_{y,1}^{(2)}$. Thus,
\[ \mathcal{V}_{y,1}(\E) = \Big\{ \big(\E, \E_{(x-y,1)}^{(1,-2)} \oplus \E_{(x,1)}^{(1,0)}  \oplus \E_{(x,1)}^{(1,0)} , q^2 +q+1\big),
 \big(\E, \E_{(y,2)}^{(2,-2)}\oplus \E_{(x,1)}^{(1,0)} , m \big) \Big\}\]
for some positive integer $m$.
By Theorem \ref{thm-multi}, the sum over all the multiplicities of edges originated in a vertex of $\mathscr{G}_{y,1}^{(3)}$ equals the number of rational points on the Grassmannian $\mathrm{Gr}(2,3)$ over the residual field of $y$.  Therefore $m = q^4 -q$, which completes the proof.
\end{proof}
\end{thm}


\begin{small}
 \bibliographystyle{alpha}
 \bibliography{elliptic1bib}
\end{small}

\end{document}